\documentclass[reqno,11pt]{amsart}
\usepackage{mathtools}
\usepackage{amsmath}
\usepackage{amssymb}
\usepackage{amsfonts,verbatim,enumitem,leftidx,dsfont}
\usepackage{yhmath}
\usepackage{graphicx}
\usepackage{mathrsfs}
\usepackage{xcolor}
\usepackage{tikz-cd}
\usepackage{tikz}
\usetikzlibrary{patterns}
\usepackage{amsthm}
\usepackage{booktabs,longtable,array}
\usepackage{microtype}
\usepackage{lmodern}
\numberwithin{equation}{section}

\usepackage{amsxtra}
\usepackage[T1]{fontenc} 
\usepackage[utf8]{inputenc}
\usepackage{csquotes}
\usepackage[english]{babel}
\usepackage[linktocpage]{hyperref}

\usetikzlibrary{arrows,matrix,positioning}
\usepackage{bm}

\usepackage{cleveref}
\hypersetup{hidelinks}
\usetikzlibrary{decorations.pathreplacing}

\usepackage{appendix}
\DeclareMathAlphabet{\mathpzc}{OT1}{pzc}{m}{it}

\usepackage{thmtools}
\usepackage{thm-restate}

\usepackage{caption}

\newtheorem{theorem}{Theorem}[section]

\newtheorem{lemma}[theorem]{Lemma}
\newtheorem{corollary}[theorem]{Corollary}

\newtheorem{proposition}[theorem]{Proposition}

\theoremstyle{definition}
\newtheorem{definition}[theorem]{Definition}
\newtheorem{example}[theorem]{Example}

\theoremstyle{remark}
\newtheorem{remark}[theorem]{Remark}



\newcommand{\op}{\operatorname}

\newcommand{\be}{\begin{equation}}
\newcommand{\ee}{\end{equation}}

\newcommand{\rank}{\op{rank}}

\newcommand{\R}{\mathbb R}

\newcommand{\Z}{\mathbb Z}

\newcommand{\La}{\Lambda}

\newcommand{\br}{\mathbb R}
\newcommand{\SO}{\op{SO}}

\newcommand{\vol}{\op{vol}}
\newcommand{\q}{\mathbb Q}

\let\hacek=\v 
\renewcommand{\v}{\mathsf v}
\newcommand{\e}{\varepsilon}
\newcommand{\Leb}{\op{Leb}}

\newcommand{\M}{\op{M}}
\renewcommand{\S}{\mathbb S}

\newcommand{\bp}{B_0^+}

\newcommand{\id}{\op{id}}

\newcommand{\GL}{\op{GL}}

\let\cedilla=\c 

\newcommand{\SL}{\op{SL}}

\newcommand{\z}{\mathbb Z}
\newcommand{\diag}{\op{diag}}

\newcommand{\covol}{\op{covol}}

\renewcommand{\r}{\mathsf r}

\newcommand{\Gr}{\op{Gr}}

\newcolumntype{L}[1]{>{\raggedright\arraybackslash}p{#1}}

\newcommand{\eqlabel}[2]
{
\begin{equation}
{#2}\label{#1}
\end{equation}
}

\allowdisplaybreaks[2]
\begin{document}

\title[Determinant values on lattices]{Determinant values on lattices}
\begin{abstract}
We study the distribution of determinant values on lattices in $\M_n(\mathbb R)$ for $n\ge 2$. Let $\Lambda<\M_n(\mathbb R)$ be a lattice whose elements have algebraic entries. We prove that if $\det(\Lambda)$ is not contained in a scalar multiple of $\mathbb Z$, then for every $a<b$, as $T\to\infty$,
$$ \#\{v\in\Lambda:\|v\|<T,\ a<\det v<b,\ \det v\ne0\}
\sim
\frac{C_n}{\covol(\Lambda)}
(b-a)T^{n(n-1)}
$$
 where $\|\cdot\|$ is the Frobenius norm and $C_n>0$ depends only on $n$.

Under an additional hypothesis on the isotropic subspaces associated with $\Lambda$, which is automatic for $n=2,3$ and is satisfied when  $n\ge4$, by every lattice spanned by scalar multiples of the elementary matrices, we also obtain an asymptotic formula for the singular lattice points. Both conclusions extend to the broader class of Diophantine lattices under the corresponding hypotheses.

For $n=2$, our theorem recovers the Eskin--Margulis--Mozes theorem on the quantitative Oppenheim problem for quadratic forms of signature $(2,2)$; more generally, our results may be viewed as higher-degree analogues of that theorem.
\end{abstract}

\author{Wooyeon Kim}
\address{Korea Institute for
Advanced Study, Seoul, Republic of Korea}
\email{wooyeonkim@kias.re.kr}

\author{Hee Oh}
\address{Department of Mathematics, Yale University, New Haven, CT}
\email{hee.oh@yale.edu}

\maketitle


\section{Introduction}
The distribution of values of a polynomial $F(x_1,\ldots,x_N)$ on lattice
points is a classical problem in number theory. For indefinite quadratic
forms, this circle of questions has been extensively studied as the quantitative Oppenheim problem; see
\cite{Margulis:1987,Margulis:1989,dani-margulis:1993,eskin-margulis-mozes:1998,eskin-margulis-mozes:2005,Kim}.  Although higher-degree forms in many variables are often accessible by the
Hardy--Littlewood circle method, the determinant lies far outside its usual
range; see \Cref{circle}.

In this paper we study the distribution of determinant values on lattices in
$\M_n(\mathbb R)$. A
lattice in $\M_n(\mathbb R)$ means a discrete subgroup of full rank,
equivalently the $\mathbb Z$-span of $n^2$ linearly independent
matrices. Given such a lattice $\Lambda$, our aim is to understand how
$$
\det(\Lambda):=\{\det v:\ v\in\Lambda\}
$$ is distributed in $\mathbb R$.

Fix a norm
$\|\cdot\|$ on $\M_n(\mathbb R)$ that is invariant under left and right
multiplication by $\SO(n)$; the Frobenius norm is the basic example. For a
lattice $\Lambda<\M_n(\mathbb R)$, real numbers $a<b$, and $T>0$, set
$$
N_\Lambda(a,b;T)
:=
\#\{v\in\Lambda:\ \|v\|<T,\ a<\det v<b\},
$$
and
$$
N_\Lambda^\times(a,b;T)
:=
\#\{v\in\Lambda:\ \|v\|<T,\ a<\det v<b,\ \det v\ne0\}.
$$

This distinction is essential when $0\in(a,b)$, since singular lattice
points may contribute on the same scale as the nonsingular main term.

\subsection{The nonsingular counting theorem}

Our first main result gives a precise asymptotic for
$N_\Lambda^\times(a,b;T)$.

\begin{theorem}\label{m1} 
Let $\Lambda<\M_n(\mathbb R)$ be a lattice whose elements have
algebraic entries, and assume that
\begin{equation}\label{irr}
\det(\Lambda)\not\subset\lambda\mathbb Z
\qquad
\text{for every }\lambda\in\mathbb R^\times.
\end{equation}
Then there exists a constant $C_{\|\cdot\|}>0$, depending only on the norm $\|\cdot\|$, such that, for every $a<b$,
\begin{equation}\label{eq:intro-regular-asymptotic}
N_\Lambda^\times(a,b;T)
\sim
\frac{C_{\|\cdot\|}}{\covol(\Lambda)}
(b-a)T^{n(n-1)}
\qquad(T\to\infty).
\end{equation}
\end{theorem}

Let $\vol$ denote Lebesgue measure on $\M_n(\mathbb R)$. All covolumes are computed with respect to $\vol$. The constant
$C_{\|\cdot\|}$ is characterized by the volume asymptotic
\begin{equation}\label{eq:intro-volume-asymptotic}
\operatorname{vol}
\{v\in\M_n(\mathbb R):\ \|v\|<T,\ a<\det v<b\}
\sim
C_{\|\cdot\|}(b-a)T^{n(n-1)},
\end{equation}
proved in \Cref{lem:volume-main-term}. Thus \Cref{m1} may also be
written as
$$
N_\Lambda^\times(a,b;T)
\sim
\frac{1}{\covol(\Lambda)}
\operatorname{vol}
\{v\in\M_n(\mathbb R):\ \|v\|<T,\ a<\det v<b\}.
$$

\subsection{The singular contribution}
The restriction $\det v\ne0$ in \Cref{m1} is necessary. Indeed, for any
fixed proper subspace $\mathsf U<\mathbb R^n$, the determinant vanishes
identically on the space of matrices whose columns lie in $\mathsf U$, and
likewise on the space of matrices whose rows lie in $\mathsf U$. We call
these spaces column- and row-isotropic, respectively. If
$\dim\mathsf U=n-1$, each has dimension $n(n-1)$, precisely the exponent
appearing in \Cref{m1}. Thus a $\mathbb Z$-submodule of
$\Lambda$ of rank $n(n-1)$ contained in such a space may contribute on the same scale as
the nonsingular main term.

The isotropic noncoincidence condition, defined in
\Cref{def:isotropic-noncoincidence}, rules out certain commensurabilities
associated with distinct maximal rational column- or row-isotropic
subspaces. Almost every lattice contains no nonzero singular matrices and
hence satisfies this condition vacuously. Moreover, the condition is
automatic for every lattice when $n=2,3$ and, when $n\ge4$, is
satisfied by every diagonal lattice\footnote{A lattice
$\Lambda<\M_n(\R)$ is called \emph{diagonal} if it is the
$\mathbb Z$-span of nonzero scalar multiples of the elementary matrices.}; see \Cref{ex:diagonal-isotropically-noncoincident}.
\begin{theorem}\label{m11}
Let $\Lambda$ be as in \Cref{m1}. Suppose either that $n\le 3$, or that
 $\Lambda$ satisfies the isotropic noncoincidence condition. Then the
limit
\begin{equation}\label{eq:intro-singular-asymptotic}
c_\Lambda^{\operatorname{sing}}
:=
\lim_{T\to\infty}T^{-n(n-1)}
\#\{v\in\Lambda:\ \|v\|< T,\ \det v=0\}
\end{equation}
exists and is finite. Moreover, $c_\Lambda^{\operatorname{sing}}>0$ if and
only if $\Lambda$ contains a $\mathbb Z$-submodule of rank $n(n-1)$ on which the
determinant vanishes identically.

Consequently, for every $a<b$,
\begin{equation}\label{eq:intro-full-asymptotic}
N_\Lambda(a,b;T)
\sim
\left(
\frac{C_{\|\cdot\|}}{\covol(\Lambda)}(b-a)
+c_\Lambda^{\operatorname{sing}}\mathbf 1_{\{0\in(a,b)\}}
\right)T^{n(n-1)}.
\end{equation}
\end{theorem}

The irrationality hypothesis \eqref{irr} in \Cref{m1} is easy to test: it is
equivalent to the existence of $v,w\in\Lambda$ such that $\det w\ne0$ and
$$
\frac{\det v}{\det w}\notin\mathbb Q.
$$
For example, when $n=3$, both theorems apply to
\begin{equation}\label{ex1}
\Lambda_{\sqrt2}
:=
\begin{pmatrix}
\sqrt2\,\mathbb Z&\mathbb Z&\mathbb Z\\
\mathbb Z&\mathbb Z&\mathbb Z\\
\mathbb Z&\mathbb Z&\mathbb Z
\end{pmatrix},
\end{equation}
since $1,\sqrt2\in\det(\Lambda_{\sqrt2})$. The matrices with first row zero
form a singular submodule of rank $6$, so
$c_{\Lambda_{\sqrt2}}^{\operatorname{sing}}>0$.
Moreover, $\Lambda_{\sqrt2}$ contains infinitely many submodules of rank $6$
on which the determinant vanishes identically. Indeed, writing
$v_1,v_2,v_3$ for the columns of $v\in\Lambda_{\sqrt2}$, each primitive
pair $(p,q)\in\mathbb Z^2$, considered up to sign, gives a distinct such
submodule:
$$
\{v\in\Lambda_{\sqrt2}:pv_2+qv_3=0\}.
$$
Thus the finite constant $c_{\Lambda_{\sqrt2}}^{\operatorname{sing}}$
may receive contributions
from infinitely many singular submodules of rank-$n(n-1)$.

\subsection{Diophantine lattices} The algebraicity assumption on $\Lambda$ in the preceding theorems can be
replaced by a Diophantine condition on rational Pl\"ucker configurations;
see \Cref{lattice_Diophantine,main}. Informally, this condition requires
polynomial separation from the exceptional isotropic summands unless the
configuration lies exactly in one of them.

The critical behavior of exact isotropic contributions is illustrated by
Katznelson's estimate \cite{katznelson:1994}: the number of 
matrices of rank $n-1$ in $\M_n(\mathbb Z)$ of norm at most $T$ is of order
$$
T^{n(n-1)}\log T.
$$
The logarithmic factor reflects the failure of summability at the critical scale
$T^{n(n-1)}$ in the split rational lattice. Exact isotropic
subspaces, which may occur in infinite families, are treated separately in
\Cref{thm:quasinullcontribution2}.

Without a Diophantine hypothesis, we still obtain the optimal growth
exponent $n(n-1)$, with a loss of $T^\varepsilon$, for every
non-determinant-rational lattice; see
\Cref{thm:general-irrational-bounds}.

\subsection{The arithmetic obstruction and determinant forms}

We say that a lattice $\Lambda<\M_n(\mathbb R)$ is
\emph{determinant-rational} if
$$
\det(\Lambda)\subset\lambda\mathbb Q
\qquad
\text{for some }\lambda\in\mathbb R^\times,
$$
or, equivalently,
$
\det(\Lambda)\subset\lambda\mathbb Z$
for some $\lambda\in\mathbb R^\times
$ (see \Cref{thm:det-rational-characterization}).
Thus the hypothesis in \Cref{m1} is precisely that $\Lambda$ is not
determinant-rational.
See \Cref{s:det-rational} for an algebraic description of all determinant-rational lattices.

The same arithmetic obstruction has a dynamical interpretation. After
rescaling, we may assume that $\Lambda$ is a unimodular lattice and hence regard it  as a
point of
$$
X=\SL_{n^2}(\mathbb R)/\SL_{n^2}(\mathbb Z).
$$
The left-right action
$$
(g,h)\cdot v=g v h^{\intercal},
\qquad (g,h\in\SL_n(\mathbb R), v\in \M_n(\br))
$$
induces the tensor-product representation
$(g,h)\mapsto g\otimes h$ on
$\M_n(\mathbb R)\simeq\mathbb R^n\otimes\mathbb R^n$. We denote its image by
$$
H:=\SL_n(\mathbb R)\otimes\SL_n(\mathbb R)<G:=\SL_{n^2}(\mathbb R).
$$
 Ratner's orbit-closure theorem \cite{ratner:1991}, together with the
maximality of $H$, gives a closed-or-dense dichotomy (\Cref{rat}); we show that the closed case is precisely the
determinant-rational case. This yields the following:
\begin{theorem}\label{prop:intro-density}
If $\Lambda<\M_n(\mathbb R)$ is not determinant-rational, then
$$
\overline{\det(\Lambda)}=\mathbb R.
$$
\end{theorem}

We next reformulate \Cref{m1} in terms of determinant forms. Choose a
$\mathbb Z$-basis $\mathcal B=\{v_{ij}:1\le i,j\le n\}$ of $\Lambda$ and
set
$$
F_{\Lambda,\mathcal B}(x)
:=
\det\!\left(\sum_{i,j}x_{ij}v_{ij}\right),
\qquad x=(x_{ij})\in\M_n(\mathbb R).
$$
We call $F_{\Lambda,\mathcal B}$ a \emph{determinant form}. This gives the following equivalent formulation of \Cref{m1}. 
\begin{theorem}\label{m2}
Let $F$ be a determinant form on $\M_n(\mathbb R)$ with algebraic
coefficients, and suppose that $F$ is not proportional to a polynomial with
rational coefficients. Given any norm $\|\cdot \|$ on $\M_n(\br)$ invariant under a maximal compact subgroup
of the stabilizer of $F$ in $\SL_{n^2}(\mathbb R)$,  there exists
$c_F=c(F,\|\cdot\|)>0$ such that, for every $a<b$,
$$
\#\{x\in\M_n(\mathbb Z):\ \|x\|<T,\ a<F(x)<b,\ F(x)\ne0\}
\sim
c_F(b-a)T^{n(n-1)}.
$$
\end{theorem}

The same conclusion holds for every Diophantine determinant form that is not
proportional to a rational polynomial; see \Cref{m2general}.

\subsection{Relation to the quantitative Oppenheim problem} When $n=2$, the determinant $\det\begin{pmatrix}x&y\\ z&w\end{pmatrix}=xw-yz $ is a quadratic form of signature
$(2,2)$, so the corresponding counting problem is a special case of the
quantitative Oppenheim theorem of Eskin--Margulis--Mozes
\cite{eskin-margulis-mozes:2005}, building on
\cite{dani-margulis:1993,eskin-margulis-mozes:1998}; see also \cite{Kim}
for signature $(2,1)$. Thus \Cref{m1,m11} may be viewed as
higher-degree analogues of that theorem.

\subsection{Comparison with the circle method}
\label{circle}
Even for the standard determinant polynomial, which has rational
coefficients, the codimension of its singular locus is too small relative
to its degree for Birch's theorem to apply. Indeed, for a rational
homogeneous form $P$ of degree $k$ in $N$ variables, Birch's theorem
requires
$$
N-\dim V_P^*>(k-1)2^k,
\qquad
V_P^*:=\{x\in\mathbb C^N:\nabla P(x)=0\};
$$
see \cite{birch:1962}. The determinant on $\M_n(\mathbb C)$ has degree
$n$ in $N=n^2$ variables, while
$$
V_{\det}^*
=
\{X\in\M_n(\mathbb C):\operatorname{rank}X\le n-2\}
$$
has dimension $N-4$. Thus the left-hand side of Birch's condition is $4$:
it equals the threshold $(n-1)2^n$ when $n=2$, so the required strict
inequality already fails, and it lies far below the threshold when
$n\ge3$.

The determinant forms considered here lie even further outside
this classical setting: their coefficients cannot all be made rational by a common rescaling, and
our problem concerns a Diophantine inequality rather than an equation.
General results for irrational inequalities, based on the
Davenport--Heilbronn method, require substantially more variables and
additional geometric hypotheses; see
\cite{RydinMyersonThesis,BrowningDH}.

\subsection{Main ideas and new difficulties for $n\ge3$}

We now outline the proof, emphasizing the points where the determinant
problem differs from the quadratic Oppenheim problem. We retain the notation
$G$, $X$, and $H$ introduced above, and regard the normalized lattice
$\Lambda$ as a point of $X$. If $\Lambda$ is not determinant-rational, then
its $H$-orbit is dense by the orbit-closure discussion above. Shah's equidistribution theorem for expanding translates of maximal compact orbits in $H$ \cite{shah} therefore gives equidistribution for compactly supported continuous functions on $X$.

To pass from orbit equidistribution to counting, we use Siegel transforms.
For $f\in C_c(\M_n(\mathbb R)\smallsetminus\{0\})$, its Siegel
transform is
\begin{equation}\label{eq:Siegeltransform}
    \widehat f(\Delta)
:=
\sum_{v\in\Delta\smallsetminus \{0\}}f(v),
\qquad
\Delta\in X.
\end{equation}
When $f$ approximates the indicator of a region, $\widehat f$ represents the
corresponding lattice-point count. The difficulty is that $\widehat f$ is unbounded on $X$.
We therefore
modify the sum in \eqref{eq:Siegeltransform} by retaining only lattice points outside the union of
the proper rational isotropic subspaces and introduce a height adapted
to this restriction. A uniform $L^{1+\theta}$-bound for the modified height
controls the resulting modified Siegel transform. The omitted isotropic
lattice points are counted separately.

The cusp is controlled by the Margulis $\alpha$-function. If $V$ is an
$r$-dimensional $\Delta$-rational subspace where $1\le r\le n^2-1$, then its Pl\"ucker vector
$\mathsf w_{\Delta,V}\in\wedge^r\M_n(\mathbb R)$
has norm equal to the covolume of $\Delta\cap V$ in $V$. The $\alpha$-function is given by
$$
\alpha(\Delta)
=
\max_{1\le r\le n^2-1}
\sup_V
\|\mathsf w_{\Delta,V}\|^{-1}.
$$
 Thus all exterior powers of the matrix space enter the cusp analysis.

The relevant diagonal ray is dictated by the singular-value geometry of a
bounded determinant window.  In the main compact singular-value
regime, $n-1$ singular values have a common large scale, while the remaining
one compensates to keep the determinant bounded. The determinant-preserving
normalization of such a  configuration is
$$
b_t=\operatorname{diag}(e^{-t},\ldots,e^{-t},e^{(n-1)t}),
\quad
a_t=b_t\otimes b_t\in H,\quad  T= e^{2t} .
$$
Let $K=\SO(n)\otimes\SO(n)$ and let $dk$ denote its  Haar probability measure. The $a_tK$-averages are the expanding translates to which Shah's
theorem applies. The same principal ray governs the height estimates and
produces the main term. 
Here the first major difference from $n=2$ appears. In that case, the relevant exterior-power representations have, up to scale, only the three weights $-1,0$ and $1$.
 For $n\ge3$, the skew
Cauchy formula gives
$$
\wedge^r\M_n(\mathbb R)
\simeq
\bigoplus_{\substack{\lambda\vdash r,\, 
\lambda\subseteq n\times n}}
{\mathsf S}_\lambda(\mathbb R^n)\otimes
{\mathsf S}_{\lambda^{\intercal}}(\mathbb R^n),
$$
where ${\mathsf S}_\lambda$ is the Schur module associated with $\lambda$
and $\lambda^{\intercal}$ is its conjugate partition.  The possible $b_t$-weights on a Schur factor are 
$-|\lambda|+ni$, with $\lambda_n\le i\le\lambda_1$. Thus a single exterior power may contain many irreducible $H$-summands with long weight
strings.

In each critical degree $kn$ where
$1\le k\le n-1$, there are two exceptional summands, namely
$
{\mathsf S}_{(n^k)}(\mathbb R^n)\otimes\mathbf 1$ and
$\mathbf 1\otimes {\mathsf S}_{(n^k)}(\mathbb R^n).
$
The nonzero decomposable vectors in these summands are, respectively, the Pl\"ucker vectors of
column- and row-isotropic $kn$-planes. 
All other irreducible summands are called \emph{nonexceptional}.

For $0<\eta<1$, $M>1$, and $1\le k\le n-1$, a $kn$-dimensional $\Delta$-rational subspace
$V$ is called \emph{$(\eta,M)$-quasi-null} if its Pl\"ucker vector lies within distance
$\eta\|\mathsf w_{\Delta,V}\|^{-M}$ of one of the exceptional summands.
If $\Delta$ is $(\eta,M)$-Diophantine, its $(\eta,M)$-quasi-null
rational subspaces are precisely its rational isotropic subspaces. 
With this terminology, the main new difficulties are the following.

\begin{enumerate}[label=\textnormal{(\arabic*)}]

\item \emph{Higher-degree sublevel estimates.}
Let  $u_\xi$ belong to the expanding horospherical subgroup of $\SL_n(\br)$ associated with $b_t$.
On a nonexceptional Schur summand, the first nonnegative-weight coordinate
of $u_\xi v$ is a
vector-valued polynomial in $\xi\in\mathbb R^{n-1}$. For $n=2$,
the corresponding polynomial is at most affine-linear, and the required
sublevel estimates are elementary. For $n\ge3$, it may have higher degree,
and its order of vanishing may vary from one zero to another.

The proof combines algebraic
and analytic inputs. We use weight estimates to obtain lower bounds for the complex codimensions
of the zero loci of the homogeneous polynomials associated with weight
vectors, and hence for the corresponding complex singularity exponents. We then combine Demailly--Koll\'ar's stability theorem for complex
singularity exponents \cite{DemaillyKollar2001} with Cutkosky's relative
monomialization theorem \cite{Cutkosky2017} to obtain uniform real
negative-moment estimates for the full horospherical polynomials, and
hence the required local contraction inequalities.

\item \emph{A uniform $L^{1+\theta}$-bound for the modified height.} The ordinary $\alpha$-function is insufficient because the local
heights on the exceptional summands have critical exponent $1$. For $h\in H$,
define
$$
\widehat\alpha_{\eta,M}(h;\Lambda)
:=\max\left\{ 1, \sup_V\|h\mathsf w_{\Lambda,V}\|^{-1}
\right\},
$$
where the supremum is over all nonzero proper $\Lambda$-rational subspaces
$V$ that are not contained in an $(\eta,M)$-quasi-null subspace. 

For the parameters $\eta,M$ supplied
by the Diophantine hypothesis on $\Lambda$, there exists $\theta>0$ such that
$$
\sup_{t\ge0}
\int_K
\widehat\alpha_{\eta,M}(a_tk;\Lambda)^{1+\theta}\,dk
<\infty.
$$
Following the avoidance-and-iteration framework of \cite{Kim}, we introduce
an auxiliary modified height and establish a global Margulis inequality for
it, with a logarithmic loss, outside an explicit exceptional set. The local
estimates in (1) handle the case in which only
a few short rational
subspaces are relevant; when many such subspaces compete, we use the intersection--sum
mechanism of Eskin--Margulis--Mozes
\cite{eskin-margulis-mozes:1998}. Since the usual intersection--sum
inequality controls covolumes but not distances from the exceptional
summands, we supplement it with the Mother Inequality of Benoist--Quint
\cite{benoist-quint:2012}.
Here the exceptional set consists of configurations in which  a rational Pl\"ucker vector
that is not quasi-null for the base lattice is moved close to an
exceptional summand. Quantitative nondivergence \cite{kleinbock:2008},
together with the Diophantine condition, shows that this set has small measure. The
downward-closed definition of the modified height ensures that the auxiliary
height majorizes $\widehat\alpha_{\eta,M}$. We first establish
contraction and avoidance for averages over the expanding horospherical
subgroup of $H$ associated with $a_t$. Iteration, with step sizes chosen to
absorb the logarithmic loss, yields a uniform horospherical moment bound.
The argument of \cite[Section~7.2]{Kim} transfers this bound to the
$K$-averages, permitting us to truncate the modified Siegel transform,
apply Shah's theorem, and remove the truncation. 

\item \emph{Infinitely many exact isotropic contributions.}
The isotropic lattice points omitted from the modified Siegel transform must
be counted separately. For $n=2$, a non-determinant-rational lattice has
only finitely many maximal rational isotropic planes. For $n\ge3$, even a
non-determinant-rational algebraic lattice may have infinitely many
top-dimensional rational isotropic subspaces, each potentially contributing
on the same $T^{n(n-1)}$ scale as the nonsingular main term.

We show that these subspaces can nevertheless be organized into finitely
many families. Within each family, the singular counting problem reduces to
counting bounded-rank matrices in a rational lattice. We use weighted asymptotic formulas and summable majorants to sum these
contributions, and combine same-type and mixed column--row overlap
estimates with the isotropic noncoincidence condition to rule out
additional contributions of main-term order. This yields the singular constant in
\eqref{eq:intro-full-asymptotic}.

\item \emph{From the principal-ray dynamics to norm balls.}
Finally, we pass from the dynamical statement along $a_t$ to the original
norm-ball counting problem. The invariance of the norm under left and right
multiplication by $\SO(n)$ allows us to use singular-value coordinates.
After factoring out the common scale of the $n-1$ large singular
values, their ratios serve as parameters in the fiber kernels. On compact ratio windows, the
modified Siegel-transform limit applies uniformly, while the unbalanced
regions are controlled by weighted shell estimates. The same fiber
integrals identify the constant $C_{\|\cdot\|}$ in
\eqref{eq:intro-volume-asymptotic} and complete the proof of the counting
asymptotic.

\end{enumerate}

In a separate paper \cite{KimOh_two}, we give a self-contained treatment of the case $n=2$, which also serves as a model for the present argument.

\subsection*{Organization of the paper}

\Cref{s:det-rational} classifies determinant-rational lattices and proves
the closed-or-dense orbit dichotomy. \Cref{s:height-main} develops the height
formalism and the Schur-functor decomposition, introduces  Diophantine lattices,
quasi-null subspaces, and isotropic subspaces, and states the main technical results. \Cref{s:sublevel-general-n,s:weightedlocalestimates,s:globalheight} establish the local and global height estimates, starting from uniform sublevel bounds and passing through weighted estimates for the $H$-action to lattice-height inequalities. \Cref{s:avoidance,s:iterations} prove avoidance and combine it with contraction in an iteration argument to obtain the required uniform integrability. \Cref{s:quasinull,s:isotropic} analyze the singular directions, first describing
the structure of rational isotropic subspaces and then proving the bounded-rank
estimates and the singular asymptotic. \Cref{s:fiberintegral,s:maincounting}
convert the dynamical estimates into counting results by establishing the relevant
fiber-integral identities and proving equidistribution for modified Siegel
transforms. \Cref{s:example} proves that algebraic lattices and algebraic
determinant forms are Diophantine, gives examples with  positive
singular constants, and shows that all diagonal lattices satisfy the
isotropic noncoincidence condition.
Finally, \Cref{app:analytic-stability} establishes the uniform negative-moment estimate used in \Cref{s:sublevel-general-n}, using the stability of complex singularity exponents and relative monomialization.

\subsection*{Acknowledgments}
 The authors  thank the Korea Institute for Advanced Study and Yale University for their hospitality;
a substantial part of this work was completed during visits to
these institutions. Part of this research was conducted while the authors were visiting the Simons Laufer Mathematical Sciences Institute (SLMath),  which is supported by NSF grant DMS-2424139. W.K. is supported
by KIAS individual grant HP101301. H.O. is partially supported by NSF grant DMS-2450703.

\subsection*{AI disclosure}
The authors used AI tools to assist with literature searches, checking proofs, and editing the manuscript. All mathematical ideas and contributions are due to the authors.
\section{Determinant-rational lattices and the orbit-closure dichotomy}
\label{s:det-rational}
In this section, we characterize determinant-rational lattices both arithmetically and dynamically. We first classify their determinant forms over $\mathbb Q$ and then prove the corresponding closed-or-dense orbit dichotomy for the left--right action of $\SL_n(\mathbb R)\times \SL_n(\mathbb R)$.

Let $n\ge2$ and set $N=n^2$. We identify $\M_n(\mathbb R)$ with
$\mathbb R^N$ via the linear isomorphism determined by
$$
E_{ij}\longleftrightarrow e_{(i-1)n+j},
\qquad 1\le i,j\le n.
$$
Equivalently, under
$\mathbb R^N\simeq\mathbb R^n\otimes\mathbb R^n$, we identify
$E_{ij}$ with $e_i\otimes e_j$.
Under this identification, the determinant is a homogeneous
polynomial of degree $n$ on $\mathbb R^N$.

The group $\GL_n(\mathbb R)\times \GL_n(\mathbb R)$ acts on
$\M_n(\mathbb R)$ by
\begin{equation}\label{gla}
    (g_1,g_2)\cdot v=g_1vg_2^{\intercal}\qquad(\text{$g_1, g_2\in \GL_n(\br), v\in \M_n(\br)$}).
\end{equation}

Given a lattice $\Lambda<\M_n(\mathbb R)$, choose a $\mathbb Z$-basis
$\mathcal B=\{v_{ij}\}$ of $\Lambda$. We associate to $\Lambda$ the
determinant form
$$
    F_\Lambda(x)=
    \det\!\Bigl(\sum_{i,j}x_{ij}v_{ij}\Bigr),
    \qquad x=(x_{ij})\in\M_n(\mathbb R).
$$
This polynomial is well-defined up to composition with an element of
$\GL_N(\mathbb Z)$. 

\begin{definition}
    We say that a lattice $\Lambda<\M_n(\br)$ is
\emph{determinant-rational} if
$$ \det(\Lambda)=F_\Lambda(\mathbb Z^N)
    \subset \lambda\mathbb Q
    \qquad
    \text{for some }\lambda\in\mathbb R^\times .
$$
\end{definition}
\subsection{Determinant forms over $\mathbb Q$}

We begin with an algebraic description of rational forms of the determinant.
A quadratic \'etale algebra over $\mathbb Q$ is either a quadratic field
extension of $\mathbb Q$ or the split algebra
$\mathbb Q\times\mathbb Q$. Denote its nontrivial
$\mathbb Q$-automorphism by $z\mapsto \overline z$.
\begin{definition}[Determinant algebras]
\label{def:determinant-algebra}
A \emph{determinant algebra of degree $n$ over $\mathbb Q$} is a triple
$$
    (\mathscr{K},\mathscr{A},\tau),
$$
satisfying one of the following two set of conditions:

\begin{enumerate}
\item $\mathscr{K}/\mathbb Q$ is a quadratic field extension, $\mathscr{A}$ is a central
simple $\mathscr{K}$-algebra of degree $n$, and $\tau$ is a unitary involution on
$\mathscr{A}$, i.e. a $\mathbb Q$-linear anti-automorphism satisfying
$$
    \tau^2=\mathrm{id},
    \qquad
    \tau(zx)=\overline z\,\tau(x)
    \quad (z\in \mathscr{K},\ x\in \mathscr{A}).
$$

\item $\mathscr{K}=\mathbb Q\times\mathbb Q$, $\mathscr{A}=\mathscr{A}_0\times \mathscr{A}_0^{\mathrm{op}}$ for
some central simple $\mathbb Q$-algebra $\mathscr{A}_0$ of degree $n$, and
$$
    \tau(x,y)=(y,x).
$$
\end{enumerate}

We write
$$
    J(\mathscr{A},\tau):=\{x\in \mathscr{A}:\tau(x)=x\}.
$$
 For $x\in J(\mathscr{A},\tau)$, define
$$
    N_{\mathscr{A},\tau}(x):=
    \begin{cases}
    \operatorname{Nrd}_{\mathscr{A}/\mathscr{K}}(x),&
    \text{in case (1)} \\
    \operatorname{Nrd}_{\mathscr{A}_0/\mathbb Q}(x_0),&
    \text{in case (2), where }x=(x_0,x_0)\in
    \mathscr{A}_0\times \mathscr{A}_0^{\mathrm{op}}.
    \end{cases}
$$
Then $N_{\mathscr{A},\tau}$ is a homogeneous polynomial of degree $n$, defined
over $\mathbb Q$, on the $n^2$-dimensional vector space $J(\mathscr{A},\tau)$.
\end{definition}
Case (1) gives the outer determinant forms, while the case (2) gives the inner forms, namely the usual
reduced-norm forms on central simple $\mathbb Q$-algebras.

We say that a determinant algebra $(\mathscr{K},\mathscr{A},\tau)$ is \emph{split over
$\mathbb R$} if there is an $\mathbb R$-linear isomorphism
$\iota:J(\mathscr{A},\tau)\otimes_{\mathbb Q}\mathbb R
    \xrightarrow{\sim}
    \M_n(\mathbb R)$
such that $ N_{\mathscr{A},\tau}(x)=\det(\iota(x))$.
Only determinant algebras split over $\mathbb R$ will occur below. An order in $\mathscr{A}$ means a subring $\mathcal O$ that is a full
$\mathbb Z$-lattice in the underlying $\mathbb Q$-vector space. If
$\tau(\mathcal O)=\mathcal O$, we write
$ \mathcal O^\tau:=\mathcal O\cap J(\mathscr{A},\tau)$, which
 is a full lattice in $J(\mathscr{A},\tau)$.

The following Frobenius--Dieudonn\'e theorem describes the linear maps
preserving the determinant up to a scalar; see \cite[Theorem~2]{MarcusMoyls}.

\begin{lemma}
\label{Fro}
Let $L$ be a field of characteristic zero, and let
$T\in\GL(\M_n(L))$ satisfy
$$
    \det(Tv)=c\,\det(v)
    \qquad
    (v\in\M_n(L))
$$
for some $c\in L^\times$. Then either
$$
    Tv=g_1vg_2
    \qquad
    (v\in\M_n(L))\quad\text{or}\quad 
    Tv=g_1v^{\intercal}g_2
    \qquad
    (v\in\M_n(L)),
$$
for some $g_1,g_2\in\GL_n(L)$ satisfying $\det(g_1)\det(g_2)=c$.
\end{lemma}

We use the following elementary uniqueness statement.

\begin{lemma}
\label{commute}
If $g_1,g_2\in\GL_n(L)$ satisfy
$$
    g_1vg_2=v
    \qquad
    \text{for all }v\in\M_n(L),
$$
then there exists $c\in L^\times$ such that
$$
    g_1=cI_n,
    \qquad
    g_2=c^{-1}I_n.
$$
\end{lemma}

\begin{proof}
Taking $v=I_n$, we obtain $g_1g_2=I_n$, so $g_2=g_1^{-1}$. Hence
$g_1vg_1^{-1}=v$ for every $v\in\M_n(L)$. Therefore $g_1$ lies in
the center of $\M_n(L)$, and so $g_1=cI_n$ for some $c\in L^\times$.
Then $g_2=c^{-1}I_n$.
\end{proof}

We also record a simple rationality criterion for polynomials.

\begin{lemma}
\label{polynomial}
Let $p\in\mathbb R[x_1,\dots,x_\ell]$. If $p(\mathbb Z^\ell)\subset\mathbb Q$,
then $p\in\mathbb Q[x_1,\dots,x_\ell]$.
\end{lemma}

\begin{proof}
We first consider the case $\ell=1$. Let $p\in\mathbb R[x]$ have degree
$m$. Let
$$
    \Delta p(x):=p(x+1)-p(x).
$$
 Then $\Delta^m p$ is the constant $m!a_m$, where $a_m$ is the leading coefficient of $p$. Moreover,
each $\Delta^k p(0)$ is an integer linear combination of
$p(0),\dots,p(k)$. Since $p(\mathbb Z)\subset\mathbb Q$, all these
finite differences are rational. It follows that every coefficient of
$p$ is rational.

For general $\ell$, we argue by induction. Write
$$
    p(x_1,\dots,x_\ell)
    =
    \sum_{j=0}^m p_j(x_1,\dots,x_{\ell-1})x_\ell^j.
$$
For each integer $r$, the polynomial $p(\cdot,r)$ takes rational values
on $\mathbb Z^{\ell-1}$. By the induction hypothesis, it has rational
coefficients. Evaluating at $m+1$ distinct integers $r$ and inverting
the corresponding Vandermonde matrix, we conclude that each $p_j$ has
rational coefficients.
\end{proof}

We use the following standard classification of twisted forms of the
determinant \cite{WaterhouseTwistedDeterminant}.

\begin{lemma}[Twisted determinant forms]
\label{lem:twisted-determinant-forms}
Let $F\in\mathbb Q[x_1,\dots,x_N]$ be a homogeneous polynomial of degree
$n$. Suppose that $F$ becomes linearly equivalent to $\det$ over
$\overline{\mathbb Q}$. Then there exist a determinant algebra
$(\mathscr{K},\mathscr{A},\tau)$ of degree $n$ over $\mathbb Q$, a
$\mathbb Q$-basis $\{w_{ij}\}$ of $J(\mathscr{A},\tau)$, and $c\in\mathbb Q^\times$ such that
$$ F(x)
    =
    c\, N_{\mathscr{A},\tau}\!\left(\sum_{i,j}x_{ij}w_{ij}\right).
$$
 Conversely, every polynomial of this
form becomes linearly equivalent to the determinant over
$\overline{\mathbb Q}$.

Moreover, $F$ is linearly equivalent over $\mathbb R$ to the ordinary
determinant if and only if $(\mathscr{K},\mathscr{A},\tau)$ is split over $\mathbb R$.
\end{lemma}

\begin{proof}
This is the classification of twisted forms of the determinant due to
Waterhouse; see \cite{WaterhouseTwistedDeterminant}. We recall the descent
mechanism in order to fix the notation.
Choose $g\in\GL_N(\overline{\mathbb Q})$ such that
$F=\det\circ g$.
For each $\sigma\in\operatorname{Gal}(\overline{\mathbb Q}/\mathbb Q)$,
put
$$
z_\sigma:=g\,\sigma(g)^{-1}.
$$
Since $F$ is defined over $\mathbb Q$, the element $z_\sigma$ stabilizes
the determinant. Moreover, $(z_\sigma)$ is a Galois $1$-cocycle with
values in the stabilizer of the determinant. By \Cref{Fro}, this
stabilizer has two connected components, consisting respectively of the
left-right transformations and the transpose-left-right transformations.
Hence, by passage to the component group, $(z_\sigma)$ induces a character
$$
\chi:\operatorname{Gal}(\overline{\mathbb Q}/\mathbb Q)
\to\mathbb Z/2\mathbb Z.
$$
The character $\chi$ determines a quadratic \'etale algebra
$\mathscr K/\mathbb Q$: if $\chi$ is nontrivial, $\mathscr K$ is the
quadratic field fixed by $\ker\chi$, while if $\chi$ is trivial, then
$\mathscr K=\mathbb Q\times\mathbb Q$.

After base change to $\mathscr K$, the component obstruction disappears,
so the descent datum takes values in the identity component of the determinant
stabilizer. Waterhouse's classification identifies such a descent datum
with a central simple $\mathscr K$-algebra $\mathscr A$ of degree $n$.
The remaining descent from $\mathscr K$ to $\mathbb Q$, in which the
nontrivial automorphism of $\mathscr K/\mathbb Q$ is accompanied by the
transpose component, gives a unitary involution $\tau$ on $\mathscr A$.
The descended $\mathbb Q$-vector space is
$
J(\mathscr A,\tau)
=
\{a\in\mathscr A:\tau(a)=a\},
$
and the descended determinant form is, up to a nonzero scalar, the
restriction of the reduced norm:
$$
N_{\mathscr A,\tau}
=
\left.
\operatorname{Nrd}_{\mathscr A/\mathscr K}
\right|_{J(\mathscr A,\tau)}.
$$
In particular, $\dim_{\mathbb Q}J(\mathscr A,\tau)=n^2=N$. Choosing the
$\mathbb Q$-basis $\{w_{ij}\}$ of $J(\mathscr A,\tau)$ corresponding,
under the descended linear identification, to the standard basis of
$\mathbb Q^N$, we obtain
$$
F(x)
=
c\,N_{\mathscr A,\tau}
\left(\sum_{i,j}x_{ij}w_{ij}\right)
$$
for some $c\in\overline{\mathbb Q}^{\times}$. Since both sides are defined
over $\mathbb Q$, the scalar $c$ is Galois invariant, and hence
$c\in\mathbb Q^\times$.

Conversely, let $(\mathscr K,\mathscr A,\tau)$ be a determinant algebra
of degree $n$ over $\mathbb Q$. After extension of scalars to
$\overline{\mathbb Q}$, the quadratic \'etale algebra $\mathscr K$ splits,
the central simple algebra $\mathscr A$ becomes a matrix algebra, and the
unitary involution becomes the standard transpose descent datum. Consequently,
$$
J(\mathscr A,\tau)\otimes_{\mathbb Q}\overline{\mathbb Q}
\simeq
\M_n(\overline{\mathbb Q}),
$$
under an identification for which $N_{\mathscr A,\tau}$ becomes the
ordinary determinant. It follows that every polynomial of the stated form
is linearly equivalent to $\det$ over $\overline{\mathbb Q}$.

Finally, applying the same classification after extension of scalars to
$\mathbb R$, the corresponding twisted determinant form is linearly
equivalent over $\mathbb R$ to the ordinary determinant precisely when
the determinant algebra $(\mathscr K,\mathscr A,\tau)$ is split over
$\mathbb R$. This proves the last assertion.
\end{proof}

 We also use the following terminology. A determinant form $F$ is called \emph{$\mathbb Q$-split} if it is proportional to
$\det\circ T$ for some $T\in\GL_N(\mathbb Q)$. 
We say that a lattice $\Lambda$ is of
$\mathbb Q$-split type if there exist
$g_1,g_2\in\GL_n(\mathbb R)$ such that
$$
\Lambda
\quad\text{is commensurable with}\quad
g_1\M_n(\mathbb Z)g_2.
$$
We now give the arithmetic characterizations of determinant-rational
lattices.

\begin{theorem}
\label{thm:det-rational-characterization}
Let $\Lambda<\M_n(\mathbb R)$ be a lattice. The following are
equivalent:
\begin{enumerate}
\item $\Lambda$ is determinant-rational;
\item $\det(\Lambda)\subset \lambda\mathbb Z$
    for some $\lambda\in\mathbb R^\times$;
\item $F_\Lambda$ is proportional to a polynomial with rational
coefficients;
\item there exist a determinant algebra $(\mathscr{K},\mathscr{A},\tau)$ of degree $n$ over
$\mathbb Q$, split over $\mathbb R$, a $\tau$-stable order
$\mathcal O\subset \mathscr{A}$, an $\mathbb R$-linear isomorphism
$$
    \iota:J(\mathscr{A},\tau)\otimes_{\mathbb Q}\mathbb R
    \xrightarrow{\sim}
    \M_n(\mathbb R),
$$ satisfying
$$
\det(\iota(x))=N_{\mathscr A,\tau}(x)
\qquad
\text{for all }x\in
J(\mathscr A,\tau)\otimes_{\mathbb Q}\mathbb R,
$$
and $g_1,g_2\in\GL_n(\mathbb R)$, such that $\Lambda$ is commensurable
with
$$
    g_1\,\iota(\mathcal O^\tau)\,g_2;
$$
\item there exist a determinant algebra $(\mathscr{K},\mathscr{A},\tau)$ of degree $n$ over
$\mathbb Q$, split over $\mathbb R$, a $\mathbb Q$-basis
$\{w_{ij}\}$ of $J(\mathscr{A},\tau)$,  and $c\in\mathbb R^\times$ such that
$$
    F_\Lambda(x)
    =
    c\,
    N_{\mathscr{A},\tau}\!\left(\sum_{i,j}x_{ij}w_{ij}\right) .
$$

\end{enumerate}
Moreover, $\Lambda$ is of $\mathbb Q$-split type if and only if $F_\La$ is $\mathbb Q$-split.
\end{theorem}

\begin{proof}
The implication $(2)\Rightarrow(1)$ is immediate. If (3) holds, then
there are $c\in\mathbb R^\times$ and
$P\in\mathbb Q[x_1,\dots,x_N]$ such that $F_\Lambda=cP$. Since the
values of $P$ on $\mathbb Z^N$ have bounded denominators, there exists
$m\ge1$ such that $ mP(\mathbb Z^N)\subset\mathbb Z$.
Thus $\det(\Lambda)=F_\Lambda(\mathbb Z^N)\subset c\,m^{-1}\mathbb Z$, which proves $(3)\Rightarrow(2)$.

We next prove $(1)\Rightarrow(3)$. If $\Lambda$ is
determinant-rational, then $F_\Lambda(\mathbb Z^N)\subset\lambda\mathbb Q$ for some $\lambda\in\mathbb R^\times$. Hence
$\lambda^{-1}F_\Lambda$ takes rational values on $\mathbb Z^N$.
By \Cref{polynomial},
$$
    \lambda^{-1}F_\Lambda\in\mathbb Q[x_1,\dots,x_N],
$$
so $F_\Lambda$ is proportional to a polynomial with rational
coefficients.

We now prove $(3)\Rightarrow(5)$. After rescaling, we may assume that
$F_\Lambda\in\mathbb Q[x_1,\dots,x_N]$. Since $F_\Lambda$ is a real
determinant form, the scheme
$\{g\in\GL_N:F_\Lambda=\det\circ g\}$ is defined over $\mathbb Q$ and
has an $\mathbb R$-point. By the Nullstellensatz, it therefore has an
$\overline{\mathbb Q}$-point,
so $F_\Lambda$ is linearly equivalent to $\det$ over
$\overline{\mathbb Q}$. By
\Cref{lem:twisted-determinant-forms}, there exist a determinant algebra
$(\mathscr{K},\mathscr{A},\tau)$ over $\mathbb Q$ and a $\mathbb Q$-basis
$\{w_{ij}\}$ of $J(\mathscr{A},\tau)$ such that
$$
    F_\Lambda(x)
    =
    c\,
    N_{\mathscr{A},\tau}\!\left(\sum_{i,j}x_{ij}w_{ij}\right)
$$
for some $c\in\mathbb R^\times$. Since $F_\Lambda$ is linearly
equivalent over $\mathbb R$ to the ordinary determinant, the determinant
algebra is split over $\mathbb R$.

We prove $(5)\Rightarrow(4)$. Let $\{v_{ij}\}$ be the
$\mathbb Z$-basis of $\Lambda$ used to define $F_\Lambda$, and let
$W$ be the $\mathbb Z$-span of the basis $\{w_{ij}\}$ in
$J(\mathscr{A},\tau)$. Choose an order $\mathcal O_0\subset\mathscr A$ and
set $\mathcal O:=\mathcal O_0\cap\tau(\mathcal O_0)$. Then $\mathcal O$ is
a $\tau$-stable order in $\mathscr A$, and $W$ and $\mathcal O^\tau$ are full lattices in the same
$\mathbb Q$-vector space $J(\mathscr{A},\tau)$, hence are commensurable. Let $ L:J(\mathscr{A},\tau)\otimes_{\mathbb Q}\mathbb R\to \M_n(\mathbb R)$
be the real linear isomorphism determined by
$L(w_{ij})=v_{ij}$. Since $(\mathscr K,\mathscr A,\tau)$ is split over $\mathbb R$, there is
an $\mathbb R$-linear isomorphism
$\iota:J(\mathscr A,\tau)\otimes_{\mathbb Q}\mathbb R
\xrightarrow{\sim}\M_n(\mathbb R)$
such that
$$
N_{\mathscr A,\tau}(x)=\det(\iota(x)).
$$
Then $T:=L\circ\iota^{-1}$ satisfies
$\det(Tv)=c\,\det(v)$ for all $v\in\M_n(\mathbb R)$.
By \Cref{Fro}, either
$$
    T(v)=g_1vg_2\quad\text{or}\quad
    T(v)=g_1v^{\intercal}g_2
$$
for some $g_1,g_2\in\GL_n(\mathbb R)$.  In the second case, replace
$\iota$ by the $\mathbb R$-linear isomorphism
$x\to \iota(x)^{\intercal}$.
Since the determinant is invariant under transpose, this new isomorphism
still satisfies
$$
N_{\mathscr A,\tau}(x)=\det(\iota(x)).
$$ Thus, after this replacement, we may assume that
$
T(v)=g_1vg_2$ for $v\in\M_n(\mathbb R)$.
Since $T=L\circ\iota^{-1}$, it follows that
$$
L(x)=T(\iota(x))=g_1\iota(x)g_2
\qquad
\bigl(x\in J(\mathscr A,\tau)\otimes_{\mathbb Q}\mathbb R\bigr).
$$

Since $L(W)=\Lambda$ and $W$ is commensurable with $\mathcal O^\tau$,
it follows that $\Lambda$ is commensurable with
$g_1\,\iota(\mathcal O^\tau)\,g_2$. This proves (4).

Finally, we prove $(4)\Rightarrow(2)$. Let $x\in\mathcal O^\tau$.
Since $\mathcal O$ is an order, $\operatorname{Nrd}_{\mathscr{A}/\mathscr{K}}(x)$ is an
algebraic integer of $\mathscr{K}$. Since $x=\tau(x)$, this reduced norm is fixed
by the nontrivial automorphism of $\mathscr{K}/\mathbb Q$, and hence lies in
$\mathbb Z$. Therefore
$$
    \det(g_1\iota(x)g_2)
    =
    \det(g_1g_2)\,N_{\mathscr{A},\tau}(x)
    \in
    \det(g_1g_2)\mathbb Z.
$$
If $\Lambda$ is commensurable with $g_1\iota(\mathcal O^\tau)g_2$, then
there exists $m\ge1$ such that $m\Lambda\subset g_1\iota(\mathcal O^\tau)g_2$. Consequently,
$$
    \det(\Lambda)
    \subset
    m^{-n}\det(g_1g_2)\mathbb Z.
$$
This proves (2).
It remains to prove the final assertion. Let
$$
L_\Lambda:\M_n(\mathbb R)\to\M_n(\mathbb R)
$$
be the linear isomorphism determined by the chosen
$\mathbb Z$-basis $\{v_{ij}\}$ of $\Lambda$, so that
$L_\Lambda(E_{ij})=v_{ij}$. Then
$$
\Lambda=L_\Lambda\M_n(\mathbb Z)
\qquad\text{and}\qquad
F_\Lambda=\det\circ L_\Lambda.
$$

Recall that, for $S_1,S_2\in\GL_N(\mathbb R)$, the lattices
$S_1\mathbb Z^N$ and $S_2\mathbb Z^N$ are commensurable if and only if $S_1^{-1}S_2\in\GL_N(\mathbb Q)$.
Suppose first that $\Lambda$ is of $\mathbb Q$-split type. Then there
exist $g_1,g_2\in\GL_n(\mathbb R)$ such that $\Lambda$ is
commensurable with $g_1\M_n(\mathbb Z)g_2$.
Let
$
S:\M_n(\mathbb R)\to\M_n(\mathbb R)$ be given by $S(X)=g_1Xg_2$.
The lattices $L_\Lambda\M_n(\mathbb Z)$ and
$S\M_n(\mathbb Z)$ are commensurable. Hence
$$
T:=S^{-1}L_\Lambda\in\GL_N(\mathbb Q).
$$
It follows that
$$
\begin{aligned}
F_\Lambda(X)
&=\det(L_\Lambda X)=\det(S(TX))=\det(g_1)\det(g_2)\det(TX).
\end{aligned}
$$
Thus
$F_\Lambda=
\det(g_1)\det(g_2)\,\det\circ T,
$
and therefore $F_\Lambda$ is $\mathbb Q$-split.

Conversely, suppose that $F_\Lambda$ is $\mathbb Q$-split. Then there
exist $c\in\mathbb R^\times$ and $T\in\GL_N(\mathbb Q)$ such that $ F_\Lambda=c\,\det\circ T$. Set $S:=L_\Lambda T^{-1}$. For every $X\in\M_n(\mathbb R)$, we have
$$
\det(SX)
=
\det(L_\Lambda T^{-1}X)
=
c\,\det(X).
$$
By \Cref{Fro}, there exist $g_1,g_2\in\GL_n(\mathbb R)$ such that either $S(X)=g_1Xg_2$
for all $X$, or
$
S(X)=g_1X^{\intercal}g_2$
for all $X$. Since $T\in\GL_N(\mathbb Q)$, the lattice
$T\M_n(\mathbb Z)$ is commensurable with $\M_n(\mathbb Z)$. Hence
$$
\begin{aligned}
\Lambda
&=L_\Lambda\M_n(\mathbb Z) =S\bigl(T\M_n(\mathbb Z)\bigr)
\end{aligned}
$$
is commensurable with $S\M_n(\mathbb Z)=g_1\M_n(\mathbb Z)g_2$.
Thus $\Lambda$ is of $\mathbb Q$-split type.

\end{proof}

\subsection{Orbit closures}

Let
$$
    G=\SL_N(\mathbb R),
    \qquad
    \Gamma=\SL_N(\mathbb Z),
    \qquad
    X=G/\Gamma.
$$
The space $X$ parametrizes unimodular lattices in $\M_n(\mathbb R)$.
For $g=(g_{ij})\in\SL_n(\mathbb R)$ and $h\in\SL_n(\mathbb R)$, their
Kronecker product is
$$
    g\otimes h
    =
    \begin{pmatrix}
    g_{11}h&\cdots&g_{1n}h\\
    \vdots&\ddots&\vdots\\
    g_{n1}h&\cdots&g_{nn}h
    \end{pmatrix}
    \in\SL_N(\mathbb R).
$$
We set
$$
    H=\SL_n(\mathbb R)\otimes\SL_n(\mathbb R).
$$
Under the identification $E_{ij}=e_i\otimes e_j$, one has
$$
    (h_1\otimes h_2)v=h_1vh_2^{\intercal}.
$$

The following consequence of Ratner's orbit-closure theorem \cite{ratner:1991} is the
dynamical input.

\begin{theorem}
\label{rat}
For every $[g]\in G/\Gamma$, the orbit $H[g]$ is either closed or
dense. If $H[g]$ is closed, then
$ H\cap g\Gamma g^{-1}$
is a lattice in $H$.
\end{theorem}

\begin{proof}
The group $H$ is
generated by one-parameter unipotent subgroups. Moreover, $H$ is a
maximal connected Lie subgroup of $G$; see \cite{dynkin}. Ratner's
orbit-closure theorem \cite{ratner:1991} therefore implies the claim.
\end{proof}

We next identify the invariant polynomials for the $H$-action.

\begin{lemma}
\label{detuni}
For any $g\in G$, the space of homogeneous  polynomials of degree $n$ on
$\M_n(\mathbb R)$ invariant under $g^{-1}Hg$ is one-dimensional and is
spanned by
$$
    F_g:=\det\circ g.
$$
\end{lemma}

\begin{proof}  Let $f$ be a homogeneous degree-$n$ polynomial invariant under
$g^{-1} Hg $ for some $g\in G$.
After replacing $f$ by $f\circ g^{-1}$, it suffices to treat the case
$g=e$. If $f(I_n)=0$, then $f$ vanishes on the $H$-orbit of $I_n$,
which is $\SL_n(\mathbb R)$. By homogeneity, $f$ vanishes on
the open subset $\br^+\SL_n(\mathbb R) $ of matrices with positive determinants. This set is Zariski dense in  
$\M_n(\br)$. Hence $f$ vanishes identically. Thus for $f\ne 0$, we may assume after rescaling that
$f(I_n)=1$.

For $X\in\GL_n(\mathbb R)$, the $H$-orbit of $X$ is determined by
$\det X$. Hence
$$
    f(X)=
    f\bigl(\operatorname{diag}(\det X,1,\ldots,1)\bigr).
$$
For every $t\ne0$, $\operatorname{diag}(t^n,1,\ldots,1)$
lies in the $H$-orbit of $tI_n$. Therefore
$$
    f\bigl(\operatorname{diag}(t^n,1,\ldots,1)\bigr)
    =
    f(tI_n)=t^n.
$$
It follows that for all $s>0$, 
$ f\bigl(\operatorname{diag}(s,1,\ldots,1)\bigr)=s$.
  Hence $f(X)=\det X$ on the open set
$\{X:\det X>0\}$. Since this set is Zariski dense in
$\M_n(\mathbb R)$, the two polynomials agree identically.
\end{proof}

We can now characterize determinant-rational lattices dynamically.

\begin{theorem}
\label{prop:closed-rational-csa}
Let $g\in G$, and let
$\Lambda=g\mathbb Z^N$.
Then the following hold:
\begin{enumerate}
\item $\Lambda$ is determinant-rational if and only if $H[g]$ is closed
in $X$.
\item If $\Lambda$ is not determinant-rational, then $H[g]$ is dense in
$X$, and
$$
    \overline{\det(\Lambda)}=\mathbb R.
$$
\end{enumerate}
\end{theorem}

\begin{proof}
 Suppose
that $H[g]$ is closed. By \Cref{rat},
$\Gamma\cap g^{-1}Hg$
is a lattice in $g^{-1}Hg$.
Let $W_0$ be the space of homogeneous degree-$n$ polynomials on
$\mathbb R^N$ invariant under $\Gamma\cap g^{-1}Hg$. This is an
$\mathbb R$-subspace of $\operatorname{Sym}^n((\mathbb R^N)^*)$ defined over $\mathbb Q$, since
$\Gamma\cap g^{-1}Hg\subset \SL_N(\mathbb Z)$. Moreover $F_g\in W_0$.
By the Borel density theorem, $\Gamma\cap g^{-1}Hg$ is Zariski dense in
$g^{-1}Hg$; see, for instance, \cite{zimmer1984ergodic}. Hence every
element of $W_0$ is invariant under $g^{-1}Hg$. By \Cref{detuni},
$$
    W_0=\mathbb R F_g.
$$
Since $W_0$ is defined over $\mathbb Q$, the line
$\mathbb R F_g$ is defined over $\mathbb Q$. Hence $F_g$, and
therefore $F_\Lambda$, is proportional to a polynomial with rational
coefficients. By \Cref{thm:det-rational-characterization}, $\Lambda$ is
determinant-rational.

Conversely, suppose that $\Lambda=g\mathbb Z^N$ is determinant-rational.
After rescaling $F_g$, we may assume that
$ F_g\in\mathbb Q[x_1,\dots,x_N]$.
Let
$$
    \mathbf L
    :=
    \operatorname{Stab}_{\SL_N}(F_g)^\circ .
$$
Since $F_g$ has rational coefficients, $\mathbf L$ is an algebraic
group defined over $\mathbb Q$. Moreover,
$\mathbf L(\mathbb R)^\circ=g^{-1}Hg$.
The group $\mathbf L$ is semisimple. By the Borel--Harish-Chandra
theorem \cite{borel-harish:1962}, $\mathbf L(\mathbb Q)\cap \SL_N(\mathbb Z)$
is a lattice in $\mathbf L(\mathbb R)^\circ$. Hence
$ g^{-1}Hg\cap \Gamma$
is a lattice in $g^{-1}Hg$. Therefore $H[g]$ is closed.
This proves (1). 
Part (2) follows immediately from \Cref{rat}: if
$\Lambda$ is not determinant-rational, then $H[g]$ is not closed by (1), and
hence is dense in $X$. Since
$\Lambda=g\mathbb Z^N\supset g\Gamma e_1$
and $ Ge_1=\mathbb R^N-\{0\}$,
we obtain
$$
    \overline{\det(\Lambda)}
    \supset
    \det\bigl(\overline{Hg\Gamma}\,e_1\bigr)
    =
    \det(Ge_1)
    =
    \det(\M_n(\mathbb R)-\{0\})
    =
    \mathbb R.
$$
\end{proof}

\section{Height functions and main technical results}
\label{s:height-main}

In this section, we introduce the height functions and the representation-theoretic
framework used in the proof of \Cref{main}. We use
$\SL_n(\mathbb R)\times\SL_n(\mathbb R)$ as a covering group of its image $H$ under the
left--right action and do not distinguish a pair from its image. We then define the Diophantine and quasi-null conditions
that isolate the exceptional subspaces for this action, and state the
principal technical results proved in the subsequent sections.

\subsection{Rational subspaces of a lattice}
Recall that $N=n^2$ and that $\M_n(\mathbb R)=\mathbb R^N$.
 We equip $\M_n(\mathbb R)$ with the
Euclidean inner product
$$
\langle v,w\rangle=\operatorname{tr}(v w^{\intercal}),\qquad(v, w\in \M_n(\br))
$$
 and use the induced inner
products on all exterior powers.

Let $V<\M_n(\mathbb R)$ be an $r$-dimensional subspace for $1\leq r\leq N-1$. A
\emph{Pl\"ucker vector} of $V$ is a nonzero decomposable vector
\begin{equation}\label{wvv}
\mathsf w_V=v_1\wedge\cdots\wedge v_r
\in \wedge^r\M_n(\mathbb R),
\end{equation}
where $v_1,\ldots,v_r$ form a basis of $V$. Thus the Pl\"ucker line
$\mathbb R\mathsf w_V$ depends only on $V$.

Let $\Delta<\M_n(\mathbb R)$ be a lattice. A subspace
$V<\M_n(\mathbb R)$ is \emph{$\Delta$-rational} if $\Delta\cap V$ is a
lattice in $V$. For such a subspace, its \emph{$\Delta$-Pl\"ucker vector}
is
\begin{equation}\label{wv}
\mathsf w_{\Delta,V}=v_1\wedge\cdots\wedge v_r,
\end{equation}
where $v_1,\ldots,v_r$ is a $\mathbb Z$-basis of $\Delta\cap V$. It is well-defined up to sign, and
$$
\|\mathsf w_{\Delta,V}\|=\operatorname{covol}_V(\Delta\cap V).
$$
We write
$$
d_\Delta(V):=
\begin{cases}
1 &\text{if } V=\{0\},\\
\|\mathsf w_{\Delta,V}\|& \text{if }V\neq\{0\}.
\end{cases}
$$

\subsection{The height function \texorpdfstring{$\alpha$}{alpha}}
Recall that $X=\SL_N(\br)/\SL_N(\z)$ is the space of unimodular lattices in $\M_n(\br)$.  The
\emph{$\alpha$-characteristics} are the functions
$\alpha_1,\ldots,\alpha_{N-1}:X\to(0,\infty)$ defined by
$$
\alpha_i(\Delta)
:=
\sup\left\{
\frac{1}{d_\Delta(V)}:
V<\M_n(\mathbb R)\text{ is $\Delta$-rational and }\dim V=i
\right\}.
$$
Equivalently, $\alpha_i(\Delta)^{-1}$ is the least covolume of an
$i$-dimensional primitive sublattice of $\Delta$. The $\alpha$-function is defined by
$$
\alpha(\Delta):=\max_{1\le i\le N-1}\alpha_i(\Delta)
$$

\subsection{Schur functors}\label{subsection:SchurFunctors}

Let
$$
H_1=\SL_n(\mathbb R)\times\{e\},
\quad
H_2=\{e\}\times\SL_n(\mathbb R),
\quad
H=\SL_n(\mathbb R)\times \SL_n(\mathbb R)=H_1H_2.
$$
The action 
$$
(g,h)\cdot v=gvh^{\intercal}
\qquad
(g,h\in\SL_n(\mathbb R),\ v\in\M_n(\mathbb R))
$$
induces an action of $H$ on every exterior power of
$\M_n(\mathbb R)$.

Let $\mathsf C$ and $\mathsf R$ be two copies of $\mathbb R^n$,
representing the column and row coordinates, respectively. Under the
identification $E_{ij}\leftrightarrow e_i\otimes f_j$, one has
$$
\M_n(\mathbb R)\simeq \mathsf C\otimes\mathsf R,
$$
and the two factors of $H$ act in the standard way on $\mathsf C$ and
$\mathsf R$. 

The exterior powers of matrix space decompose into irreducible $H$-summands
by the skew Cauchy formula for all $1\le r\le N-1$:
$$
\wedge^r\M_n(\mathbb R)
\simeq
\bigoplus_{\substack{\lambda\vdash r\\ \lambda\subseteq n\times n}}
{\mathsf S}_\lambda(\mathbb R^n)\otimes
{\mathsf S}_{\lambda^{\intercal}}(\mathbb R^n),
$$
where $\lambda$ ranges over
partitions of $r$ whose Young diagram lies in the $n\times n$ box,
${\mathsf S}_\lambda$ denotes the corresponding Schur module, and
$\lambda^{\intercal}$ denotes the conjugate partition. We use the standard
notation for Schur functors; see \cite[Section~4.2]{FultonHarris} or
\cite[Chapter~I]{FultonYT}.

If $\lambda=(\lambda_1,\ldots,\lambda_n)$, with
$\lambda_1\ge\cdots\ge\lambda_n\ge0$, then 
${\mathsf S}_\lambda(\mathbb R^n)$, viewed as an $\SL_n(\mathbb R)$-representation, has highest weight
$$
\sum_{j=1}^{n-1}(\lambda_j-\lambda_{j+1})\omega_j
$$ where $\omega_1,\ldots, \omega_{n-1}$ are the fundamental weights of $\SL_n(\br)$.
It is therefore trivial precisely when $\lambda=(q^n)$ for some
$q\ge0$. Consequently, a trivial factor can occur in
$\wedge^i\M_n(\mathbb R)$ only when $n\mid  i$.

For $1\le k\le n-1$, define
$$
\mathcal M_{k,1}
:=
\bigl(\wedge^{kn}\M_n(\mathbb R)\bigr)^{H_2}
\simeq {\mathsf S}_{(n^k)}(\mathsf C)\otimes\mathbf 1,
$$
$$
\mathcal M_{k,2}
:=
\bigl(\wedge^{kn}\M_n(\mathbb R)\bigr)^{H_1}
\simeq \mathbf 1\otimes {\mathsf S}_{(n^k)}(\mathsf R).
$$
Thus $\mathcal M_{k,1}$ and $\mathcal M_{k,2}$ are the maximal subspaces
on which $H_2$ and $H_1$, respectively, act trivially. Since the skew
Cauchy decomposition is multiplicity-free, it gives an $H$-equivariant
direct-sum decomposition
\begin{equation}\label{eq:decompositionatdimnn-1}
\wedge^{kn}\M_n(\mathbb R)
=
\mathcal M_{k,0}\oplus\mathcal M_{k,1}\oplus\mathcal M_{k,2},
\end{equation}
where $\mathcal M_{k,0}$ is the sum of the remaining irreducible summands.
For each $m=1,2$, let
\begin{equation}\label{eq:pikmdefinition}
    \pi_{k,m}:\wedge^{kn}\M_n(\mathbb R)\to\mathcal M_{k,m}
\end{equation}
be the $H$-equivariant projection associated with
\eqref{eq:decompositionatdimnn-1}. For $r\ge 1$ and
$m=1,2$, put
$$
\pi_{k,m}^{(r)}:=\wedge^r\pi_{k,m}:
\wedge^r\!\bigl(\wedge^{kn}\M_n(\mathbb R)\bigr)
\to \wedge^r\mathcal M_{k,m}.
$$
Thus
$$
\pi_{k,m}^{(r)}(w_1\wedge\cdots\wedge w_r)
=
\pi_{k,m}(w_1)\wedge\cdots\wedge\pi_{k,m}(w_r).
$$

\subsection{Isotropic subspaces}
For a matrix $v\in\M_n(\mathbb R)$, let $\operatorname{Col}(v)$ and
$\operatorname{Row}(v)$ denote  the subspaces of $\mathbb R^n$
spanned by the columns and the rows of $v$, respectively. For a subspace
$\mathsf U<\mathbb R^n$, define
\begin{equation*}
\begin{aligned}
\mathcal L(\mathsf U)
&:=\{v\in\M_n(\mathbb R):\operatorname{Col}(v)\subset \mathsf U\},\\
\mathcal R(\mathsf U)
&:=\{v\in\M_n(\mathbb R):\operatorname{Row}(v)\subset \mathsf U\}.
\end{aligned}
\end{equation*}
Both spaces have dimension $n\dim \mathsf U$.
\begin{definition}
Let $1\le k\le n-1$. A $kn$-dimensional subspace
$V<\M_n(\mathbb R)$ is called \emph{column-isotropic} if
$$
V=\mathcal L(\mathsf U)
$$
for some $k$-dimensional subspace $\mathsf U<\mathbb R^n$, and
\emph{row-isotropic} if
$$
V=\mathcal R(\mathsf U)
$$
for some $k$-dimensional subspace $\mathsf U<\mathbb R^n$.
We call $V$ \emph{isotropic} if it is either column-isotropic or
row-isotropic.
\end{definition}

Every isotropic subspace is contained in the singular hypersurface
$\{X:\det X=0\}$. We use the term \emph{isotropic} only for the two
families above; an arbitrary linear subspace on which the determinant
vanishes identically need not be isotropic in this sense. In the maximal
dimension $n(n-1)$, however, every $n(n-1)$-dimensional linear subspace of $\{\det=0\}$ is
column- or row-isotropic, by a classical theorem of Dieudonn\'e
\cite{dieudonne:1948}; this fact will be used in the proof of
\Cref{m11}.

The exceptional summands in \eqref{eq:decompositionatdimnn-1} have the
following interpretation.

\begin{lemma}\label{lem:isotropic_iff_pi1_pi2}
Let $V<\M_n(\mathbb R)$ have dimension $kn$ for some $1\le k\le n-1$. Then
$$
V\text{ is column-isotropic}
\quad\Longleftrightarrow\quad
\mathsf w_V\in\mathcal M_{k,1};
$$
$$
V\text{ is row-isotropic}
\quad\Longleftrightarrow\quad
\mathsf w_V\in\mathcal M_{k,2}.
$$
\end{lemma}

\begin{proof}
Suppose first that $V=\mathcal L(\mathsf U)$ with $\dim \mathsf U=k$. Then $H_2$ preserves $V$. Under the
identification $V\simeq \mathsf U\otimes\mathsf R$, the determinant of the action
of $h\in H_2$ on $V$ is $(\det h)^k=1$. Hence $H_2$ fixes
$\mathsf w_V$, and therefore $\mathsf w_V\in\mathcal M_{k,1}$.

Conversely, suppose that $\mathsf w_V\in\mathcal M_{k,1}$. Then
$\mathsf w_V$ is fixed by $H_2$. Since a nonzero decomposable exterior
vector determines its underlying subspace, $V$ is $H_2$-invariant. As an
$H_2$-module,
$
\M_n(\mathbb R)\simeq\mathsf C\otimes\mathsf R
$
is the tensor product of a trivial multiplicity space $\mathsf C$ with the
irreducible standard $H_2$-module $\mathsf R$. Hence every
$H_2$-submodule is of the form $\mathsf U\otimes\mathsf R$ for a unique
subspace $\mathsf U<\mathsf C$. Since $\dim V=kn$, one has $\dim \mathsf U=k$, and
thus $V=\mathcal L(\mathsf U)$. The row statement follows by interchanging the
two factors of $H$.
\end{proof}

\subsection{A rationality obstruction for exceptional summands}

The next observation clarifies the relation between determinant-rationality
and the exceptional summands. 
\begin{proposition}[A rational exceptional summand forces determinant-rationality]
\label{prop:rational-exceptional-summand-implies-det-rational}
Let $\Delta<\M_n(\mathbb R)$ be a lattice and
$1\le k\le n-1$. Suppose that, for some $m\in\{1,2\}$, the subspace
$\mathcal M_{k,m}\subset\wedge^{kn}\M_n(\mathbb R)$ is rational with
respect to the exterior lattice $\wedge^{kn}\Delta$. Then $\Delta$ is
determinant-rational.
\end{proposition}

\begin{proof}  Let $\Delta=g\mathbb Z^N$ for some $g\in \GL_N(\br)$. By assumption, 
$$
M_g:=(\wedge^{kn}g^{-1})\mathcal M_{k,1}
\subset \wedge^{kn}\mathbb R^N 
$$
is a $\mathbb Q$-rational subspace of
$\wedge^{kn}\mathbb R^N$. Hence its pointwise stabilizer
$$
\mathbf L
:=
\left\{
h\in\SL_N:\ (\wedge^{kn}h)v=v
\text{ for every }v\in M_g
\right\}
$$
is a $\mathbb Q$-algebraic subgroup of $\SL_N$. We claim that
\begin{equation}\label{eq:pointwise-stabilizer-exceptional}
\mathbf L(\mathbb R)^\circ=g^{-1}H_2g.
\end{equation}
 It is enough to prove  that the identity component $P$ of
the pointwise stabilizer of $\mathcal M_{k,1}$ in
$\SL(\mathsf C\otimes\mathsf R)$ is $H_2$. Let $\Gr(k,\mathsf C)$ denote the Grassmannian of $k$-dimensional subspaces of $\mathsf C$.
For $\mathsf U\in\Gr(k,\mathsf C)$, set
$$
W_{\mathsf U}:=\mathsf U\otimes\mathsf R.
$$ Let
$Z\in\mathfrak{sl}(\mathsf C\otimes\mathsf R)
$
belong to $\operatorname{Lie}(P)$. Then 
$
Z\mathsf w_{W_{\mathsf U}}=0$ for every $\mathsf U\in\Gr(k,\mathsf C).
$
This implies that
$$
Z(\mathsf U\otimes\mathsf R)\subset \mathsf U\otimes\mathsf R
\qquad
\text{for every }\mathsf U\in\Gr(k,\mathsf C).
$$
We claim that $Z$ preserves each subspace
$\mathbb Rc\otimes\mathsf R$, $0\ne c\in\mathsf C$. If $k=1$, this is
immediate. If $k>1$, then
$\mathbb Rc
=
\bigcap_{\substack{\mathsf U\in\Gr(k,\mathsf C),\, c\in \mathsf U}}\mathsf U,
$
and therefore
$
\mathbb Rc\otimes\mathsf R
=
\bigcap_{\substack{\mathsf U\in\Gr(k,\mathsf C),\,  c\in \mathsf U}}
(\mathsf U\otimes\mathsf R).
$
Since $Z$ preserves every $\mathsf U\otimes\mathsf R$, it preserves this
intersection.
Consequently, for every $0\ne c\in\mathsf C$, there exists
$B_c\in\operatorname{End}(\mathsf R)$ such that
$$
Z(c\otimes r)=c\otimes B_c r
\qquad(r\in\mathsf R).
$$
If $c_1,c_2$ are linearly independent, then applying the same statement to
$c_1+c_2$ and using linearity gives
$$
c_1\otimes B_{c_1}r+c_2\otimes B_{c_2}r
=
(c_1+c_2)\otimes B_{c_1+c_2}r
\qquad(r\in\mathsf R).
$$
Hence
$$
B_{c_1}=B_{c_2}=B_{c_1+c_2}.
$$
The same conclusion for proportional nonzero vectors follows by scaling.
Thus $B_c$ is independent of $c$; write it simply as $B$. We have shown
that
$$
Z=\mathrm{id}_{\mathsf C}\otimes B.
$$
Since $Z\in\mathfrak{sl}(\mathsf C\otimes\mathsf R)$, we have $B\in\mathfrak{sl}(\mathsf R)$.
Thus
$
\operatorname{Lie}(P)
=
\mathrm{id}_{\mathsf C}\otimes\mathfrak{sl}(\mathsf R).
$
Since $H_2$ is connected and is contained in $P$, it follows that
$P=H_2$.
Conjugating by $g$ proves \eqref{eq:pointwise-stabilizer-exceptional}.
In particular, $g^{-1}H_2g$ is
defined over $\mathbb Q$. Let $\mathcal P_n$ be the real vector space of
homogeneous degree-$n$ polynomials on $\mathbb R^N$. Then the subspace
$
\mathcal P_n^{\,g^{-1}H_2g}
$
is a 
linear subspace of $\mathcal P_n$ defined over $\mathbb Q$.
We now identify this invariant subspace. Set
$$
F_g:=\det\circ g.
$$
After composing polynomials with $g^{-1}$, the computation reduces to the
standard right $\SL(\mathsf R)$-action on
$\mathsf C\otimes\mathsf R$. By the Cauchy decomposition,
$$
\operatorname{Sym}^n((\mathsf C\otimes\mathsf R)^*)
\simeq
\bigoplus_{\lambda\vdash n}
{\mathsf S}_\lambda(\mathsf C^*)\otimes
{\mathsf S}_\lambda(\mathsf R^*).
$$
The restriction of ${\mathsf S}_\lambda(\mathsf R^*)$ to
$\SL(\mathsf R)$ contains a trivial subrepresentation if and only if
$\lambda=(1^n)$. Hence the space of right-$\SL(\mathsf R)$-invariant
homogeneous degree-$n$ polynomials is
$$
{\mathsf S}_{(1^n)}(\mathsf C^*)\otimes
{\mathsf S}_{(1^n)}(\mathsf R^*)
=
\wedge^n\mathsf C^*\otimes\wedge^n\mathsf R^*,
$$
which is one-dimensional and is spanned by the determinant. Therefore
$$
\mathcal P_n^{\,g^{-1}H_2g}
=
\mathbb R F_g.
$$

Since this one-dimensional subspace is defined over $\mathbb Q$, the line
$\mathbb R F_g$ is defined over $\mathbb Q$. Equivalently, $F_g$ is
proportional to a polynomial with rational coefficients. By
\Cref{thm:det-rational-characterization}, the lattice
$\Delta=g\mathbb Z^N$ is determinant-rational.

The proof for $\mathcal M_{k,2}$ is identical, with $H_2$ replaced by $H_1$.
\end{proof}

\begin{remark}
When $n=2$, the only possible value of $k$ is $1$,
$\mathcal M_{1,1}$ is three-dimensional, and the Pl\"ucker lines of
column-isotropic $2$-planes form a nondegenerate conic in
$\mathbb P(\mathcal M_{1,1})$. If a non-determinant-rational lattice
$\Delta$ had three distinct $\Delta$-rational column-isotropic planes, the
corresponding three rational points on this conic would be noncollinear and
hence span $\mathbb P(\mathcal M_{1,1})$. Thus $\mathcal M_{1,1}$ would be
rational with respect to $\wedge^2\Delta$, contradicting the preceding
proposition.

Therefore a non-determinant-rational lattice in $\M_2(\mathbb R)$ has at
most two rational column-isotropic planes and, similarly, at most two
rational row-isotropic planes. Consequently, $\{\det=0\}$ contains at most
four maximal $\Delta$-rational linear subspaces. This is the finiteness
phenomenon used in the quadratic determinant case
\cite{eskin-margulis-mozes:2005}.
\end{remark}

\subsection{Diophantine lattices and determinant forms}
The following condition requires rational Pl\"ucker configurations either
to lie in the exceptional subspaces $\mathcal M_{k,1}$ and
$\mathcal M_{k,2}$, where local contraction may fail, or to remain
polynomially separated from them.

\begin{definition}[Diophantine lattices]\label{lattice_Diophantine}
Let $\eta>0$ and $M>1$. A lattice
$\Lambda<\M_n(\mathbb R)$ is \emph{$(\eta,M)$-Diophantine} if  for every $1\le k\le n-1$,  $m\in\{1,2\}$, 
$r\in\{1,\dim\mathcal M_{k,m}\}$, and for every collection of $kn$-dimensional
$\Lambda$-rational subspaces $V_1,\ldots,V_r$, 
the vector
$
\mathbf w
:=
\mathsf w_{\Lambda,V_1}\wedge\cdots\wedge
\mathsf w_{\Lambda,V_r},
$ 
satisfies, whenever it is nonzero,
$$\text{either }
\mathbf w=\pi_{k,m}^{(r)}(\mathbf w)\quad \text{ or }\quad 
\|\mathbf w-\pi_{k,m}^{(r)}(\mathbf w)\|
>
\eta\,\|\mathbf w\|^{-M}.
$$
A lattice is \emph{Diophantine} if it is $(\eta,M)$-Diophantine for some
$\eta>0$ and $M>1$.
\end{definition}

 The case
$r=1$ controls individual rational subspaces, while the case
$r=\dim\mathcal M_{k,m}$ rules out configurations of rational
Pl\"ucker vectors that nearly span an entire exceptional summand.

The Diophantine property is invariant under left and right multiplication by
elements of $\GL_n(\mathbb R)$, and under transposition. Indeed, these
linear transformations preserve the two exceptional families (with
transposition interchanging them), and the relevant norms are comparable
under every fixed invertible linear map.

\begin{proposition}\label{prop:algebraicityimpliesDiophantine}
Every lattice in $\M_n(\mathbb R)$ consisting of matrices with
algebraic entries is Diophantine.
\end{proposition}
The proof is given in \Cref{s:example}.
Our first counting theorem is the following extension of
\Cref{m1}. \begin{theorem}\label{main}
Let $\Lambda<\M_n(\mathbb R)$ be a Diophantine lattice that is not
determinant-rational. Then, for every $a<b$,
$$
N_\Lambda^\times(a,b;T)
\sim
\frac{C_{\|\cdot\|}}{\covol(\Lambda)}
(b-a)T^{n(n-1)}
\qquad (T\to\infty),
$$
where $C_{\|\cdot\|}$ is the constant appearing in
\eqref{eq:intro-volume-asymptotic}. Equivalently,
$$
N_\Lambda^\times(a,b;T)
\sim
\frac{1}{\covol(\Lambda)}
\operatorname{vol}
\{v\in\M_n(\mathbb R):\ \|v\|<T,\ a<\det v<b\}.
$$
\end{theorem}
We also record the corresponding terminology for determinant forms.
\begin{definition}\label{form_Diophantine}
A determinant form $F$ is \emph{Diophantine} if, for one (equivalently,
every) realization $F=F_{\Lambda,\mathcal B}$, the lattice $\Lambda$ is
Diophantine in the sense of \Cref{lattice_Diophantine}.
\end{definition}

This definition is independent of the realization. Indeed, if
$F=\det\circ g=\det\circ g'$, then $g'g^{-1}$ preserves the determinant;
by \Cref{Fro}, it is given by left and right multiplication, possibly
followed by transposition. The preceding invariance observation therefore
applies. A change of integral coordinates only changes the chosen basis of the same lattice.
Theorem \ref{main} is therefore equivalent to the following:

\begin{theorem}\label{m2general}
Let $F$ be a Diophantine determinant form on $\M_n(\mathbb R)$. If
$F$ is not proportional to a polynomial with rational coefficients, then
the asymptotic formula of \Cref{m2} holds for $F$.
\end{theorem}

\subsection{Quasi-null subspaces and the modified height function}

Motivated by \Cref{lem:isotropic_iff_pi1_pi2}, we now introduce
subspaces whose Pl\"ucker vectors are close to one of the exceptional
summands.

For $0\le\eta<1$, $M\ge1$, and $1\le k\le n-1$, let
\be\label{hkd}
\mathscr Q_{kn,\eta,M}
:=
\left\{
0\ne v\in\wedge^{kn}\M_n(\mathbb R):
\min_{m=1,2}\|v-\pi_{k,m}(v)\|
\le \eta\|v\|^{-M}
\right\}.
\ee

\begin{definition}
Let $0\le\eta<1$, $M\ge1$, and 
$\Delta<\M_n(\mathbb R)$ be a lattice. For $1\le k\le n-1$, a
$kn$-dimensional $\Delta$-rational subspace $V$ is
\emph{$(\eta,M)$-quasi-null} if 
$$\mathsf w_{\Delta,V}\in  \mathscr Q_{kn,\eta,M}.$$
We call it \emph{column quasi-null} when the inequality holds for $m=1$,
and \emph{row quasi-null} when it holds for $m=2$. 
\end{definition}
Let $\mathscr Q_{kn,\eta,M}(\Delta)$ denote the collection of all
$(\eta,M)$-quasi-null $\Delta$-rational subspaces of dimension $kn$.

For $1\le i\le N-1$, let
$$\widetilde{\mathscr Q}_{i,\eta,M}(\Delta)$$ be the collection of all
$i$-dimensional $\Delta$-rational subspaces $V$ 
such that $V\subset V'$ for some $\Delta$-rational subspace $V'\in \mathscr Q_{kn,\eta,M}(\Delta)$ and some $1\le k\le n-1$. Thus $\widetilde{\mathscr Q}_{i,\eta,M}(\Delta)$ is the
$i$-dimensional part of the downward closure of the quasi-null subspaces.

\begin{definition}[Modified Margulis height function]
\label{def:modified-Margulis-height}
For $h\in H$ and $\Delta\in X$, define
\begin{equation}\label{eq:modified-alpha}
\widehat\alpha_{\eta,M}(h;\Delta)
:=
\max_{1\le i\le N-1}\widehat\alpha_{i,\eta,M}(h;\Delta),
\end{equation}
where
\begin{equation}\label{eq:hat-local-height}
\widehat\alpha_{i,\eta,M}(h;\Delta)
:=
\max\left\{
1,\ 
\sup_{\substack{
V\text{ is $\Delta$-rational},\ \dim V=i\\
V\notin\widetilde{\mathscr Q}_{i,\eta,M}(\Delta)
}}
\frac{1}{d_{h\Delta}(hV)}
\right\}.
\end{equation}
Here the supremum over the empty set is understood to be $0$.
\end{definition}

It follows directly from the defining quasi-null inequality that
whenever $1\leq i\leq N-1$, $0\le \eta'\le \eta<1$, $h\in H$ and $\Delta\in X$,
\be \label{lem:modified-height-monotonicity}
\widetilde{\mathscr Q}_{i,\eta',M}(\Delta)
\subset
\widetilde{\mathscr Q}_{i,\eta,M}(\Delta),
\qquad
\widehat\alpha_{\eta,M}(h;\Delta)
\le
\widehat\alpha_{\eta',M}(h;\Delta).
\ee

\subsection{Main technical results}

For $t\ge0$, define
\begin{equation}\label{btdef}
b_t:=\operatorname{diag}(e^{-t},\ldots,e^{-t},e^{(n-1)t})
\in\SL_n(\mathbb R),
\end{equation}
and
\begin{equation}\label{atdef}
a_t:=(b_t,b_t)\in H.
\end{equation}
Let $$K:=\SO(n)\times\SO(n)$$ and let $dk$ denote its probability Haar
measure.

The first result is the following  uniform integrability estimate:

\begin{theorem}\label{thm:modifieduniformboundedness}
Let $\Lambda\in X$ be an $(\eta,M)$-Diophantine lattice for some $0<\eta<1$ and $M>1$. Then there exists
$\theta>0$ such that
$$
\sup_{t\ge0}
\int_K
\widehat\alpha_{\eta,M}(a_tk;\Lambda)^{1+\theta}\,dk
<\infty.
$$
\end{theorem}

For a lattice $\La<\M_n(\br)$, define
\be\label{iso}
\Lambda_{\operatorname{iso}}
:=
\bigcup_{V}\,\bigl(\Lambda\cap V\bigr),
\ee
where $V<\M_n(\br)$ ranges over all $\Lambda$-rational column- and row-isotropic subspaces. Thus $\Lambda_{\operatorname{iso}}$ is the set of lattice points lying in such a subspace. 

The singular asymptotic requires the {\it isotropic noncoincidence} condition. For a lattice $\Delta<\M_n(\mathbb R)$ and
a $\Delta$-rational subspace $V<\M_n(\mathbb R)$, define the quotient
lattice
$$
\Delta_V:=p_{V^\perp}(\Delta)<V^\perp,
$$
where $p_{V^\perp}$ denotes orthogonal projection onto $V^\perp<\M_n(\br)$. Since the kernel of $p_{V^\perp}|_\Delta$ is $\Delta\cap V$, the
projected lattice $\Delta_V$ is naturally identified with the quotient $\Delta/(\Delta\cap V)$.
\begin{definition}
\label{def:isotropic-noncoincidence}
We say that $\Delta$ satisfies the \emph{column isotropic noncoincidence
condition} if, for every
$\lceil n/2\rceil\le k\le n-2$, there do not exist distinct proper
$\Delta$-rational column-isotropic subspaces $V_1,V_2$, each maximal
with respect to inclusion among all proper $\Delta$-rational
column-isotropic subspaces, such that
$\dim V_1=\dim V_2=kn$
and, for some $h\in\GL_n(\mathbb R)$,
$$
h\Delta_{V_1}
\quad\text{is commensurable with}\quad
\Delta_{V_2}.
$$

We define the \emph{row isotropic noncoincidence condition} analogously,
with $\Delta_{V_1}h$ in place of $h\Delta_{V_1}$. We say that $\Delta$
satisfies the \emph{isotropic noncoincidence condition} if it satisfies
both the column and row conditions.

\end{definition}
The isotropic noncoincidence condition is generic in the
measure-theoretic sense. Haar-almost every lattice
$\Delta\in X$ satisfies
$$
\Delta\cap\{v\in\M_n(\R):\det v=0\}=\{0\}.
$$
To see this, write $\Delta=g\mathbb Z^N$. For each
$0\ne z\in\mathbb Z^N$ the equation $\det(gz)=0$ defines a proper
algebraic subset of $\SL_N(\R)$; taking the countable union over $z$
proves the assertion.
For such a lattice, there are no nonzero $\Delta$-rational isotropic
subspaces, so the isotropic noncoincidence condition holds vacuously.

For $n=3$, we have $\lceil n/2\rceil=2$ and $n-2=1$ and hence
the above noncoincidence condition is vacuous. Moreover, every diagonal lattice satisfies the noncoincidence condition; see Example
\ref{ex:diagonal-isotropically-noncoincident}.

\begin{theorem}\label{thm:quasinullcontribution2}
Let $\Lambda<\M_n(\mathbb R)$ be a Diophantine lattice that is not of
$\mathbb Q$-split type and satisfies the
isotropic noncoincidence condition. Then the limit
$$
c_{\Lambda,\mathrm{iso}}
:=
\lim_{T\to\infty}
T^{-n(n-1)}
\#\{v\in\Lambda_{\operatorname{iso}}:\|v\|< T\}\geq 0
$$
exists and is finite. Moreover, $c_{\Lambda,\mathrm{iso}}>0$ if and
only if $\Lambda$ contains a rank-$n(n-1)$ submodule on which the
determinant vanishes identically.
\end{theorem}
We will show later that the determinant-zero points outside $\Lambda_{\operatorname{iso}}$ contribute $o(T^{n(n-1)})$. Consequently,
$c_{\Lambda,\mathrm{iso}}$  agrees with the constant
$c_\Lambda^{\operatorname{sing}}$ in
\eqref{eq:intro-singular-asymptotic}; see Subsection \ref{isoo}.

\section{Sublevel estimates}
\label{s:sublevel-general-n}
The goal of this section is to prove the uniform sublevel estimate needed for
the contraction inequalities in
\Cref{prop:generallocalcontraction}. Working with one Schur module
${\mathsf S}_\lambda(\mathbb R^n)$ at a time, we reduce the problem to the
first nonnegative-weight coordinate of a horospherical translate. This
coordinate is a vector-valued polynomial, which may have degree greater
than one when $n\ge3$.

Using the semistandard Young tableau model for Schur modules, we relate the
$b_t$-weight of a vector to the number of independent horospherical
directions that act nontrivially on it.  This gives lower bounds for the
codimension of the space of directions annihilating the vector, which in
turn control the complex singularity exponents of the associated homogeneous
polynomials.  We then use stability of complex singularity exponents and the
uniform negative-moment result of \Cref{app:analytic-stability}.

Recall the notation $b_t$ from \eqref{btdef}, and let $U_0<\SL_n(\br)$ be
its expanding horospherical subgroup:
$$
U_0 =
\left\{
u_\xi:=
\begin{pmatrix}
\operatorname{Id}_{n-1} & 0\\
\xi & 1
\end{pmatrix}
:\xi\in\R^{n-1}
\right\}.
$$
We identify its Lie algebra with
$$
\mathfrak u_0
\simeq
\sum_{1\le j\le n-1}\R E_{nj}
\simeq
\R^{n-1},
$$
and for $\xi\in\R^{n-1}$, we write
$$
Y_\xi:=\log u_\xi=\sum_{j=1}^{n-1}\xi_jE_{nj}\in\mathfrak u_0.
$$
Thus $u_\xi=\exp(Y_\xi)$, and the map $\xi\mapsto Y_\xi$ is linear.
For $1\le i\le n-1$, put
$$
Y_i:=Y_{e_i}=E_{ni}.
$$
In this section, we fix
$$
V={\mathsf S}_\lambda(\R^n),
$$
where $\lambda=(\lambda_1, \cdots,  \lambda_n)$ is a partition with
$\lambda_1\ge  \cdots\ge  \lambda_n\ge 0 $  whose Young diagram fits inside the
$n\times n$ box. Since we work with $\SL_n(\R)$-representations, we may
replace $\lambda$ by the equivalent partition $(\lambda_1-\lambda_n \ge \cdots \ge \lambda_n-\lambda_n)$, and  therefore assume throughout that
\begin{equation}\label{eq:lambda-reduced}
\lambda_n=0.
\end{equation} In particular, $|\lambda|\le n(n-1)$.
We assume that $$\lambda\ne0.$$  Fix an inner product
on $V$ for which the weight spaces below are orthogonal. We use the same
symbol $Y_\xi$ for the derived action of $Y_\xi\in\mathfrak u_0$ on $V$.
The Young diagram of $\lambda$ has at most $n-1$ nonzero
rows and at most $n$ columns.

We shall occasionally use the standard semistandard-tableau model for
Schur modules; see, for example, \cite[Chapter~I]{FultonYT} or
\cite[Section~4.2]{FultonHarris}. Recall that a semistandard Young
tableau of shape $\lambda$ is a filling of the Young diagram of $\lambda$
with entries in $\{1,\ldots,n\}$ that are non-decreasing along rows and
strictly increasing down columns. Such tableaux index a weight
basis of ${\mathsf S}_\lambda(\mathbb C^n)$, with the weight determined by
the number of occurrences of each entry. In particular, if a tableau
contains $i$ entries equal to $n$, then its $b_t$-weight is
$-|\lambda|+ni$. We will also use the immediate consequence of column
strictness that any fixed label occurs at most once in each column.

Set
$$
\mu_i:=-|\lambda|+ni 
\quad (0\le i\le \lambda_1),
$$
and let $V_i$ be the $\log b_1$-eigenspace of weight $\mu_i$. Note that $0$ is a weight for $\log b_1$ if and only if $n\mid |\lambda|$. For a nonzero
vector $w\in V_i$, put $$\operatorname{wt}(w):=\mu_i.$$ We have
$$
V=\bigoplus_{i=0}^{\lambda_1}V_i.
$$
Indeed, in the semistandard-tableau model, the index $i$ is the number of
entries equal to $n$. Moreover,
$$
[\log b_1,E_{nj}]=nE_{nj},
\quad\text{
so } \quad Y_\xi V_i\subset V_{i+1} \quad \text{for all $\xi\in\R^{n-1}$ and $0\le i<\lambda_1$}.$$

Let
$$
L=
\begin{pmatrix}
\SL_{n-1}(\R)&0\\
0&1
\end{pmatrix}.
$$
Since $L$ commutes with $b_t$, every $V_i$ is $L$-invariant. Let
$$
p_i:V\to V_i
$$
denote the projection onto $V_i$. Because $\mathfrak u_0$ is abelian, for every $v\in V$ and $\xi\in\R^{n-1}$,
$$
u_\xi v=\sum_{k\ge 0}\frac{1}{k!}Y_\xi^k v.
$$
Since $Y_\xi^k$ maps $V_i$ to $V_{i+k}$, we obtain, for
$0\le m\le \lambda_1$,
\begin{equation}\label{eq:polynomialcoordinate}
p_m(u_\xi v)
=
\sum_{i=0}^{m}\frac{1}{(m-i)!}Y_\xi^{m-i}p_i(v).
\end{equation}
The formula is triangular with respect to the weight decomposition: the
$m$-th component depends only on $p_0(v),\ldots,p_m(v)$. In particular,
$p_m(u_\xi v)$ is a polynomial in $\xi$ of degree at most $m$.

Define
\begin{equation}\label{alphalambda}
\alpha_\lambda:=
\begin{cases}
\dfrac{n-1}{n-2}
&\text{if }n\mid |\lambda| \text{ and }\lambda\neq (n^k),\ 1\le k\le n-1\\ 
1
&\text{otherwise}.
\end{cases}
\end{equation}

\begin{remark}\label{rem:sublevel-n2} The first case can occur only when $n\ge3$.
For $n=2$, the only nontrivial $\lambda$ with $2\mid |\lambda|$ is
$\lambda=(2)$.
\end{remark}

The goal of this section is to prove the following uniform sublevel estimate.

\begin{proposition}[Sublevel estimate]\label{prop:sublevelestimate}
Let $1\le \ell_\lambda\le \lambda_1$ be the
smallest integer such that $\mu_{\ell_\lambda}\ge 0$.
For every bounded open subset $\Omega\subset \R^{n-1}$ and $0<\delta<\alpha_\lambda$, there exists $C=C(V,\Omega,\delta)>0$ such that, for
every $v\in V$ with $\max_{0\le i\le \ell_{\lambda}}\|p_i(v)\|>0$ and $0<\varepsilon<1$,
\be\label{dt}
\operatorname{Leb}
\Bigl(
\bigl\{\xi\in\Omega:\|p_{\ell_{\lambda}}(u_\xi v)\|\le \varepsilon \max_{0\le i\le \ell_{\lambda}}\|p_i(v)\|\bigr\}
\Bigr)
\le
C\varepsilon^{\alpha_\lambda-\delta}
\ee
where $\operatorname{Leb}$ denotes the Lebesgue measure on $\mathbb R^{n-1}$.
\end{proposition}

\subsection{A weight estimate and a codimension bound}
\label{s:parabolic-codim}
In this subsection, set
\[
V_i^{\mathbb C}:=V_i\otimes_{\mathbb R}\mathbb C,
\]
and continue to write $Y_\xi$, $p_i$, and $\operatorname{wt}$ for the
complexified action, projections, and weight function, respectively.
We identify
$\mathfrak u_0^{\mathbb C}$ with $\mathbb C^{n-1}$ as above.

For $0\ne v\in\mathsf S_\lambda(\mathbb C^n)$, define the linear
subspace
$$
\mathcal Z^{\mathbb C}(v)
:=
\{\xi\in\mathbb C^{n-1}:Y_\xi v=0\}.
$$
We write
$$
\operatorname{codim}_{\mathbb C}\mathcal Z^{\mathbb C}(v)
:=
(n-1)-\dim_{\mathbb C}\mathcal Z^{\mathbb C}(v).
$$

If $v$ is a $\log b_1$-weight vector of negative weight, define
\begin{equation}\label{eq:def-mv}
\mathsf m(v)
:=
\min\bigl\{m\in\mathbb Z_{\ge0}:
\operatorname{wt}(v)+nm\ge0\bigr\}
=
\left\lceil\frac{-\operatorname{wt}(v)}{n}\right\rceil.
\end{equation}
We call $\mathsf m(v)$ the $U_0$-raising depth of $v$.

Recall that scalar matrices $e^tI_n$ act on
$\mathsf S_\lambda(\mathbb C^n)$ by multiplication by
$e^{t|\lambda|}$. Thus the infinitesimal action of
$I_n=\sum_{i=1}^nE_{ii}
$
is $|\lambda|\operatorname{Id}$. Consequently, using the same notation
for elements of $\mathfrak{gl}_n$ and their induced actions on
$\mathsf S_\lambda(\mathbb C^n)$, we have
$$
\log b_1
=
-\sum_{i=1}^{n-1}E_{ii}+(n-1)E_{nn}
= -I_n + n E_{nn}=
-|\lambda|\operatorname{Id}+nE_{nn}.
$$
It follows that every $\log b_1$-weight vector satisfies
$$
E_{nn}v=qv,
\qquad
\operatorname{wt}(v)=-|\lambda|+nq
$$
for some $q\in\mathbb Z_{\ge0}$. Hence, for a negative-weight vector,
\begin{equation}\label{eq:weight-divisibility}
\operatorname{wt}(v)=-n\mathsf m(v)
\quad\Longleftrightarrow\quad
n\mid|\lambda|.
\end{equation}

Since $|\lambda|\le n(n-1)$, we have
$1\le\mathsf m(v)\le n-1$.
Moreover,
\begin{equation}\label{mv}
\mathsf m(v)\le n-2
\quad\text{if }n\mid|\lambda|
\text{ and }\lambda\ne(n^{n-1}).
\end{equation}
Indeed, if $n\mid|\lambda|$ and $\mathsf m(v)=n-1$, then
\eqref{eq:weight-divisibility} gives
$
\operatorname{wt}(v)=-n(n-1).
$
Since
$$
\operatorname{wt}(v)\ge-|\lambda|\ge-n(n-1),
$$
equality holds throughout. Thus $|\lambda|=n(n-1)$, which forces
$\lambda=(n^{n-1})$.

\begin{lemma}[Codimension bound]
\label{lem:codimensionbound}
Let $0\ne v\in\mathsf S_\lambda(\mathbb C^n)$ be a
$\log b_1$-weight vector of negative weight. Then
\begin{equation}\label{eq:kernel-weight-bound}
\operatorname{codim}_{\mathbb C}\mathcal Z^{\mathbb C}(v)
\ge
\mathsf m(v)
=
\left\lceil\frac{-\operatorname{wt}(v)}{n}\right\rceil.
\end{equation}
If $n\mid|\lambda|$ and equality holds, then
$\lambda=(n^{\mathsf m(v)}).
$
Consequently, if $n\mid|\lambda|$ and
$\lambda\ne(n^{\mathsf m(v)})$, then
$$
\operatorname{codim}_{\mathbb C}\mathcal Z^{\mathbb C}(v)
\ge\mathsf m(v)+1.
$$

Moreover, the polynomial
\begin{equation}\label{eq:critical-power-nonzero}
\xi\mapsto Y_\xi^{\mathsf m(v)}v
\quad\text{is not identically zero}.
\end{equation}
\end{lemma}

\begin{proof}
Set
$
\mathsf d:=\dim_{\mathbb C}\mathcal Z^{\mathbb C}(v)$
and $\mathsf c:=n-1-\mathsf d.
$
Let
$$
L_{\mathbb C}
=
\left\{
\begin{pmatrix}
h&0\\
0&1
\end{pmatrix}
:h\in\SL_{n-1}(\mathbb C)
\right\}.
$$
Since $L_{\mathbb C}$ commutes with $b_1$ and acts transitively on
the $\mathsf d$-dimensional subspaces of $\mathbb C^{n-1}$, after
replacing $v$ by an $L_{\mathbb C}$-translate, we may assume that
$$
\mathcal Z^{\mathbb C}(v)
=
\operatorname{span}_{\mathbb C}\{e_1,\ldots,e_{\mathsf d}\}.
$$
Thus
$
Y_i v=0$ for all $1\le i\le\mathsf d$.

Choose a $U(n)$-invariant Hermitian form on
$\mathsf S_\lambda(\mathbb C^n)$. Then $Y_i=E_{ni}$ has adjoint
$Y_i^*=E_{in}$. Write
$
E_{nn}v=qv$ and
$\operatorname{wt}(v)=-|\lambda|+nq. $

Set
$$
A:=\sum_{i=1}^{\mathsf d}E_{ii},
\qquad
P:=\sum_{j=\mathsf d+1}^{n-1}E_{jj}.
$$
Both $A$ and $P$ are nonnegative self-adjoint operators. In the
semistandard-tableau model, $P$ counts entries with labels
$\mathsf d+1,\ldots,n-1$. There are $\mathsf c$ such labels, and
each can occur at most once in each column. Since the Young diagram of
$\lambda$ has $\lambda_1\le n$ columns,
$$
0\le P\le\mathsf c\lambda_1\operatorname{Id}
\le n\mathsf c\operatorname{Id}.
$$

Using
$[Y_i^*,Y_i]=E_{ii}-E_{nn}$
and $Y_i v=0$, we obtain
$$
\mathsf d q\|v\|^2-\langle Av,v\rangle
=
\sum_{i=1}^{\mathsf d}\|Y_i^*v\|^2
\ge0.
$$
Since
$A+P+E_{nn}=|\lambda|\operatorname{Id}$ and $
\mathsf d+\mathsf c=n-1,$
it follows that
$$
\begin{aligned}
\bigl(\operatorname{wt}(v)+n\mathsf c\bigr)\|v\|^2
={}&
\sum_{i=1}^{\mathsf d}\|Y_i^*v\|^2
+\mathsf c q\|v\|^2 +
\langle(n\mathsf c\operatorname{Id}-P)v,v\rangle
\ge0.
\end{aligned}
$$
Thus
\begin{equation}\label{eq:auxiliary-weight-bound}
\operatorname{wt}(v)\ge-n\mathsf c.
\end{equation}
Since $\mathsf c$ is an integer, this gives
$\mathsf c
\ge
\left\lceil\frac{-\operatorname{wt}(v)}{n}\right\rceil
= \mathsf m(v)$,
proving \eqref{eq:kernel-weight-bound}.

Suppose now that $n\mid|\lambda|$ and equality holds in
\eqref{eq:kernel-weight-bound}. By
\eqref{eq:weight-divisibility},
$$
\operatorname{wt}(v)
=
-n\mathsf m(v)
=
-n\mathsf c.
$$
 So every
nonnegative term in the preceding identity vanishes. Since
$\mathsf c=\mathsf m(v)\ge1$, we obtain
$$
q=0,\qquad Av=0,\qquad Pv=n\mathsf c\,v.
$$
Thus every semistandard tableau occurring in $v$ uses only the
$\mathsf c$ labels
$$
\mathsf d+1,\ldots,n-1,
$$
and the Young diagram of $\lambda$ has
$|\lambda|=n\mathsf c$
boxes. By column strictness, each of these $\mathsf c$ labels can occur at
most once in a given column, so each column contains at most $\mathsf c$
boxes. Since the diagram has at most $n$ columns, the equality
$|\lambda|=n\mathsf c$ forces it to have exactly $n$ columns, each
containing exactly $\mathsf c$ boxes. Hence
$$
\lambda=(n^{\mathsf c})=(n^{\mathsf m(v)}).
$$
Therefore, if $\lambda\ne(n^{\mathsf m(v)})$, equality in the codimension
bound is impossible; since the codimension is an integer, it is at least
$\mathsf m(v)+1$.

It remains to prove \eqref{eq:critical-power-nonzero}. Set
$\mathsf m:=\mathsf m(v)$, and suppose that $Y_\xi^{\mathsf m}v\equiv0$.
For fixed $\xi_1,\ldots,\xi_{\mathsf m}\in\mathbb C^{n-1}$, the polynomial
$Y_{t_1\xi_1+\cdots+t_{\mathsf m}\xi_{\mathsf m}}^{\mathsf m}v$ vanishes
identically in $t_1,\ldots,t_{\mathsf m}$. Since the operators $Y_\xi$
commute, the coefficient of $t_1\cdots t_{\mathsf m}$ is
$\mathsf m!Y_{\xi_1}\cdots Y_{\xi_{\mathsf m}}v$. Hence
$$
Y_{\xi_1}\cdots Y_{\xi_{\mathsf m}}v=0
$$
for all $\xi_1,\ldots,\xi_{\mathsf m}\in\mathbb C^{n-1}$. Choose $k<\mathsf m$
maximal such that
$$
w:=Y_{\xi_1}\cdots Y_{\xi_k}v\ne0
$$
for some $\xi_1,\ldots,\xi_k\in\mathbb C^{n-1}$. Then $Y_\eta w=0$ for every
$\eta\in\mathbb C^{n-1}$, so
$
\mathcal Z^{\mathbb C}(w)=\mathbb C^{n-1}.
$
The argument proving \eqref{eq:auxiliary-weight-bound}, applied to $w$
with $\mathsf c=0$, therefore gives
$\operatorname{wt}(w)\ge0$.
On the other hand,
$$
\operatorname{wt}(w)
=
\operatorname{wt}(v)+nk
\le
\operatorname{wt}(v)+n(\mathsf m-1)
<0,
$$
a contradiction.
\end{proof}

\subsection{Complex singularity exponents and uniform negative moments}
\label{s:analytic-stability}

We use complex singularity exponents to obtain uniformity in the polynomial
families arising below. If
$$
R=(R_1,\ldots,R_q):(\mathbb C^\ell,z_0)\to\mathbb C^q
$$
is a germ at $z_0$ of a holomorphic map, define
$$
{c}_{z_0}(R)
:=
\sup\left\{c>0:
\left(\sum_{j=1}^q|R_j|^2\right)^{-c}
\text{ is locally integrable near }z_0
\right\}.
$$
Here and below, we set $c_{z_0}(R)=0$ when $R$ is the zero germ. Thus
$c_{z_0}(R)>0$ if and only if $R$ is not the zero germ at $z_0$.
If $R(z_0)\ne 0$, then ${c}_{z_0}(R)=+\infty$. If $R(z_0)=0$ but $R$ is not the zero germ, then $0<c_{z_0}(R)<+\infty$.
This is the complex singularity exponent of $R$, or equivalently the
log canonical threshold at $z_0$ of the ideal generated by its coordinates $R_1,\ldots,R_q$.

We use the following consequence of \Cref{lem:complex-stability-slicing}(1)
in Appendix~A, which compares the complex singularity exponent with that
of the first nonzero homogeneous Taylor term. If
$$
R(z_0+z)
=
R_m^{\mathrm{hom}}(z)+R_{m+1}^{\mathrm{hom}}(z)+\cdots,
\qquad
R_m^{\mathrm{hom}}\ne0,
$$
where $R_j^{\mathrm{hom}}$ denotes the homogeneous component of degree $j$
in the Taylor expansion of $R$ at $z_0$,
then
\begin{equation}\label{eq:lct-initial-degeneration}
c_{z_0}(R)\ge c_0(R_m^{\mathrm{hom}}).
\end{equation}
Indeed, the formula
$$
F(r,z):=
\begin{cases}
r^{-m}R(z_0+rz),&r\ne0,\\
R_m^{\mathrm{hom}}(z),&r=0,
\end{cases}
$$
defines a holomorphic family near $(0,0)$. For $r\ne0$, multiplication by
the nonzero scalar $r^{-m}$ and the biholomorphic change of variables
$z\mapsto z_0+rz$ give
$$
c_0(F_r)=c_{z_0}(R),
$$
where $F_r:=F(r,\cdot)$. Applying \Cref{lem:complex-stability-slicing}(1) at $r=0$ and letting the
persistent exponent tend to $c_0(F_0)=c_0(R_m^{\mathrm{hom}})$ proves
\eqref{eq:lct-initial-degeneration}.

The following uniformity statement follows  immediately from
\Cref{prop:multiparameter-stability} in Appendix~A and compactness of the
parameter set. For each fixed parameter $a_0$, apply that
proposition to
$\mathcal Q(a_0+h,x)$. The uniform lower bound on the complex thresholds
gives a neighborhood of $a_0$ on which the negative moments are uniformly
bounded. 
\begin{lemma}[Uniform negative moments from complex thresholds]
\label{lem:uniform-negative-moments-complex}
Let $m, \ell, q \ge1$, let $A\subset\mathbb R^m$ be compact and let
$B\subset\mathbb R^\ell$ be bounded and open. Let
$\mathcal{Q}:\mathbb R^m\times\mathbb R^\ell\to \mathbb R^q $
be a polynomial map. Denote by
$\mathcal{Q}_a^{\mathbb C}$ the complexification of $\mathcal{Q}_a(x):=\mathcal{Q}(a,x)$. Let $\alpha>0$ and suppose that
\begin{equation}\label{eq:uniform-complex-threshold-hypothesis}
\inf_{a\in A} c_z(\mathcal{Q}_a^{\mathbb C})\ge\alpha
\end{equation}
for every $z$ in some complex neighborhood of $\overline{B}$. Then, for every $0<\beta<\alpha$,
\begin{equation}\label{eq:uniform-negative-moment-CM}
\sup_{a\in A}\int_B\|\mathcal{Q}_a(x)\|^{-\beta}\,dx<\infty.
\end{equation}
\end{lemma}

\subsection{Sublevel estimate for weight vectors}
\label{s:critical-weight-vector}

We first establish a complex singularity-exponent bound for the polynomials $\xi\mapsto Y_\xi^{\mathsf m(v)} v$ attached to $\log b_1$-weight vectors; the real
sublevel estimate will follow immediately.

\begin{proposition}[Complex threshold for the polynomial $\xi\mapsto Y_\xi^{\mathsf m(v)} v$]
\label{prop:critical-weight-complex-threshold}
For any negative $\log b_1$-weight vector $0\ne v\in{\mathsf S}_\lambda(\mathbb C^n)$, we have
\begin{equation}\label{eq:critical-complex-threshold}
c_z\bigl(\xi\mapsto Y_\xi^{\mathsf m(v)} v\bigr)\ge\alpha_{\lambda}
\qquad(z\in\mathbb C^{n-1})
\end{equation}
where $\alpha_\lambda$ is defined by \eqref{alphalambda}.
\end{proposition}

\begin{proof}
Set $P_v(\xi):=Y_\xi^{\mathsf m(v)} v$. We argue by induction on $\mathsf m(v)$. By
\Cref{lem:codimensionbound}, $P_v$ is not identically zero and hence is
not the zero germ at any point. Fix a
zero $z$ of $P_v$.
Suppose first that $z\notin\mathcal Z^{\mathbb C}(v)$. This case is empty
when $\mathsf m(v)=1$. For $\mathsf m(v) \ge2$, let $j$ be the largest integer with
$0\le j\le \mathsf m(v)-1$ such that $Y_z^jv\ne0$, and put
$$
q:=\mathsf m(v)-j,
\qquad w:=Y_z^jv.
$$
Then $1\le q<\mathsf m(v)$, and the weight identities give
$\mathsf m(w)=q$. By the nonvanishing conclusion of
\Cref{lem:codimensionbound}, $Y_\xi^qw\not\equiv0$. Since the operators
$Y_\xi$ commute, the first nonzero homogeneous Taylor
term of $P_v$ at $z$ is
$$
\binom {\mathsf m(v)} q Y_\xi^qw.
$$
The induction hypothesis and \eqref{eq:lct-initial-degeneration} therefore
give
$$
c_z(P_v)\ge c_0(\xi\mapsto Y_\xi^qw)\ge\alpha_{\lambda}.
$$
It remains to consider $z\in\mathcal Z^{\mathbb C}(v)$. Since
$Y_zv=0$ and the operators $Y_\xi$ commute, one has
$
P_v(z+\xi)=P_v(\xi)$.
Thus $P_v$ is constant along the affine subspaces parallel to
$\mathcal Z^{\mathbb C}(v)$, and it suffices to work at the origin. Choose a complex-linear complement $\mathbb C^{n-1}
=
\mathcal Z^{\mathbb C}(v)\oplus\mathcal N$ and put
$$
\mathsf{c}:=\dim_{\mathbb C}\mathcal N
=
\operatorname{codim}_{\mathbb C}\mathcal Z^{\mathbb C}(v).
$$
Writing $\operatorname{pr}_{\mathcal N}$ for the projection associated with
this decomposition, we have
$$
P_v
=
(P_v|_{\mathcal N})\circ\operatorname{pr}_{\mathcal N}.
$$
The variables in $\mathcal Z^{\mathbb C}(v)$ therefore do not occur in
$P_v$. By Fubini's theorem, adjoining these variables does not
change the local complex singularity exponent; in particular,
$$
c_0(P_v)=c_0(P_v|_{\mathcal N}).
$$
We may consequently regard $P_v|_{\mathcal N}$ as a nonzero homogeneous
holomorphic map on $\mathbb C^{\mathsf{c}}$.

We use the following elementary fact about homogeneous maps. Let
$P:\mathbb C^{\mathsf{c}}\to\mathbb C^s$ be a nonzero homogeneous holomorphic map of
degree $\mathsf m(v)$. Then
\begin{equation}\label{eq:lct-homogeneous-map}
c_0(P)
=
\min\left\{
\frac {\mathsf{c}}{\mathsf m(v)},
\inf_{\substack{y\in\mathbb C^{\mathsf{c}}\smallsetminus\{0\},\, P(y)=0}}
c_y(P)
\right\},
\end{equation}
where the infimum over the empty set is understood to be $+\infty$.

Indeed, near any direction in $\mathbb P^{\mathsf{c}-1}(\mathbb C)$, after permuting
coordinates if necessary, write $x=t(1,w)\in \mathbb C^{\mathsf c}$ with $t\in\mathbb C$ and $w\in\mathbb C^{\mathsf{c}-1}$. By homogeneity,
$$
P(t(1,w))=t^{\mathsf m(v)}P(1,w),
$$
while the Euclidean volume form satisfies $dx\asymp |t|^{2\mathsf{c}-2}\,dt\,dw$. Hence, by Fubini's theorem, $\|P(x)\|^{-2\gamma}$ is locally integrable in such a chart precisely when
$\gamma<\frac{\mathsf{c}}{\mathsf m(v)}$
and $\|P(1,w)\|^{-2\gamma}$ is locally integrable in the $w$-variables.
For a nonzero point $y=(1,w)$, the radial coordinate is nonvanishing, so
the latter condition is equivalent to $\gamma<c_y(P)$. Covering
$\mathbb P^{\mathsf{c}-1}(\mathbb C)$ by its finitely many standard affine charts
gives \eqref{eq:lct-homogeneous-map}.
We now apply \eqref{eq:lct-homogeneous-map} to $P=P_v|_{\mathcal N}$. Every nonzero zero of $P_v|_{\mathcal N}$ lies outside
$\mathcal Z^{\mathbb C}(v)$, so the first part of the proof bounds the
second term in \eqref{eq:lct-homogeneous-map} from below by $\alpha_{\lambda}$.
By \Cref{lem:codimensionbound}, $\mathsf{c}\ge \mathsf m(v)$. Hence
$\mathsf{c}/\mathsf m(v)\ge1=\alpha_{\lambda}$ when $\alpha_{\lambda}=1$. If
$\alpha_{\lambda}=(n-1)/(n-2)$, the strengthened conclusion of
\Cref{lem:codimensionbound} and \eqref{mv} give
$$
\mathsf{c}\ge \mathsf m(v)+1,
\qquad \mathsf m(v)\le n-2,
$$
and consequently
$$
\frac{\mathsf{c}}{\mathsf m(v)}\ge1+\frac{1}{\mathsf m(v)}
\ge\frac{n-1}{n-2}=\alpha_{\lambda}.
$$
Equation \eqref{eq:lct-homogeneous-map} now proves
\eqref{eq:critical-complex-threshold}.
\end{proof}

\begin{lemma}[Sublevel estimate for weight vectors]
\label{lem:critical-weight-sublevel}
Let $0\ne v\in{\mathsf S}_\lambda(\mathbb R^n)$ be a negative $\log b_1$-weight vector. For
every bounded open set $\Omega\subset\mathbb R^{n-1}$ and $0<\delta<\alpha_\lambda$,
there exists $C=C(\delta,\Omega,v)>0$ such that for all $0<\varepsilon<1$,
$$
\operatorname{Leb}
\left\{\xi\in\Omega:
\|Y_\xi^{\mathsf m(v)}v\|\le\varepsilon\|v\|\right\}
\le C\varepsilon^{\alpha_{\lambda}-\delta}.
$$

\end{lemma}

\begin{proof}
By \Cref{prop:critical-weight-complex-threshold} and
\Cref{lem:uniform-negative-moments-complex}, applied with a singleton
parameter set, for every $0<\beta<\alpha_{\lambda}$, one has
\begin{equation}\label{eq:critical-negative-moment}
\int_\Omega
\|Y_\xi^{\mathsf m(v)}v\|^{-\beta}\,d\xi<\infty.
\end{equation}
Choose $\alpha_{\lambda}-\delta<\beta<\alpha_{\lambda}$. Chebyshev's inequality then gives
$$
\operatorname{Leb}
\left\{\xi\in\Omega:
\|Y_\xi^{\mathsf m(v)}v\|\le\varepsilon\|v\|\right\}
\le
(\varepsilon\|v\|)^\beta
\int_\Omega
\|Y_\xi^{\mathsf m(v)}v\|^{-\beta}\,d\xi
\ll_{\delta,\Omega,v}
\varepsilon^{\alpha_{\lambda}-\delta}.
$$
\end{proof}

\subsection{Triangular $U_0$-polynomials}
\label{s:triangular-U-polynomials}
We next pass from the homogeneous polynomials $\xi\mapsto Y_\xi^{\mathsf{m}(v)} v$ to the polynomial expressions in \eqref{eq:polynomialcoordinate}, which have a triangular structure with respect to the weight decomposition.

\begin{proposition}[Complex threshold for triangular $U_0$-polynomials]
\label{prop:triangular-U-local-integrability}
Let $\ell_\lambda$ be as in \Cref{prop:sublevelestimate}. Let
$$
Q(\xi)=
\sum_{q=0}^{\ell_\lambda}\frac1{q!}Y_\xi^qv_{\ell-q},
\qquad v_j\in V_j\otimes_{\mathbb R}\mathbb C,
$$
and suppose that $Q\not\equiv0$. Then
\begin{equation}\label{eq:triangular-complex-threshold}
\inf_{z\in \mathbb C^{n-1}} c_z(Q)\ge\alpha_\lambda.
\end{equation}
In particular, when the coefficients $v_j$ are real, $\|Q\|^{-\beta}$ is
locally integrable on $\mathbb R^{n-1}$ for every
$0<\beta<\alpha_\lambda$.
\end{proposition}

\begin{proof}
Let $\ell=\ell_{\lambda}$ for simplicity. It suffices to prove $ c_{z_0}(Q)
\ge\alpha_\lambda$ for every zero $z_0$ of $Q$. Since $U_0$ is abelian,
$$
Q(z_0+\xi)
=
\sum_{q=0}^{\ell}\frac1{q!}Y_\xi^qw_{\ell-q},
$$
where
$
w_j:=p_j\!\left(u_{z_0}\sum_{i=0}^{\ell}v_i\right).
$
Here the action and the projections have been complexified. Since
$Q(z_0)=0$, one has $w_\ell=0$. Since $Q\not\equiv0$, its translate
$$
\xi\mapsto Q(z_0+\xi)
$$
is not identically zero. Hence there exists $1\le q\le\ell$ such that
$w_{\ell-q}\ne0$; let $q$ be the smallest such integer.
The terms of degree less than $q$ vanish, and the first
nonzero homogeneous Taylor term is
\begin{equation}\label{eq:triangular-initial-form}
\frac1{q!}Y_\xi^qw_{\ell-q}.
\end{equation}
Indeed, $w_{\ell-q}$ has weight $\mu_\ell-nq$, and
$0\le\mu_\ell<n$, so $\mathsf m(w_{\ell-q})=q$; \Cref{lem:codimensionbound}, applied to $w_{\ell-q}$, then shows that the map in
\eqref{eq:triangular-initial-form} is nonzero.
\Cref{prop:critical-weight-complex-threshold} and
\eqref{eq:lct-initial-degeneration} yield
$$
c_{z_0}(Q)
\ge c_0(\xi\mapsto Y_\xi^qw_{\ell-q})
\ge\alpha_\lambda,
$$
as required.
If the coefficients are real, let $B_0\subset\mathbb R^{n-1}$ be any
bounded open ball. By \eqref{eq:triangular-complex-threshold} and
\Cref{lem:uniform-negative-moments-complex}, applied with a singleton
parameter set,
$$
\int_{B_0}\|Q(\xi)\|^{-\beta}\,d\xi<\infty
\qquad(0<\beta<\alpha_\lambda).
$$
Thus $\|Q\|^{-\beta}$ is locally integrable on $\mathbb R^{n-1}$.
\end{proof}

\subsection{Uniformity in the normalized family and proof of the sublevel estimate}
\label{s:proof-sublevel-theorem}
\begin{proof}[Proof of \Cref{prop:sublevelestimate}]
Let $\ell=\ell_\lambda$, and choose a bounded open ball
$B\subset\mathbb R^{n-1}$ such that $\overline\Omega\subset B$. Set
$
A_\ell(v):=\max_{0\le i\le\ell}\|p_i(v)\|.
$
By homogeneity in $v$, it suffices to consider the normalized case $A_\ell(v)=1$.
Write $v_i:=p_i(v)$ for $0\le i\le\ell$. By
\eqref{eq:polynomialcoordinate},
\begin{equation}\label{eq:Pv-expansion-sublevel}
P_v(\xi):=p_\ell(u_\xi v)
=
\sum_{q=0}^{\ell}\frac{1}{q!}Y_\xi^qv_{\ell-q}.
\end{equation}

Let
$$
\mathcal A_\ell
:=
\left\{
(v_0,\ldots,v_\ell)\in V_0\oplus\cdots\oplus V_\ell:
\max_{0\le i\le\ell}\|v_i\|=1
\right\}.
$$
This is a compact subset of $V_0\oplus\cdots\oplus V_\ell$. Consider
the polynomial map
$$
\widetilde P:
(V_0\oplus\cdots\oplus V_\ell)\times\mathbb R^{n-1}
\to V_\ell,
\qquad
\widetilde P((v_0,\ldots,v_\ell),\xi)
=
\sum_{q=0}^{\ell}\frac{1}{q!}Y_\xi^qv_{\ell-q}.
$$

We first claim that for every parameter in $\mathcal A_\ell$,
the corresponding polynomial 
 is not identically zero. Fix
$(v_0,\ldots,v_\ell)\in\mathcal A_\ell$, and let $i_0$ be the
smallest index such that $v_{i_0}\ne0$. The highest-degree homogeneous
term of the corresponding polynomial is
$$
\frac{1}{(\ell-i_0)!}Y_\xi^{\ell-i_0}v_{i_0}.
$$
If $i_0=\ell$, this is the nonzero constant $v_\ell$. If
$i_0<\ell$, then $v_{i_0}$ has negative weight and
$$
\mathsf m(v_{i_0})=\ell-i_0,
$$
because $0\le\mu_\ell<n$. The nonvanishing conclusion of
\Cref{lem:codimensionbound}, applied to the complexification of
$v_{i_0}$, shows that
$$
Y_\xi^{\ell-i_0}v_{i_0}\not\equiv0.
$$
Thus the corresponding polynomial is not identically zero. It follows from \Cref{prop:triangular-U-local-integrability} that
$$
c_z\bigl(\widetilde P_{\boldsymbol v}^{\mathbb C}\bigr)
\ge\alpha_\lambda
\qquad
(\boldsymbol v\in\mathcal A_\ell,\ z\in\mathbb C^{n-1}).
$$
After choosing Euclidean coordinates on the coefficient space,
\Cref{lem:uniform-negative-moments-complex}, applied with parameter set
$\mathcal A_\ell$, gives, for every $0<\beta<\alpha_\lambda$,
\begin{equation}\label{eq:uniform-negative-moment-P}
\sup_{\boldsymbol v\in\mathcal A_\ell}
\int_B
\|\widetilde P(\boldsymbol v,\xi)\|^{-\beta}\,d\xi
<\infty.
\end{equation}
Consequently, Chebyshev's inequality gives
\begin{equation}\label{eq:normalized-sublevel-beta}
\operatorname{Leb}
\left\{
\xi\in B:
\|\widetilde P(\boldsymbol v,\xi)\|\le\varepsilon
\right\}
\ll_{\beta,B,V}\varepsilon^\beta
\end{equation}
uniformly for $\boldsymbol v\in\mathcal A_\ell$ and
$0<\varepsilon<1$.
Now choose
$
\alpha_\lambda-\delta<\beta<\alpha_\lambda.
$
Then \eqref{eq:normalized-sublevel-beta} gives
$$
\operatorname{Leb}
\left\{
\xi\in B:
\|p_\ell(u_\xi v)\|\le\varepsilon
\right\}
\ll_{\delta,B,V}
\varepsilon^{\alpha_\lambda-\delta}
$$
whenever $A_\ell(v)=1$, where again we used $0<\e<1$. This proves \Cref{prop:sublevelestimate}.
\end{proof}

The next lemma supplies the weak higher-weight sublevel estimate needed in the
borderline case $\mu_{\ell_{\lambda}}=0$.
\begin{lemma}[Sublevel estimate at the next weight]
\label{lem:next-weight-sublevel}
Assume that $\mu_{\ell_\lambda}=0$. For every
bounded open set $\Omega\subset\mathbb R^{n-1}$, there exists
$C=C(V,\Omega)>0$ such that, for every $v\in V$ with
$\max_{0\le j\le\ell_\lambda}\|p_j(v)\|>0$ and $0<\rho<1$,
$$
\Leb\left\{\xi\in\Omega:
 \|p_{\ell_{\lambda}+1}(u_\xi v)\|
 \le \rho\max_{0\le j\le\ell_\lambda}\|p_j(v)\|
\right\}
\le C\rho^{1/(\ell_{\lambda}+1)}.
$$
\end{lemma}

\begin{proof}
Set $\ell:=\ell_\lambda$. Since $\mu_\ell=0$, one has
$|\lambda|=n\ell$. Since $\lambda\ne0$ and $\lambda_n=0$, we have
$|\lambda|<n\lambda_1$, and hence $\ell<\lambda_1$. Thus
$V_{\ell+1}$ is defined.

For $0\le i\le\ell$ and $0\ne v_i\in V_i$, put $q:=\ell-i$. We claim that
$$
Y_\xi^{q+1}v_i\not\equiv0.
$$
Suppose otherwise. Arguing as in the last paragraph of the proof of
\Cref{lem:codimensionbound}, polarization allows us to choose $r\le q$
maximal such that
$$
w:=Y_{\xi_1}\cdots Y_{\xi_r}v_i\ne0
$$
for some $\xi_1,\ldots,\xi_r$. Then $Y_\eta w=0$ for every $\eta$.
The estimate \eqref{eq:auxiliary-weight-bound}, whose proof does not use
the negativity of the weight, applied to $w$ with $\mathsf c=0$ gives
$\operatorname{wt}(w)\ge0$. On the other hand,
$$
\operatorname{wt}(w)
=
\mu_i+nr
=
-nq+nr
\le0.
$$
Hence $r=q$ and $\operatorname{wt}(w)=0$. The identity preceding
\eqref{eq:auxiliary-weight-bound}, again with $\mathsf c=0$, now gives
$$
\sum_{j=1}^{n-1}\|Y_j^*w\|^2=0.
$$
Thus $Y_jw=Y_j^*w=0$ for every $j$. The operators
$E_{nj}=Y_j$, $E_{jn}=Y_j^*$, together with their commutators, generate
$\mathfrak{sl}_n$. Hence $w$ is $\mathfrak{sl}_n$-fixed. Since
$\lambda_n=0$ and $\lambda\ne0$, the module
$\mathsf S_\lambda(\mathbb C^n)$ is a nontrivial irreducible
$\SL_n$-module and therefore contains no nonzero invariant vector.
This contradiction proves the claim.

Let
$\mathcal P_{\le\ell+1}(\mathbb R^{n-1},V_{\ell+1})$
denote the space of $V_{\ell+1}$-valued polynomial maps of degree at most
$\ell+1$, and let
$\mathcal P_0(\mathbb R^{n-1},V_{\ell+1})$
denote its subspace of constant maps. It follows that the linear map
$\bigoplus_{i=0}^{\ell}V_i
\to
\mathcal P_{\le\ell+1}(\mathbb R^{n-1},V_{\ell+1})/
\mathcal P_0(\mathbb R^{n-1},V_{\ell+1}),
$
which sends $(v_0,\ldots,v_\ell)$ to the class of
$$
\sum_{i=0}^{\ell}
\frac1{(\ell+1-i)!}Y_\xi^{\ell+1-i}v_i,
$$
is injective. Indeed, its restriction to each $V_i$ is injective by the
claim, while the corresponding homogeneous degrees $\ell+1-i$ are
distinct.

Set
$A_\ell(v):=\max_{0\le j\le\ell}\|p_j(v)\|$, and choose a closed ball
$B$ whose interior contains $\overline\Omega$. Equip the polynomial
space with the supremum norm on $B$ and the quotient with the induced
quotient norm. Since the preceding linear map is injective, finite
dimensionality gives a constant $c_0>0$ such that
$$
\sup_{\xi\in B}\|p_{\ell+1}(u_\xi v)\|
\ge c_0A_\ell(v)
$$
whenever $A_\ell(v)>0$. The constant term $p_{\ell+1}(v)$ disappears
upon passing to the quotient.
Choose $\xi_0\in B$ at which this supremum is attained, and choose a unit
functional $\varphi\in V_{\ell+1}^*$ satisfying
$$
\varphi\bigl(p_{\ell+1}(u_{\xi_0}v)\bigr)
=
\|p_{\ell+1}(u_{\xi_0}v)\|.
$$
The scalar polynomial
$
P(\xi):=\varphi\bigl(p_{\ell+1}(u_\xi v)\bigr)
$
has degree at most $\ell+1$ and satisfies
$$
\sup_{\xi\in B}|P(\xi)|\ge c_0A_\ell(v).
$$
The multivariable Remez inequality \cite{BrudnyiGanzburg1973} gives
$$
\Leb\{\xi\in B:|P(\xi)|\le\varepsilon\sup_B|P|\}
\ll_{B,\ell}\varepsilon^{1/(\ell+1)}.
$$
Since $\Omega\subset B$ and
$
|P(\xi)|\le\|p_{\ell+1}(u_\xi v)\|,
$
the desired sublevel set is contained in the preceding Remez sublevel
set with $\varepsilon=c_0^{-1}\rho$. After adjusting the implied
constant when $c_0^{-1}\rho\ge1$, we obtain
$$
\Leb\left\{\xi\in\Omega:
 \|p_{\ell+1}(u_\xi v)\|\le\rho A_\ell(v)
\right\}
\ll_{V,\Omega}\rho^{1/(\ell+1)}.
$$
This proves the lemma.
\end{proof}

\section[Weighted local estimates for the H-action]{Weighted local estimates for the \texorpdfstring{$H$}{H}-action on \texorpdfstring{$\wedge^i \M_n(\mathbb R)$}{exterior powers of matrix space}}
\label{s:weightedlocalestimates}
In this section, we convert the sublevel estimates of the previous section
into local contraction estimates for the $H$-representations occurring in
the exterior powers of $\M_n(\mathbb R)$. We first establish one-factor
estimates for Schur modules along $b_t$, and then pass to the full
left--right action along $a_t=(b_t,b_t)$ by averaging over compact subsets
of its expanding horospherical subgroup.

On the nonexceptional summands, these estimates give contraction for a range
of exponents slightly larger than $1$. On the exceptional summands, the
ordinary local heights have critical exponent $1$: they contract below this
exponent but exhibit only bounded expansion above it. In the final
subsection, we introduce modified local heights that penalize proximity to
the exceptional summands and recover contraction beyond exponent $1$. These
estimates provide the local input for the global Margulis inequality in
\Cref{s:globalheight}.

\medskip 
We retain the notation and conventions of the preceding section. Thus
$$
V={\mathsf S}_\lambda(\mathbb R^n),
\qquad \lambda_n=0,
$$
where the Young diagram of $\lambda$ is contained in the $n\times n$
box, and
$$
V=\bigoplus_{j=0}^{\lambda_1}V_j,
\qquad
\mu_j=\mu_0+nj=-|\lambda|+nj.
$$
We assume that $V$ is nontrivial, so $\mu_0<0$. In particular,
$\mu_0\in\mathbb Z$, $|\mu_0|\le n(n-1)$, and
$n\mid\mu_0$ if and only if $n\mid|\lambda|$.

\subsection{Contraction inequality for $\mathsf{S}_{\lambda}(\R^n)$}

Recall the exponent $\alpha_\lambda$ given in
\eqref{alphalambda}, and define
$\gamma_\lambda$ as in the following table:
\begin{table}[ht]
\centering
\begin{tabular}{ccc}
\toprule
Type of $\lambda$
& $\alpha_\lambda$
& $\gamma_\lambda$ \\
\midrule
$n\nmid|\lambda|$
& $1$
& $1+\dfrac{1}{n^2}$ \\
\midrule
$n\mid|\lambda|$ and $\lambda\neq(n^k)$ for all $1\le k\le n-1$
& $\dfrac{(n-1)}{(n-2)}$
& $\dfrac{(n-1)}{(n-2)}$ \\
\midrule
$\lambda=(n^k)$ for some $1\le k\le n-1$
& $1$
& $1$ \\
\bottomrule
\end{tabular}
\caption{The exponents $\alpha_\lambda$ and $\gamma_\lambda$.}
\label{tab:alpha-gamma}
\end{table}

As noted in \Cref{rem:sublevel-n2}, the second case does not occur when $n=2$.

The main one-factor estimate is the following.

\begin{proposition}[Contraction for $\|\cdot\|^{-\beta}$]
\label{prop:generallocalcontraction}
For every $0<\beta<\gamma_\lambda$ and bounded open subset
$\Omega\subset\mathbb R^{n-1}$, there exists
$c=c(V,\beta,\Omega)>0$ such that, for every $t\ge1$ and 
$0\ne v\in V$,
$$
\int_\Omega \|b_tu_\xi v\|^{-\beta}\,d\xi
\le
c\exp\!\left(-\frac14\min\{\beta,\gamma_\lambda-\beta\}t\right)
\|v\|^{-\beta}.
$$
\end{proposition}
\begin{proof}
We may replace the given norm by an equivalent norm for which
$\|v\|=\max_{0\le j\le\lambda_1}\|p_j(v)\|$.
Set
$A_q(v):=\max_{0\le j\le q}\|p_j(v)\|.
$
Since $\Omega$ is bounded and the action of $u_\xi$ is triangular with
respect to the weight decomposition, there exists $C\ge1$ such that for every $\xi\in\Omega$, $q$, and $v\in V$,
\begin{equation}
\label{eq:triangular-boundedness-section5}
C^{-1}\|v\|\le\|u_\xi v\|\le C\|v\|,
\qquad
A_q(u_\xi v)\le CA_q(v).
\end{equation}

As before, let
$$
\ell:=\ell_\lambda
=\min\{j:\mu_j\ge0\}
=\left\lceil\frac{|\mu_0|}{n}\right\rceil.
$$
We distinguish two cases.

\medskip
\noindent
\textbf{Case 1:} $A_\ell(v)<(2C^2)^{-1}\|v\|$.

If $\ell=\lambda_1$, this case is empty. Otherwise,
\eqref{eq:triangular-boundedness-section5} implies
$$
\max_{\ell+1\le j\le\lambda_1}\|p_j(u_\xi v)\|
\ge C^{-1}\|v\|.
$$
Since $\mu_{\ell+1}\ge n$, we have
$$
\|b_tu_\xi v\|
\ge C^{-1}e^{\mu_{\ell+1}t}\|v\|
\ge C^{-1}e^{nt}\|v\|.
$$
Consequently,
$$
\int_\Omega\|b_tu_\xi v\|^{-\beta}\,d\xi
\ll_{V,\beta,\Omega}e^{-n\beta t}\|v\|^{-\beta},
$$
which is stronger than the required estimate.

\medskip
\noindent
\textbf{Case 2:} $A_\ell(v)\ge(2C^2)^{-1}\|v\|$.

The triangular formula \eqref{eq:polynomialcoordinate} and the boundedness
of $\Omega$ give a constant $C_1\ge1$ such that
$$
\|p_j(u_\xi v)-p_j(v)\|
\le C_1\max_{0\le i<j}\|p_i(v)\|
\qquad(1\le j\le\ell,\ \xi\in\Omega).
$$
Choose an index at which $A_\ell(v)$ is attained. If the maximum norm among preceding coordinates is larger than $(2C_1)^{-1}$ times the norm of the current
coordinate, replace the current index by an earlier index at which that maximum is attained, and
repeat. The process stops after at most $\ell$ steps. Its final index $r$
satisfies
\begin{equation}
\label{eq:dominant-weight-coordinate}
\|p_r(v)\|\ge(2C_1)^{-\ell}A_\ell(v)
\quad\text{and}\quad
\max_{0\le i<r}\|p_i(v)\|
\le(2C_1)^{-1}\|p_r(v)\|.
\end{equation}
Consequently, for $c_0:=2^{-1}(2C_1)^{-\ell}$,
\begin{equation}
\label{eq:dominant-coordinate-stable}
\|p_r(u_\xi v)\|
\ge\frac12\|p_r(v)\|
\ge c_0A_\ell(v)
\qquad(\xi\in\Omega).
\end{equation}

Set
$$
F(\xi):=A_\ell(v)^{-1} {\|p_\ell(u_\xi v)\|},
\qquad
s:=n(\ell-r)t.
$$
By \Cref{prop:sublevelestimate}, for every
$0<a<\alpha_\lambda$,
\begin{equation}
\label{eq:sublevel-normalized-section5}
\Leb\{\xi\in\Omega:F(\xi)\le\varepsilon\}
\ll_{a,V,\Omega}\varepsilon^a
\qquad(0<\varepsilon<1).
\end{equation}
Moreover, \eqref{eq:dominant-coordinate-stable} gives
\begin{equation}
\label{eq:two-coordinate-lower-bound}
\|b_tu_\xi v\|
\gg A_\ell(v)e^{\mu_rt}\max\{1,e^sF(\xi)\}.
\end{equation}

We first suppose that $\mu_\ell>0$, equivalently
$n\nmid|\lambda|$. In this case $\alpha_\lambda=1$ and
$\gamma_\lambda=1+n^{-2}$. Set
$$
\gamma:=\frac{n\ell}{|\mu_0|}
=1+\frac{\mu_\ell}{|\mu_0|}.
$$
Since $\mu_\ell$ is a positive integer and
$|\mu_0|\le n(n-1)$,
$$
\gamma\ge1+\frac1{n(n-1)}
>1+\frac1{n^2}
=\gamma_\lambda.
$$

Assume first that $0<\beta<1$, and choose
$\beta<a<1$. Decomposing the range of $F$ into
$\{F\le e^{-s}\}$ and the successive level ranges between $e^{-s}$
and $1$, and then using
\eqref{eq:sublevel-normalized-section5}, gives
$$
\int_\Omega\max\{1,e^sF(\xi)\}^{-\beta}\,d\xi
\ll_{a,\beta,V,\Omega}e^{-\beta s}.
$$
Together with \eqref{eq:two-coordinate-lower-bound}, this yields
$$
\begin{aligned}
\int_\Omega\|b_tu_\xi v\|^{-\beta}\,d\xi
&\ll A_\ell(v)^{-\beta}e^{-\beta\mu_rt-\beta s}=A_\ell(v)^{-\beta}e^{-\beta\mu_\ell t}
\ll e^{-\beta t}\|v\|^{-\beta}.
\end{aligned}
$$

Now let $1\le\beta<\gamma_\lambda$. Choose $a<1$ sufficiently close
to $1$ that
$$
\gamma a-\beta\ge\frac12(\gamma-\beta).
$$
If $r=\ell$, \eqref{eq:two-coordinate-lower-bound} already gives decay
at rate $\beta\mu_\ell\ge\beta$. Suppose that $r<\ell$. Then
$$
\frac{n(\ell-r)}{|\mu_r|}
=
\frac{n(\ell-r)}{|\mu_0|-nr}
\ge\gamma.
$$
The same decomposition of the range of $F$, now with $a<\beta$, gives
$$
\int_\Omega\max\{1,e^sF(\xi)\}^{-\beta}\,d\xi
\ll_{a,\beta,V,\Omega}e^{-as}.
$$
Therefore,
$$
\begin{aligned}
\int_\Omega\|b_tu_\xi v\|^{-\beta}\,d\xi
&\ll
A_\ell(v)^{-\beta}e^{-\beta\mu_rt-as}\\
&=
A_\ell(v)^{-\beta}
\exp\left(
-\left(an(\ell-r)-\beta|\mu_r|\right)t
\right)\\
&\le
A_\ell(v)^{-\beta}
\exp\left(
-|\mu_r|(\gamma a-\beta)t
\right).
\end{aligned}
$$
By the choice of $a$, the resulting decay exponent is at least
$$
\frac12|\mu_r|(\gamma-\beta)
\ge\frac12(\gamma-\beta)
\ge\frac12(\gamma_\lambda-\beta).
$$
This proves the proposition when $\mu_\ell>0$.
It remains to consider $\mu_\ell=0$. Then
$$
\gamma_{\lambda}=\alpha_\lambda=
\begin{cases}
1&\text{if }\lambda=(n^k),\\
\dfrac{n-1}{n-2}
&\text{if }\lambda\ne(n^k)\text{ for every }1\le k\le n-1,
\end{cases}.
$$
 For $0<\rho<1$, define
$$
D^+(\rho):=
\left\{\xi\in\Omega:
\max_{\ell+1\le j\le\lambda_1}
\|p_j(u_\xi v)\|\le\rho A_\ell(v)
\right\};
$$
$$
D(\rho):=
\left\{\xi\in\Omega:
\max_{\ell\le j\le\lambda_1}
\|p_j(u_\xi v)\|\le\rho A_\ell(v)
\right\}.
$$
Let $0<\rho_1<\rho_2\le1$. On
$\Omega\smallsetminus D^+(\rho_1)$,
we have
$\|b_tu_\xi v\|\ge e^{nt}\rho_1A_\ell(v)$,
and hence
\begin{equation}
\label{eq:borderline-contribution-one}
\int_{\Omega\smallsetminus D^+(\rho_1)}
\|b_tu_\xi v\|^{-\beta}\,d\xi
\ll e^{-n\beta t}\rho_1^{-\beta}A_\ell(v)^{-\beta}.
\end{equation}

Since
$$
D^+(\rho_1)
\subset
\left\{\xi\in\Omega:
\|p_{\ell+1}(u_\xi v)\|\le\rho_1A_\ell(v)
\right\},
$$
\Cref{lem:next-weight-sublevel} gives
$$
\Leb(D^+(\rho_1))
\ll\rho_1^{1/(\ell+1)}.
$$
If $\xi\in D^+(\rho_1)\smallsetminus D(\rho_2)$, then
$\|p_\ell(u_\xi v)\|>\rho_2A_\ell(v)$. Consequently,
\begin{equation}
\label{eq:borderline-contribution-two}
\int_{D^+(\rho_1)\smallsetminus D(\rho_2)}
\|b_tu_\xi v\|^{-\beta}\,d\xi
\ll
\rho_1^{1/(\ell+1)}
\rho_2^{-\beta}A_\ell(v)^{-\beta}.
\end{equation}

Finally, on $D(\rho_2)$ one has $F\le\rho_2$. Since
$\mu_r=-n(\ell-r)$, \eqref{eq:two-coordinate-lower-bound} becomes
$$
\|b_tu_\xi v\|
\gg A_\ell(v)\max\{e^{-n(\ell-r)t},F(\xi)\}.
$$
Decomposing $\{F\le\rho_2\}$ into the range
$F\le e^{-n(\ell-r)t}$ and the successive level ranges up to $\rho_2$,
and using \eqref{eq:sublevel-normalized-section5}, gives
\begin{equation}
\label{eq:borderline-contribution-three}
\int_{D(\rho_2)}
\|b_tu_\xi v\|^{-\beta}\,d\xi
\ll\rho_2^{a-\beta}A_\ell(v)^{-\beta}.
\end{equation}
where $a:=\alpha_\lambda-\frac1{10}(\alpha_\lambda-\beta)>\beta $.
If $n=2$, then $\ell=1$ and $\lambda=(2)$, so
$\alpha_\lambda=\gamma_\lambda=1$. Take
$\rho_1=e^{-3t/2}$ and $\rho_2=e^{-3t/10}$.
The three terms in
\eqref{eq:borderline-contribution-one}--\eqref{eq:borderline-contribution-three}
then decay at rates at least
$$
\frac{\beta}{2},
\qquad
\frac34-\frac{3\beta}{10},
\qquad
\frac{27}{100}(1-\beta),
$$
respectively. Each is at least
$\frac14\min\{\beta,1-\beta\}$, so the required estimate follows.

We may therefore assume that $n\ge3$. Take
$\rho_1=e^{-(n-1)t}$ and $\rho_2=e^{-\frac{n-1}{2(\ell+1)}t}$.
The three terms in
\eqref{eq:borderline-contribution-one}--\eqref{eq:borderline-contribution-three}
then decay at rates
$$
\beta,
\qquad
\frac{n-1}{\ell+1}\left(1-\frac\beta2\right),
\qquad
\frac{n-1}{2(\ell+1)}(a-\beta),
$$
respectively. Since $\gamma_\lambda\le2$, $\ell+1\le n$, and
$\beta<\gamma_\lambda$,
$$
\frac{n-1}{\ell+1}\left(1-\frac\beta2\right)
\ge
\frac13(\gamma_\lambda-\beta).
$$
Moreover, $a-\beta=\frac9{10}(\gamma_\lambda-\beta)$
and hence
$$
\frac{n-1}{2(\ell+1)}(a-\beta)
\ge
\frac3{10}(\gamma_\lambda-\beta).
$$
Thus all three rates are at least
$$
\frac14\min\{\beta,\gamma_\lambda-\beta\}.
$$
Since $A_\ell(v)\asymp\|v\|$ in Case~2, we conclude that
$$
\int_\Omega\|b_tu_\xi v\|^{-\beta}\,d\xi
\ll
\exp\!\left(
-\frac14\min\{\beta,\gamma_\lambda-\beta\}t
\right)
\|v\|^{-\beta}.
$$
This completes the proof.
\end{proof}

\begin{corollary}[Bounded expansion for $\|\cdot\|^{-\beta}$]
\label{prop:generallocalexpansion}
For every $\beta>\gamma_\lambda$, and bounded open subset
$\Omega\subset\mathbb R^{n-1}$, there exists
$c=c(V,\beta,\Omega)>0$ such that for every $t\ge1$ and $0\ne v\in V$,
$$
\int_\Omega\|b_tu_\xi v\|^{-\beta}\,d\xi
\le c e^{n^2(\beta-\gamma_\lambda)t}\|v\|^{-\beta}.
$$

\end{corollary}

\begin{proof}
We have 
$$
\|b_tu_\xi v\|\gg e^{-n(n-1)t}\|v\|,
$$
with the implied constant independent of all $\xi\in\Omega$,
because the smallest $b_t$-weight is at least $-n(n-1)$. Set
$$
\delta:=\min\left\{\frac{\beta-\gamma_\lambda}{n-1},\frac{\gamma_\lambda}{2}\right\},
\qquad
\beta_0:=\gamma_\lambda-\delta.
$$
Then $0<\beta_0<\gamma_\lambda$, so
\Cref{prop:generallocalcontraction} gives a uniform bound for the
$\beta_0$-moment. Therefore
$$
\begin{aligned}
\int_\Omega\|b_tu_\xi v\|^{-\beta}\,d\xi
&\ll e^{n(n-1)(\beta-\beta_0)t}\|v\|^{-(\beta-\beta_0)}
   \int_\Omega\|b_tu_\xi v\|^{-\beta_0}\,d\xi\\
&\ll e^{n(n-1)(\beta-\gamma_\lambda +\delta)t}\|v\|^{-\beta}.
\end{aligned}
$$
The definition of $\delta$ gives
$n(n-1)(\beta-\gamma_\lambda +\delta)\le n^2 (\beta-\gamma_\lambda)$, proving the claim.
\end{proof}

\subsection{One-factor estimates along the principal ray}

We now specialize to the representations occurring in
\begin{equation}\label{vall}
\mathcal V_{\op{all}}
:=\bigoplus_{1\le i\le n^2-1}\wedge^i\M_n(\mathbb R),
\end{equation}
viewed as a representation of
$H=\SL_n(\mathbb R)\times\SL_n(\mathbb R)$. Recalling
\eqref{eq:decompositionatdimnn-1}, set
$$
\mathcal V_{\op{gen}}
:=
\bigoplus_{\substack{1\le i\le n^2-1, \; n\nmid i}}
\wedge^i\M_n(\mathbb R)
\oplus\bigoplus_{k=1}^{n-1}\mathcal M_{k,0}.
$$

\begin{lemma}[Contraction on generic $\SL_n(\mathbb R)$-types]
\label{lem:generalweightedcontraction}
Let $W_1\otimes W_2\subset\mathcal V_{\op{gen}}$ be an irreducible
$H$-subrepresentation, let $m\in\{1,2\}$, and let
$\Omega\subset\mathbb R^{n-1}$ be bounded and open. For every $\frac12\le\beta\le1+\frac1{2n^2}$,
there exists $c=c(W_m,\beta,\Omega)>0$ such that
for every $t\ge1$ and nonzero $v\in W_m$,
$$
\int_\Omega\|b_tu_\xi v\|^{-\beta}\,d\xi
\le c e^{-t/(8n^2)}\|v\|^{-\beta}.
$$

\end{lemma}

\begin{proof}
Let $\lambda$ be the reduced partition corresponding to $W_m$. This
partition is nontrivial: a trivial reduced partition would give one of the
two exceptional summands, which is excluded from
$\mathcal V_{\op{gen}}$. Moreover, $\lambda$ is not of the form
$(n^j)$. Indeed, if the reduced partition were $(n^j)$, then the
$n\times n$ box condition would force the corresponding original partition
itself to be $(n^j)$. Its size would be $jn$, and the resulting summand
would again be exceptional. Hence $\gamma_\lambda\geq1+n^{-2}$. Since
$\min\left\{\beta,1+\frac1{n^2}-\beta\right\}
\ge\frac1{2n^2}$,
\Cref{prop:generallocalcontraction}, applied to $b_t$, gives
$$
\int_\Omega\|b_tu_\xi v\|^{-\beta}\,d\xi
\ll e^{-t/(8n^2)}\|v\|^{-\beta},
$$
which is the asserted estimate.
\end{proof}

The same estimate holds for every finite direct sum of generic types. Indeed,
use the maximum norm, choose a component of maximal norm, and apply the
preceding lemma to that invariant component. In addition, H\"older's
inequality gives
\begin{equation}
\label{eq:generic-small-exponent-bound}
\int_\Omega\|b_tu_\xi w\|^{-q}\,d\xi
\ll_{q,\Omega}\|w\|^{-q}
\qquad
\left(0\le q\le1+\frac1{2n^2}\right)
\end{equation}
for every such finite direct sum.

\begin{lemma}[Estimates for $\mathcal M_{k,1}$ and $\mathcal M_{k,2}$]
\label{lem:generalweightedcontraction'}
Let $1\le k\le n-1$, $m\in\{1,2\}$, and let
$\Omega\subset\mathbb R^{n-1}$ be bounded and open. For every
$\beta\in \left[\frac{1}2,1\right) \cup \left(1,1+ \frac{1}{2n^2}\right]$,
there exists $c=c(\beta,\Omega)>0$ such that, for every $t\ge1$ and
 nonzero vector $v$ in the nontrivial $H_m$-module
$\mathcal M_{k,m}$,
$$
\int_\Omega\|b_tu_\xi v\|^{-\beta}\,d\xi
\le
\begin{cases}
c e^{-\frac15(1-\beta)t}\|v\|^{-\beta},
& \frac12\le\beta<1\\ 
c e^{2n^2(\beta-1)t}\|v\|^{-\beta},
& 1<\beta\le1+\frac{1}{2n^2}.
\end{cases}
$$
\end{lemma}

\begin{proof}
As an $H_m\simeq\SL_n(\mathbb R)$-module,
$\mathcal M_{k,m}\simeq{\mathsf S}_{(n^k)}(\mathbb R^n),$
and hence $\gamma_\lambda=1$. For $\frac12\le\beta<1$,
\Cref{prop:generallocalcontraction} gives the first estimate. For
$1<\beta\le1+\frac1{2n^2}$, the second estimate follows from
\Cref{prop:generallocalexpansion}.
\end{proof}
\subsection{Local estimates for $H=\SL_n(\mathbb R)\times\SL_n(\mathbb R)$} We now pass from the one-factor estimates to estimates for the full left--right action. We apply the one-factor estimates successively in the two horospherical variables and use the tensor-product structure of the $H$-summands.

For $\boldsymbol{\xi}=(\xi_1,\xi_2)\in\mathbb R^{n-1}\times\mathbb R^{n-1}$, put
$$
u_{\boldsymbol{\xi}}:=
\left(
\begin{pmatrix}\operatorname{Id}_{n-1}&0\\ \xi_1&1\end{pmatrix},
\begin{pmatrix}\operatorname{Id}_{n-1}&0\\ \xi_2&1\end{pmatrix}
\right)\in H.
$$
We write $u_1(\xi_1)=u_{(\xi_1,0)}$ and
$u_2(\xi_2)=u_{(0,\xi_2)}$, so that
$u_{\boldsymbol{\xi}}=u_1(\xi_1)u_2(\xi_2)$. Let
$$
U_m:=\{u_m(\xi_m):\xi_m\in\mathbb R^{n-1}\},
\qquad
U:=U_1U_2,
$$
for $m\in\{1,2\}$. Write
$$
B_{U_m}(r):=\{u_m(\xi_m)\in U_m:\|\xi_m\|\le r\}.
$$

We use the Haar measures $du_m$ and $du$ induced by Lebesgue measures in these coordinates. Using the maximum norm $\|\boldsymbol{\xi}\|=\max (\|\xi_1\|, \|\xi_2\|)$, we 
have $B_{U}(r)=B_{U_1}(r)B_{U_2}(r)$.

\begin{proposition}
\label{lem:localcontractioninequalitygenerali}
\label{cor:generalweightedcontractiontensor}
Let $\Omega\subset\mathbb R^{n-1}\times\mathbb R^{n-1}$ be bounded and
open. For every
$\frac12\le\beta\le1+\frac1{2n^2}$, there exists $c>0$ such that
for every $t\ge1$ and nonzero $v\in\mathcal V_{\op{gen}}$, \begin{equation}
\label{eq:contractioninequalityatleasthalf}
\int_\Omega\|a_tu_{\boldsymbol{\xi}} v\|^{-\beta}\,d\boldsymbol{\xi}
\le c e^{-t/(5n^2)}\|v\|^{-\beta}.
\end{equation}
 Moreover, for all
$0\le\beta\le1+\frac1{2n^2}$,
\begin{equation}
\label{eq:contractioninequalitysmallrange}
\int_\Omega\|a_tu_{\boldsymbol{\xi}} v\|^{-\beta}\,d\boldsymbol{\xi}
\ll_{\beta,\Omega}\|v\|^{-\beta}.
\end{equation}
\end{proposition}

\begin{proof}
Choose bounded open sets
$\Omega_i\subset\mathbb R^{n-1}$ such that
$\Omega\subset\Omega_1\times\Omega_2$. Equip
$\mathcal V_{\op{gen}}$ with the maximum norm associated with its finite
decomposition into irreducible $H$-modules. Given $0\ne v\in
\mathcal V_{\op{gen}}$, choose an irreducible component $v'$ such that
$\|v'\|=\|v\|$. Since the summands are $H$-invariant,
$$
\|a_tu_{\boldsymbol{\xi}}v\|
\ge\|a_tu_{\boldsymbol{\xi}}v'\|.
$$
It therefore suffices to consider
$v\in W_1\otimes W_2$ in a fixed irreducible summand.

Viewed as an $H_1$-module, $W_1\otimes W_2$ is a finite direct sum of
copies of $W_1$; similarly, as an $H_2$-module, it is a finite direct
sum of copies of $W_2$. Applying the direct-sum extension of
\Cref{lem:generalweightedcontraction} first in the $\xi_1$-variable
and then in the $\xi_2$-variable, using the slightly weaker decay rate
$e^{-t/(10n^2)}$ at each step, gives
$$
\begin{aligned}
&\int_{\Omega_2}\int_{\Omega_1}
 \|(b_tu_{\xi_1},b_tu_{\xi_2})v\|^{-\beta}
 \,d\xi_1\,d\xi_2\\
& \ll
e^{-t/(10n^2)}
\int_{\Omega_2}\|(e,b_tu_{\xi_2})v\|^{-\beta}\,d\xi_2 \ll
e^{-t/(5n^2)}\|v\|^{-\beta}.
\end{aligned}
$$
This proves \eqref{eq:contractioninequalityatleasthalf}.

For $\frac12\le\beta\le1+\frac1{2n^2}$,
\eqref{eq:contractioninequalitysmallrange} follows immediately from
\eqref{eq:contractioninequalityatleasthalf}. If
$0<\beta<\frac12$, H\"older's inequality and the case $\beta=\frac12$
give
$$
\begin{aligned}
\int_\Omega\|a_tu_{\boldsymbol{\xi}}v\|^{-\beta}
\,d\boldsymbol{\xi}
&\le
\Leb(\Omega)^{1-2\beta}
\left(
\int_\Omega\|a_tu_{\boldsymbol{\xi}}v\|^{-1/2}
\,d\boldsymbol{\xi}
\right)^{2\beta} \ll_{\beta,\Omega}\|v\|^{-\beta}.
\end{aligned}
$$
The case $\beta=0$ is immediate.
\end{proof}

\begin{proposition}
[Bounded expansion on $\mathcal M_{k,1}\oplus\mathcal M_{k,2}$]
\label{lem:boundedexpansionfornn-1}
Let $1\le k\le n-1$, and let
$\Omega\subset\mathbb R^{n-1}\times\mathbb R^{n-1}$ be bounded and
open. For every
$
\beta\in
\left[\frac12,1\right)
\cup
\left(1,1+\frac1{2n^2}\right],
$
there exists $c=c(\beta,\Omega)>0$ such that, for every $t\ge1$ and
nonzero
$v\in\mathcal M_{k,1}\oplus\mathcal M_{k,2}$,
$$
\int_\Omega\|a_tu_{\boldsymbol{\xi}} v\|^{-\beta}\,d\boldsymbol{\xi}
\le
\begin{cases}
c e^{-\frac15(1-\beta)t}\|v\|^{-\beta},
&\frac12\le\beta<1\\ 
c e^{2n^2(\beta-1)t}\|v\|^{-\beta},
&1<\beta\le1+\frac1{2n^2}.
\end{cases}
$$ \end{proposition}

\begin{proof}
By the same maximum-norm reduction as in the proof of
\Cref{lem:localcontractioninequalitygenerali}, it suffices to consider $v\in\mathcal M_{k,i}$ for some $i\in\{1,2\}$. The other factor of $H$
acts trivially on $\mathcal M_{k,i}$. After enlarging $\Omega$ to a
product and integrating over the variable on which the integrand is constant, the result follows
directly from \Cref{lem:generalweightedcontraction'}.
\end{proof}

\subsection{Modified local height functions and local contraction}
\label{s:modifiedlocalheight}
At a critical degree $kn$, $1\le k\le n-1$, the ordinary local height $\|\cdot\|^{-\beta}$ has critical exponent
$1$ because of the two exceptional summands
$\mathcal M_{k,1}$ and $\mathcal M_{k,2}$. To obtain contraction exponents larger than $1$, we modify the local height by penalizing both small vectors and proximity to either exceptional summand.

\begin{definition}[Modified local height]
\label{modifiedlocalheight}
Fix $1\le k\le n-1$ and $0<\theta<(2n)^{-6}$. For
$v\notin\mathcal M_{k,1}\cup\mathcal M_{k,2}$, define
$$
\phi_{kn,\theta}(v)
:=
\|v\|^{-1+4\theta}
\|v-\pi_{k,1}(v)\|^{-2\theta}
\|v-\pi_{k,2}(v)\|^{-2\theta}.
$$
For $v\in\mathcal M_{k,1}\cup\mathcal M_{k,2}$, set
$\phi_{kn,\theta}(v)=\infty$. We write $\phi_{kn}=\phi_{kn,\theta}$
when $\theta$ is fixed.
\end{definition} Choose the maximum norm associated with
$\mathcal M_{k,0}\oplus\mathcal M_{k,1}\oplus\mathcal M_{k,2}$, and
write let $\pi_{i, 0}$ denote the projection onto $\mathcal M_{k,0}$. Write $v_m=\pi_{k,m}(v)$ for $m\in \{0,1,2\}$. With the convention $0^{-q}=\infty$, define
$$
F_{00}(v):=\|v_0\|^{-1},
$$
$$
F_{12}(v):=\|v_1\|^{-1+2\theta}\|v_2\|^{-2\theta},
\qquad
F_{21}(v):=\|v_2\|^{-1+2\theta}\|v_1\|^{-2\theta},
$$
and
$$
F_{10}(v):=\|v_1\|^{-1+2\theta}\|v_0\|^{-2\theta},
\qquad
F_{20}(v):=\|v_2\|^{-1+2\theta}\|v_0\|^{-2\theta}.
$$
A direct comparison of the three component norms gives
\begin{equation}
\label{eq:phi-as-minimum}
\phi_{kn}(v)
=
\min\{F_{00}(v),F_{12}(v),F_{21}(v),F_{10}(v),F_{20}(v)\}.
\end{equation}
Equivalently,
$$
\phi_{kn}(v)=
\begin{cases}
\|v_0\|^{-1},&\|v\|=\|v_0\|,\\
\|v_1\|^{-1+2\theta}
\max\{\|v_0\|,\|v_2\|\}^{-2\theta},
&\|v\|=\|v_1\|,\\ 
\|v_2\|^{-1+2\theta}
\max\{\|v_0\|,\|v_1\|\}^{-2\theta},
&\|v\|=\|v_2\|.
\end{cases}
$$

\begin{proposition}[Contraction for $\phi_{kn}$]
\label{prop:localMargulisineq;n-1}
Let $1\le k\le n-1$, $0<\theta<(2n)^{-6}$, and
$1\le\beta\le1+\frac{\theta}{4n^2}$.
There exists $c=c(\beta,\theta)>0$ such that, for every
$t\ge1$ and 
$v\notin\mathcal M_{k,1}\cup\mathcal M_{k,2}$,
$$
\int_{B_{U}(1)}\phi_{kn}(a_tuv)^\beta\,du
\le c e^{-\theta t/(20n^2)}\phi_{kn}(v)^\beta.
$$
\end{proposition}
 For the remainder of this subsection, fix $k$ and $\theta$ as in \Cref{prop:localMargulisineq;n-1}. We treat the five cases corresponding to
\eqref{eq:phi-as-minimum}.
\begin{lemma}
\label{lem:pi0pi0case}
Let $1\le\beta\le1+\frac1{2n^2}$. For every nonzero
$v_0\in\mathcal M_{k,0}$ and $t\ge1$,
$$
\int_{B_U(1)}\|a_tuv_0\|^{-\beta}\,du
\ll e^{-t/(5n^2)}\|v_0\|^{-\beta}.
$$
\end{lemma}

\begin{proof}
The projections are $H$-equivariant, and
$\mathcal M_{k,0}\subset\mathcal V_{\op{gen}}$. Apply
\Cref{lem:localcontractioninequalitygenerali} to $v_0$.
\end{proof}

\begin{lemma}
\label{lem:pi1pi2case;n-1}
Let $1\le\beta\le1+\frac{\theta}{4n^2}$.
If $v_1,v_2\ne0$, then
$$
\int_{B_{U}(1)}
\|\pi_{k,1}(a_tuv)\|^{-(1-2\theta)\beta}
\|\pi_{k,2}(a_tuv)\|^{-2\theta\beta}\,du
\ll e^{-\frac{\theta t}{20n^2}}\|v_1\|^{-(1-2\theta)\beta}\|v_2\|^{-2\theta\beta}.
$$
The same estimate holds with $1$ and $2$ interchanged.
\end{lemma}

\begin{proof} Let
$q:=(1-2\theta)\beta$ and $r:=2\theta\beta$.
We have $q\ge\frac12$, $r<\frac12$, and
$$
1-q
=1-(1-2\theta)\beta
\ge 2\theta-\frac{\theta}{4n^2}
>\theta.
$$
The first factor depends only on $u_1$, and the second only on $u_2$.
\Cref{lem:generalweightedcontraction'} therefore gives
$$
\int_{B_{U_1}(1)}\|b_tu_1v_1\|^{-q}\,du_1
\ll e^{-(1-q)t /5}\|v_1\|^{-q}.
$$
For the exponent $r<1/2$, H\"older's inequality and the estimate at
exponent $1/2$ give
$$
\int_{B_{U_2}(1)}\|b_tu_2v_2\|^{-r}\,du_2
\ll\|v_2\|^{-r}.
$$
Fubini's theorem completes the proof.
\end{proof}

\begin{lemma}
\label{lem:pi0pi1case;n-1}
Let $m\in\{1,2\}$ and
$1\le\beta\le1+\frac{\theta}{4n^2}$.
If $v_0,v_m\ne0$, then
$$
\int_{B_{U}(1)}
\|\pi_{k,0}(a_tuv)\|^{-2\theta\beta}
\|\pi_{k,m}(a_tuv)\|^{-(1-2\theta)\beta}\,du
\ll e^{-\frac{\theta t}{20n^2}}\|v_0\|^{-2\theta\beta}\|v_m\|^{-(1-2\theta)\beta}.
$$
\end{lemma}

\begin{proof} Let
$q_0:=2\theta\beta$ and $q_1:=(1-2\theta)\beta$.
We treat $m=1$; the other case is symmetric. Set
$$
s:=1+\frac1{2n^2},
\qquad
a:=\frac{q_0}{s},
\qquad
r:=\frac{q_1}{1-a}.
$$
Then $0<a<1$, and
$$
\begin{aligned}
1-a-q_1
&=1-\beta+2\theta\beta\left(1-\frac1s\right)\\
&=1-\beta+\frac{2\theta\beta}{2n^2+1}\ge -\frac{\theta}{4n^2}+\frac{2\theta}{2n^2+1}\ge \frac{\theta}{2n^2}.
\end{aligned}
$$
Consequently $\frac12<r<1$, and
\begin{equation}
\label{eq:mixed-holder-margin}
(1-a)(1-r)=1-a-q_1
\ge \frac{\theta}{2n^2}.
\end{equation}

Write $g_j=b_tu_j$. For fixed $u_2$, H\"older's inequality
in the $u_1$-variable gives
$$
\begin{aligned}
&\int_{B_{U_1}(1)}
 \|(g_1,g_2)v_0\|^{-q_0}\|g_1v_1\|^{-q_1}\,du_1\\
&\quad\le
\left(\int_{B_{U_1}(1)}\|(g_1,g_2)v_0\|^{-s}\,du_1\right)^a
\left(\int_{B_{U_1}(1)}\|g_1v_1\|^{-r}\,du_1\right)^{1-a}.
\end{aligned}
$$
Every irreducible $H_1$-type occurring in $\mathcal M_{k,0}$ is
nontrivial and generic. Indeed, a trivial reduced partition would give the
summand $\mathcal M_{k,2}$. If the reduced partition were $(n^j)$
with $j\ge1$, the $n\times n$ box condition would force the original
partition itself to be $(n^j)$. Since its size is $kn$, one would have
$j=k$, giving the summand $\mathcal M_{k,1}$. Both possibilities are
excluded from $\mathcal M_{k,0}$. Hence the direct-sum form of
\Cref{lem:generalweightedcontraction} gives
$$
\int_{B_{U_1}(1)}\|(g_1,g_2)v_0\|^{-s}\,du_1
\ll e^{-t/(10n^2)}\|(e,g_2)v_0\|^{-s}.
$$
\Cref{prop:generallocalcontraction}, applied along $b_t$, gives
$$
\int_{B_{U_1}(1)}\|g_1v_1\|^{-r}\,du_1
\ll e^{-(1-r)t/5}\|v_1\|^{-r}.
$$
Using \eqref{eq:mixed-holder-margin} and then integrating in $u_2$, we
obtain
$$
\begin{aligned}
&\int_{B_{U}(1)}
 \|a_tuv_0\|^{-q_0}\|a_tuv_1\|^{-q_1}\,du \ll e^{-\frac{\theta t}{20n^2} }\|v_1\|^{-q_1}
\int_{B_{U_2}(1)}\|(e,g_2)v_0\|^{-q_0}\,du_2.
\end{aligned}
$$
Finally, $q_0<1/2$, so
\eqref{eq:generic-small-exponent-bound}, applied to the $H_2$-action on
$\mathcal M_{k,0}$, bounds the last integral by
$\|v_0\|^{-q_0}$.
\end{proof}

\begin{proof}[Proof of \Cref{prop:localMargulisineq;n-1}]
For each $v\notin\mathcal M_{k,1}\cup\mathcal M_{k,2}$, choose one of
the five functions in \eqref{eq:phi-as-minimum}, say $F$, such that
$$
\phi_{kn}(v)=F(v).
$$
The value $F(v)$ is finite, so every component appearing with a negative
power in $F$ is nonzero. Since $\phi_{kn}(w)\le F(w)$ for every $w$,
$$
\int_{B_{U}(1)}\phi_{kn}(a_tuv)^\beta\,du
\le
\int_{B_{U}(1)}F(a_tuv)^\beta\,du.
$$
If $F=F_{00}$, apply \Cref{lem:pi0pi0case}; if
$F=F_{12}$ or $F_{21}$, apply
\Cref{lem:pi1pi2case;n-1}; and if $F=F_{10}$ or $F_{20}$,
apply \Cref{lem:pi0pi1case;n-1}. In every case, the resulting right-hand
side is at most
$
C e^{-\theta t/(20n^2)}F(v)^\beta
$.
This proves the proposition.
\end{proof}

\section{The global Margulis inequality for the modified height}
\label{s:globalheight}
The goal of this section is to prove a global Margulis inequality for the modified height outside the exceptional set (\Cref{thm:globalcontraction}). This estimate globalizes the
local estimates of \Cref{s:weightedlocalestimates} and will later be
combined with avoidance and iteration to obtain the uniform
$L^{1+\theta}$-bound for some $\theta>0$.
The inequality has a logarithmic loss and an additive exponential error.
Its proof combines the intersection--sum inequality of
Eskin--Margulis--Mozes \cite[Lemma~5.6]{eskin-margulis-mozes:1998} with
the Mother Inequality of Benoist--Quint
\cite[Proposition~3.1]{benoist-quint:2012} to control both covolumes and
projections
to the exceptional summands.
\medskip 

Set $N=n^2$, and fix
$$
0<\theta<(2n)^{-6}.
$$

For $1\le i\le N-1$, define
\begin{equation}\label{theta}
\theta_i:=10^{i-N}\theta
\quad \text{ and }\quad
\beta_i:=\left(1-\frac{\theta_i}{8n^2}\right)^{-1}.
\end{equation}
We also set
\begin{equation}\label{eta}
\tau_i:=
\left(1+\kappa i(N-i)\theta_1\right)^{-1},
\quad
\kappa:=\frac{1}{16n^2(N-1)},\quad\vartheta:=\frac{\theta_1}{100n}.
\end{equation}

The numerical choices in this paragraph are bookkeeping devices. Their only
purpose is to create a hierarchy of small exponents with enough room for the
local contraction estimates, the intersection--sum inequalities, and the
iteration argument. The parameters $\beta_i>1$ will be used for the
modified local heights; the parameters $\tau_i<1$ will be used for the
ordinary height $\bar\alpha$, where the concavity in the index $i$ is
needed when two short rational subspaces are replaced by their intersection
and sum. The rapidly increasing sequence
$\theta_1\ll\theta_2\ll\cdots\ll\theta_{N-1}$ is chosen only to make these
concavity inequalities strict.

The choices above have three elementary consequences. First,
\begin{equation}\label{eq:beta-local-range}
1<\beta_i\le 1+\frac{1}{4n^2}
\qquad (1\le i\le N-1),
\end{equation}
and, if $i=kn$, then
\begin{equation}\label{eq:beta-critical-range}
1\le\beta_i\le 1+\frac{\theta_i}{4n^2}.
\end{equation}
Indeed, $(1-x)^{-1}\le1+2x$ for the values of
$x=\theta_i/(8n^2)$ under consideration. Thus $\beta_{kn}$ lies in
the range of \Cref{prop:localMargulisineq;n-1}.

Second,
\begin{equation}\label{eq:tau-range}
\frac12<\tau_i<1,
\qquad
1-\tau_i\ge \frac{\vartheta}{10}
\quad\text{whenever }n\mid i.
\end{equation}
For the second assertion, write
$x_i=\kappa i(N-i)\theta_1$. If $n\mid i$, then
$$
x_i\ge \kappa n(N-n)\theta_1
=\frac{\theta_1}{16(n+1)}.
$$
Since $x_i<1$, it follows that
$1-\tau_i=x_i/(1+x_i)\ge x_i/2\ge\vartheta/10$.

Third,
\begin{equation}\label{eq:beta-tau-comparison}
\beta_i^{-1}<\tau_i
\qquad (1\le i\le N-1).
\end{equation}
Indeed, $\beta_i^{-1}=1-\theta_i/(8n^2)$, whereas
$$
\kappa i(N-i)\theta_1
\le \frac{\theta_i}{16n^2};
$$
the latter inequality is immediate for $i=1$, and for $i\ge2$ it
follows from $i(N-i)\le i(N-1)\le 2^{i-1}(N-1)\le10^{i-1}(N-1)$.
Therefore
$$
\tau_i=\frac1{1+x_i}\ge1-x_i
>1-\frac{\theta_i}{8n^2}=\beta_i^{-1}.
$$

For later endpoint conventions, put
$$
\alpha_0^+(\Delta)=\alpha_N^+(\Delta):=1,
\qquad
\alpha_i^+(\Delta):=\max\{1,\alpha_i(\Delta)\}
\quad (1\le i\le N-1).
$$

\subsection{The global and modified heights}
Define
\begin{equation}\label{eq:bar-alpha-definition}
\bar\alpha_\theta(\Delta)
:=
\max_{1\le i\le N-1}\alpha_i^+(\Delta)^{\tau_i}.
\end{equation}
We also introduce a second auxiliary height:
\begin{equation}\label{eq:bar-alpha-prime-definition}
\bar\alpha'_\theta(\Delta)
:=
\sum_{i=1}^{N-1}\alpha_i^+(\Delta)^{1+\frac{\theta_1}{100n^2}}.
\end{equation}
We usually write
$\bar\alpha,\bar\alpha'$ when $\theta$ is fixed. Since
$\tau_i\ge1-C_n\theta_1$, one has
\begin{equation}\label{eq:bar-alpha-comparison}
\max\{1,\alpha(\Delta)\}^{1-C_n\theta_1}
\le \bar\alpha(\Delta)
\le \max\{1,\alpha(\Delta)\},
\end{equation}
and
\begin{equation}\label{eq:bar-alpha'-comparison}
\max\{1,\alpha(\Delta)\}^{1+C_n^{-1}\theta}
\le \bar\alpha'(\Delta)
\ll_n \max\{1,\alpha(\Delta)\}^{2}.
\end{equation}
Here and below, $C_n$ denotes a positive constant depending only on $n$,
whose value may change from line to line. Unsubscripted implied constants may
also depend on the fixed parameter $\theta$.
 When we classify a component as maximal, we use the
maximum of the three component norms associated with
$\mathcal M_{k,0}\oplus\mathcal M_{k,1}\oplus\mathcal M_{k,2}$;
this auxiliary maximum is uniformly equivalent to the Euclidean norm.

\begin{definition}[Auxiliary weighted modified height]
\label{def:auxiliary-weighted-modified-height}
For $1\le i\le N-1$ and $v\in \wedge^i \M_n(\br)$, define
$$
\phi_i(v):=
\begin{cases}
\phi_{kn,\theta_{kn}}(v)&\text{ for } i=kn,\ 1\le k\le n-1,\\ 
\|v\|^{-1} &\text{ for } n\nmid i.
\end{cases}
$$
At the endpoints we use the convention
$\phi_0=\phi_N=1$.

For $1\le i\le N-1$, $h\in H$ and $\Delta\in X$, set
\begin{equation}\label{eq:modified-local-global-height}
\widetilde\alpha_{i,\eta,M}(h;\Delta)
:=
\max\left\{
1,\
\sup_{\substack{
V\text{ is $\Delta$-rational},\ \dim V=i\\
V\notin\widetilde{\mathscr Q}_{i,\eta,M}(\Delta)\\
0<\|h\mathsf w_{\Delta,V}\|\le1}}
\phi_i(h\mathsf w_{\Delta,V})^{\beta_i}
\right\}.
\end{equation}
Finally, define
\begin{equation}\label{eq:modified-global-height}
\widetilde\alpha_{\eta,M,\theta}(h;\Delta)
:=
\bar\alpha_\theta(h\Delta)
+
\max_{1\le i\le N-1}
\widetilde\alpha_{i,\eta,M}(h;\Delta).
\end{equation}
We suppress $\theta$ from the notation when it is fixed. At degrees
$0$ and $N$, we set
$$
\widehat\alpha_{0,\eta,M}
=\widehat\alpha_{N,\eta,M}
=
\widetilde\alpha_{0,\eta,M}
=\widetilde\alpha_{N,\eta,M}
:=1.
$$
\end{definition}

For $i=kn$, the supremum in
\eqref{eq:modified-local-global-height} never evaluates
$\phi_i=\phi_{kn,\theta_{kn}}$ on an exact exceptional vector: such a
vector belongs to $\mathscr Q_{kn,\eta,M}$, so its corresponding
subspace belongs to
$\widetilde{\mathscr Q}_{kn,\eta,M}(\Delta)$ and is excluded from the
supremum. Hence every value of $\phi_i$ occurring there is finite.

Let
$$
\beta_*:=\min_{1\le i\le N-1}\beta_i=\beta_1,
\qquad
\tau_*:=\min_{1\le i\le N-1}\tau_i>\frac12.
$$
At a critical degree $i=km$,
the fixed projections satisfy $\|v-\pi_{k,m}(v)\|\ll_n\|v\|$ for $m=1,2$. Together with the definition at noncritical degrees, this gives
$\phi_i(v)\gg_n\|v\|^{-1}$ for every $i$. Thus, for each vector contributing to $\widehat{\alpha}_{i,\eta,M}(h;\Delta)$, the term
$\|hv\|^{-\beta_i}$ is $O(1)$ when $\|hv\|>1$ and is dominated by the
corresponding term in $\widetilde\alpha_{i,\eta,M}(h;\Delta)$ when
$\|hv\|\le1$. Therefore
\begin{equation}\label{eq:sharp-majorized-by-tilde}
\widehat{\alpha}_{i,\eta,M}(h;\Delta)^{\beta_i}
\ll_n \widetilde\alpha_{i,\eta,M}(h;\Delta),
\qquad
\widehat{\alpha}_{\eta,M}(h;\Delta)^{\beta_*}
\ll_n \widetilde\alpha_{\eta,M}(h;\Delta).
\end{equation}

\begin{lemma}\label{lem:hatalpha-majorized-by-tilde}
If $0<\theta'\le \beta_*-1$, then
$$
\widehat\alpha_{\eta,M}(h;\Delta)^{1+\theta'}
\ll_n
\widetilde\alpha_{\eta,M,\theta}(h;\Delta)
\qquad ((h,\Delta)\in H\times X).
$$
\end{lemma}

\begin{proof}
Since
$
\widehat\alpha_{\eta,M}(h;\Delta)
=
\max_{1\le i\le N-1}
\widehat\alpha_{i,\eta,M}(h;\Delta),
$
and $\widehat\alpha_{i,\eta,M}(h;\Delta)\ge1$, the assumption
$1+\theta'\le\beta_*\le\beta_i$ gives
$$
\widehat\alpha_{\eta,M}(h;\Delta)^{1+\theta'}
\le
\max_{1\le i\le N-1}
\widehat\alpha_{i,\eta,M}(h;\Delta)^{\beta_i}.
$$
The claim now follows from \eqref{eq:sharp-majorized-by-tilde} and the
definition of $\widetilde\alpha_{\eta,M,\theta}$.
\end{proof}

\subsection{Concavity of the exponents}
The exponents are chosen so that products arising from intersections and
sums of rational subspaces are absorbed by the global height. In the first
inequality below, the asymmetric exponents $\beta_{i+j}^{-1}$ and
$\beta_{i-j}$ reflect that the intersection factor is controlled by the
ordinary height, while the sum factor is controlled by the modified
height.

\begin{lemma}\label{lem:exponent-concavity}
Let $1\le i\le N-1$ and $1\le j\le\min\{i,N-i\}$.

\begin{enumerate}
\item If $i-j\ge1$ and $i+j\le N-1$, then
\begin{equation}\label{eq:beta-concavity}
2\beta_i^{-1}-\beta_{i+j}^{-1}-\beta_{i-j}
\ge \frac{\theta_1}{2n^2}.
\end{equation}
Consequently, for every $0<\rho\le1$ and $x,y,z\ge0$,
\begin{equation}\label{eq:Youngineq}
x^{\beta_i/2}y^{\beta_i/2}z
\ll
\rho\left(x^{\beta_{i+j}}+y^{1/\beta_{i-j}}\right)
+
1+\left(\rho^{-1}z\right)^{C_n\theta_1^{-1}}.
\end{equation}

\item If $i-j\ge1$ and $i+j\le N-1$, then
\begin{equation}\label{eq:tau-concavity}
2\tau_i^{-1}-\tau_{i+j}^{-1}-\tau_{i-j}^{-1}
=2\kappa j^2\theta_1.
\end{equation}
Consequently, for every $0<\rho\le1$,
\begin{equation}\label{eq:Youngineq'}
x^{\tau_i/2}y^{\tau_i/2}z
\ll
\rho\left(x^{\tau_{i+j}}+y^{\tau_{i-j}}\right)
+
1+\left(\rho^{-1}z\right)^{C_n\theta_1^{-1}}.
\end{equation}
\end{enumerate}
If $i-j=0$, we omit the $y$-factor on the left and the $y$-term on the
right; if $i+j=N$, we omit the $x$-factor and the $x$-term. \end{lemma}

\begin{proof}
Assume first that $i-j\ge1$ and $i+j\le N-1$. Set $a=(8n^2)^{-1}$. Since
$\beta_r^{-1}=1-a\theta_r$ and
$\beta_r\le1+2a\theta_r$,
$$
\begin{aligned}
2\beta_i^{-1}-\beta_{i+j}^{-1}-\beta_{i-j}
&\ge
 a\bigl(\theta_{i+j}-2\theta_i-2\theta_{i-j}\bigr)=a\theta_i\bigl(10^j-2-2\cdot10^{-j}\bigr)
\ge \frac{\theta_1}{2n^2}.
\end{aligned}
$$
For $\rho=1$, \eqref{eq:Youngineq} follows from weighted Young's
inequality with exponents
$$
\frac{2\beta_{i+j}}{\beta_i},
\qquad
\frac{2}{\beta_i\beta_{i-j}},
\qquad
\frac{2\beta_i^{-1}}
{2\beta_i^{-1}-\beta_{i+j}^{-1}-\beta_{i-j}}.
$$
The last exponent is at most $4n^2\theta_1^{-1}$. For general $\rho$,
put
$$
\gamma:=\frac{\beta_i}{2\beta_{i+j}}
+\frac{\beta_i\beta_{i-j}}2<1
$$
and apply the case $\rho=1$ to
$\rho^{1/\beta_{i+j}}x$, 
$\rho^{\beta_{i-j}}y$ and $\rho^{-\gamma}z$. This gives \eqref{eq:Youngineq}, after enlarging $C_n$. The identity \eqref{eq:tau-concavity} follows immediately from
$\tau_r^{-1}=1+\kappa r(N-r)\theta_1$, and the same scaled version of Young's inequality gives \eqref{eq:Youngineq'}. The endpoint cases follow by
omitting the absent factors.
\end{proof}

\subsection{Estimates for ordinary global heights}
For $s\ge 2\log n$, $u\in B_{U}(1)$, and
$v\in\wedge^i\M_n(\mathbb R)$, one has
\begin{equation}\label{eq:logLipschitzforwedge}
 e^{-4n^3 s}\|v\|
\le \|a_suv\|
\le e^{4n^3 s}\|v\|.
\end{equation}
At a critical degree $i=kn$, equivariance of the projections gives 
\begin{equation}\label{eq:logLipschitzformodifiedht}
e^{-4n^3s}\phi_{kn}(v)
\le \phi_{kn}(a_suv)
\le e^{4n^3s}\phi_{kn}(v)
\end{equation}
whenever $\phi_{kn}(v)<\infty$.

We first record the subharmonic estimate for $\bar\alpha$. We deduce it from the local contraction estimates of the previous section by following the argument of \cite{eskin-margulis-mozes:1998}. We normalize Haar measure $du$ on the abelian group $U$ by
$\operatorname{vol}(B_{U}(1))=1$.
\begin{lemma}[Margulis inequality for $\bar{\alpha}$]\label{prop:globalboundedexpansion'}
There are constants $c_n,C_n>0$ such that, for every $\Delta\in X$ and
$s\ge 2\log n$
\begin{equation}\label{eq:globalboundedexpansion'}
\int_{B_{U}(1)}
\bar\alpha(a_su\Delta)\,du
\ll
 e^{-c_n\theta_1 s}\bar\alpha(\Delta)
+
 e^{C_n\theta_1^{-1}s}.
\end{equation}
\end{lemma}

\begin{proof}
 Decompose
$\wedge^i\M_n(\mathbb R)$ into its generic part and, when $n\mid i$,
its two exceptional summands, and use the maximum norm. If a maximal component
of $v$ is generic, \Cref{lem:localcontractioninequalitygenerali}
applies. If it is exceptional, apply
\Cref{lem:boundedexpansionfornn-1}. The estimate
\eqref{eq:tau-range} places $\tau_i$ in the contracting range of that
proposition. This gives the same conclusion. Applying these estimates on a
fixed bounded open neighborhood of $B_{U}(1)$ and then restricting the
integral, we obtain, uniformly in $1\le i\le N-1$,
\begin{equation}\label{eq:local-tau-contraction}
\int_{B_{U}(1)}\|a_suv\|^{-\tau_i}\,du
\ll e^{-c_n\theta_1 s}\|v\|^{-\tau_i}.
\end{equation}

Fix $i$, and choose an $i$-dimensional $\Delta$-rational subspace
$V$ with
$d_\Delta(V)^{-1}=\alpha_i(\Delta)$. Such a subspace exists because
the primitive decomposable vectors in $\wedge^i\Delta$ form a discrete
set. Let
$$
\mathcal P_i(\Delta,s)
:=
\left\{
W:
W\text{ is $\Delta$-rational},\ \dim W=i,
d_\Delta(W)\le e^{8n^3s}d_\Delta(V)
\right\}.
$$
This set is finite, again by discreteness of the primitive exterior lattice.

\noindent{\bf{Case (1).}} $\#\mathcal P_i(\Delta,s)\leq 1 $.

If $\mathcal P_i(\Delta,s)=\{V\}$, then
\eqref{eq:logLipschitzforwedge} shows that $V$ realizes the least
covolume after applying $a_su$. Hence
\eqref{eq:local-tau-contraction} yields
$$
\int_{B_{U}(1)}
\alpha_i^+(a_su\Delta)^{\tau_i}\,du
\ll
 e^{-c_n\theta_1 s}
\alpha_i^+(\Delta)^{\tau_i}+1.
$$

\noindent{\bf{Case (2).}} $\#\mathcal P_i(\Delta,s)\geq 2 $.

Suppose that $\mathcal P_i(\Delta,s)$ contains two distinct
subspaces $W_1,W_2$. Write
$$
\dim(W_1\cap W_2)=i-j,
\qquad
\dim(W_1+W_2)=i+j.
$$
The intersection-sum inequality
\cite[Lemma~5.6]{eskin-margulis-mozes:1998} gives
\begin{equation}\label{eq:EMMintersectionsum}
    d_\Delta(W_1\cap W_2)d_\Delta(W_1+W_2)
\ll d_\Delta(W_1)d_\Delta(W_2).
\end{equation}
Consequently,
\begin{equation}\label{eq:two-short-subspaces}
\alpha_i^+(\Delta)^2
\ll
e^{16n^3s}
\alpha_{i-j}^+(\Delta)\alpha_{i+j}^+(\Delta).
\end{equation}
On the other hand, \eqref{eq:logLipschitzforwedge} implies
$$
\alpha_i^+(a_su\Delta)^{\tau_i}
\ll e^{4n^3s}
\alpha_i^+(\Delta)^{\tau_i}.
$$
Combining this with \eqref{eq:two-short-subspaces}, we obtain
$$
\alpha_i^+(a_su\Delta)^{\tau_i}
\ll
e^{2C_ns}
\alpha_{i-j}^+(\Delta)^{\tau_i/2}
\alpha_{i+j}^+(\Delta)^{\tau_i/2}.
$$
Apply \eqref{eq:Youngineq'} with
$\rho=e^{-c_n\theta_1 s}$, decreasing $c_n$ if necessary.
Since each height term produced by \eqref{eq:Youngineq'} is bounded by
$\bar\alpha(\Delta)$, this gives, pointwise in $u$,
$$
\alpha_i^+(a_su\Delta)^{\tau_i}
\ll
 e^{-c_n\theta_1 s}\bar\alpha(\Delta)
+e^{C_n\theta_1^{-1}s}.
$$
Summing over $i$ proves \eqref{eq:globalboundedexpansion'}.
\end{proof}

We next deduce a bounded-expansion estimate for $\bar\alpha'_\theta$ from
the one-vector estimates.
\begin{proposition}[Bounded expansion for $\bar{\alpha}'$]\label{prop:globalboundedexpansion''}
There is a constant $C_n>0$ such that, for every $\Delta\in X$ and 
$s\ge 2\log n$,
\begin{equation}\label{eq:globalboundedexpansion''}
\int_{B_{U}(1)}
\bar\alpha'_\theta(a_su\Delta)\,du
\ll
e^{C_n\theta_1s}
\left(\bar\alpha'_\theta(\Delta)+1\right).
\end{equation}
\end{proposition}

\begin{proof}  Set
$\chi:=\frac{\theta_1}{100n^2}$.
Since $1<1+\chi<1+\frac1{2n^2}$, we apply the local estimates on a fixed bounded open neighborhood of
$B_{U}(1)$. They give, for every $1\le i\le N-1$ and nonzero
$v\in\wedge^i\M_n(\mathbb R)$,
\begin{equation}\label{eq:one-vector-alpha-prime-bounded-expansion}
\int_{B_{U}(1)}
\|a_suv\|^{-1-\chi}\,du
\ll_{n,\theta}
e^{2n^2\chi s}\|v\|^{-1-\chi}.
\end{equation}
Indeed, use the maximum norm associated with the decomposition into the
generic part and, when $n\mid i$, the two exceptional summands. Choose
an invariant component $v'$ of maximal norm. Since the summands are
$H$-invariant, we have $\|a_suv\|\ge\|a_suv'\|$. If $v'$ is generic, apply
\Cref{lem:localcontractioninequalitygenerali} at exponent $1+\chi$.
If $v'$ is exceptional, apply
\Cref{lem:boundedexpansionfornn-1} at the same exponent. This proves
\eqref{eq:one-vector-alpha-prime-bounded-expansion}.

Choose a sufficiently large constant $C_n$. Since $\theta$
is fixed, after increasing a lower threshold for $s$ if necessary, the
implied constant in
\eqref{eq:one-vector-alpha-prime-bounded-expansion} is absorbed by
$e^{C_n\chi s/4}$. Fix such an $s$. For $1\le i\le N-1$, choose an
$i$-dimensional $\Delta$-rational subspace $V$ such that
$d_\Delta(V)^{-1}=\alpha_i(\Delta)$, and put
$$
\mathcal P_i(\Delta,s)
:=
\left\{
W:
W\text{ is $\Delta$-rational},\ \dim W=i,
d_\Delta(W)\le e^{8n^3s}d_\Delta(V)
\right\}.
$$
This is a finite set containing $V$.

\noindent{\bf{Case (1).}} $\#\mathcal P_i(\Delta,s)=1$.

In this case $\mathcal P_i(\Delta,s)=\{V\}$. If an $i$-dimensional
$\Delta$-rational subspace realizes the least covolume after applying
$a_su$, then \eqref{eq:logLipschitzforwedge} shows that it belongs to
$\mathcal P_i(\Delta,s)$. Hence $V$ realizes the least covolume after
applying $a_su$ for every $u\in B_{U}(1)$. It follows from
\eqref{eq:one-vector-alpha-prime-bounded-expansion} that
\begin{equation}\label{eq:alpha-prime-case-one}
\int_{B_{U}(1)}
\alpha_i^+(a_su\Delta)^{1+\chi}\,du
\le
e^{C_n\chi s/4}\alpha_i^+(\Delta)^{1+\chi}+1.
\end{equation}

\noindent{\bf{Case (2).}} $\#\mathcal P_i(\Delta,s)\ge2$.

Choose two distinct subspaces $W_1,W_2\in\mathcal P_i(\Delta,s)$, and
write
$$
\dim(W_1\cap W_2)=i-j,
\qquad
\dim(W_1+W_2)=i+j.
$$
The intersection--sum inequality
\cite[Lemma~5.6]{eskin-margulis-mozes:1998} gives
$$
d_\Delta(W_1\cap W_2)d_\Delta(W_1+W_2)
\ll_n d_\Delta(W_1)d_\Delta(W_2).
$$
Consequently,
$$
\alpha_i^+(\Delta)^2
\ll_n
e^{16n^3s}
\alpha_{i-j}^+(\Delta)\alpha_{i+j}^+(\Delta).
$$
Combining this with \eqref{eq:logLipschitzforwedge}, and increasing the
lower threshold for $s$ if necessary, gives, pointwise in $u$,
\begin{equation}\label{eq:alpha-prime-case-two}
\alpha_i^+(a_su\Delta)^{1+\chi}
\le
e^{C_ns/4}
\left(
\alpha_{i-j}^+(\Delta)\alpha_{i+j}^+(\Delta)
\right)^{(1+\chi)/2}.
\end{equation}
The identity
$$
(i-j)(N-i+j)+(i+j)(N-i-j)
=
2i(N-i)-2j^2
$$
and Young's inequality give
\begin{align*}
&e^{-C_nsi(N-i)}
\left(
\alpha_{i-j}^+(\Delta)\alpha_{i+j}^+(\Delta)
\right)^{(1+\chi)/2} \le
\tfrac{e^{-C_ns}}2
e^{-C_ns(i-j)(N-i+j)}
\alpha_{i-j}^+(\Delta)^{1+\chi}
\\ \qquad &+
\tfrac{e^{-C_ns}}2
e^{-C_ns(i+j)(N-i-j)}
\alpha_{i+j}^+(\Delta)^{1+\chi}.
\end{align*}
Here, a term of degree $0$ or $N$ is interpreted as $1$. Multiplying
\eqref{eq:alpha-prime-case-one} and
\eqref{eq:alpha-prime-case-two} by $e^{-C_nsi(N-i)}$, summing over $i$,
and enlarging the fixed lower threshold for $s$ once more, we obtain
\begin{equation}\label{eq:weighted-alpha-prime-one-step}
\begin{aligned}
&\int_{B_{U}(1)}
\left(
1+\sum_{i=1}^{N-1}
e^{-C_nsi(N-i)}
\alpha_i^+(a_su\Delta)^{1+\chi}
\right)\,du\\
&\qquad\le
e^{C_n\chi s}
\left(
1+\sum_{i=1}^{N-1}
e^{-C_nsi(N-i)}
\alpha_i^+(\Delta)^{1+\chi}
\right).
\end{aligned}
\end{equation}
Indeed, after multiplication by the degree weight, each term from
Case~(2) carries the factor $e^{-3C_ns/4}$, and there are at most $N-1$
possible degrees. Increasing the fixed lower threshold for $s$ absorbs
their sum into the right-hand side of
\eqref{eq:weighted-alpha-prime-one-step}.

Fix one $s$ for which \eqref{eq:weighted-alpha-prime-one-step} holds and
iterate that estimate in blocks of length $s$. The comparison of iterated
horospherical averages in \cite[Lemma~6.1]{Kim}, recorded below as
\Cref{lem:randomwalkgeneral}, shows that, for every $m\ge1$,
\begin{align*}
&\int_{B_{U}(1/3)}
\left(
1+\sum_{i=1}^{N-1}
e^{-C_nsi(N-i)}
\alpha_i^+(a_{ms}u\Delta)^{1+\chi}
\right)\,du\\
&\qquad\ll
e^{C_n\chi ms}
\left(
1+\sum_{i=1}^{N-1}
e^{-C_nsi(N-i)}
\alpha_i^+(\Delta)^{1+\chi}
\right).
\end{align*}
For an arbitrary time $t$, write $t=ms+s'$, where $m\in\mathbb N$ and
$0\le s'<s$. Applying the elementary log-Lipschitz estimate on this fixed range to the single
$a_{s'}$-step reduces the estimate at time $t$ to the preceding estimate at
time $ms$. Since $s$ is now fixed in terms of $n$ and $\theta$, this
changes only the implied constant. Moreover,
$$
e^{-C_ns\lfloor N^2/4\rfloor}\bar\alpha'_\theta(\Delta)
\le
\sum_{i=1}^{N-1}
e^{-C_nsi(N-i)}\alpha_i^+(\Delta)^{1+\chi}
\le
\bar\alpha'_\theta(\Delta).
$$
We therefore obtain
$$
\int_{B_{U}(1/3)}
\bar\alpha'_\theta(a_tu\Delta)\,du
\ll_{n,\theta}
e^{C_n\chi t}\left(\bar\alpha'_\theta(\Delta)+1\right)
\qquad(t\ge0).
$$
Finally, cover $B_{U}(1)$ by finitely many sets $u_0B_{U}(1/3)$ with $u_0$
in a fixed bounded subset of $U$. Apply the preceding estimate to
$u_0\Delta$. Since $U$ is abelian, $u_0u=uu_0$, and the log-Lipschitz
estimate on a fixed bounded subset gives
$\bar\alpha'_\theta(u_0\Delta)\asymp_n
\bar\alpha'_\theta(\Delta)$. Summing over the finite cover gives the same
estimate over $B_{U}(1)$. Since
$\chi=\theta_1/(100n^2)$, this proves
\eqref{eq:globalboundedexpansion''}.
\end{proof}

\subsection{A consequence of the Mother Inequality}

For a critical degree $i=kn$, define
\begin{equation}\label{eq:bar-phi-definition}
\bar\phi_i(v)
:=
\|v-\pi_{k,1}(v)\|^{-1/2}
\|v-\pi_{k,2}(v)\|^{-1/2}\in (0, \infty],\quad (v\in \wedge^i(\M_n(\br)))
\end{equation}
Then
\begin{equation}\label{eq:phi-decomposition-global}
\phi_i(v)=\|v\|^{-1+4\theta_i}\bar\phi_i(v)^{4\theta_i}.
\end{equation}

Relative to the split torus of $H$, which is
the product of diagonal subgroups of $\SL_n(\br)$, let $q_\lambda$ denote the
$H$-equivariant projection onto the isotypic subspace of  highest
weight $\lambda$. On each exterior degree, this is the sum of the relevant
projections in the skew Cauchy decomposition. In the notation used here,
the Mother Inequality of Benoist--Quint
\cite[Proposition~3.1]{benoist-quint:2012} states the following:

\begin{theorem}[Benoist--Quint Mother Inequality]
\label{thm:BQ-Mother}
For all highest weights $\lambda,\mu$ and all decomposable vectors
$u\in\wedge^r\M_n(\mathbb R), v\in\wedge^s\M_n(\mathbb R), w\in\wedge^t\M_n(\mathbb R) $
with $r+s+t\le N$, one has
\begin{equation}\label{eq:BQ-Mother-form}
\|q_\lambda(u)\|\,\|q_\mu(u\wedge v\wedge w)\|
\ll
\sum_{\nu+\rho\succeq\lambda+\mu}
\|q_\nu(u\wedge v)\|\,\|q_\rho(u\wedge w)\|.
\end{equation}
Here $\nu$ and $\rho$ range over the highest weights occurring in
$\wedge^{r+s}\M_n(\mathbb R)$ and
$\wedge^{r+t}\M_n(\mathbb R)$, respectively, and
$\nu+\rho\succeq\lambda+\mu$ means that
$\nu+\rho-\lambda-\mu$ is a nonnegative integral combination of the
chosen positive roots. The implied constant is independent of
$u,v,w$.
\end{theorem}

 Suppose now that
$w_1,w_2$ are $\Delta$-Pl\"ucker vectors of two $\Delta$-rational
subspaces of dimension $kn$, and let $w_-$ and $w_+$ be 
$\Delta$-Pl\"ucker vectors of their intersection and sum. The 
intersection lattice is primitive in each of the two original lattices. By extending a
primitive basis of the intersection to bases of the two lattices and adjusting
signs if necessary, we may therefore choose decomposable exterior vectors $u,v,w$ such
that
$$
u=w_-,\quad u\wedge v=w_1,\quad u\wedge w=w_2,\quad
u\wedge v\wedge w=mw_+
$$
for an integer $m\ge1$. Applying \eqref{eq:BQ-Mother-form} and discarding
the factor $m$ on the left yields
\begin{equation}\label{eq:BQ-specialized-form}
\|q_\lambda(w_-)\|\,\|q_\mu(w_+)\|
\ll
\max_{\nu+\rho\succeq\lambda+\mu}
\|q_\nu(w_1)\|\,\|q_\rho(w_2)\|.
\end{equation}
If $\deg w_+<N$, then
\begin{equation}\label{eq:Mother-coarse-consequence}
\|w_-\|D_+(w_+)\ll \max\{\|w_1\|\|w_2-\pi_{k,1}(w_2)\|,\|w_2\|\|w_1-\pi_{k,1}(w_1)\|\},
\end{equation}
where
$$
D_+(w_+):=
\begin{cases}
\|w_+\|& \text{ if } n\nmid\deg w_+,\\
\|w_+-\pi_{\deg(w_+)/n,1}(w_+)\|&\text{ if } n\mid\deg w_+.
\end{cases}
$$

At any exterior degree, an irreducible summand is trivial under the second
$\SL_n$-factor only when the degree is a multiple of $n$, in which case
the summand is $\mathcal M_{\deg/n,1}$. Hence, if the output projection
lies outside this exceptional summand, the two input projections cannot
both lie in $\mathcal M_{k,1}$, since their product would be trivial under
the second factor. Thus, in every nonzero term of
\eqref{eq:BQ-specialized-form}, at least one input projection is controlled
by the distance from $\mathcal M_{k,1}$, while the other is bounded by the
norm of the corresponding input. Taking the maximum over the finitely many
isotypic projections and using norm equivalence gives
\eqref{eq:Mother-coarse-consequence}.

The following lemma treats the case $\dim(W_1+W_2)<N$. The case
$\dim(W_1+W_2)=N$ is handled separately in the proof of
\Cref{thm:globalcontraction}.
\begin{lemma}\label{lem:BQinequality}
Let $\Delta\in X$, and let $W_1,W_2$ be distinct
$\Delta$-rational subspaces of dimension $i=kn$ such that $W_1+W_2\ne \M_n(\br)$. Let $j$ be such that 
$$
\dim(W_1\cap W_2)=i-j\quad 
\text{ and }\quad
\dim(W_1+W_2)=i+j.
$$
Let $w_m=\mathsf w_{\Delta,W_m}$, $m=1,2$, and assume that
$w_m\notin\mathcal M_{k,1}\cup\mathcal M_{k,2}$ and
$$
\|\pi_{k,1}(w_m)\|
=\max_{q=0,1,2}\|\pi_{k,q}(w_m)\|
\qquad (m=1,2).
$$
Suppose also that, for some $\mathsf L\ge1$,
$$
\mathsf L^{-1}\|w_2\|\le\|w_1\|\le\mathsf L\|w_2\|\;\; \text{ and }\;\; 
\mathsf L^{-1}\bar\phi_i(w_2)
\le\bar\phi_i(w_1)
\le\mathsf L\bar\phi_i(w_2).
$$
Then
\begin{equation}\label{eq:BQ-modified-height-inequality}
\phi_i(w_1)\phi_i(w_2)
\ll
\mathsf L^{8\theta_i}
\|\mathsf w_{\Delta,W_1\cap W_2}\|^{-1}
\phi_{i+j}(\mathsf w_{\Delta,W_1+W_2}).
\end{equation}
The same conclusion holds with $\pi_{k,1}$ replaced by $\pi_{k,2}$.
\end{lemma}

\begin{proof}
We prove the first assertion. Write
$$
w_-:=\mathsf w_{\Delta,W_1\cap W_2},
\qquad
w_+:=\mathsf w_{\Delta,W_1+W_2}.
$$
Set
$$P:=\max\{\|w_1\|\|w_2-\pi_{k,1}(w_2)\|,\|w_2\|\|w_1-\pi_{k,1}(w_1)\|\}.$$
Since the first exceptional component is maximal, norm equivalence gives
$$
\bar\phi_i(w_m)\asymp(\|w_m\|\|w_m-\pi_{k,1}(w_m)\|)^{-1/2}.
$$
The two comparison assumptions imply
\begin{equation}\label{eq:P-barphi-comparison}
P\ll \mathsf L^2
\bar\phi_i(w_1)^{-1}\bar\phi_i(w_2)^{-1}.
\end{equation}
The standard intersection-sum inequality \cite[Lemma~5.6]{eskin-margulis-mozes:1998} gives
\begin{equation}\label{eq:ordinary-plucker-Mother-proof}
\|w_{-}\|\|w_+\|\ll \|w_1\|\|w_2\|,
\end{equation}
and \eqref{eq:Mother-coarse-consequence} gives
\begin{equation}\label{eq:projected-plucker-Mother-proof}
\|w_{-}\|D_+(w_+)\ll P.
\end{equation}

If $n\nmid i+j$, then $D_+(w_+)=\|w_{+}\|$. Using
\eqref{eq:P-barphi-comparison}--\eqref{eq:projected-plucker-Mother-proof}
in \eqref{eq:phi-decomposition-global}, we obtain
$$
\phi_i(w_1)\phi_i(w_2)
\ll \mathsf L^{8\theta_i}(\|w_{-}\|\|w_{+}\|)^{-1}
=\mathsf L^{8\theta_i}\|w_{-}\|^{-1}\phi_{i+j}(w_+).
$$

Suppose that $n\,|\,i+j$. If $D_+(w_+)=0$, then $w_+$ lies in the first exceptional
summand and the right-hand side of
\eqref{eq:BQ-modified-height-inequality} is infinite, so there is nothing
to prove. We may therefore assume that $D_+(w_+)>0$. From \eqref{eq:ordinary-plucker-Mother-proof} and
\eqref{eq:projected-plucker-Mother-proof},
$$
(\|w_1\|\|w_2\|)^{-(1-4\theta_i)}P^{-4\theta_i}
\ll
\|w_{-}\|^{-1}\|w_{+}\|^{-1+4\theta_i}D_+(w_+)^{-4\theta_i}.
$$
Since $D_+(w_+)\ll \|w_{+}\|$, the definition of $\phi_{i+j}$ and norm
equivalence give
$$
\phi_{i+j}(w_+)
\gg \|w_{+}\|^{-1}
\left(\frac{D_+(w_+)}{\|w_{+}\|}\right)^{-2\theta_{i+j}}.
$$
Moreover, $\theta_{i+j}=10^j\theta_i\ge2\theta_i$, and
$D_+(w_+)/\|w_{+}\|$ is bounded above by a constant depending only on the fixed
projections. Hence
$$
\|w_{+}\|^{-1+4\theta_i}D_+(w_+)^{-4\theta_i}
=\|w_{+}\|^{-1}\left(\frac{D_+(w_+)}{\|w_{+}\|}\right)^{-4\theta_i}
\ll \phi_{i+j}(w_+).
$$
This proves \eqref{eq:BQ-modified-height-inequality}. The row-exceptional
case is symmetric.
\end{proof}

\subsection{Main global Margulis inequality for $\widetilde{\alpha}_{\eta,M}$}\label{subsec:globalMargulis}

Fix $0<\eta<1$, $M\ge1$, $M'\ge1$, and $b\ge1$. 
For $1\le k\le n-1$, recall the notation
$\mathscr Q_{kn,\eta,M}$ and $\pi_{k,m}$ from
\eqref{hkd} and \eqref{eq:pikmdefinition}.
For 
$s\ge 2\log n$, let
$$\mathcal E_{ k ,s,\eta,M,M'}$$ be the set of pairs
$(h,\Delta)\in H\times X$ for which there is a $kn$-dimensional
$\Delta$-rational subspace $V$ satisfying
$$
\mathsf w_{\Delta,V}\notin
\mathscr Q_{ k  n,\eta,M},
\qquad
0<\|h\mathsf w_{\Delta,V}\|
\le e^{4n^3s},
$$
and
\begin{equation}\label{twocases}
\min_{m=1,2}
\|h\mathsf w_{\Delta,V}
-
\pi_{ k ,m}(h\mathsf w_{\Delta,V})\|
\le \begin{cases}
e^{-M's}
&\text{if } \alpha(h\Delta)
\le e^{20n^3s},\\ 
\alpha(h\Delta)^{-M'b}
&\text{if }\alpha(h\Delta)
>e^{20n^3s}.
\end{cases}
\end{equation}

Set
\begin{equation}\label{eq:exceptional-set-definition}
\mathcal E_{s,\eta,M,M'}
:=
\bigcup_{ k =1}^{n-1}
\mathcal E_{ k ,s,\eta,M,M'}.
\end{equation}
We suppress the fixed parameter $b$ from the notation
$\mathcal E_{s,\eta,M,M'}$.
The constant $b=b(n)$ will be fixed in the avoidance argument of \Cref{s:avoidance} (see \Cref{prop:avoidance'}). 

We first note that the identity eventually lies outside the exceptional set.

\begin{lemma}\label{lem:quasinullatidentity}
Let $\Lambda\in X$, $0<\eta<1$, and $M\ge1$. There exists
$T_0=T_0(\eta,M,\Lambda)\ge 2\log n$ such that, whenever $s\ge T_0$,
$$
(e,\Lambda)\notin
\mathcal E_{s,\eta,M,10n^3M}.
$$
\end{lemma}

\begin{proof}
For all sufficiently large $s$, the first line of
\eqref{twocases} applies. If $(e,\Lambda)$ belonged to the
exceptional set, there would be a vector
$v=\mathsf w_{\Lambda,V}\notin
\mathscr Q_{ k  n,\eta,M}$ with
$\|v\|\le e^{4n^3s}$ and
$$
\min_{m=1,2}\|v-\pi_{ k ,m}(v)\|
\le e^{-10n^3Ms}.
$$
For sufficiently large $s$, the right-hand side is at most
$\eta e^{-4n^3Ms}\le\eta\|v\|^{-M}$. Thus
$v\in\mathscr Q_{ k  n,\eta,M}$, a contradiction.
\end{proof}

We now combine the preceding estimates with the Mother Inequality.

\begin{theorem}[Main global Margulis inequality outside the exceptional set]\label{thm:globalcontraction}
Fix $b\ge1$ and $M'\ge1$. There exists
$\theta_0=\theta_0(n,b,M')>0$ such that for every
$0<\theta\le\theta_0$, $0<\eta<1$, $M\ge1$,
$s\ge 2\log n$ and
$(h,\Delta)\notin\mathcal E_{s,\eta,M,M'}$,
\begin{equation}\label{eq:globalcontraction}
\begin{aligned}
\int_{B_U(1)}
\widetilde\alpha(a_suh;\Delta)\,du
\ll{}&
e^{-c_n\theta s}
\widetilde\alpha(h;\Delta)
\log\!\bigl(3+\widehat\alpha_{\eta,M}(h;\Delta)\bigr)
+e^{C_n(M'\theta+\theta^{-1})s}.
\end{aligned}
\end{equation} Here the implicit constant depends only on $b$, $M'$, and $\theta$.
\end{theorem}

\begin{proof}
By reducing $\theta_0$, we may assume that
\begin{equation}\label{eq:theta-small-global-contraction}
C_nbM'\theta\le\frac1{20}.
\end{equation}

All local integral estimates below are applied on a fixed bounded open
neighborhood of $B_{U}(1)$ and then restricted to $B_{U}(1)$.
Fix $1\le i\le N-1$. Let
$\mathcal P_i=\mathcal P_i(h,\Delta,s)$ be the finite collection
of $i$-dimensional $\Delta$-rational subspaces $V$ such that
\begin{equation}\label{eq:relevant-subspaces-definition}
V\notin\widetilde{\mathscr Q}_{i,\eta,M}(\Delta),
\quad
0<\|h\mathsf w_{\Delta,V}\|\le1,
\quad 
\phi_i(h\mathsf w_{\Delta,V})^{\beta_i}
>e^{-12n^3s}\widetilde\alpha_{i,\eta,M}(h;\Delta).
\end{equation}
Finiteness follows because the corresponding primitive vectors in
$\wedge^i(h\Delta)$ lie in a bounded set.
We first bound the contribution of subspaces outside
$\mathcal P_i$. Let $W$ be any $\Delta$-rational subspace occurring
in the supremum in
$\widetilde\alpha_{i,\eta,M}(a_suh;\Delta)$, and write $w:=\mathsf w_{\Delta,W}$. Thus $W\notin\widetilde{\mathscr Q}_{i,\eta,M}(\Delta)$ and $0<\|a_suhw\|\le1$, and its contribution to the supremum is
$\phi_i(a_suhw)^{\beta_i}$. By
\eqref{eq:logLipschitzforwedge}, we have $\|hw\|\le e^{4n^3s}$.

If
$\|hw\|\le1$ and $W\notin\mathcal P_i$, then
\eqref{eq:logLipschitzformodifiedht} and
\eqref{eq:relevant-subspaces-definition} give
\begin{equation}
\begin{aligned}
        \phi_i(a_suhw)^{\beta_i}
&\le e^{4n^3\beta_i s}\phi_i(hw)^{\beta_i}
\\&\le e^{(4n^3\beta_i-12n^3)s}
\widetilde\alpha_{i,\eta,M}(h;\Delta)
\le e^{-4n^3s}\widetilde\alpha_{i,\eta,M}(h;\Delta),
\end{aligned}
\end{equation}
because $\beta_i<2$.

If $\|hw\|>1$ and $n\nmid i$, the same expression is
$O(e^{2C_ns})$. Finally, suppose $i=kn$ and $\|hw\|>1$. Since
$W\notin\widetilde{\mathscr Q}_{i,\eta,M}(\Delta)$, its own Pl\"ucker vector does not
belong to $\mathscr Q_{kn,\eta,M}$. Let
\begin{equation}\label{eq:exceptional-scale}
    \varepsilon_{s,\eta,M,M'}(h;\Delta)
:=
\begin{cases}
e^{-M's}
&\text{if } \alpha(h\Delta)
\le e^{20n^3s},\\ 
\alpha(h\Delta)^{-bM'}
&\text{if }\alpha(h\Delta)
>e^{20n^3s}.\end{cases}
\end{equation}  Because
$(h,\Delta)\notin\mathcal E_{s,\eta,M,M'}$,
$$
\min_{m=1,2}\|hw-\pi_{k,m}(hw)\|
>\varepsilon_{s,\eta,M,M'}(h;\Delta).
$$
Since $\|hw\|>1$ and $-1+4\theta_i<0$, the definition of $\phi_i$ gives
$\phi_i(hw)^{\beta_i}\ll
\varepsilon_{s,\eta,M,M'}(h;\Delta)^{-4\theta_i\beta_i}$.
Using \eqref{eq:logLipschitzformodifiedht} and the two alternatives in
\eqref{eq:exceptional-scale},
we obtain
\begin{equation}\label{eq:outside-P-error}
\phi_i(a_suhw)^{\beta_i}
\ll
e^{2C_ns}
\left(
 e^{C_nM'\theta s}
 +\alpha(h\Delta)^{C_nbM'\theta}
\right).
\end{equation}
With the convention that the maximum over the empty set is zero, it follows that
\begin{equation}\label{eq:competitor-reduction}
\begin{aligned}
\widetilde\alpha_{i,\eta,M}(a_suh;\Delta)
\ll{}&
\max_{V\in\mathcal P_i}
\phi_i(a_suh\mathsf w_{\Delta,V})^{\beta_i}
+e^{-4n^3s}\widetilde\alpha_{i,\eta,M}(h;\Delta)\\
&+e^{2C_ns}
\left(e^{C_nM'\theta s}+\alpha(h\Delta)^{C_nbM'\theta}\right).
\end{aligned}
\end{equation}

Fix a sufficiently large constant $C_0=C_0(n)$. 

\noindent{\bf{Case (1).}} First suppose that
\begin{equation}\label{eq:few-competitors}
\#\mathcal P_i
\le 3C_0\log(3+\widehat{\alpha}_{\eta,M}(h;\Delta)).
\end{equation}
For noncritical $i$, every summand of $\wedge^i\M_n(\mathbb R)$ is
generic, and \Cref{lem:localcontractioninequalitygenerali}
applies. For $i=kn$, use
\Cref{prop:localMargulisineq;n-1}. Thus
$$
\int_{B_{U}(1)}
\phi_i(a_suhv)^{\beta_i}\,du
\ll e^{-c_n\theta_1 s}\phi_i(hv)^{\beta_i}
$$
for every vector $v=\mathsf{w}_{\Delta,V}$ with $V\in \mathcal P_i$. Integrating \eqref{eq:competitor-reduction} and using
\eqref{eq:few-competitors} gives the required decaying term. For the error
term, the definition of $\bar\alpha$ gives $\alpha(h\Delta)
\le \bar\alpha(h\Delta)^{1/\tau_*}
\le \widetilde\alpha_{\eta,M}(h;\Delta)^{1/\tau_*}$.
Together with \eqref{eq:theta-small-global-contraction} and
$\tau_*>1/2$, this implies
$$
\alpha(h\Delta)^{C_nbM'\theta}
\ll 1+\widetilde\alpha_{\eta,M}(h;\Delta)^{1/2}.
$$
The inequality
$x\widetilde\alpha_{\eta,M}(h;\Delta)^{1/2}
\le \rho\widetilde\alpha_{\eta,M}(h;\Delta)+C\rho^{-1}x^2$, with
$\rho=e^{-c_n\theta_1 s}$, absorbs this sublinear power. Consequently,
\begin{multline}\label{eq:few-competitors-bound}
\int_{B_{U}(1)}
\widetilde\alpha_{i,\eta,M}(a_suh;\Delta)\,du
\\ \ll
 e^{-c_n\theta_1 s}\widetilde\alpha_{\eta,M}(h;\Delta)
 \log(3+\widehat{\alpha}_{\eta,M}(h;\Delta))
+e^{C_n(M'\theta+\theta_1^{-1})s}.
\end{multline}

\noindent{\bf{Case (2).}} Now suppose that
\begin{equation}\label{large}
\#\mathcal P_i
> 3C_0\log(3+\widehat{\alpha}_{\eta,M}(h;\Delta)).
\end{equation} If $n\nmid i$, put
$\mathcal R_i^0=\mathcal P_i$ and
$\mathcal R_i^1=\mathcal R_i^2=\emptyset$. If $i=kn$, partition
$\mathcal P_i$ into three sets. Let $\mathcal R_i^0$ consist of
those $V$ for which
$$
\phi_i(h\mathsf w_{\Delta,V})
\le e^{4n^3s}\|h\mathsf w_{\Delta,V}\|^{-1}.
$$
If the generic component of $h\mathsf w_{\Delta,V}$ were maximal, then
\eqref{eq:phi-as-minimum} and equivalence of the fixed norms would give
$\phi_i(h\mathsf w_{\Delta,V})\asymp_n
\|h\mathsf w_{\Delta,V}\|^{-1}$. Since $s\ge2\log n$, such a subspace
belongs to $\mathcal R_i^0$. Thus every remaining subspace has a maximal
exceptional component.
Place each remaining subspace in $\mathcal R_i^1$ or
$\mathcal R_i^2$ according to whether a maximal component of
$h\mathsf w_{\Delta,V}$ is the $\pi_{k,1}$-component or the
$\pi_{k,2}$-component, assigning ties to $\mathcal R_i^1$. At least one of
these three sets contains more than
$C_0\log(3+\widehat{\alpha}_{\eta,M}(h;\Delta))$ elements.

\noindent{\bf{Case (2-a)} $\#\mathcal R_i^0\ge 2$.} 
Let $W_1\ne W_2$ be two subspaces in $\mathcal R_i^0$. Set
$$
\dim(W_1\cap W_2)=i-j,
\qquad
\dim(W_1+W_2)=i+j.
$$
If $i+j<N$ and
$W_1+W_2\in
\widetilde{\mathscr Q}_{i+j,\eta,M}(\Delta)$,
then the downward-closure property and
$W_1\subset W_1+W_2$ would imply
\[
W_1\in\widetilde{\mathscr Q}_{i,\eta,M}(\Delta),
\]
contrary to the definition of $\mathcal P_i$. Therefore
\[
W_1+W_2
\notin
\widetilde{\mathscr Q}_{i+j,\eta,M}(\Delta).
\]
When $i+j=N$, we use the endpoint convention introduced above.
The Pl\"ucker covolume inequality and
\eqref{eq:relevant-subspaces-definition} give
\begin{equation}\label{eq:R0-product-bound}
\widetilde\alpha_{i,\eta,M}(h;\Delta)^{2/\beta_i}
\ll
e^{2C_ns}
\alpha_{i-j}^+(h\Delta)
\widehat\alpha_{i+j,\eta,M}(h;\Delta).
\end{equation}
Taking the $\beta_i/2$-power of
\eqref{eq:R0-product-bound} and absorbing the additional log-Lipschitz
factor from \eqref{eq:competitor-reduction} into $e^{C_ns}$, we obtain
$$
e^{2C_ns}\widetilde\alpha_{i,\eta,M}(h;\Delta)\ll e^{2C_ns}
\widehat\alpha_{i+j,\eta,M}(h;\Delta)^{\beta_i/2}
\alpha_{i-j}^+(h\Delta)^{\beta_i/2}.
$$
Apply \eqref{eq:Youngineq} with
$x=\widehat\alpha_{i+j,\eta,M}(h;\Delta)$,
$y=\alpha_{i-j}^+(h\Delta)$, $z=e^{2C_ns}$, and
$\rho=e^{-c_n\theta_1s}$, omitting an endpoint variable when necessary.
This gives
$$
e^{2C_ns}\widetilde\alpha_{i,\eta,M}(h;\Delta)
\ll
 e^{-c_n\theta_1 s}
\left(
\widehat\alpha_{i+j,\eta,M}(h;\Delta)^{\beta_{i+j}}
+
\alpha_{i-j}^+(h\Delta)^{1/\beta_{i-j}}
\right)
+e^{C_n\theta_1^{-1}s},
$$
with the absent endpoint term omitted. By \eqref{eq:beta-tau-comparison} and
\eqref{eq:sharp-majorized-by-tilde}, both displayed height terms are
$\ll \widetilde\alpha_{\eta,M}(h;\Delta)$. Consequently,
\begin{equation}\label{eq:R0-contraction}
e^{2C_ns}\widetilde\alpha_{i,\eta,M}(h;\Delta)
\ll e^{-c_n\theta_1 s}\widetilde\alpha_{\eta,M}(h;\Delta)
+e^{C_n\theta_1^{-1}s}.
\end{equation}
This settles the case (2-a).

\noindent{\bf{Case (2-b)} $\#\mathcal R_i^0\le1$.} 

Since $\mathcal P_i$ satisfies \eqref{large}, either
$\mathcal R_i^1$ or $\mathcal R_i^2$ then contains more than
$C_0\log(3+\widehat{\alpha}_{\eta,M}(h;\Delta))$ elements. By symmetry, assume that
$\mathcal R_i^1$ is large. Every
$W\in\mathcal R_i^1$ satisfies
$$
\widehat{\alpha}_{\eta,M}(h;\Delta)^{-1}
\le\|h\mathsf w_{\Delta,W}\|\le1.
$$
The interval
$[\widehat{\alpha}_{\eta,M}(h;\Delta)^{-1},1]$ is covered by at most
$\ll \log(3+\widehat{\alpha}_{\eta,M}(h;\Delta))$ dyadic
subintervals. Since $C_0$ was chosen sufficiently
large, the pigeonhole principle gives distinct
$W_1,W_2\in\mathcal R_i^1$ whose Pl\"ucker norms are comparable by an
absolute factor. By \eqref{eq:relevant-subspaces-definition} and the
definition of $\widetilde\alpha_{i,\eta,M}(h;\Delta)$, the
$\phi_i$-value of each lies between
$e^{-12n^3s/\beta_i}\widetilde\alpha_{i,\eta,M}(h;\Delta)^{1/\beta_i}$ and
$\widetilde\alpha_{i,\eta,M}(h;\Delta)^{1/\beta_i}$. They are therefore comparable by a factor
$e^{2C_ns}$. Using \eqref{eq:phi-decomposition-global} and the comparability
of their Pl\"ucker norms, we obtain
\begin{equation}\label{eq:barphi-comparison-competitors}
\mathsf L^{-1}\bar\phi_i(hw_2)
\le\bar\phi_i(hw_1)
\le\mathsf L\bar\phi_i(hw_2) \quad\text{for }\mathsf L:= e^{2C_ns/\theta_i},
\end{equation}
 where $w_m=\mathsf w_{\Delta,W_m}$. Set
$$
w_-:=\mathsf w_{\Delta,W_1\cap W_2},
\qquad
w_+:=\mathsf w_{\Delta,W_1+W_2}.
$$
If $i+j<N$, then, as above,
$W_1+W_2\notin\widetilde{\mathscr Q}_{i+j,\eta,M}(\Delta)$.

\noindent{\bf{Case (2-b-(i))}} Suppose first that $i+j<N$ and $\|hw_+\|\le1$.  Apply
\Cref{lem:BQinequality} to the lattice $h\Delta$ and the subspaces
$hW_1,hW_2$. Together with
\eqref{eq:barphi-comparison-competitors}, this yields
$$
\phi_i(hw_1)\phi_i(hw_2)
\ll e^{2C_ns}
\|hw_-\|^{-1}\phi_{i+j}(hw_+).
$$
Combining this with \eqref{eq:relevant-subspaces-definition} gives
$$
\widetilde\alpha_{i,\eta,M}(h;\Delta)^{2/\beta_i}
\ll
e^{2C_ns}
\alpha_{i-j}^+(h\Delta)
\phi_{i+j}(hw_+).
$$
Because $W_1+W_2\notin\widetilde{\mathscr Q}_{i+j,\eta,M}(\Delta)$ and
$\|hw_+\|\le1$, the last factor satisfies
$$
\phi_{i+j}(hw_+)
\le
\widetilde\alpha_{i+j,\eta,M}(h;\Delta)^{1/\beta_{i+j}}.
$$

Consequently, after absorbing the additional log-Lipschitz factor from
\eqref{eq:competitor-reduction} into $e^{C_ns}$, we obtain
$$
e^{2C_ns}\widetilde\alpha_{i,\eta,M}(h;\Delta)\ll e^{2C_ns}
\widetilde\alpha_{i+j,\eta,M}(h;\Delta)^{\beta_i/(2\beta_{i+j})}
\alpha_{i-j}^+(h\Delta)^{\beta_i/2}.
$$
Apply \eqref{eq:Youngineq} with
$x=\widetilde\alpha_{i+j,\eta,M}(h;\Delta)^{1/\beta_{i+j}}$,
$y=\alpha_{i-j}^+(h\Delta)$, $z=e^{2C_ns}$, and
$\rho=e^{-c_n\theta_1s}$, using its one-factor version when $i-j=0$.
Using
\eqref{eq:beta-tau-comparison}, we obtain
\begin{equation}\label{eq:R1-Mother-contraction}
e^{2C_ns}\widetilde\alpha_{i,\eta,M}(h;\Delta)
\ll e^{-c_n\theta_1 s}\widetilde\alpha_{\eta,M}(h;\Delta)
+e^{C_n\theta_1^{-1}s}.
\end{equation}

\noindent{\bf{Case (2-b-(ii))}} Now suppose that either $\|hw_+\|>1$ or $i+j=N$. In the latter
case, $hw_+$ is a top-degree Pl\"ucker vector of the unimodular lattice
$h\Delta$, and hence $\|hw_+\|=1$. Thus, in either case, the standard intersection-sum inequality and the comparability of $\|hw_1\|$ and
$\|hw_2\|$ imply
$$
\|hw_-\|\ll\|hw_1\|^2.
$$
Hence $\|hw_1\|^{-1}\ll
\alpha_{i-j}^+(h\Delta)^{1/2}$. In the region where the first
exceptional component is maximal, norm equivalence gives
$$
\phi_i(hw_1)\asymp_n
\|hw_1\|^{-1+2\theta_i}
\|hw_1-\pi_{k,1}(hw_1)\|^{-2\theta_i}.
$$
Because $(h,\Delta)$ is outside the exceptional set,
$$
\|hw_1-\pi_{k,1}(hw_1)\|
>\varepsilon_{s,\eta,M,M'}(h;\Delta).
$$
Using the formula for $\phi_i$ in the region where the first exceptional
component is maximal, we obtain
$$
\phi_i(hw_1)
\ll
\alpha_{i-j}^+(h\Delta)^{1/2}
\left(
 e^{2M'\theta_is}
 +\alpha(h\Delta)^{2bM'\theta_i}
\right).
$$
The relevance condition for $W_1$, together with
$\alpha_{i-j}^+(h\Delta)
\le \widetilde\alpha_{\eta,M}(h;\Delta)^{1/\tau_{i-j}}$ when $i-j\ge1$
(and $\alpha_0^+=1$ otherwise) and
$\alpha(h\Delta)\le \widetilde\alpha_{\eta,M}(h;\Delta)^{1/\tau_*}$, now gives
$$
\widetilde\alpha_{i,\eta,M}(h;\Delta)
\ll
e^{2C_ns}e^{C_nM'\theta s}
\widetilde\alpha_{\eta,M}(h;\Delta)^{q_i},
$$
where $q_i:=
\frac{\beta_i}{2\tau_{i-j}}
+
\frac{2bM'\theta_i\beta_i}{\tau_*}$ if $i-j\ge 1$ and
$q_i:=
\frac{2bM'\theta_i\beta_i}{\tau_*}$ if $i-j=0$.

There are only finitely many admissible pairs $(i,j)$. As $\theta\to0$,
$\beta_i$ and $\tau_*$ tend to $1$, as does $\tau_{i-j}$ when $i-j\ge 1$,
while $bM'\theta_i$ tends to $0$.
Thus $q_i$ tends to $1/2$, or to $0$ when $i-j=0$.
After decreasing $\theta_0(n,b,M')$, we therefore have $q_i\le2/3$
uniformly. Hence
\begin{equation}\label{eq:R1-endpoint-sublinear}
\widetilde\alpha_{i,\eta,M}(h;\Delta)
\ll e^{2C_ns}e^{C_nM'\theta s}
\widetilde\alpha_{\eta,M}(h;\Delta)^{2/3}.
\end{equation}
The elementary scaled inequality
$x\widetilde\alpha_{\eta,M}(h;\Delta)^{2/3}
\le \rho\widetilde\alpha_{\eta,M}(h;\Delta)+C\rho^{-2}x^3$, with
$\rho=e^{-c_n\theta_1 s}$, therefore gives the same conclusion as
\eqref{eq:R1-Mother-contraction}. The case of $\mathcal R_i^2$ is
identical.

In Case (2), the log-Lipschitz estimate gives
$$
\max_{V\in\mathcal P_i}
\phi_i(a_suh\mathsf w_{\Delta,V})^{\beta_i}
\ll e^{2C_ns}\widetilde\alpha_{i,\eta,M}(h;\Delta).
$$
Combining this with \eqref{eq:competitor-reduction},
\eqref{eq:R0-contraction}, \eqref{eq:R1-Mother-contraction}, and
\eqref{eq:R1-endpoint-sublinear}, and absorbing the same sublinear error as
in Case (1), yields \eqref{eq:few-competitors-bound}. Thus this
bound holds for every $1\le i\le N-1$. Summing over $i$ and adding
\Cref{prop:globalboundedexpansion'} gives the corresponding estimate in
terms of $\theta_1$. Since $\theta_1=10^{1-N}\theta$ and $N=n^2$, replacing
$\theta_1$ by $\theta$ only changes the dimensional constants. Absorbing
this change into $c_n$ and $C_n$ yields \eqref{eq:globalcontraction}.
\end{proof}

\section{Avoidance of near-exceptional Pl\"ucker vectors}
\label{s:avoidance}

The global Margulis inequality (\Cref{thm:globalcontraction}) from the previous section holds outside an
exceptional set $\mathcal E_{s,\eta,M,M'}$. This set records the event that a rational
Pl\"ucker vector that is not already quasi-null for the base lattice is moved close to one of the column- or row-isotropic summands. The purpose of this section is to show that, under the Diophantine condition, this exceptional event has small measure in the relevant horospherical
averages.

The proof has two ingredients. First, we use a quantitative nondivergence estimate for a
prescribed set of primitive Pl\"ucker vectors (\Cref{prop:Kle08variant}). Applied to the Pl\"ucker
vectors that are not already quasi-null, it shows that either the desired
measure estimate holds, or several such vectors have a wedge that
remains uniformly small throughout the averaging ball. Second, the
instability of nontrivial $H$-representations forces this wedge to be
almost entirely contained in an exterior power of a column- or row-isotropic summand.
The Diophantine condition then upgrades this near-containment to exact
containment. This would force the original Pl\"ucker
vectors themselves to be quasi-null, contradicting the way the prescribed
set was chosen.

Thus the avoidance argument concerns only near-exceptional non-isotropic subspaces.
Exact isotropic subspaces are not excluded here; they are omitted from the
modified height and counted separately in the singular counting theorem (\Cref{thm:quasinullcontribution2}).
\medskip

\begin{proposition}[Avoidance estimate]\label{prop:avoidance'}
There exists $b=b(n)\ge1$ such that the following holds. Let
$
0<\eta<1,\;  M>1,
$
and let $\Delta\in X$ be an $(\eta,M)$-Diophantine lattice. If
$
M'\ge 10n^3bM^2$, $s\ge 2\log n$, and
$t\ge 10n^3bM's$,
then
\begin{equation}\label{eq:avoidance}
\begin{aligned}
&\int_{B_{U}(1)}
\alpha(a_tu\Delta)^{2}\, 
\mathds{1}_{\mathcal E_{s,\eta,M,M'}}(a_tu;\Delta)\,du \ll
e^{-\frac{M'}{bM}s}
+ e^{-\frac{1}{bM}t}.
\end{aligned}
\end{equation}
The implicit constant depends only on
$\Delta,\eta$, and $M$.
\end{proposition}

The remainder of the section is devoted to the proof.

\subsection{Instability and selected-vector nondivergence}
We first record a quantitative instability statement. 

\begin{lemma}\label{lem:liegroupinstability}
Let $V$ be a nontrivial irreducible finite-dimensional $H$-module. There
exist  $c=c(V)>0$ and $\kappa=\kappa(V)>0$ such that
for every $0\ne v\in V$ and $t\ge1$,
$$
\sup_{u\in B_{U}(1)}\|a_tuv\|
\ge c\,  e^{\kappa t}\|v\|.
$$
\end{lemma}

\begin{proof}
 We first record that the subspace of $U$-fixed vectors
$V^{U}$ is non-zero and that $a_t$ expands $V^{U}$ exponentially. Indeed,
after complexification, each irreducible constituent is a tensor product of
irreducible modules for the two $\SL_n$-factors. For one factor, write the
module as ${\mathsf S}_\lambda(\mathbb C^n)$ after adding a determinant
twist, with $\lambda_n=0$. The determinant twist is trivial on $\SL_n$.
The $U_i$-fixed vectors form the top component for the maximal parabolic
with unipotent radical $U_i$. Along the principal element
$b_t=\operatorname{diag}(e^{-t},\ldots,e^{-t},e^{(n-1)t})$, every nonzero
weight on this top component is at least a positive multiple of $t$, unless
the representation of that factor is trivial. Since $V$ is nontrivial,
there are constants $c_1,\kappa_1>0$, depending only on $V$, such that
\begin{equation}\label{eq:UH-fixed-expansion}
\|a_tw\|
\ge c_1e^{\kappa_1t}\|w\|
\qquad
(w\in V^{U},\ t\ge0).
\end{equation}
Since $V^U$ is the maximal $a_t$-weight space, let
$P_U:V\to V^U$ be the projection along the sum of the remaining
$a_t$-weight spaces. Then $P_U$ commutes with $a_t$.
Define $q:V\to[0,\infty)$ by
$$
q(v):=\sup_{u\in B_U(1)}\|P_U(uv)\|.
$$
Absolute homogeneity and the triangle inequality are immediate, so $q$
is a seminorm. We claim that it is a norm. Indeed, if $q(v)=0$, then the polynomial map
$u\mapsto P_{U}(uv)$ vanishes on $B_{U}(1)$, hence on all of $U$. The span
of $Uv$ is then a nonzero finite-dimensional $U$-module with no nonzero
$U$-fixed vector, which is impossible for a unipotent group. Hence
$q(v)=0$ only for $v=0$. Therefore
$q(v)\gg_V\|v\|$. Hence we obtain
$$
\sup_{u\in B_{U}(1)}\|a_tuv\|
\gg_V
\sup_{u\in B_{U}(1)}\|a_tP_{U}(uv)\|
\gg_V e^{\kappa_1t}q(v)
\gg_V e^{\kappa_1t}\|v\|.
$$
This proves the lemma.
\end{proof}

We next state the form of quantitative non-divergence needed here. This is
in the spirit of the quantitative non-divergence theorem of
Kleinbock--Margulis \cite{kleinbock-margulis:1998}, but the short vectors
must belong to a prescribed subset of the primitive vectors. The bad
alternative must also retain the rank: for a rank-$m$ wedge, the lower bound
must occur at the natural scale $\rho^m$, rather than at a rank-independent
scale. The weighted-poset theorem of Kleinbock
\cite[Theorem~2.1]{kleinbock:2008} gives precisely this form.

Recall that for given constants $C, \tau>0$, a continuous function $f$ on a ball $B\subset\mathbb R^\ell$
is \emph{$(C,\tau)$-good} if, for every ball $B_0\subset B$ and 
$\varepsilon>0$,
$$
\operatorname{Leb}\{x\in B_0:|f(x)|<\varepsilon\}
\le
C\left(\frac{\varepsilon}{\sup_{B_0}|f|}\right)^\tau
\operatorname{Leb}(B_0).
$$
A continuous map $h:B\to\GL_k(\mathbb R)$ is called $(C,\tau)$-good if the function
$x\mapsto\|h(x)w\|$ is $(C,\tau)$-good for every
$0\ne w\in\wedge^i\mathbb R^k$ and $1\le i\le k$.

When $\Upsilon$ is the set of all primitive vectors in $\Delta$, the following is the usual quantitative nondivergence dichotomy of Kleinbock \cite{kleinbock:2008}; the
point of the present formulation is that the same dichotomy holds for any
prescribed subset of primitive vectors.
\begin{proposition}[Selected-vector quantitative nondivergence]
\label{prop:Kle08variant}
For every $\ell,k\in\mathbb N$, there is 
$c=c(\ell,k)>0$ with the following property. Let
$x_0\in \mathbb R^\ell$, $r>0$,  and
$h:B(x_0,3^kr)\to\GL_k(\mathbb R)$ be a $(C,\tau)$-good map for some $C, \tau>0$. Let
$\Delta<\mathbb R^k$ be a lattice and let
$\emptyset\ne\Upsilon\subset\Delta$ be a subset consisting of primitive
vectors. For every
$0<\varepsilon\le\rho\le1$, one of the following alternatives holds:

\begin{enumerate}
\item there are linearly independent
$w_1,\dots,w_m\in\Upsilon$ such that
$$
\sup_{x\in B(x_0,r)}
\|h(x)(w_1\wedge\cdots\wedge w_m)\|
\le \rho^m;
$$

\item
$$
\operatorname{Leb}\left\{x\in B(x_0, r):
\min_{w\in\Upsilon}\|h(x)w\|\le\varepsilon\right\}
\le
c\; C\; \left(\frac{\varepsilon}{\rho}\right)^\tau
\operatorname{Leb}(B(x_0,r)).
$$
\end{enumerate}
\end{proposition}
\begin{proof}
Set $B:=B(x_0,r)$ and $\widetilde B:=B(x_0,3^kr)$.
Let $\mathfrak P_\Upsilon$ be the partially ordered set of nonzero
$\Delta$-rational subspaces $W<\mathbb R^k$ spanned by finite subsets
of $\Upsilon$, ordered by inclusion.

For $W\in\mathfrak P_\Upsilon$, set
$\Delta_W:=\Delta\cap W$, and let $\mathsf w_{\Delta,W}$ be a primitive
Pl\"ucker vector of $\Delta_W$. If $m=\dim W$, define
$$
\iota_\Upsilon(W)
:=
\min
\left\{
[\Delta_W:\mathbb Zw_1+\cdots+\mathbb Zw_m]:
\begin{array}{l}
w_1,\dots,w_m\in\Upsilon\cap W,\\
\operatorname{span}_{\mathbb R}\{w_1,\dots,w_m\}=W
\end{array}
\right\}.
$$
 Set
$$
\psi_W(x)
:=
\iota_\Upsilon(W)\,
\|h(x)\mathsf w_{\Delta,W}\|.
$$
This function is $(C,\tau)$-good on $\widetilde B$ by the hypothesis on $h$. Moreover,
\begin{equation}\label{eq:selected-minimal-wedge}
\psi_W(x)
=
\min
\left\{
\|h(x)(w_1\wedge\cdots\wedge w_m)\|:
\begin{array}{l}
w_1,\dots,w_m\in\Upsilon\cap W,\\
\operatorname{span}_{\mathbb R}\{w_1,\dots,w_m\}=W
\end{array}
\right\}.
\end{equation}
Indeed, each wedge $w_1\wedge\cdots\wedge w_m$ is, up to sign, the
corresponding index times $\mathsf w_{\Delta,W}$. On $\mathfrak P_\Upsilon$, consider the weight function:
$$
W\mapsto \zeta(W):=\rho^{\dim W}.
$$
Conditions (A0), (A1), and (A3) of
\cite[Theorem~2.1]{kleinbock:2008} are satisfied. Condition (A0)
follows because every strictly increasing chain in
$\mathfrak P_\Upsilon$ has at most $k$ elements, and (A1) follows from
the goodness of the functions $\psi_W$. For (A3), observe that
$$
\psi_W(x)<\zeta(W)\le1
\quad\to\quad
\|h(x)\mathsf w_{\Delta,W}\|<1.
$$
For each fixed $x$, the lattice $h(x)\Delta$ has only finitely many
primitive subgroups of bounded covolume.

If condition (A2) fails, then for some
$W\in\mathfrak P_\Upsilon$,
$$
\sup_{x\in B}\psi_W(x)<\rho^{\dim W}.
$$
Choosing vectors that attain the minimum defining
$\iota_\Upsilon(W)$ and using
\eqref{eq:selected-minimal-wedge} gives alternative~(1).
We may therefore assume that (A2) holds, namely,
$$
\sup_{x\in B}\psi_W(x)\ge\rho^{\dim W}
\qquad(W\in\mathfrak P_\Upsilon).
$$
The conclusion is immediate when $\varepsilon=\rho$, after enlarging
the constant, so suppose that $\varepsilon<\rho$. Choose
$$
1<\lambda<{\rho}{\varepsilon}^{-1}
\qquad\text{and set}\qquad
\delta:={\lambda\varepsilon}{\rho}^{-1}<1.
$$
By \cite[Theorem~2.1]{kleinbock:2008},
$$
\operatorname{Leb}
\bigl(B\smallsetminus\Phi(\delta,\mathfrak P_\Upsilon)\bigr)
\ll_{\ell,k}
C\delta^\tau\operatorname{Leb}(B),
$$
where $\Phi(\delta,\mathfrak P_\Upsilon)$ denotes the set of
$\delta$-marked points.
We claim that every
$x\in\Phi(\delta,\mathfrak P_\Upsilon)$ satisfies
$$
\|h(x)w\|>\varepsilon
\qquad(w\in\Upsilon).
$$
Let
$W_1\subsetneq\cdots\subsetneq W_j$
be a marking flag at $x$, possibly empty, and put $W_0:=\{0\}$.
Fix $w\in\Upsilon$. If $j\ge1$ and $w\in W_j$, let $i$ be the least
index such that $w\in W_i$; otherwise put $i:=j+1$. In either case,
set
$$
U:=W_{i-1}+\mathbb Rw.
$$
Then $U\in\mathfrak P_\Upsilon$,
$\dim U=\dim W_{i-1}+1$, and $U$ is comparable with every member of
the marking flag. The marking inequalities therefore give
$$
\psi_U(x)
\ge
\delta\rho^{\dim U}
=
\lambda\varepsilon\rho^{\dim W_{i-1}}.
$$

Choose vectors in $\Upsilon\cap W_{i-1}$ attaining
$\iota_\Upsilon(W_{i-1})$; when $i=1$, use
the empty wedge and the convention
$\psi_{W_0}=1$. By
\eqref{eq:selected-minimal-wedge} and the submultiplicativity of the
exterior norm,
$$
\psi_U(x)
\le
\psi_{W_{i-1}}(x)\|h(x)w\|
\le
\rho^{\dim W_{i-1}}\|h(x)w\|.
$$
Consequently,
$$
\|h(x)w\|\ge\lambda\varepsilon>\varepsilon,
$$
which proves the claim.

Thus the sublevel set in alternative~(2) is contained in
$B\smallsetminus\Phi(\delta,\mathfrak P_\Upsilon)$. Hence
$$
\operatorname{Leb}\left\{
x\in B:
\min_{w\in\Upsilon}\|h(x)w\|\le\varepsilon
\right\}
\ll_{\ell,k}
C\left(\frac{\lambda\varepsilon}{\rho}\right)^\tau
\operatorname{Leb}(B).
$$
Letting $\lambda\to1$ proves alternative~(2).
\end{proof}

\subsection{Avoiding almost-isotropic Pl\"ucker vectors}
For $0\le\eta<1$, $M\ge1$, and $1\le k\le n-1$, recall the cone of Pl\"ucker vectors that are
$(\eta,M)$-close to one of the exceptional summands:
$$
\mathscr Q_{kn,\eta,M}
=
\left\{
0\ne w\in\wedge^{kn}\M_n(\mathbb R):
\min_{m=1,2}\|w-\pi_{k,m}(w)\|
\le \eta\|w\|^{-M}
\right\}.
$$

For $1\le k\le n-1$, $s\ge 2\log n$ and
$0<\varepsilon<1$, define
$$
\Xi_k(s,\varepsilon)
:=
\left\{
0\ne v\in\wedge^{kn}\M_n(\mathbb R):
\|v\|\le e^{4n^3s},
\min_{m=1,2}\|v-\pi_{k,m}(v)\|\le\varepsilon
\right\}.
$$
Let $\mathcal K_{\eta,M}(s,\varepsilon)$ be the set of pairs
$(h,\Delta)\in H\times X$ for which there exist $1\le k\le n-1$ and a
$kn$-dimensional $\Delta$-rational subspace $V$ such that
$$
\mathsf w_{\Delta,V}\notin\mathscr Q_{kn,\eta,M}
\qquad\text{and}\qquad
h\mathsf w_{\Delta,V}\in\Xi_k(s,\varepsilon).
$$
Thus $\mathcal K_{\eta,M}(s,\varepsilon)$ records the event that a
rational Pl\"ucker vector that is not already quasi-null for the base
lattice is moved close to an exceptional summand at scale
$s$.

The following quantitative estimate is the main input in the proof of \Cref{prop:avoidance'}.

\begin{proposition}\label{prop:avoidance0}
There exists $b_1=b_1(n)>1$ such that the following holds. Let
$0<\eta<1$, $M>1$, and let $\Delta\in X$ be
$(\eta,M)$-Diophantine. If
$s\ge 2\log n$, $0<\varepsilon<e^{-b_1Ms}$, $t\ge1$, and
$t\ge b_1\log \varepsilon^{-1},$
then
\begin{equation}\label{eq:avoidance0}
\operatorname{Leb}
\left\{u\in B_{U}(1):
(a_tu,\Delta)
\in\mathcal K_{\eta,M}(s,\varepsilon)
\right\}
\ll
\varepsilon^{1/(b_1M)},
\end{equation}
with the implicit constant depending only on
$\Delta,\eta$, and $M$.
\end{proposition}
\begin{proof}
For $1\le k\le n-1$, let
$$
\Upsilon_{k,\eta,M}
:=
\left\{
\mathsf w_{\Delta,V}:
V\text{ is $kn$-dimensional and $\Delta$-rational},
\mathsf w_{\Delta,V}\notin\mathscr Q_{kn,\eta,M}
\right\}.
$$
After choosing one sign for each Pl\"ucker vector, this is a subset of
the primitive vectors of the exterior lattice
$\wedge^{kn}\Delta<\wedge^{kn}\M_n(\mathbb R)$. Indices $k$ for which
$\Upsilon_{k,\eta,M}$ is empty are omitted below.

Fix $k$. We first treat proximity to $\mathcal M_{k,1}$; the row case
is identical. Recall the decomposition
$$
\wedge^{kn}\M_n(\mathbb R)
=
\mathcal M_{k,0}\oplus\mathcal M_{k,1}\oplus\mathcal M_{k,2}.
$$
Choose a sufficiently small constant $c_0=c_0(n)>0$, and define
$\psi=\psi_{k,1,s,\varepsilon}\in
\GL(\wedge^{kn}\M_n(\mathbb R))$ by
$$
\psi|_{\mathcal M_{k,1}}
=
c_0\varepsilon e^{-4n^3s}\id,
\qquad
\psi|_{\mathcal M_{k,0}\oplus\mathcal M_{k,2}}
=
c_0\id.
$$
The constant $c_0$ may be chosen so that
$$
\|y\|\le e^{4n^3s}
\quad\text{and}\quad
\|y-\pi_{k,1}(y)\|\le\varepsilon
\quad\to\quad
\|\psi(y)\|\le\varepsilon.
$$
Indeed, the relevant projections have operator norms bounded in terms
of $n$ alone. Define $\psi_{k,2,s,\varepsilon}$ similarly, with
$\mathcal M_{k,2}$ in place of $\mathcal M_{k,1}$. By the definition
of $\mathcal K_{\eta,M}(s,\varepsilon)$,
\begin{align}
&\operatorname{Leb}
\left\{u\in B_U(1):
(a_tu,\Delta)\in\mathcal K_{\eta,M}(s,\varepsilon)
\right\}
\notag\\
&\qquad\le
\sum_{k=1}^{n-1}\sum_{m=1}^2
\operatorname{Leb}
\left\{u\in B_U(1):
\min_{w\in\Upsilon_{k,\eta,M}}
\|\psi_{k,m,s,\varepsilon}a_tuw\|
\le\varepsilon
\right\}.
\label{eq:avoidance-union-selected}
\end{align}

Let
$$
D:=
\max_{1\le k\le n-1}
\dim\wedge^{kn}\M_n(\mathbb R),
\qquad
\delta:=(20DM)^{-1}.
$$
Since $0<\delta<1$, we have
$0<\varepsilon\le\varepsilon^{1-\delta}\le1$. For every $k$ and $m$,
the map
$$
u\mapsto\psi_{k,m,s,\varepsilon}a_tu
$$
is $(C,\tau)$-good on the enlarged ball required in
\Cref{prop:Kle08variant}, where $C>0$ and $\tau>0$ depend only on $n$.
Indeed, all matrix coefficients in the relevant exterior powers are
polynomials of degree bounded only in terms of $n$, so the required
goodness follows from the multivariable Remez inequality; in the top
exterior degree, the corresponding function is a nonzero constant.

Fix $k$ and $m$, and abbreviate
$\psi=\psi_{k,m,s,\varepsilon}$. Applying
\Cref{prop:Kle08variant} with the parameters
$\varepsilon$ and $\varepsilon^{1-\delta}$, we obtain
\begin{equation}\label{eq:quantitativenondivergenceforwedge}
\operatorname{Leb}
\left\{u\in B_U(1):
\min_{w\in\Upsilon_{k,\eta,M}}
\|\psi a_tuw\|\le\varepsilon
\right\}
\ll
\varepsilon^{\tau\delta},
\end{equation}
unless there are linearly independent
$w_1,\dots,w_r\in\Upsilon_{k,\eta,M}$, with
$1\le r\le\dim\wedge^{kn}\M_n(\mathbb R)$, such that, for
$$
\mathbf w:=w_1\wedge\cdots\wedge w_r
\in\wedge^r\bigl(\wedge^{kn}\M_n(\mathbb R)\bigr),
$$
one has
\begin{equation}\label{eq:shortsubspaceW}
\sup_{u\in B_U(1)}
\|\psi a_tu\mathbf w\|
\le
\varepsilon^{r(1-\delta)}.
\end{equation}
Here and below, the same symbols denote the induced actions on exterior
powers. We show that this alternative is impossible when $\varepsilon$
is sufficiently small. The remaining range of $\varepsilon$ will be
absorbed into the implicit constant in \eqref{eq:avoidance0}.

It suffices to consider $m=1$. Write
$$
\wedge^{kn}\M_n(\mathbb R)
=
\mathcal M_{k,1}\oplus\mathcal M_{k,*},
\qquad
\mathcal M_{k,*}
:=
\mathcal M_{k,0}\oplus\mathcal M_{k,2},
$$
and put $q:=\dim\mathcal M_{k,1}$. The induced decomposition
\begin{equation}\label{eq:number-of-exceptional-factors}
\wedge^r\bigl(\wedge^{kn}\M_n(\mathbb R)\bigr)
=
\bigoplus_j
\left(\wedge^j\mathcal M_{k,1}\right)
\wedge
\left(\wedge^{r-j}\mathcal M_{k,*}\right)
\end{equation}
is $H$-invariant. On the $j$-th summand, the induced action of
$\psi$ is multiplication by
$c_0^r\varepsilon^je^{-4n^3sj}$.
The common factor $c_0^r$ will be absorbed into the implicit constants.
Let $P_j$ denote the projection onto the $j$-th summand. In particular,
$$
P_r=\pi_{k,1}^{(r)}:=\wedge^r\pi_{k,1},
$$
with the convention that $P_r=0$ when $r>q$.

By complete reducibility, we may choose, in each summand of
\eqref{eq:number-of-exceptional-factors}, an $H$-invariant complement
to its $H$-fixed subspace. This gives an $H$-equivariant decomposition
$$
\wedge^r\bigl(\wedge^{kn}\M_n(\mathbb R)\bigr)
=
V_{\mathrm{tr}}\oplus V_{\mathrm{nt}},
$$
where
$V_{\mathrm{tr}}=\left(
\wedge^r\bigl(\wedge^{kn}\M_n(\mathbb R)\bigr)
\right)^H$
and $V_{\mathrm{nt}}$ is a sum of nontrivial irreducible
$H$-modules. Let $P_{\mathrm{tr}}$ and $P_{\mathrm{nt}}$ denote the
corresponding projections. By construction, these projections commute
with $\psi$ and with every $P_j$.

Since $H$ acts trivially on $V_{\mathrm{tr}}$,
\eqref{eq:shortsubspaceW} implies
\begin{equation}\label{eq:trivial-top-bound}
\|P_rP_{\mathrm{tr}}\mathbf w\|
\ll
\varepsilon^{-r}e^{4n^3sr}
\varepsilon^{r(1-\delta)},
\qquad
\|(I-P_r)P_{\mathrm{tr}}\mathbf w\|
\ll
\varepsilon^{-(r-1)}e^{4n^3s(r-1)}
\varepsilon^{r(1-\delta)}.
\end{equation}
Indeed, up to the fixed factor $c_0^r$, the map $\psi$ scales the
top summand by $\varepsilon^re^{-4n^3sr}$ and every off-top summand
by at least
$\varepsilon^{r-1}e^{-4n^3s(r-1)}$.

Since $r\le D$, $\delta=(20DM)^{-1}$, and
$\varepsilon<e^{-b_1Ms}$, choosing $b_1=b_1(n)$ sufficiently large
gives
\begin{align}
\varepsilon^{-r}e^{4n^3sr}\varepsilon^{r(1-\delta)}
&=
e^{4n^3rs}\varepsilon^{-r\delta}
\ll
\varepsilon^{-1/(4M)},
\label{eq:top-exponent-bookkeeping}\\
\varepsilon^{-(r-1)}e^{4n^3s(r-1)}
\varepsilon^{r(1-\delta)}
&=
e^{4n^3s(r-1)}\varepsilon^{1-r\delta}
\ll
\varepsilon^{1/2}.
\label{eq:offtop-exponent-bookkeeping}
\end{align}

On $V_{\mathrm{nt}}$, the smallest singular value of the induced
action of $\psi$ is bounded below, up to a constant depending only on
$n$, by $\varepsilon^re^{-4n^3sr}$. Hence
\eqref{eq:shortsubspaceW} gives
$$
\sup_{u\in B_U(1)}
\|a_tuP_{\mathrm{nt}}\mathbf w\|
\ll
\varepsilon^{-r}e^{4n^3sr}
\varepsilon^{r(1-\delta)}
\ll
\varepsilon^{-1/(4M)}.
$$
Applying \Cref{lem:liegroupinstability} to the finitely many irreducible
summands of $V_{\mathrm{nt}}$, and increasing $b_1(n)$ if necessary,
we see that the condition $t\ge b_1\log\varepsilon^{-1}$ implies
\begin{equation}\label{eq:expansionboundirr1}
\|P_{\mathrm{nt}}\mathbf w\|
\ll
\varepsilon^{1/2}.
\end{equation}

Combining
\eqref{eq:trivial-top-bound}--\eqref{eq:expansionboundirr1}, we obtain
\begin{equation}\label{eq:almostisotropic0}
0<\|\mathbf w\|\ll\varepsilon^{-1/(4M)},
\qquad
\|\mathbf w-\pi_{k,1}^{(r)}(\mathbf w)\|
\ll\varepsilon^{1/2}.
\end{equation}
There exists $c_\Delta>0$ such that
$\|\mathbf w\|\ge c_\Delta$, because $\mathbf w$ is a nonzero vector
in one of the finitely many exterior lattices
$\wedge^r(\wedge^{kn}\Delta)$.

If $r>q$, then $\pi_{k,1}^{(r)}=0$, and
\eqref{eq:almostisotropic0} contradicts
$\|\mathbf w\|\ge c_\Delta$ for sufficiently small $\varepsilon$. Suppose next that $0<r<q$, and put
$$
\mathbf w'
:=
\pi_{k,1}^{(r)}(\mathbf w)
\in\wedge^r\mathcal M_{k,1}.
$$
The vector $\mathbf w'$ is decomposable, and for sufficiently small
$\varepsilon$, it is nonzero and satisfies
$\|\mathbf w'\|\asymp\|\mathbf w\|$. No nonzero decomposable vector
in $\wedge^r\mathcal M_{k,1}$ belongs to $V_{\mathrm{tr}}$:
otherwise the corresponding proper nonzero $r$-dimensional subspace
of the irreducible $H_1$-module $\mathcal M_{k,1}$ would be
$H_1$-invariant. By compactness of
$\operatorname{Gr}(r,\mathcal M_{k,1})$, there exists $c=c(n)>0$
such that
$$
\|P_{\mathrm{nt}}\mathbf w'\|
\ge c\|\mathbf w'\|
$$
for every decomposable
$\mathbf w'\in\wedge^r\mathcal M_{k,1}$. On the other hand,
$P_{\mathrm{nt}}$ commutes with $\pi_{k,1}^{(r)}$, so
\eqref{eq:expansionboundirr1} gives
$$
\|P_{\mathrm{nt}}\mathbf w'\|
\ll\varepsilon^{1/2}.
$$
Since $\|\mathbf w'\|\asymp\|\mathbf w\|\ge c_\Delta$, this is
impossible for sufficiently small $\varepsilon$.

It remains to consider $r=q$. If
$\mathbf w\ne\pi_{k,1}^{(q)}(\mathbf w)$, then the
$(\eta,M)$-Diophantine condition and
\eqref{eq:almostisotropic0} give
$$
\|\mathbf w-\pi_{k,1}^{(q)}(\mathbf w)\|
\ge
\eta\|\mathbf w\|^{-M}
\gg
\eta\varepsilon^{1/4},
$$
contradicting the second bound in
\eqref{eq:almostisotropic0} for sufficiently small $\varepsilon$.
Hence
$$
\mathbf w=\pi_{k,1}^{(q)}(\mathbf w).
$$
Since $\mathbf w$ is a nonzero decomposable $q$-vector and
$\dim\mathcal M_{k,1}=q$, the subspace
$\operatorname{span}\{w_1,\ldots,w_q\}$ is exactly
$\mathcal M_{k,1}$. Thus each $w_i$ belongs to
$\mathcal M_{k,1}$, and hence
$$
\min_{m=1,2}\|w_i-\pi_{k,m}(w_i)\|=0.
$$
It follows that $w_i\in\mathscr Q_{kn,\eta,M}$, contradicting
$w_i\in\Upsilon_{k,\eta,M}$.

We have therefore ruled out the exceptional alternative in
\Cref{prop:Kle08variant} for all sufficiently small $\varepsilon$.
It follows from
\eqref{eq:avoidance-union-selected} and
\eqref{eq:quantitativenondivergenceforwedge} that
$$
\operatorname{Leb}
\left\{u\in B_U(1):
(a_tu,\Delta)\in\mathcal K_{\eta,M}(s,\varepsilon)
\right\}
\ll
\varepsilon^{\tau\delta}.
$$
Finally, enlarge $b_1=b_1(n)$ so that
${b_1}^{-1}\le {\tau}/({20D})$.
Since $\tau\delta=\tau/(20DM)$, we then have
$\varepsilon^{\tau\delta}\le
\varepsilon^{1/(b_1M)}$. Enlarging the implicit constant to cover the
remaining bounded range of $\varepsilon$ proves
\eqref{eq:avoidance0}.
\end{proof}

\subsection{Proof of the avoidance estimate}

\begin{proof}[Proof of \Cref{prop:avoidance'}]
We first obtain an upper bound for $\alpha(a_tu\Delta)$ on the exceptional set. We claim
that there is $C_0=C_0(\Delta,\eta,M)\ge1$ such that
\begin{equation}\label{eq:alpha-cap-on-exceptional-set}
(a_tu,\Delta)\in
\mathcal E_{s,\eta,M,M'}
\quad\to\quad
\alpha(a_tu\Delta)\le C_0e^{t/(b^2M)}.
\end{equation}
Indeed, suppose that the pair is exceptional and that
$\alpha(a_tu\Delta)>e^{20n^3s}$, so that the second line of
\eqref{eq:exceptional-scale} applies. There are
$1\le k\le n-1$, $m\in\{1,2\}$, and $v=\mathsf w_{\Delta,V}\notin\mathscr Q_{kn,\eta,M}$ such that
$$0<\|a_tuv\|\le e^{4n^3s},\qquad
\|a_tuv-
\pi_{k,m}(a_tuv)\|
\le \alpha(a_tu\Delta)^{-bM'}.
$$
The projections are $H$-equivariant. Hence the log-Lipschitz estimate for
exterior powers gives
\begin{equation}\label{eq:pullback-exceptional-witness}
\|v\|
\ll e^{4n^3(t+s)},
\qquad
\|v-\pi_{k,m}(v)\|
\ll e^{4n^3t}\alpha(a_tu\Delta)^{-bM'}.
\end{equation}
If
$\alpha(a_tu\Delta)>C_0e^{t/(b^2M)}$, then, after choosing $C_0$ sufficiently large,
$$
e^{4n^3t}\alpha(a_tu\Delta)^{-bM'}
\le
\eta
\left(e^{4n^3t}e^{4n^3s}\right)^{-M}.
$$
To verify this inequality, take logarithms and use
$$
M'\ge 10n^3bM^2,
\qquad
t\ge 10n^3bM's.
$$
The negative term
$-bM'\log \alpha(a_tu\Delta)$ then dominates
$4n^3(M+1)t+4n^3Ms$; the remaining
constant is absorbed by $C_0$. Together with
\eqref{eq:pullback-exceptional-witness}, this would imply
$$
\|v-\pi_{k,m}(v)\|
\le\eta\|v\|^{-M},
$$
contrary to the choice of $v$. Finally, the hypothesis
$t\ge 10n^3bM's$ and the lower bound on $M'$ imply
$e^{20n^3s}\ll e^{t/(b^2M)}$. This proves
\eqref{eq:alpha-cap-on-exceptional-set}.

For $i\ge0$, set
\begin{equation}\label{eq:avoidance-height-levels}
\mathsf{h}_i:=e^{10n^3(2s+i)}.
\end{equation}
Let $I$ be the least nonnegative integer such that
$$
\mathsf{h}_I\ge C_0e^{t/(b^2M)}.
$$
Then \eqref{eq:alpha-cap-on-exceptional-set} implies that
$\mathcal E_{s,\eta,M,M'}$ is covered by the union of the following sets:
\begin{align*}
E_0
&:=
\left\{u\in B_{U}(1):
\alpha(a_tu\Delta)\le \mathsf{h}_0,
(a_tu,\Delta)
\in\mathcal K_{\eta,M}(s,e^{-M's})
\right\},\\
E_i
&:=
\left\{u\in B_{U}(1):
\alpha(a_tu\Delta)\le \mathsf{h}_i,
(a_tu,\Delta)
\in\mathcal K_{\eta,M}(s,\mathsf{h}_{i-1}^{-M'})
\right\},
\qquad 1\le i\le I.
\end{align*}
Moreover, for a constant $I_0=I_0(C_0,n)$,
\begin{equation}\label{eq:avoidance-number-of-levels}
I\le I_0+{t}({10n^3b^2M})^{-1}.
\end{equation}
Here we enlarged the scale
$\mathsf{h}_{i-1}^{-bM'}$ occurring in
\eqref{eq:exceptional-scale} to $\mathsf{h}_{i-1}^{-M'}$, which is legitimate
because $b\ge1$.

\Cref{prop:avoidance0} applies to $E_0$ and gives
\begin{equation}\label{eq:E0-measure}
\operatorname{Leb}(E_0)
\ll e^{-M's/(b_1M)}.
\end{equation}
Indeed, its two hypotheses follow from $M'\ge 10n^3bM^2$ and
$t\ge 10n^3bM's$, provided $b$ is chosen larger than $b_1$.

For $i\ge1$, put $\varepsilon_i:=\mathsf{h}_{i-1}^{-M'}$, and call the
index $i$ \emph{low} if
\begin{equation}\label{eq:avoidance-low-index-condition}
b_1\log{\varepsilon_i^{-1}}
=b_1M'\log \mathsf{h}_{i-1}
\le {t}/{2}.
\end{equation}
For every low index, \Cref{prop:avoidance0} applies and gives
\begin{equation}\label{eq:low-level-measure}
\operatorname{Leb}(E_i)
\ll \mathsf{h}_{i-1}^{-M'/(b_1M)}.
\end{equation}
Indeed, \eqref{eq:avoidance-low-index-condition} supplies the required lower
bound for $t$, while
$\varepsilon_i<e^{-b_1Ms}$ follows from
$M'\ge 10n^3bM^2$, after increasing $b$ in terms of $b_1$ and $n$.

For the remaining indices, set
\begin{equation}\label{eq:avoidance-tail-scale}
\varepsilon_*:=e^{-t/(2b_1)}.
\end{equation}
If $i$ is not low, then $\varepsilon_i<\varepsilon_*$, and hence
$$
\mathcal K_{\eta,M}(s,\varepsilon_i)
\subset
\mathcal K_{\eta,M}(s,\varepsilon_*).
$$
Moreover, $ t\ge b_1\log\frac1{\varepsilon_*}=\frac{t}{2}$,
and the assumptions on $M'$, $s$, and $t$ imply
$\varepsilon_*<e^{-b_1Ms}$, provided $b=b(n)$ is sufficiently large.
Thus \Cref{prop:avoidance0} also gives
\begin{equation}\label{eq:high-level-measure}
\operatorname{Leb}(E_i)
\ll e^{-t/(2b_1^2M)}
\qquad\text{for every non-low index }i.
\end{equation}

We have
\begin{align}
&\int_{B_{U}(1)}
\alpha(a_tu\Delta)^2
\mathds{1}_{\mathcal E_{s,\eta,M,M'}}
(a_tu;\Delta)\,du
\notag\\
&\qquad\le
\mathsf{h}_0^2\operatorname{Leb}(E_0) +
\sum_{\substack{1\le i\le I, i\ \mathrm{low}}}
\mathsf{h}_i^2\operatorname{Leb}(E_i)
+
\sum_{\substack{1\le i\le I, i\ \mathrm{not\ low}}}
\mathsf{h}_i^2\operatorname{Leb}(E_i).
\label{eq:avoidance-three-sums}
\end{align}
By \eqref{eq:E0-measure} and the lower bound on $M'$,
\begin{equation}\label{eq:avoidance-first-sum}
\mathsf{h}_0^2\operatorname{Leb}(E_0)
\ll e^{-M's/(bM)}.
\end{equation}
Since $\mathsf{h}_i=e^{10n^3}\mathsf{h}_{i-1}$ and
$M'/(b_1M)$ is larger than a sufficiently large constant depending only
on $n$, \eqref{eq:low-level-measure} gives the convergent geometric bound
\begin{equation}\label{eq:avoidance-low-sum}
\sum_{\substack{1\le i\le I, i\ \mathrm{low}}}
\mathsf{h}_i^2\operatorname{Leb}(E_i)
\ll e^{-M's/(bM)}.
\end{equation}
It remains to sum over the non-low indices. By the minimal choice of $I$,
$$
\mathsf{h}_I\le e^{10n^3}\max\left\{\mathsf{h}_0,C_0e^{t/(b^2M)}\right\}.
$$
The assumptions relating $s$, $t$, and $M'$ therefore give
\begin{equation}\label{eq:max-height-level}
\mathsf{h}_I^2
\ll e^{{3t}/({b^2M})}.
\end{equation}
Choose $b=b(n)$ so large that $\frac{1}{2b_1^2}-\frac{3}{b^2}
>\frac{2}{b}$. Using \eqref{eq:high-level-measure} and
\eqref{eq:max-height-level}, and absorbing the factor $I+1$ into half of
the remaining exponential decay, we obtain
\begin{equation}\label{eq:avoidance-high-sum}
\sum_{\substack{1\le i\le I,  i\ \mathrm{not\ low}}}
\mathsf{h}_i^2\operatorname{Leb}(E_i)
\ll e^{-t/(bM)}.
\end{equation}
Combining
\eqref{eq:avoidance-three-sums},
\eqref{eq:avoidance-first-sum},
\eqref{eq:avoidance-low-sum}, and
\eqref{eq:avoidance-high-sum} proves \eqref{eq:avoidance}.
\end{proof}

\section{Uniform moment bounds for the modified height via iteration}
\label{s:iterations}

We now prove uniform moment bounds for the $U$-averages of the modified
height $\widehat{\alpha}_{\eta,M}$ along the principal ray $a_t$ by
iterating the global Margulis inequality
(\Cref{thm:globalcontraction}) together with the avoidance estimate
(\Cref{prop:avoidance'}). The argument parallels \cite{Kim}, but the
multiplicative logarithmic loss in our Margulis inequality requires the
iteration to treat the near-exceptional and deep-cusp contributions
separately.

The following $U$-average estimate implies
\Cref{thm:modifieduniformboundedness}, passing from $U$- to $K$-averages.

\begin{proposition}\label{prop:maintheoremforu}
Let $0<\eta<1$, $M>1$, and let $\Lambda\in X$ be an
$(\eta,M)$-Diophantine lattice. Then there
exists $\theta'>0$ such that
\begin{equation}\label{supa}
\sup_{t\ge0}
\int_{B_{U}(1)} \widehat{\alpha}_{\eta,M}(a_tu;\Lambda)^{1+\theta'}\,du
<\infty.
\end{equation}
\end{proposition}

We begin with the following standard comparison lemma for iterated
horospherical averages.

\begin{lemma}[Comparison of iterated averages]\label{lem:randomwalkgeneral}
Let $m\in\mathbb N$ and $t_1,\ldots,t_m\ge1$. Then, for every 
measurable function $f:H\times X\to[0,\infty)$ and $\Delta\in X$, we have
\begin{align*}
&\int_{B_{U}(1/3)}
 f(a_{t_1+\cdots+t_m}u;\Delta)\,du \le \int_{B_{U}(1)^m}
 f(a_{t_m}u_m\cdots a_{t_1}u_1;\Delta)
 \,du_1\cdots du_m;
\end{align*}
\begin{align*}
&\int_{B_{U}(1)^m}
 f(a_{t_m}u_m\cdots a_{t_1}u_1;\Delta)
 \,du_1\cdots du_m \le \int_{B_{U}(2)}
 f(a_{t_1+\cdots+t_m}u;\Delta)\,du.
\end{align*}

\end{lemma}
\begin{proof} We adapt the convolution argument of \cite[Lemma~6.1]{Kim}. Write $u_j=u_{\xi_j}$, where $\xi_j$ belongs to the unit ball in the
parameter space of $U$,
set $T:=t_1+\cdots+t_m$. Since
$a_tu_\xi a_{-t}=u_{e^{nt}\xi}$, we have
$$
a_{t_m}u_{\xi_m}\cdots a_{t_1}u_{\xi_1}
=
a_Tu_{\xi_1+e^{-nt_1}\xi_2+\cdots+
e^{-n(t_1+\cdots+t_{m-1})}\xi_m}.
$$
Since $t_j\ge1$ and $n\ge2$, the sum of the radii contributed by all
terms after $\xi_1$ is at most
$\sum_{j\ge1}e^{-nj}<1/3.
$
Let $\nu$ be the pushforward of the normalized product measure on
$B_U(1)^{m-1}$ under the map
$$
(\xi_2,\ldots,\xi_m)
\mapsto
e^{-nt_1}\xi_2+\cdots+
e^{-n(t_1+\cdots+t_{m-1})}\xi_m.
$$
Then $\nu$ is a probability measure supported in $B_U(1/3)$ and the
pushforward of the product measure on $B_U(1)^m$ has density
$\kappa=\mathbf 1_{B_U(1)}*\nu
$
with respect to Haar measure on $U$. Here we use the normalization
$\operatorname{vol}(B_U(1))=1$. It follows that
$$
\kappa=1\text{ on }B_U(1/3),\qquad
0\le\kappa\le1\text{ on }U,\qquad
\operatorname{supp}\kappa\subset B_U(2).
$$
Consequently,
$$
\int_{B_U(1)^m}
f(a_{t_m}u_m\cdots a_{t_1}u_1;\Delta)\,du_1\cdots du_m
=
\int_U f(a_Tu;\Delta)\kappa(u)\,du.
$$
The asserted inequalities now follow from the preceding bounds on
$\kappa$.
\end{proof}

We will also use the following technical decomposition lemma from \cite{Kim}.

\begin{lemma}\cite[Lemma 6.2]{Kim}\label{lem:ndecomposition}
Let $D>1$, $0<\delta<\frac{1}{1+D}$, and $T>0$ be given.
For any $t\ge T/\delta$,
there exists a finite sequence $\{s_i\}_{1\le i\le I}$ such that
$$
t=s_1+\cdots+s_I,\;
s_1=Ds_2,\;
s_i=(1+\delta)s_{i+1}\;  (2\le i\le I-1), \; T\le s_I\le 2T.$$
\end{lemma}

\subsection{Proof of \Cref{prop:maintheoremforu} and \Cref{thm:modifieduniformboundedness}}
For bounded $t$, the claim follows immediately from the log-Lipschitz
property of $\widehat\alpha_{\eta,M}$. It therefore suffices to consider
sufficiently large $t$.

Let $b\ge1$ be the constant from \Cref{prop:avoidance'}. Recall from
\eqref{theta} that the first exponent associated with any height parameter is
obtained by multiplication by $10^{1-N}$.

We use the constants $c_n,C_n$ from Section~6 throughout the
iteration. After decreasing $c_n$ and increasing $C_n$ once and for all, we
assume that $0<c_n\le1\le C_n$, $C_n\ge2c_n^{-1}$, and that the
bounded-expansion estimate used below has multiplicative and additive rates
at most $C_n\theta'$ and $C_n$, respectively, while the decay rate in
\Cref{thm:globalcontraction} is at least $4c_n\theta$ and its additive rate
is at most $C_n(M'\theta+\theta^{-1})$. We choose the constants so that
these estimates hold with either $B_{U}(1)$ or
$B_{U}(2)$ as the averaging domain. This fixed-ball extension follows by
running the proofs of \Cref{prop:globalboundedexpansion'',thm:globalcontraction}
on $B_{U}(2)$; the constants $c_n,C_n$ remain fixed below.

Set
$$
M':=bM^2\max\{100n^{10},20C_n\}
\quad\text{and}\quad
D:=20n^3bM'.
$$
Choose $\theta>0$ sufficiently small that \Cref{thm:globalcontraction} holds
for both $B_{U}(1)$ and $B_{U}(2)$ and
$$
\theta<\frac12\min\left\{(2n)^{-6},(10bM')^{-1}\right\}.
$$
Set
$C_*:=C_n(M'\theta+\theta^{-1})$
and choose $\theta'>0$ so small that
\begin{equation}\label{eq:theta-prime-choice}
\theta'
<
\min\left\{
\frac{\theta}{200n^2 10^{N-1}},
\frac{c_n\theta M'}{4000C_nbMDC_*},
\frac{M'}{80C_nbM(1+D)}
\right\}.
\end{equation}
Set
$$
\theta^{\sharp}:=100n^2 10^{N-1}\theta'.
$$
Then
$$
\theta^{\sharp}
\le
100n^2 10^{N-1}\frac{\theta}{200n^2 10^{N-1}}
=
\frac{\theta}{2}
<
(2n)^{-6},
\qquad
(\theta^{\sharp})_1=100n^2\theta'.
$$
Hence the auxiliary height
$\bar\alpha'_{\theta^{\sharp}}$ is well defined, and
\eqref{eq:globalboundedexpansion''} applies to it. Moreover, each
$\alpha_i^+(\Delta)$ in $\bar\alpha'_{\theta^{\sharp}}(\Delta)$ is raised
to the power $1+\frac{(\theta^{\sharp})_1}{100n^2}
=1+\theta'$. Consequently,
\begin{equation}\label{eq:fixed-auxiliary-height-comparison}
\alpha(\Delta)^{1+\theta'}
\le \bar\alpha'_{\theta^{\sharp}}(\Delta)
\ll_n \max\{1,\alpha(\Delta)\}^{2}.
\end{equation}

Note that $0<\theta'<\theta<(2n)^{-6}$. By
\Cref{lem:hatalpha-majorized-by-tilde} and the choice
$\theta'\le\theta/(200n^2 10^{N-1})<\beta_*-1$, we have
\begin{equation}\label{eq:hatalpha-majorized}
\widehat\alpha_{\eta,M}(h;\Delta)^{1+\theta'}
\ll_n \widetilde\alpha_{\eta,M,\theta}(h;\Delta)
\qquad
\text{for all } (h,\Delta)\in H\times X.
\end{equation}

For $s\ge0$, define
$$
\mathcal{T}_{s,\theta}:=
\{(h,\Delta)\in H\times X:
\log \widehat{\alpha}_{\eta,M}(h;\Delta)>e^{c_n\theta s}\}.
$$

By \Cref{prop:globalboundedexpansion''}, for all sufficiently
large $s$ (with the threshold allowed to depend on the now-fixed
$\theta^{\sharp}$),
\begin{equation}\label{eq:one-step-trivial}
\int_{B_{U}(1)}
\bar{\alpha}'_{\theta^{\sharp}}(a_suh\Delta)\,du
\le
 e^{C_n\theta's}\bar{\alpha}'_{\theta^{\sharp}}(h\Delta)
+
 e^{C_ns}.
\end{equation}
Also, outside $\mathcal T_{2s,\theta}$ we have
$\log(3+\widehat\alpha_{\eta,M})\ll
e^{2c_n\theta s}$. Thus \Cref{thm:globalcontraction}, the choice of
$c_n$, and a further increase of the lower threshold for $s$ give, for
all sufficiently large $s$ and all
$(h,\Delta)\notin \mathcal E_{s,\eta,M,M'}\cup \mathcal{T}_{2s,\theta}$,
\begin{equation}\label{eq:one-step-good}
\int_{B_{U}(1)}
\widetilde\alpha_{\eta,M}(a_suh;\Delta)\,du
\le
e^{-c_n\theta s}\widetilde\alpha_{\eta,M}(h;\Delta)
+
e^{C_*s}.
\end{equation}
Moreover, by \Cref{lem:quasinullatidentity}, for all sufficiently large
$s$,
$$
(\id,\Lambda)\notin \mathcal E_{s,\eta,M,10n^3M}.
$$
Since $M'\ge10n^3M$, we have
$\mathcal E_{s,\eta,M,M'}\subset\mathcal E_{s,\eta,M,10n^3M}$.
Since $\widehat\alpha_{\eta,M}(\Lambda)$ is fixed, we also have
$(\id,\Lambda)\notin\mathcal T_{2s,\theta}$ for all sufficiently large
$s$. Hence
\begin{equation}\label{eq:one-step-base}
\int_{B_{U}(1)}
\widetilde\alpha_{\eta,M}(a_su;\Lambda)\,du
\le
e^{-c_n\theta s}\widetilde\alpha_{\eta,M}(\Lambda)
+
e^{C_*s}.
\end{equation}

Choose $0<\delta<\frac1{1+D}$ so that
\begin{equation}\label{eq:deltarange}
20C_nbM\theta'<M'\delta
<
\min\left\{
\frac{c_n\theta M'}{100DC_*},\frac{M'}{2(1+D)}
\right\}.
\end{equation}
The last two bounds in \eqref{eq:theta-prime-choice} show that this interval is nonempty. We finally choose $T\ge 2\log n$ sufficiently
large that \eqref{eq:one-step-trivial}, \eqref{eq:one-step-good}, and
\eqref{eq:one-step-base} hold for every $s\ge T$, and enlarge $T$ below
whenever a further condition depending only on the parameters fixed above
is required.
\medskip

Now fix a sufficiently large $t\ge1$.
By \Cref{lem:ndecomposition}, there exists a sequence
$\{s_i\}_{1\le i\le I}$ such that
$$
t=s_1+\cdots+s_I,
\qquad
s_i\ge0,\quad\text{and}
$$
$$
s_1=Ds_2,\qquad
s_i=(1+\delta)s_{i+1}\ (2\le i\le I-1),\qquad
T\le s_I\le 2T.
$$
For $1\le i\le I-1$, write
$$
s'_{i+1}:=s_{i+1}+\cdots+s_I.
$$
We also put $s'_{I+1}:=0$.
The geometric relations above also give
\begin{equation}\label{eq:tail-bound}
\Bigl(1-\frac{s_I}{s_i}\Bigr)\delta^{-1}s_i
\le s'_{i+1}
\le \delta^{-1}s_i
\qquad (2\le i\le I-1).
\end{equation}

We introduce the iterated integral
$$
\mathsf Z_{t}
:=
\int_{B_{U}(1)^I}
\widehat\alpha_{\eta,M}(a_{s_I}u_I\cdots a_{s_1}u_1;\Lambda)^{1+\theta'}
\,du_1\cdots du_I.
$$

By \Cref{lem:randomwalkgeneral},
\begin{equation}\label{eq:Ztexpression}
\int_{B_{U}(1/3)}
\widehat{\alpha}_{\eta,M}(a_tu;\Lambda)^{1+\theta'}\,du
\le \mathsf Z_{t}.
\end{equation}
Thus it suffices, for the moment, to show that
\begin{equation}\label{zt}
\sup_{t\ge0}\mathsf Z_{t}<\infty.
\end{equation}

\medskip

For $1\le m\le I-1$, we consider exceptional sets
$$
\widetilde{\mathcal{E}}_m:=\begin{cases}\mathcal E_{s_{m+1},\eta,M,M'}\cup \mathcal E_{s_{m+1}',\eta,M,M'}&\text{ if }s_{m+1}\leq C_n\theta^{-1}\log t,\\ \mathcal E_{s_{m+1},\eta,M,M'}&\text{otherwise}. \end{cases}
$$
Define
$$
\overline\Theta_m
:=
\Bigl\{(u_1,\dots,u_m)\in B_{U}(1)^m :
(a_{s_m}u_m\cdots a_{s_1}u_1,\Lambda)
\in \widetilde{\mathcal{E}}_m\Bigr\},
$$
and set
$$
\Theta_m:=\overline\Theta_m\times B_{U}(1)^{I-m}\subset B_{U}(1)^I,
\qquad
\Theta:=\bigcup_{m=1}^{I-1}\Theta_m.
$$
We also  define inductively
$$\overline\Omega_m':=\Bigl\{(u_1,\dots,u_m)\in B_{U}(1)^m :
(a_{s_m}u_m\cdots a_{s_1}u_1,\Lambda)
\in \mathcal T_{s_{m+1},\theta}\Bigr\},$$
$$\overline\Omega_m:=\overline\Omega_m'\smallsetminus\bigcup_{i=1}^{m-1}\left(\overline\Omega_i' \times B_{U}(1)^{m-i}\right),$$
$$\Omega_m:=\overline\Omega_m\times B_{U}(1)^{I-m}\subset B_{U}(1)^I, \qquad \Omega:=\bigcup_{m=1}^{I-1}\Omega_m.$$

We then have
$$
\mathsf Z_{t}\le \mathsf Y_{t}+\sum_{m=1}^{I-1}\mathsf T_{t,m}+\sum_{m=1}^{I-1}\mathsf E_{t,m},
$$
where
$$
\mathsf Y_{t}
:=
\int_{B_{U}(1)^I\smallsetminus (\Theta\cup \Omega)}
\widehat\alpha_{\eta,M}(a_{s_I}u_I\cdots a_{s_1}u_1;\Lambda)^{1+\theta'}
\,du_1\cdots du_I,
$$
$$
\mathsf T_{t,m}
:=
\int_{\Omega_m\smallsetminus \Theta}
\widehat\alpha_{\eta,M}(a_{s_I}u_I\cdots a_{s_1}u_1;\Lambda)^{1+\theta'}
\,du_1\cdots du_I,
$$
$$
\mathsf E_{t,m}
:=
\int_{\Theta_m}
\widehat\alpha_{\eta,M}(a_{s_I}u_I\cdots a_{s_1}u_1;\Lambda)^{1+\theta'}
\,du_1\cdots du_I.
$$

\medskip
\noindent{\bf {Estimate of the bad part $\mathsf E_{t,m}$.}}
Fix $1\le m\le I-1$. Write
$$
\mathsf E_{t,m}
=
\int_{\overline\Theta_m}
J(u_1,\dots,u_m)\,du_1\cdots du_m,
$$
where
$$
J(u_1,\dots,u_m)
:=
\int_{B_{U}(1)^{I-m}}
\widehat\alpha_{\eta,M}(a_{s_I}u_I\cdots a_{s_1}u_1;\Lambda)^{1+\theta'}
\,du_{m+1}\cdots du_I.
$$

Recall that
$$
\widehat\alpha_{\eta,M}^{1+\theta'}
\le \alpha^{1+\theta'}\leq \bar{\alpha}'_{\theta^{\sharp}}.
$$
Set
$$
x_m:=a_{s_m}u_m\cdots a_{s_1}u_1\Lambda.
$$
Since $s_i\ge T$ for every $i$, directly iterating the affine inequality
\eqref{eq:one-step-trivial} over the last $I-m$ steps gives
\begin{equation}\label{eq:Badestimate1}
\begin{aligned}
J(u_1,\dots,u_m)
&\le
e^{C_n\theta's'_{m+1}}\bar\alpha'_{\theta^{\sharp}}(x_m)+
\sum_{i=m+1}^I
e^{C_ns_i+C_n\theta'(s_{i+1}+\cdots+s_I)}.
\end{aligned}
\end{equation}
Since $\delta<1$, \eqref{eq:tail-bound} gives
$$
s'_{m+1}\le(1+\delta^{-1})s_{m+1}\le2\delta^{-1}s_{m+1},
\qquad
s_{i+1}+\cdots+s_I\le\delta^{-1}s_{m+1}
\quad(i\ge m+1).
$$
Moreover, the geometric definition of the $s_i$ gives
$I-m\ll1+\delta^{-1}\log(2+\frac{s_{m+1}}{T})$. After increasing $T$, we may therefore
assume that
$$
I-m+1\le \exp\left(\frac{M'}{20bM}s_{m+1}\right).
$$
Since $\bar\alpha'_{\theta^{\sharp}}\ge1$, the lower bound in
\eqref{eq:deltarange} and the choice of $M'$ imply
$$
C_n\theta'\delta^{-1}<\frac{M'}{20bM},
\qquad
C_n\le\frac{M'}{20bM}.
$$
It follows from \eqref{eq:Badestimate1} that
\begin{equation}\label{eq:Badestimate1-simplified}
J(u_1,\ldots,u_m)
\ll
e^{\frac{M'}{5bM}s_{m+1}}
\bar\alpha'_{\theta^{\sharp}}(x_m).
\end{equation}

Now set
$$
\varphi(h;\Delta)
:=
\alpha(h\Delta)^{2}\,
\mathds{1}_{\widetilde{\mathcal E}_{m}}(h,\Delta).
$$
We also have $s_1=Ds_2\ge Ds_{m+1}$. If
$s_{m+1}\le C_n\theta^{-1}\log t$, then
$s'_{m+1}\le\delta^{-1}s_{m+1}$. Since
$
s_1\ge \frac{D}{D+1+\delta^{-1}}t,
$
for sufficiently large $t$, we have
$$
s_1\ge D\delta^{-1}C_n\theta^{-1}\log t
\ge Ds'_{m+1}.
$$
By \Cref{prop:avoidance'} and \Cref{lem:randomwalkgeneral},
$$
\begin{aligned}
&\int_{\overline{\Theta}_m}
\bar{\alpha}'_{\theta^{\sharp}}(a_{s_m}u_m\cdots a_{s_1}u_1\Lambda)
\,du_1\cdots du_m\\ 
&\ll
\int_{B_{U}(1)^m}
\varphi(a_{s_m}u_m\cdots a_{s_1}u_1;\Lambda)\,du_1\cdots du_m \\
&\le
\int_{B_{U}(2)}
\varphi(a_{s_1+\cdots+s_m}u;\Lambda)\,du.
\end{aligned}
$$
The proof of \Cref{prop:avoidance'} gives the same estimate with
$B_U(R)$ in place of $B_U(1)$ for every fixed $R\ge1$, with the implicit
constant also allowed to depend on $R$. Indeed, the log-Lipschitz bounds are
uniform on $B_U(R)$, the Remez argument gives the required goodness on its
fixed enlargement, and the instability estimate remains valid because
$B_U(1)\subset B_U(R)$. We apply this extension with $R=2$. Applying
\Cref{prop:avoidance'} to the one or
two exceptional sets appearing in $\widetilde{\mathcal E}_m$, and using
$s'_{m+1}\ge s_{m+1}$ when the second exceptional set
is present, gives
$$
\int_{\overline{\Theta}_m}
\bar{\alpha}'_{\theta^{\sharp}}(a_{s_m}u_m\cdots a_{s_1}u_1\Lambda)
\,du_1\cdots du_m
\ll
e^{-\frac{M'}{bM}s_{m+1}}.
$$
Here we used \eqref{eq:fixed-auxiliary-height-comparison} and $\alpha(\Delta)\gg_n1$ on $X$.
Combining this with \eqref{eq:Badestimate1-simplified}, we get
for $1\le m\le I-1$,
\begin{equation}\label{eq:Badestimate3}
\begin{aligned}
\mathsf E_{t,m}
&\ll
e^{\frac{M'}{5bM}s_{m+1}}
e^{-\frac{M'}{bM}s_{m+1}} \ll e^{-s_{m+1}}.
\end{aligned}
\end{equation}

Summing over $m$,
$$
\sum_{m=1}^{I-1}\mathsf E_{t,m}
\ll
\sum_{m=1}^{I-1} e^{-s_{m+1}}
\le
\sum_{j=0}^{\infty}e^{-(1+\delta)^jT}
\le
e^{-T}\sum_{j=0}^{\infty}e^{-j\delta T}
\ll
\delta^{-1}T^{-1}.
$$

Choose $0<\varepsilon\le 1/100$, and let $I'$ be the largest
integer such that
$$
s_{I'}\ge \varepsilon^{-1}s_I.
$$
After increasing the lower bound on $t$, we have
$s_2\ge\varepsilon^{-1}s_I$, and hence $I'\ge2$.
Since $I'$ is maximal and
$s_{I'}=(1+\delta)^{I-I'}s_I$, we have
$$
\varepsilon^{-1}
\le (1+\delta)^{I-I'}
< (1+\delta)\varepsilon^{-1}.
$$
In particular,
$$
I-I'\asymp \delta^{-1}\log\frac1\varepsilon,
$$
with absolute implied constants.
Moreover, maximality and $s_I\le2T$ give
\begin{equation}\label{eq:I-prime-upper}
s_{I'}<(1+\delta)\varepsilon^{-1}s_I
<4\varepsilon^{-1}T.
\end{equation}

\medskip
\noindent{\bf{Estimate of the cusp part $\mathsf T_{t,m}$.}}
Fix $1\leq m\leq I-1$, and put
$$
x_i=(a_{s_i}u_i\cdots a_{s_1}u_1,\Lambda),\qquad 1\leq i\leq m.
$$
Define
$$
\overline\Omega_m^{\circ}
:=
\overline\Omega_m\smallsetminus
\bigcup_{i=1}^{m}
\left(\overline\Theta_i\times B_U(1)^{m-i}\right).
$$
Since
$$
\Omega_m\smallsetminus\Theta
\subset
\overline\Omega_m^{\circ}\times B_U(1)^{I-m},
$$
we may discard the exceptional restrictions on the last $I-m$ variables
when taking an upper bound.
For $(u_1,\ldots,u_m)\in\overline\Omega_m^{\circ}$, we have
$x_m\in \mathcal T_{s_{m+1},\theta}$. Hence, by the log-Lipschitz property of $\widehat\alpha_{\eta,M}$,
\eqlabel{eq:alphalowerbound}{4n^3t+O_{\Lambda}(1)
\geq \log\widehat\alpha_{\eta,M}(x_m)
>
\exp(c_n\theta s_{m+1}).}
Taking logarithms and using $C_n\geq 2c_n^{-1}$, we obtain, for sufficiently
large $t$, $s_{m+1}\leq C_n\theta^{-1}\log t$.
By the definition of $\widetilde {\mathcal{E}}_m$, this implies that $\mathcal E_{s'_{m+1},\eta,M,M'}\subset \widetilde{\mathcal{E}}_m$. Since $(u_1,\ldots,u_m)\in\overline\Omega_m^{\circ}$, we therefore have
$x_m\notin \mathcal E_{s'_{m+1},\eta,M,M'}$.

We next show that $x_m\notin \mathcal T_{2s'_{m+1},\theta}$ for sufficiently large $T$. If $m\geq 2$, then the definition of $\Omega_m$ implies that the preceding
intermediate point has not entered the corresponding cusp set, that is, $x_{m-1}\notin \mathcal T_{s_m,\theta}$. Thus
$$
\log\widehat\alpha_{\eta,M}(x_{m-1})
\leq \exp(c_n\theta s_m).
$$
Using the log-Lipschitz property of $\widehat\alpha_{\eta,M}$, we obtain
$$
\log\widehat\alpha_{\eta,M}(x_m)
\leq
\exp(c_n\theta s_m)+O(s_m).
$$
Since $2s'_{m+1}\geq \frac{2}{1+\delta}s_m>s_m$, the right-hand side is smaller than $\exp(c_n\theta 2s'_{m+1})$ for $T$ sufficiently large. Hence $x_m\notin \mathcal T_{2s'_{m+1},\theta}$. The case
$m=1$ is similar, since $\Lambda$ is fixed and the log-Lipschitz property gives
$$
\log\widehat\alpha_{\eta,M}(a_{s_1}u_1;\Lambda)\ll s_1+1,
$$
whereas $\exp(c_n\theta 2s'_2)$ grows exponentially in $s'_2$. Thus the same conclusion holds  for sufficiently large $T$. Consequently, 
$$
x_m\notin \mathcal E_{s'_{m+1},\eta,M,M'}\cup \mathcal T_{2s'_{m+1},\theta}.
$$
We may therefore apply the contraction estimate \eqref{eq:one-step-good} with $s=s'_{m+1}$.

    Applying \Cref{lem:randomwalkgeneral} to the last $I-m$ factors and
    then the $B_U(2)$-version of \eqref{eq:one-step-good}, we get
    \begin{equation}\label{eq:Jtildeiterationupperbound}{\begin{aligned}
        \widetilde{J}(u_1,\ldots,u_m)&=\int_{B_{U}(1)^{I-m}}\widetilde{\alpha}_{\eta,M}\big(a_{s_I}u_I\cdots a_{s_1}u_1;\Lambda\big)du_{m+1}\cdots du_I\\&\le \int_{B_U(2)} \widetilde{\alpha}_{\eta,M}\big(a_{s_{m+1}'}ua_{s_m}u_{m}\cdots a_{s_1}u_1;\Lambda\big)du\\&\leq e^{-c_n\theta s'_{m+1}}\widetilde{\alpha}_{\eta,M}\big(a_{s_m}u_m\cdots a_{s_1}u_1;\Lambda\big)+e^{C_*s_{m+1}'}\\&\leq 2e^{-c_n\theta s'_{m+1}}\widetilde{\alpha}_{\eta,M}\big(a_{s_m}u_m\cdots a_{s_1}u_1;\Lambda\big).
    \end{aligned}}\end{equation}
The last inequality requires the slightly stronger comparison
$$
\widetilde\alpha_{\eta,M}(x_m)
\ge
\exp\bigl((C_*+c_n\theta)s'_{m+1}\bigr).
$$
Indeed, \eqref{eq:hatalpha-majorized},
$x_m\in\mathcal T_{s_{m+1},\theta}$, and
$s'_{m+1}\le\delta^{-1}s_{m+1}$ give
$$
\log\widetilde\alpha_{\eta,M}(x_m)
\ge
(1+\theta')e^{c_n\theta s_{m+1}}-O_n(1),
$$
which is larger than
$(C_*+c_n\theta)s'_{m+1}$ for  sufficiently large $T$.

Since \eqref{eq:hatalpha-majorized} holds and the integrand is
nonnegative,
$$
\mathsf T_{t,m}
\ll_n
\int_{\overline\Omega_m^{\circ}}
\widetilde J(u_1,\ldots,u_m)\,du_1\cdots du_m.
$$

For $(u_1,\ldots,u_m)\in\overline\Omega_m^{\circ}$ and every
$1\le i<m$, the intermediate point
$$
(a_{s_i}u_i\cdots a_{s_1}u_1,\Lambda)
$$
lies outside the two sets relevant to the next increment $s_{i+1}$:
$$
(a_{s_i}u_i\cdots a_{s_1}u_1,\Lambda)
\notin
\mathcal E_{s_{i+1},\eta,M,M'}
\cup
\mathcal T_{2s_{i+1},\theta}.
$$
Here the exclusion from $\mathcal E_{s_{i+1},\eta,M,M'}$ follows
from being outside $\Theta$, while the exclusion from
$\mathcal T_{2s_{i+1},\theta}$ follows from the definition of
$\Omega_m$, since
$$
\mathcal T_{2s_{i+1},\theta}
\subset
\mathcal T_{s_{i+1},\theta}.
$$
Thus we may apply the contraction estimate \eqref{eq:one-step-good}
successively to the steps $s_m,\dots,s_2$, and then use
\eqref{eq:one-step-base} for the first step $s_1$. This gives
\begin{equation}\label{eq:Goodestimate0}
\begin{aligned}
&\int_{\overline\Omega_m^{\circ}}
\widetilde{\alpha}_{\eta,M}
(a_{s_m}u_m\cdots a_{s_1}u_1;\Lambda)
\,du_1\cdots du_m                                      \\
& \le
e^{-c_n\theta(s_1+\cdots+s_m)}
\widetilde{\alpha}_{\eta,M}(\Lambda)                  +
e^{C_*s_1
-c_n\theta(s_2+\cdots+s_m)}     +
\sum_{i=2}^m
e^{C_*s_i
-c_n\theta(s_{i+1}+\cdots+s_m)}.
\end{aligned}
\end{equation}
Here empty sums are interpreted as $0$.

Combining \eqref{eq:Jtildeiterationupperbound} and \eqref{eq:Goodestimate0}, and using the relation
$$
(s_{i+1}+\cdots+s_m)+s'_{m+1}
=
s'_{i+1},
$$
we obtain
\begin{align}
\mathsf T_{t,m}
\ll{}&
e^{-c_n\theta t}
\widetilde{\alpha}_{\eta,M}(\Lambda) + e^{C_*s_1-c_n\theta s'_2}
+
\sum_{i=2}^m
e^{C_*s_i
-c_n\theta s'_{i+1}}.
\label{eq:Ttmestimate-corrected}
\end{align}

Summing \eqref{eq:Ttmestimate-corrected} over $1\le m\le I-1$, and
interchanging the order of summation in the last term, gives
$$
\begin{aligned}
\sum_{m=1}^{I-1}\mathsf T_{t,m}
\ll{}&
I e^{-c_n\theta t}
\widetilde{\alpha}_{\eta,M}(\Lambda)  +
I e^{C_*s_1-c_n\theta s'_2} +
\sum_{i=2}^{I-1}
(I-i)
e^{C_*s_i
-c_n\theta s'_{i+1}}.
\end{aligned}
$$
The harmless polynomial factors $I$ and $I-i$ are absorbed by the 
exponential estimates established below in
\eqref{eq:exponentialsum0}--\eqref{eq:exponentialsum2}, after increasing the
initial threshold $T$ if necessary. Hence
\begin{equation}\label{eq:Dtmestimate}
\sum_{m=1}^{I-1}\mathsf T_{t,m}\ll 1.
\end{equation}

\noindent{\bf{Estimate of the good part $\mathsf Y_{t}$.}}
On $B_{U}(1)^I\smallsetminus(\Theta \cup \Omega)$, every intermediate point
$ (a_{s_m}u_m\cdots a_{s_1}u_1,\Lambda) $ lies outside the exceptional set for the next increment $s_{m+1}$:
$$
(a_{s_m}u_m\cdots a_{s_1}u_1,\Lambda)
\notin \mathcal E_{s_{m+1},\eta,M,M'}.
$$
Since the path is outside $\Omega$, for every $1\leq m< I$, we also have $(a_{s_m}u_m\cdots a_{s_1}u_1,\Lambda)\notin \mathcal T_{s_{m+1},\theta}$. Thus we may
apply the contraction estimate \eqref{eq:one-step-good} successively to the
steps $s_I,\dots,s_2$, and then use \eqref{eq:one-step-base}
for the first step $s_1$:
\begin{equation}\label{eq:Goodestimate1}
\begin{aligned}
\mathsf Y_{t}
&\le
\int_{B_{U}(1)^I\smallsetminus (\Theta\cup \Omega)}
\widetilde{\alpha}_{\eta,M}(a_{s_I}u_I\cdots a_{s_1}u_1;\Lambda)
\,du_1\cdots du_I \\
&\le
e^{-c_n\theta t}
\widetilde{\alpha}_{\eta,M}(\Lambda)
+
e^{C_*s_1-c_n\theta(s_2+\cdots+s_I)}  +
\sum_{m=2}^I
e^{C_*s_m
-c_n\theta(s_{m+1}+\cdots+s_I)} \\
&=
e^{-c_n\theta t}
\widetilde{\alpha}_{\eta,M}(\Lambda)
+
e^{C_*s_1-c_n\theta s_2'} +\sum_{m=2}^I
e^{C_*s_m-c_n\theta s_{m+1}'}.
\end{aligned}
\end{equation}
If $2\le m\le I'$, then
$$
s_{m+1}'
=
s_{m+1}+\cdots+s_I
\ge
\left(1-\frac{s_I}{s_{I'}}\right)\delta^{-1}s_m
\ge
(1-\varepsilon)\delta^{-1}s_m.
$$
Hence, using the upper bound in \eqref{eq:deltarange}, we have
\begin{equation}\label{eq:exponentialsum0}
\begin{aligned}
& \sum_{m=2}^{I'}
e^{C_*s_m-c_n\theta s_{m+1}'}
\le
\sum_{m=2}^{I'}
e^{\left(C_*-c_n(1-\varepsilon)\delta^{-1}\theta\right)s_m} \\
&\le
\sum_{m=2}^{I'}
e^{\left(C_*\delta-c_n(1-\varepsilon)\theta\right)(I-m)T}  \le
\sum_{j=I-I'}^\infty e^{-\frac{c_nj\theta T}{2}}
\ll
\varepsilon^{\frac{c_n\delta^{-1}\theta T}{4}},
\end{aligned}
\end{equation}
where we used $s_m\ge(1+\delta)^{I-m}T\ge(I-m)\delta T$,
$\varepsilon\le1/100$, and
$C_*\delta\le c_n\theta/(100D)$.

For $m=1$, we have $I'\ge2$ once $t$ is sufficiently large, and hence
$s_I\le\varepsilon s_2$. Thus the same geometric-tail calculation and
the upper bound in \eqref{eq:deltarange} give
$$
c_n\theta s_2'
\ge
\frac{c_n(1-\varepsilon)\theta s_2}{\delta}
=
\frac{c_n(1-\varepsilon)\theta s_1}{D\delta}
>
C_*s_1,
$$
and hence
\begin{equation}\label{eq:exponentialsum1}
e^{C_*s_1-c_n\theta s_2'}\le 1.
\end{equation}

Finally,
\begin{equation}\label{eq:exponentialsum2}
\begin{aligned}
\sum_{m=I'+1}^{I}
e^{C_*s_m-c_n\theta(s_{m+1}+\cdots+s_I)}
&\le
(I-I')e^{C_*s_{I'}}
\ll
\delta^{-1}\log \varepsilon^{-1} \,
e^{4\varepsilon^{-1}C_*T}.
\end{aligned}
\end{equation}

Combining \eqref{eq:Goodestimate1}, \eqref{eq:exponentialsum0},
\eqref{eq:exponentialsum1}, and \eqref{eq:exponentialsum2}, we conclude that
$\mathsf Y_{t}$ is uniformly bounded for all sufficiently large
$t$.

Combining the bounds for the good, bad, and cusp parts, we conclude that
$$
\mathsf Z_{t}
\leq 
\mathsf Y_{t}+\sum_{m=1}^{I-1}\mathsf T_{t,m}+\sum_{m=1}^{I-1}\mathsf E_{t,m}
$$
is uniformly bounded for all sufficiently large $t\ge1$.
By \eqref{eq:Ztexpression}, this proves
$$
\sup_{t\ge0}
\int_{B_{U}(1/3)}
\widehat\alpha_{\eta,M}(a_tu;\Lambda)^{1+\theta'}\,du
<\infty.
$$

It remains to enlarge the averaging ball. Choose
$u_1,\ldots,u_\ell\in U$ such that
$$
B_{U}(1)\subset\bigcup_{j=1}^{\ell}B_{U}(1/3)u_j.
$$
Each $u_j\Lambda$ is Diophantine. Since $u_j$ preserves every exceptional
summand and acts invertibly on every exterior power, norm comparison gives
parameters $0<\eta_j<1$ such that $u_j\Lambda$ is
$(\eta_{j},M)$-Diophantine and, in every dimension $r$,
$$
u_jV\in
\widetilde{\mathscr Q}_{r,\eta_j,M}(u_j\Lambda)
\quad\to\quad
V\in\widetilde{\mathscr Q}_{r,\eta,M}(\Lambda).
$$
Consequently, for every $h\in H$,
$$
\widehat\alpha_{\eta,M}(hu_j;\Lambda)
\le
\widehat\alpha_{\eta_j,M}(h;u_j\Lambda).
$$
The small-ball estimate just proved applies to each of the finitely many
triples $(u_j\Lambda,\eta_j,M)$. Therefore
$$
\begin{aligned}
&\int_{B_{U}(1)}
\widehat\alpha_{\eta,M}(a_tu;\Lambda)^{1+\theta'}\,du \ll
\sum_{j=1}^{\ell}
\int_{B_{U}(1/3)}
\widehat\alpha_{\eta_j,M}(a_tu;u_j\Lambda)^{1+\theta'}\,du
\ll1
\end{aligned}
$$
uniformly in $t$.
This proves \Cref{prop:maintheoremforu}.

\medskip
We deduce \Cref{thm:modifieduniformboundedness} from
\Cref{prop:maintheoremforu} by a transfer argument following
\cite[Section~7.2]{Kim}. Let $U^-<H$ be the contracting horospherical
subgroup for $a_t$, opposite to $U$, and let $P^-:=N_H(U^-)$.
Set
$M_K:=P^-\cap K$ which is the centralizer of $\{a_t:t\in\mathbb R\}$ in $K$.
Since the modified height is left $K$-invariant and $M_K$
commutes with $a_t$, the function
$k\mapsto\widehat\alpha_{\eta,M}(a_tk;\Lambda)$ descends to
$M_K\backslash K\simeq P^-\backslash H$.
By compactness, finitely many right translates of relatively compact
pieces of the open $U$-cell give local charts covering this
quotient, with representatives
$k=p(u)uk_j$, where $k_j\in K$, $u$ ranges over a subset of
$B_U(1)$, and $p(u)\in P^-$.
The elements $a_tp(u)a_{-t}$ remain in a fixed compact set
for all $t\ge0$.
Since each $k_j$ preserves the norms and commutes with the
exceptional projections, $k_j\Lambda$ satisfies the same
Diophantine condition, and
$$
\widehat\alpha_{\eta,M}(hk_j;\Lambda)
=
\widehat\alpha_{\eta,M}(h;k_j\Lambda)
\qquad (h\in H).
$$
Choose $\theta>0$ no larger than the finitely many exponents
provided by \Cref{prop:maintheoremforu} for these lattices.
Bounded coordinate Jacobians and norm comparison give
$$
\int_K \widehat\alpha_{\eta,M}(a_tk;\Lambda)^{1+\theta}\,dk
\ll
\sum_j \int_{B_U(1)}
\widehat\alpha_{\eta,M}(a_tu;k_j\Lambda)^{1+\theta}\,du,
$$
with an implied constant independent of $t\ge0$.
The right-hand side is uniformly bounded by
\Cref{prop:maintheoremforu}, proving
\Cref{thm:modifieduniformboundedness}.\qed

We shall also use the corresponding first-moment estimate for the ordinary
height. The standard iteration based on
\Cref{prop:globalboundedexpansion'}, followed by the same transfer from
$U$-averages to $K$-averages, gives, for every fixed $\Delta\in X$,
\begin{equation}\label{eq:standardalphauniformboundednesscompact}
\sup_{t\ge0}
\int_K\bar\alpha(a_tk\Delta)\,dk<\infty.
\end{equation}
This is the usual first-moment nondivergence estimate for the standard
Margulis height; compare \cite[Section~7]{eskin-margulis-mozes:1998}.

\section{Structure of isotropic subspaces}\label{s:quasinull}
In this section, we establish the structural results on isotropic subspaces
needed for the proof of \Cref{thm:quasinullcontribution2}. We first record
how the Diophantine hypothesis identifies quasi-null subspaces with
isotropic ones; the rest of the section does not use this hypothesis.

\subsection{Reduction to isotropic subspaces}

We first isolate the rational subspaces lying in column- or row-isotropic
subspaces. For $1\le k\le n-1$, let
$\mathcal X^{(1)}_{kn,\mathrm{iso}}(\Delta)$ and
$\mathcal X^{(2)}_{kn,\mathrm{iso}}(\Delta)$ denote respectively the
collections of $kn$-dimensional $\Delta$-rational column- and
row-isotropic subspaces. Set
$$
\mathcal X_{kn,\mathrm{iso}}(\Delta)
:=
\mathcal X^{(1)}_{kn,\mathrm{iso}}(\Delta)
\cup
\mathcal X^{(2)}_{kn,\mathrm{iso}}(\Delta).
$$
For $m=1,2$, let
$\mathscr Q^{(m)}_{kn,\eta,M}(\Delta)$ denote the collection of
$kn$-dimensional $\Delta$-rational subspaces $V$ such that
$$
\|\mathsf w_{\Delta,V}-\pi_{k,m}(\mathsf w_{\Delta,V})\|
\le
\eta\|\mathsf w_{\Delta,V}\|^{-M}.
$$
Set
$$
\mathscr Q_{kn,\eta,M}(\Delta)
:=
\mathscr Q^{(1)}_{kn,\eta,M}(\Delta)
\cup
\mathscr Q^{(2)}_{kn,\eta,M}(\Delta).
$$

For $1\le i\le N-1$, let
$$
\widetilde{\mathcal X}_{i,\mathrm{iso}}(\Delta)
$$
denote the collection of all $i$-dimensional $\Delta$-rational
subspaces $V$ for which there exist $1\le k\le n-1$ and
$V'\in\mathcal X_{kn,\mathrm{iso}}(\Delta)$ such that $V\subset V'$. Thus
$\widetilde{\mathcal X}_{i,\mathrm{iso}}(\Delta)$ is the
collection of $i$-dimensional subspaces contained in rational isotropic
subspaces.

\begin{lemma}\label{lem:quasinull-isotropic}
Let $0<\eta<1$, let $M>1$, and suppose that $\Delta$ is
$(\eta,M)$-Diophantine. Then, for every $1\le k\le n-1$ and
$m\in\{1,2\}$,
$$
\mathscr Q^{(m)}_{kn,\eta,M}(\Delta)
=
\mathcal X^{(m)}_{kn,\mathrm{iso}}(\Delta).
$$
Consequently,
$$
\widetilde{\mathscr Q}_{i,\eta,M}(\Delta)
=
\widetilde{\mathcal X}_{i,\mathrm{iso}}(\Delta)
\qquad (1\le i\le N-1).
$$
\end{lemma}

\begin{proof}
Fix $k$ and $m$. By
\Cref{lem:isotropic_iff_pi1_pi2}, a $kn$-dimensional
$\Delta$-rational subspace $V$ belongs to
$\mathcal X^{(m)}_{kn,\mathrm{iso}}(\Delta)$ if and only if
$
\mathsf w_{\Delta,V}
=
\pi_{k,m}(\mathsf w_{\Delta,V}).
$
If this equality holds, $V$ is $(\eta,M)$-quasi-null; otherwise, the
Diophantine condition excludes it. This proves the first equality. The second equality follows by taking the unions over $k$ and $m$ and
then taking downward closures.
\end{proof}

\subsection{Orthogonal-complement duality}
Let $\langle v,w\rangle:=\operatorname{tr}(vw^{\intercal})
$ be the Euclidean inner product on $\M_n(\R)$, and extend it to exterior
powers. Fix the orientation for which the lexicographically ordered standard matrix units $(E_{ij})_{1\le i,j\le n}$ form a positively oriented
orthonormal basis, and let $\mathrm{vol}$ be the corresponding unit volume
form.  The Hodge-star
$$
\ast:\wedge^r\M_n(\R)\to\wedge^{N-r}\M_n(\R)
$$
is characterized by
$$
\omega\wedge(\ast \omega')
=
\langle\omega,\omega'\rangle\,\mathrm{vol} \qquad
\left(\omega,\omega'\in\wedge^r\M_n(\R)\right).
$$
It is an isometry. 

Throughout this section and \Cref{s:isotropic}, we retain the
identification $\M_n(\R)\simeq\mathsf C\otimes\mathsf R$ of Section~3
and, when convenient, further identify the row factor $\mathsf R$ with
$(\R^n)^*$ via the Euclidean inner product. Thus
$\mathcal L(\mathsf U)=\mathsf U\otimes(\R^n)^*$ and $\mathcal R(\mathsf U)=\R^n\otimes \mathsf U$ for
$\mathsf U<\R^n$.

For a lattice $\Delta<\M_n(\br)$, let
$$
\Delta^*
:=
\{v\in\M_n(\mathbb R):\operatorname{tr}(v^{\intercal}w)\in\mathbb Z
\text{ for all }w\in\Delta\}
$$
be its dual lattice.

\begin{lemma}\label{lem:lattice-orthogonal-duality}
Let $\Delta<\M_n(\R)$ be a lattice, and let
$V<\M_n(\R)$ be a proper $\Delta$-rational subspace. Then $V^\perp$ is $\Delta^*$-rational and, up to sign,
\begin{equation}\label{eq:hodge-plucker-duality}
\mathsf w_{\Delta^*,V^\perp}
=
\frac{\ast\mathsf w_{\Delta,V}}{\covol(\Delta)}.
\end{equation}
Consequently,
\begin{equation}\label{eq:covol-orthogonal-duality}
\covol_{V^\perp}(\Delta^*\cap V^\perp)
=
\frac{\covol_V(\Delta\cap V)}{\covol(\Delta)}.
\end{equation}
Moreover, for every subspace $\mathsf U<\R^n$,
$$
\mathcal L(\mathsf U)^\perp=\mathcal L(\mathsf U^\perp)
\quad\text{and}\quad
\mathcal R(\mathsf U)^\perp=\mathcal R(\mathsf U^\perp).
$$
 Thus orthogonal complementation preserves the column-isotropic and
row-isotropic classes.
\end{lemma}

\begin{proof}
The sublattice $\Delta\cap V$ is primitive in $\Delta$. Choose a
$\mathbb Z$-basis $v_1,\ldots,v_r$ of $\Delta\cap V$, and extend it
to a $\mathbb Z$-basis $v_1,\ldots,v_N$ of $\Delta$. Let
$v_1^*,\ldots,v_N^*$ be the dual basis of $\Delta^*$. Since a vector
of $\Delta^*$ lies in $V^\perp$ exactly when its first $r$ dual
coordinates vanish,
$$
\Delta^*\cap V^\perp
=
\mathbb Zv_{r+1}^*\oplus\cdots\oplus\mathbb Zv_N^*.
$$
After changing the sign of the extended basis if necessary, the standard
Hodge-star identity for a basis and its dual gives
$$
\ast\bigl(v_1\wedge\cdots\wedge v_r\bigr)
=
\covol(\Delta)
\,v_{r+1}^*\wedge\cdots\wedge v_N^*.
$$
This proves \eqref{eq:hodge-plucker-duality}, up to the choice
 of sign of the Pl\"ucker vectors. Taking norms and using that the Hodge-star
is an isometry yields \eqref{eq:covol-orthogonal-duality}.
The orthogonal-complement identities follow immediately from the
tensor-product descriptions of $\mathcal L(\mathsf U)$ and
$\mathcal R(\mathsf U)$ above.
\end{proof}

\subsection{Structure in codimension $n$}
When $\Delta$ is $(\eta,M)$-Diophantine, \Cref{lem:quasinull-isotropic} reduces the
codimension-$n$ quasi-null family to the isotropic family. 

\begin{proposition}\label{lem:global-isotropic-structure}
Let $\Delta<\M_n(\R)$ be a lattice with at least one $\Delta$-rational column-isotropic subspace of dimension $n(n-1)$. Then there exist subspaces
$\mathsf U_1,\ldots,\mathsf U_p<\R^n$, $p\ge 0$, such that
$$
\mathsf U_1\oplus\cdots\oplus \mathsf U_p\subset\R^n,
\qquad
\dim \mathsf U_j\ge2.
$$
All but at most $n$ of the $\Delta$-rational column-isotropic subspaces
$V$ of dimension $n(n-1)$ satisfy
$$
V^\perp\subset\mathcal L(\mathsf U_j)
\quad\text{for a unique $j$}.
$$
For every $j$, the lattice
$\Delta^*\cap\mathcal L(\mathsf U_j)$ is column-$\mathbb Q$-split; that
is, it is commensurable with a tensor-product lattice
$\Lambda_{1,j}\otimes\Lambda_{2,j}$, where $\Lambda_{1,j}$ is a full
lattice in $\mathsf U_j$ and $\Lambda_{2,j}$ is a full lattice in
$(\R^n)^*$.

If $\Delta$ is not of $\mathbb Q$-split type, then every $\mathsf U_j$ is
proper in $\R^n$. The analogous assertions hold for row-isotropic
subspaces.
\end{proposition}

\begin{proof}
We prove the column statement. Let
$$
\mathscr L
:=
\left\{
V^\perp:
V\in\mathcal X^{(1)}_{n(n-1),\mathrm{iso}}(\Delta)
\right\}.
$$
By \Cref{lem:lattice-orthogonal-duality}, each member of $\mathscr L$ is
a $\Delta^*$-rational subspace of the form
$\mathcal{L}(\mathbb Ru)$ for some nonzero $u\in\R^n$. Let
$$
\mathsf U:=\operatorname{span}\left\{
 u:\mathcal{L}(\mathbb Ru)\in\mathscr L
\right\}.
$$
Choose $L_1,\ldots,L_\ell\in\mathscr L$, with
${L}_r=\mathcal{L}(\mathbb Ru_r)$, so that $u_1,\ldots,u_\ell$ is a basis of
$\mathsf U$. Thus $\ell=\dim \mathsf U\le n$. For each $r$, the inverse image of ${L}_r\cap\Delta^*$ under the
isomorphism $(\R^n)^*\to {L}_r$, $\lambda\mapsto u_r\otimes\lambda$,
is a lattice $\Lambda_r<(\R^n)^*$; hence
$$
{L}_r\cap\Delta^*=u_r\otimes\Lambda_r.
$$
The direct sum
$$
\Upsilon_0
:=
\bigoplus_{r=1}^{\ell}u_r\otimes\Lambda_r
$$
is a full lattice in $\mathcal{L}(\mathsf U)$ contained in $\Delta^*$. The space
$\mathcal{L}(\mathsf U)$ is $\Delta^*$-rational, since it is the sum of finitely many
$\Delta^*$-rational subspaces. Therefore $\Upsilon_0$ and
$\Upsilon:=\Delta^*\cap \mathcal{L}(\mathsf U)$ are commensurable. In particular, there is
an integer $q\ge1$ such that
\begin{equation}\label{eq:global-isotropic-lattice-inclusion}
\Upsilon\subset q^{-1}\Upsilon_0.
\end{equation}

Define an equivalence relation on $\{1,\ldots,\ell\}$ by
$$
r\sim r'
\quad\Longleftrightarrow\quad
c\Lambda_r\text{ is commensurable with }\Lambda_{r'}
\text{ for some }c\in\R^\times.
$$
To see how this relation controls $\mathscr L$, take
$L=\mathcal{L}(\mathbb Ru)\in\mathscr L$, and write
$$
u=\sum_{r=1}^{\ell}c_ru_r,
\qquad
L\cap\Delta^*=u\otimes\Lambda
$$
for a lattice $\Lambda<(\R^n)^*$. From
\eqref{eq:global-isotropic-lattice-inclusion}, comparison of the
$u_r$-coordinates gives
$$
c_r\Lambda\subset q^{-1}\Lambda_r
\qquad(1\le r\le\ell).
$$
Whenever $c_r\ne0$, both sides are full lattices in $(\R^n)^*$, so
$c_r\Lambda$ is commensurable with $\Lambda_r$. It follows that
$$
\operatorname{supp}(u):=\{r:c_r\ne0\}
$$
is contained in a single $\sim$-equivalence class.

Let $E_1,\ldots,E_p$ be the equivalence classes having at least two
elements, and set
$$
\mathsf U_j:=\operatorname{span}\{u_r:r\in E_j\}.
$$
If there is no equivalence class containing at least two elements, we
take $p=0$; in this case all assertions involving the subspaces
$\mathsf U_j$ are understood vacuously. 

Because the $E_j$ are disjoint subsets of a basis index set,
$\mathsf U_1\oplus\cdots\oplus \mathsf U_p$ is direct. If
$\#\operatorname{supp}(u)\ge2$, then
$\operatorname{supp}(u)\subset E_j$ for a unique $j$, and hence
$L\subset \mathcal{L}(\mathsf U_j)$. If $\operatorname{supp}(u)=\{r\}$, then
$L={L}_r$; this subspace is also contained in some $\mathcal{L}(\mathsf U_j)$ unless
$\{r\}$ is a singleton equivalence class. Thus the only possible exceptions
are among $\ell\le n$ subspaces $L_r$. This proves the
asserted finite decomposition and the uniqueness statement.

It remains to verify the splitting assertion. Fix $j$ and choose $r_j\in E_j$. For
each $r\in E_j$, choose $c_r\in\R^\times$ such that
$c_r\Lambda_r$ is commensurable with $\Lambda_{r_j}$, and put
$\widetilde u_r:=c_r^{-1}u_r$. Then $u_r\otimes\Lambda_r
=
\widetilde u_r\otimes(c_r\Lambda_r)$
is commensurable with
$\widetilde u_r\otimes\Lambda_{r_j}$. Consequently,
$$
\bigoplus_{r\in E_j}u_r\otimes\Lambda_r
$$
is commensurable with
$$
\Lambda_{1,j}\otimes\Lambda_{2,j},
\qquad
\Lambda_{1,j}:=\bigoplus_{r\in E_j}\mathbb Z\widetilde u_r,
\qquad
\Lambda_{2,j}:=\Lambda_{r_j}.
$$
The lattice $\bigoplus_{r\in E_j}u_r\otimes\Lambda_r$ is a full sublattice of
$\Delta^*\cap \mathcal{L}(\mathsf U_j)$; hence $\Delta^*\cap \mathcal{L}(\mathsf U_j)$ is
column-$\mathbb Q$-split.

Finally, suppose that $\Delta$ is not of $\mathbb Q$-split type. If
$\mathsf U_j=\R^n$, then $\mathcal{L}(\mathsf U_j)=\M_n(\R)$, and the splitting just proved
would imply that $\Delta^*$ is commensurable with a full tensor-product lattice.
Dualizing would give the same conclusion for $\Delta$, contrary to the
hypothesis. Thus every $\mathsf U_j$ is proper.

The row statement follows by applying the column statement to
$\Delta^{\intercal}$ and then transposing back.
\end{proof}

\subsection{Structure in higher codimension}

We first reformulate the isotropic noncoincidence condition in terms of the
dual lattice. For every $\Delta$-rational subspace $V$, one has
$$
\Delta_V^*=V^\perp\cap\Delta^*.
$$
Suppose that $h\Delta_{V_1}$ is commensurable with $\Delta_{V_2}$ for
some $h\in\GL_n(\mathbb R)$. Then $hV_1^\perp=V_2^\perp$. Let
$$
T_h:=h|_{V_1^\perp}:V_1^\perp\to V_2^\perp.
$$
Taking dual lattices, the commensurability condition is equivalent to
$$
(T_h^{-1})^*(V_1^\perp\cap\Delta^*)
\quad\text{being commensurable with}\quad
V_2^\perp\cap\Delta^*.
$$
Since $V_1^\perp$ and $V_2^\perp$ are column-type subspaces,
$(T_h^{-1})^*$ is the restriction of left multiplication by an element of
$\GL_n(\mathbb R)$. Conversely, every such restricted inverse adjoint arises
in this way. Thus the column isotropic noncoincidence condition is
equivalent to requiring that there be no $g\in\GL_n(\mathbb R)$ such that
$$
g(V_1^\perp\cap\Delta^*)
\quad\text{is commensurable with}\quad
V_2^\perp\cap\Delta^*.
$$
Likewise, the row condition is equivalent to requiring that there be no
$g\in\GL_n(\mathbb R)$ such that
$$
(V_1^\perp\cap\Delta^*)g
\quad\text{is commensurable with}\quad
V_2^\perp\cap\Delta^*.
$$

\begin{proposition}
\label{prop:finitely-many-maximal-isotropic}
Let $\Delta<\M_n(\R)$ be a lattice satisfying the column isotropic
noncoincidence condition. For every
$\left\lceil\frac n2\right\rceil\le k\le n-2$,
there are at most $n$ proper $\Delta$-rational column-isotropic
subspaces of dimension $kn$ that are maximal, with respect to inclusion,
among all proper $\Delta$-rational column-isotropic subspaces.
The analogous statement holds in the row case.
\end{proposition}
\begin{proof}
We prove the column statement. Fix
$\lceil n/2\rceil\le k\le n-2$. Thus
$2\le n-k \le\left\lfloor\frac n2\right\rfloor$.
Suppose for contradiction that there are $J>n$ distinct maximal
proper $\Delta$-rational column-isotropic subspaces
$V_1,\ldots,V_J$ of dimension $kn$. Set $W_i:=V_i^\perp$. By
\Cref{lem:lattice-orthogonal-duality}, each $W_i$ is a
$\Delta^*$-rational column-isotropic subspace of dimension $(n-k)n$,
so
$$
W_i=\mathcal{L}(\mathsf F_i)
$$
for an $(n-k)$-dimensional subspace $\mathsf F_i<\R^n$.

Each $W_i$ is minimal among the nonzero $\Delta^*$-rational
column-isotropic subspaces. Indeed, if
$0\ne Z\subsetneq W_i$ were such a subspace, then
$Z^\perp$ would be a proper $\Delta$-rational column-isotropic
subspace strictly containing $V_i=W_i^\perp$, contradicting the
maximality of $V_i$.

Choose a maximal subfamily $Y_1,\ldots,Y_r$ of
$\{W_1,\ldots,W_J\}$ whose sum is direct, and set
$$
\mathcal Y:=Y_1\oplus\cdots\oplus Y_r.
$$
Write $Y_j=\mathcal{L}(\mathsf E_j)$, where $\dim \mathsf E_j=n-k$. The sum
$\mathsf E_1\oplus\cdots\oplus \mathsf E_r$ is direct; hence
\begin{equation}\label{eq:number-direct-isotropic-summands}
r(n-k)\le n.
\end{equation}
Moreover, every $W_i$ is contained in $\mathcal Y$. Indeed,
$\mathcal Y$ is $\Delta^*$-rational and
$$
W_i\cap\mathcal Y
=
\mathcal{L}\bigl(\mathsf F_i\cap(\mathsf E_1\oplus\cdots\oplus \mathsf E_r)\bigr)
$$
is a $\Delta^*$-rational column-isotropic subspace of $W_i$. By
minimality, it is either $0$ or $W_i$. If $W_i\cap\mathcal Y=0$, then
$W_i$ could be added to the family, contrary to maximality.

Since $J>n\ge r$, choose a member
$$
W=\mathcal{L}(\mathsf F)\in\{W_1,\ldots,W_J\}\smallsetminus\{Y_1,\ldots,Y_r\}.
$$
For $1\le j\le r$, put
$$
L_j:=Y_j\cap\Delta^*,
\qquad
Q:=L_1\oplus\cdots\oplus L_r.
$$
The lattices $Q$ and $\mathcal Y\cap\Delta^*$ are commensurable.
Since $W\subset\mathcal Y$ is $\Delta^*$-rational, it is also
$Q$-rational. Let $p_j:\mathcal Y\to Y_j$ be the coordinate projection and set $\mathsf E:=\mathsf E_1\oplus\cdots\oplus \mathsf E_r$. Then
$$
p_j=P_j\otimes\id_{(\R^n)^*}
$$
for the coordinate projection $P_j:\mathsf E\to \mathsf E_j$. Therefore
$p_j(W)=\mathcal{L}(P_j\mathsf F)$ is column-isotropic. The map $p_j$ is rational with
respect to $Q$, so $p_j(W)$ is $L_j$-rational and hence
$\Delta^*$-rational. By the minimality of $Y_j$,
\begin{equation}\label{eq:projection-zero-or-surjective}
p_j(W)=0
\quad\text{or}\quad
p_j(W)=Y_j.
\end{equation}
Choose $a$ such that $p_a(W)=Y_a$. Since $W$ and $Y_a$ have the
same dimension, $p_a|_W:W\to Y_a$ is an isomorphism.

For $j\ne a$, define
$$
\phi_j:=p_j\circ(p_a|_W)^{-1}:Y_a\to Y_j;
$$
this map is zero when $p_j(W)=0$. Then
$$
W=
\left\{
 y+\sum_{j\ne a}\phi_j(y):y\in Y_a
\right\}.
$$
Since $p_j=P_j\otimes\id_{(\R^n)^*}$, there are
linear maps $A_j:\mathsf E_a\to \mathsf E_j$ such that
\begin{equation}\label{eq:phijtensor}
\phi_j=A_j\otimes\id_{(\R^n)^*}.
\end{equation}
Define $\Psi\in\GL(\mathcal Y)$ by
$$
\Psi|_{Y_a}
=
\id_{Y_a}+\sum_{j\ne a}\phi_j,
\qquad
\Psi|_{Y_b}=\id_{Y_b}\quad(b\ne a).
$$
Then $\Psi(Y_a)=W$. We claim that $\Psi(Q)$ is commensurable with $Q$. Since $W$ is
$Q$-rational, $W\cap Q$ is a full lattice in $W$, and
$$
L_a':=p_a(W\cap Q)
$$
is a full sublattice of $L_a$. Choose $m\ge1$ with
$mL_a\subset L_a'$. For $j\ne a$,
$$
\phi_j(L_a')\subset L_j,
$$
and therefore $m\phi_j(L_a)\subset L_j$. It follows that
$m\Psi(Q)\subset Q$. Since both are full lattices in $\mathcal Y$,
$\Psi(Q)$ and $Q$ are commensurable.

Set $A:=\Psi|_{Y_a}:Y_a\to W$. The displayed formula for $W$ and the injectivity
of $p_a|_W$ give
$$
A(L_a)=\Psi(Q)\cap W.
$$
Because $\Psi(Q)$ and $Q$ are commensurable, the lattices
$A(L_a)$ and $Q\cap W$ are commensurable. Since $Q$ is
commensurable with $\mathcal Y\cap\Delta^*$, we conclude that
\begin{equation}\label{eq:dual-quotient-lattices-commensurable}
A(Y_a\cap\Delta^*)
\quad\text{is commensurable with}\quad
W\cap\Delta^*.
\end{equation}

Let $V_a:=Y_a^\perp$ and $V:=W^\perp$. These are distinct elements of
$\mathcal X^{(1)}_{kn,\mathrm{iso}}(\Delta)$. Because
$$
\Delta_{V_a}^*=Y_a\cap\Delta^*,
\qquad
\Delta_V^*=W\cap\Delta^*,
$$ dualizing \eqref{eq:dual-quotient-lattices-commensurable} yields
$$
(A^*)^{-1}\Delta_{V_a}
\quad\text{is commensurable with}\quad
\Delta_V.
$$
By \eqref{eq:phijtensor}, the map $A$ has the form
$$
A=B\otimes\id_{(\R^n)^*}
$$
for an isomorphism $B:\mathsf E_a\to \mathsf F$. Hence
$$
(A^*)^{-1}=(B^*)^{-1}\otimes\id_{(\R^n)^*}.
$$
Extend $(B^*)^{-1}:\mathsf E_a\to \mathsf F$ to an element $g\in\GL_n(\R)$. Left
multiplication by $g$ on $\M_n(\R)$ restricts to $(A^*)^{-1}$ on
$Y_a=\mathcal{L}(\mathsf E_a)$. Therefore
$$
g\Delta_{V_a}
\quad\text{is commensurable with}\quad
\Delta_V.
$$
This contradicts the column isotropic noncoincidence condition, because $V_a\ne V$ are
maximal proper $\Delta$-rational column-isotropic subspaces of dimension
$kn$. Thus $J>n$ is impossible. The row case follows by transposition.
\end{proof}

\section{Counting singular matrices}
\label{s:isotropic}

In this section, we prove \Cref{thm:quasinullcontribution2}. The proof has
three steps. First, we establish an asymptotic formula for bounded-rank
matrices in a rational lattice. Second, we combine this formula with
fiberwise lattice-point counting for a rational projection. Finally,
using the structural results from \Cref{s:quasinull}, we reduce the column-
and row-isotropic contributions to finitely many such fibered counting
problems and show that their overlap is of lower order. This gives the
asymptotic for the isotropic part $\Lambda_{\operatorname{iso}}$; the
non-isotropic singular matrices are treated later in
\Cref{s:maincounting}.

Throughout this section, all implied constants depend only on the lattices,
linear maps, and norms under consideration, and are independent of~$T$.

\subsection{Bounded-rank matrices in rational lattices}
We first prove the bounded-rank asymptotic used throughout this section.
Katznelson proved the case
$\Delta=\M_{k,n}(\mathbb Z)$ in \cite{katznelson:1994}. Our argument is
different even in that case: it decomposes matrices according to their
primitive left kernels and also provides the summable majorants needed
later for weighted limits.

Throughout the remainder of this subsection, let $1\le k\le n-1$.
\begin{theorem}[Bounded-rank asymptotic in rational lattices]
\label{rational-lattice}
Let
$\Delta<\M_{k,n}(\mathbb Q)$ be a full-rank lattice. 
For any $1\le r\le k-1$, there is a constant $c_{\Delta,r}>0$ such that
$$
\#\{A\in\Delta:\ \|A\|< T,\ \rank(A)=r\}
\sim
c_{\Delta,r}T^{nr}
\qquad (T\to\infty).
$$
\end{theorem}

For $1\le s\le k$, let $\mathcal P_s(k)$ denote the
set of primitive rank-$s$ sublattices of $\mathbb Z^k$.
For $R>0$, let $\mathcal P_s(k;R)$ denote the set
of primitive rank-$s$ sublattices $L<\mathbb Z^k$ satisfying
$\covol(L)\le R$.
\begin{lemma}[Primitive sublattices]
\label{primitive}
For every $1\le s\le k$, we have
$$
\#\mathcal P_s(k;R)
\ll
R^k(\log R)^s \qquad \text{ for all $R\ge 3$}.
$$
\end{lemma}

\begin{proof}
Fix $L\in\mathcal P_s(k;R)$, and let
$\lambda_1(L)\le\cdots\le\lambda_s(L)$ be its successive minima. Choose
linearly independent vectors $v_i\in L$ with $\|v_i\|=\lambda_i(L)$. By
Minkowski's second theorem~\cite{cassels:1971},
$$
\prod_{i=1}^s\|v_i\|
\asymp_{k,s}
\covol(L)
\le R.
$$
Since $L<\mathbb Z^k$, every nonzero vector in $L$ has norm bounded below
by a positive constant depending only on the chosen norm. After shifting the dyadic scale by a fixed amount, choose integers
$m_i\ge0$ such that $\|v_i\|\asymp 2^{m_i}$. Then
$\sum_i m_i\ll\log_2(R) +O(1)$. For a fixed tuple
$(m_1,\ldots,m_s)$, the number of possible ordered tuples
$(v_1,\ldots,v_s)$ is
$$
\ll
\prod_{i=1}^s2^{km_i}
=
2^{k\sum_i m_i}.
$$
The rational span of $(v_1,\ldots,v_s)$ determines $L$, because $L$ is
primitive. Each tupe therefore contributes $O(R^k)$ possibilities, and there are $O((\log R)^s)$ admissible tuples.
Thus
$
\#\mathcal P_s(k;R)
\ll
R^k(\log R)^s.
$
\end{proof}

\begin{proof}[Proof of \Cref{rational-lattice}]
We argue by induction on $k$. The assertion is empty for $k=1$. After
multiplying $\Delta$ by a nonzero integer, we may assume that
$\Delta<\M_{k,n}(\mathbb Z)$. Fix $1\le r\le k-1$, and put $s:=k-r$. For
$L\in\mathcal P_s(k)$, define
\begin{equation}\label{vl}
V_L
:=
\{A\in\M_{k,n}(\mathbb R):u^{\intercal}A=0\text{ for every }u\in L\},
\qquad
\Delta_L:=\Delta\cap V_L.
\end{equation}
Then $V_L$ has dimension $nr$, and $\Delta_L$ is a full-rank lattice in
$V_L$.

For a rational matrix $A$ of rank $r$, let
$$
L(A):=\ker(A^{\intercal})\cap\mathbb Z^k.
$$
Then $L(A)$ is a primitive rank-$s$ sublattice. Conversely, if
$A\in V_L$ has rank $r$, then $L\subset L(A)$, and primitivity together
with equality of ranks gives $L=L(A)$. Hence the set of the rank-$r$ matrices
 decomposes disjointly according to the primitive left kernel:
\begin{equation}\label{eq:rank-kernel-decomposition}
\#\{A\in\Delta:\ \|A\|< T,\ \rank(A)=r\}
=
\sum_{L\in\mathcal P_s(k)}N_L(T),
\end{equation}
where
$N_L(T)
:=\#\{A\in\Delta_L:\ \|A\|< T,\ \rank(A)=r\}$.

For a fixed $L$, the standard lattice-point asymptotic in the
$nr$-dimensional space $V_L$ gives
$$
\#\{A\in\Delta_L:\ \|A\|< T\}
\sim
\beta_LT^{nr},
\qquad
\beta_L
:=
\frac{\vol_{V_L}(B_1\cap V_L)}{\covol(\Delta_L)}.
$$
After an
integral change of basis in $\mathbb Z^k$, the matrices in $\Delta_L$ of rank at most $r-1$ can be identified with bounded-rank matrices in a
full-rank rational lattice in $\M_{r,n}(\mathbb R)$. By the induction
hypothesis, their number is $O_L(T^{n(r-1)})$; when $r=1$, the only
rank-zero matrix is $0$. Therefore
\begin{equation}\label{eq:fixed-kernel-asymptotic}
N_L(T)
\sim
\beta_LT^{nr}.
\end{equation}

It remains to justify summation over $L$. Set
$$
L^{\perp}_{\mathbb Z}
:=
\{x\in\mathbb Z^k:u^{\intercal}x=0\text{ for every }u\in L\}.
$$
Then
$$
\M_{k,n}(\mathbb Z)\cap V_L
=
\bigl(L^{\perp}_{\mathbb Z}\bigr)^n
\quad\text{columnwise}.
$$
Since $L$ is primitive, the Hodge-star identification of the primitive
Pl\"ucker vectors of $L$ and $L^{\perp}_{\mathbb Z}$ gives
$\covol(L^{\perp}_{\mathbb Z})=\covol(L)$. Hence
$$
\covol\bigl(\M_{k,n}(\mathbb Z)\cap V_L\bigr)
=
\covol(L)^n.
$$
Because $\Delta_L$ is a sublattice of
$\M_{k,n}(\mathbb Z)\cap V_L$, we have
$\covol(\Delta_L)\ge\covol(L)^n$. Since the volumes
$\vol_{V_L}(B_1\cap V_L)$ are uniformly bounded as $V_L$ varies in the
Grassmannian,
\begin{equation}\label{eq:beta-majorant}
\beta_L\ll\covol(L)^{-n}.
\end{equation}
A successive-minima lattice-point estimate gives, uniformly in $L$ and $T\ge1$,
\begin{equation}\label{eq:NL-majorant}
N_L(T)
\ll
T^{nr}\covol(L)^{-n}.
\end{equation}
Indeed, if $N_L(T)>0$, the rank-$r$ condition forces the $r$th successive
minimum of $L_{\mathbb Z}^{\perp}$ to be $O(T)$; Minkowski's second theorem
then gives the displayed bound.
By \Cref{primitive},
$$
\sum_{L\in\mathcal P_s(k)}\covol(L)^{-n}
\ll
\sum_{j\ge0}
2^{-jn}\#\mathcal P_s(k;2^{j+1})
\ll
\sum_{j\ge0}2^{-j(n-k)}(1+j)^s
<
\infty,
$$
because $k<n$. Dominated convergence applied to
\eqref{eq:rank-kernel-decomposition}, using
\eqref{eq:fixed-kernel-asymptotic} and \eqref{eq:NL-majorant}, gives
$$
T^{-nr} \#\{A\in\Delta:\ \|A\|< T,\ \rank(A)=r\}
\to
\sum_{L\in\mathcal P_s(k)}\beta_L
=:
c_{\Delta,r}>0.
$$
This proves the theorem.
\end{proof}

\begin{corollary}[Crude bounded-rank estimate]
\label{cor:crude-bounded-rank}
Let $1\le m\le \ell$, let $0\le r\le m-1$, and let $T\ge3$. Then
$$
\#\{A\in\M_{m,\ell}(\mathbb Z):\ \|A\|< T,\ \rank(A)=r\}
\ll
T^{\ell r}(\log T)^{m+1}.
$$
If $m<\ell$, the logarithmic factor may be omitted.
\end{corollary}

\begin{proof}
The case $r=0$ is immediate. For $r\ge1$, decompose the rank-$r$
matrices according to their primitive rank-$(m-r)$ left kernel, as in the
proof of \Cref{rational-lattice}. The contribution of a fixed kernel $L$ is
$$
\ll T^{\ell r}\covol(L)^{-\ell}.
$$
Moreover, if such a matrix has norm at most $T$, then $\covol(L)\ll T^r$:
choose $r$ linearly independent columns and use the Hodge-star relation
between the saturated column lattice and its primitive orthogonal complement.
Thus it remains to sum $\covol(L)^{-\ell}$ over primitive $L$ with
$\covol(L)\ll T^r$. The dyadic estimate in \Cref{primitive} is summable when
$\ell>m$, and in the square case $\ell=m$ gives at most
$O((\log T)^{m+1})$.
\end{proof}

\subsection{A uniform rank-extension estimate}
\begin{lemma}[Uniform rank-extension estimate]
\label{lem:rank-extension} Let $m,q,\ell\ge1$ and $T\ge3$, and let
$A\in\M_{m,q}(\mathbb Z)$ satisfy
$\rank(A)<m$ and $\|A\|\le T$.
Then, uniformly in $A$,
$$
\begin{aligned}
\#&\left\{
B\in\M_{m,\ell}(\mathbb Z):\ \|B\|< T,\
\rank\begin{pmatrix}A&B\end{pmatrix}\le m-1
\right\}
\\&\hspace{35mm}
\ll
\begin{cases}
T^{m\ell}& \text{ if } \ell<m-\rank(A),\\ 
T^{\ell(m-1)}(\log T)^m& \text{ otherwise}
\end{cases}
\end{aligned}
$$
The analogous estimate for vertical concatenation follows by transposition.
\end{lemma}
\begin{proof}
Set
$
r:=\rank(A)
$
and
$
d:=m-r.
$
If $\ell<d$, then for every $B\in\M_{m,\ell}(\mathbb Z)$,
$$
\rank\begin{pmatrix}A&B\end{pmatrix}
\le r+\ell<m.
$$
 The trivial bound
$O(T^{m\ell})$ therefore proves the first case. Assume henceforth that
$\ell\ge d$. Let $E<\mathbb R^m$ be the column space of $A$, let
$p:\mathbb R^m\to E^\perp$ be the orthogonal projection, and set
$$
E_{\mathbb Z}:=E\cap\mathbb Z^m,
\qquad
L:=p(\mathbb Z^m),
\qquad
c:=\covol(L).
$$
The lattices $E_{\mathbb Z}$ and $L$ have ranks $r$ and $d$,
respectively. Since
$
\covol(E_{\mathbb Z})\,c=1
$
and $E_{\mathbb Z}$ contains $r$ linearly independent columns of
$A$, all of norm $O(T)$, we have
$$
1\le c^{-1}=\covol(E_{\mathbb Z})\ll T^r.
$$

Writing $B=(b_1,\ldots,b_\ell)$ in columns, let
$$
p(B):=(p(b_1),\ldots,p(b_\ell))\in L^\ell.
$$
Each nonempty fiber of $\mathbb Z^m\to L$ is an affine translate of
$E_{\mathbb Z}$. Since all successive minima of $E_{\mathbb Z}$ are
$O(T)$, the standard lattice-point estimate gives, uniformly in $y\in L$,
$$
\#\{b\in\mathbb Z^m:p(b)=y,\ \|b\|\ll T\}
\ll
\frac{T^r}{\covol(E_{\mathbb Z})}
=cT^r.
$$
Thus, for every $(y_1,\ldots,y_\ell)\in L^\ell$, the number of matrices
$B=(b_1,\ldots,b_\ell)$ with $\|B\|<T$ and
$p(b_j)=y_j$ for all $j$ is $\ll(cT^r)^\ell$.
The column space of $\begin{pmatrix}A&B\end{pmatrix}$ is
$
E+\operatorname{span}_{\mathbb R}\{b_1,\ldots,b_\ell\}.
$
Since $\ker p=E$, it follows that
$$
\rank\begin{pmatrix}A&B\end{pmatrix}
=
r+\dim\operatorname{span}_{\mathbb R}
\{p(b_1),\ldots,p(b_\ell)\}.
$$

It remains to show that
\begin{multline}\label{eq:projected-rank-bound}
\#\left\{
(y_1,\ldots,y_\ell)\in L^\ell:
\|y_j\|\ll T,\
\dim\operatorname{span}_{\mathbb R}\{y_1,\ldots,y_\ell\}\le d-1
\right\} \\
\ll
c^{-\ell}T^{\ell(d-1)}(\log T)^m.
\end{multline}

Note that
$L^*=E^\perp\cap\mathbb Z^m$.
Fix $1\le s\le d-1$, and consider a tuple
$(y_1,\ldots,y_\ell)$ whose span $W$ has dimension $s$. Then
$K:=L^*\cap W^\perp$
is a primitive rank-$d-s$ sublattice of $L^*$ and determines $W$
uniquely. Here and below, orthogonal complements are taken inside
$E^\perp$. Set
$
L_K:=L\cap K^\perp
$
so that
$\covol(L_K)=c\,\covol(K)$.

If $K$ arises from such a tuple, then $L_K$ contains $s$ linearly
independent vectors of norm $O(T)$. Consequently,
$$
c\,\covol(K)=\covol(L_K)\ll T^s\; \text{ and }\; 
\#\{y\in L_K:\|y\|\ll T\}
\ll
\frac{T^s}{c\,\covol(K)}.
$$
Let $N_s$ denote the number of tuples
$(y_1,\ldots,y_\ell)\in L^\ell$ such that
$$
\|y_j\|\ll T
\qquad\text{and}\qquad
\dim\operatorname{span}_{\mathbb R}\{y_1,\ldots,y_\ell\}=s.
$$
Then
\begin{equation}\label{eq:projected-rank-s}
N_s
\ll
c^{-\ell}T^{s\ell}
\sum_{\substack{
K< L^*\ \mathrm{primitive}\\
\rank(K)=d-s,\ \covol(K)\ll T^s/c}}
\covol(K)^{-\ell}.
\end{equation}

We use the following uniform estimate: for $R\ge1$ and $\ell\ge d$,
\begin{equation}\label{eq:weighted-primitive-sum}
\sum_{\substack{
K<L^*\ \mathrm{primitive}\\
\rank(K)=d-s,\ \covol(K)\le R}}
\covol(K)^{-\ell}
\ll
\begin{cases}
(1+\log R)^{d-s},&\ell=d,\\
1,&\ell>d.
\end{cases}
\end{equation}
Indeed, Minkowski's second theorem provides linearly independent vectors
$u_1,\ldots,u_{d-s}\in K$ such that
$$
\covol(K)
\le
\prod_{i=1}^{d-s}\|u_i\|
\ll
\covol(K).
$$
Since $K$ is primitive in $L^*$, it is determined by
$\operatorname{span}_{\mathbb R}\{u_1,\ldots,u_{d-s}\}$. Moreover,
$L^*\subset\mathbb Z^m$ is $1$-separated in the $d$-dimensional space
$E^\perp$, and hence
$$
\#\{u\in L^*:\|u\|\le X\}\ll(1+X)^d
$$
uniformly in $E$. If $\|u_i\|\asymp2^{j_i}$, the contribution of this
dyadic range to the sum in \eqref{eq:weighted-primitive-sum} is
$
\ll 2^{(d-\ell)(j_1+\cdots+j_{d-s})}.
$
When $\ell=d$, there are $O((1+\log R)^{d-s})$ relevant dyadic ranges;
when $\ell>d$, the contributions form a convergent geometric series.
This proves \eqref{eq:weighted-primitive-sum}.

Apply \eqref{eq:weighted-primitive-sum} with
$R\asymp T^s/c$. Since $c^{-1}\ll T^r$,
$
R\ll T^{s+r}\le T^{m-1}. $
It follows from \eqref{eq:projected-rank-s} that
$$
N_s
\ll
c^{-\ell}T^{s\ell}(\log T)^{d-s}
\ll
c^{-\ell}T^{\ell(d-1)}(\log T)^m.
$$
The rank-zero tuple is the zero tuple and satisfies the same bound because
$c\le1$. Summing over $0\le s\le d-1$ proves
\eqref{eq:projected-rank-bound}.

Finally, multiplying the projected count by the fiber bound gives
$$
(cT^r)^\ell
\cdot
c^{-\ell}T^{\ell(d-1)}(\log T)^m
=
T^{\ell(m-1)}(\log T)^m.
$$
This proves the second case. The vertical estimate follows by
transposition.
\end{proof}

\subsection{Weighted and fibered bounded-rank asymptotics}
We now record the weighted form of the bounded-rank asymptotic. 
\begin{lemma}[Weighted bounded-rank asymptotic]
\label{lem:weighted-bounded-rank}
Let $2\le k\le n-1$, and let
$\Delta<\M_{k,n}(\mathbb Q)$ be a full-rank lattice. Define a measure
$$
\mu_\Delta
:=
\sum_{L\in\mathcal P_1(k)}
\frac{1}{\covol(\Delta_L)}\,\vol_{V_L}
$$ where $V_L$ and $\Delta_L$ are as in \eqref{vl}, and
$\vol_{V_L}$ denotes the Lebesgue measure on $V_L$ induced by the
ambient Euclidean inner product.
Then $\mu_\Delta$ is locally finite and $
\mu_\Delta\bigl(\{A:\rank(A)\ne k-1\}\bigr)=0.
$. Moreover, for every
$f\in C_c(\M_{k,n}(\mathbb R))$,
\begin{equation}\label{eq:boundedrankasymptotic}
\lim_{T\to \infty} T^{-n(k-1)}
\sum_{\substack{A\in\Delta,\;  \rank(A)\le k-1}}
f(T^{-1}A) =
\int_{\M_{k,n}(\mathbb R)}f\,d\mu_\Delta.
\end{equation}
\end{lemma}

\begin{proof} After rescaling $\Delta$, we may assume that
$\Delta<\M_{k,n}(\mathbb Z)$. Since $k<n$, it follows from \eqref{eq:beta-majorant} and
\Cref{primitive} that for $\beta_L:=\frac{\vol_{V_L}(B_1\cap V_L)}{\covol(\Delta_L)}$,
$$
\sum_{L\in\mathcal P_1(k)}\beta_L
\ll
\sum_{L\in\mathcal P_1(k)}\covol(L)^{-n}
<\infty.
$$
Since
$\dim V_L=n(k-1)$, for any $R>1$, 
$$
\mu_\Delta(B_R) = R^{n(k-1)}
\sum_{L\in\mathcal P_1(k)}\beta_L
<\infty.
$$
Thus $\mu_\Delta$ is locally finite.
For each $L$, the locus of matrices of rank at most $k-2$ is a proper algebraic subset
of $V_L$ and hence has $\vol_{V_L}$-measure zero. Thus $\mu_\Delta$ gives full measure to the rank-$(k-1)$ locus.

Rank-$(k-1)$ matrices decompose disjointly according to their unique
primitive one-dimensional left kernels. For each $L$, standard Riemann-sum
convergence for $\Delta_L$ gives the desired limit termwise. The majorant
\eqref{eq:NL-majorant} and the summability established above allow us to apply
dominated convergence over $L$. Finally, \Cref{rational-lattice} shows that
matrices of rank at most $k-2$ contribute $o(T^{n(k-1)})$, so including
them does not affect the limit.
\end{proof}

\begin{lemma}
\label{lem:fibered-rank-count}
Let $W<\M_n(\mathbb R)$ be a subspace with
$\dim W=kn$ for some $1\le k\le n-1$, and let
$p_W:\M_n(\mathbb R)\to W$ denote the orthogonal projection. Let
$\Delta<\M_n(\mathbb R)$ be a lattice such that
$\Delta\cap W^\perp$ is a lattice in $W^\perp$ and
$\Delta_W:=p_W(\Delta)$ is a lattice in $W$.

Let $\iota:W\to\M_{k,n}(\mathbb R)$ be a linear isomorphism such that
$\iota(\Delta_W)$ is commensurable with a full-rank lattice in
$\M_{k,n}(\mathbb Q)$. Then, for every norm on $\M_n(\mathbb R)$,
there exists $c_{\Delta,W}>0$ such that
$$
\#\left\{
v\in\Delta:\ \|v\|<T,\
\rank\bigl(\iota(p_W(v))\bigr)\le k-1
\right\}
\sim
c_{\Delta,W}T^{n(n-1)}.
$$
\end{lemma}

\begin{proof}
If $k=1$, then the assertion is
the usual lattice-point asymptotic for the rank-$n(n-1)$ lattice
$\Delta\cap W^\perp$. Assume henceforth that $k\ge2$.
Let
$\mathsf B:=\{v\in\M_n(\mathbb R):\|v\|<1\}.$
For $x\in W$, define
$$
\mathsf B_x
:=
\{z\in W^\perp:x+z\in\mathsf B\},
\qquad
\psi(x)
:=
\operatorname{vol}_{W^\perp}(\mathsf B_x).
$$
The function $\psi$ is bounded, and its support is contained in the
compact set $\overline{p_W(\mathsf B)}$. For each $A\in\Delta_W$, choose $z_A\in W^\perp$ such that
$A+z_A\in\Delta$. Then
$$
\Delta\cap p_W^{-1}(A)
=
A+z_A+(\Delta\cap W^\perp).
$$
Moreover,
$$
\|A+z_A+z\|<T
\quad\Longleftrightarrow\quad
z_A+z\in T\mathsf B_{T^{-1}A}.
$$

We now justify the required uniformity in $A$. Each
$\mathsf B_x$ is convex. Since $\mathsf B$ is bounded, there exists
$R>0$ such that
$$
\mathsf B_x
\subset
\{z\in W^\perp:\|z\|_{\mathrm{Euc}}\le R\}
\qquad(x\in W).
$$
Consequently, the volumes of all positive-dimensional orthogonal
projections of $\mathsf B_x$ are uniformly bounded in $x$. The
successive-minima form of the Lipschitz principle therefore gives,
uniformly in $A$ and in the translate $z_A$,
\begin{equation}\label{eq:fiberwisecounting}
\begin{aligned}
&\#\left(
(z_A+\Delta\cap W^\perp)
\cap T\mathsf B_{T^{-1}A}
\right)\\
&\qquad=
\frac{T^{n^2-kn}}
{\operatorname{covol}(\Delta\cap W^\perp)}
\psi(T^{-1}A)
+
O_{\Delta,W,\|\cdot\|}
\bigl(T^{n^2-kn-1}+1\bigr).
\end{aligned}
\end{equation}

Only those $A$ for which
$\mathsf B_{T^{-1}A}\ne\emptyset$ contribute. For such $A$,
we have $
T^{-1}A\in p_W(\mathsf B),
$
and hence $\|A\|\ll T$. By \Cref{rational-lattice},
$$
\#\left\{
A\in\Delta_W:\|A\|\ll T,\
\rank(\iota(A))\le k-1
\right\}
\ll T^{n(k-1)}.
$$
Thus the total contribution of the error term in
\eqref{eq:fiberwisecounting} is
$$
O\left(
\bigl(T^{n^2-kn-1}+1\bigr)T^{n(k-1)}
\right)
=
o(T^{n(n-1)}).
$$

Transporting \Cref{lem:weighted-bounded-rank} through $\iota$ gives a
locally finite measure $\mu_{\Delta_W}$ on $W$. The function $\psi$ is
continuous on the interior of $p_W(\mathsf B)$, so its discontinuities
are contained in $\partial p_W(\mathsf B)$. Moreover,
$\mu_{\Delta_W}$ is a locally finite sum of Lebesgue measures on
$n(k-1)$-dimensional linear subspaces through the origin. Since $0$ is
an interior point of $p_W(\mathsf B)$, the intersection of its boundary
with each of these subspaces is Lebesgue null. Hence
$\partial p_W(\mathsf B)$ is $\mu_{\Delta_W}$-null.

The bounded-rank asymptotic therefore applies to $\psi$ and gives
$$
\sum_{{A\in\Delta_W,
\rank(\iota(A))\le k-1}}
\psi(T^{-1}A)
=
T^{n(k-1)}
\left(
\int_W\psi\,d\mu_{\Delta_W}+o(1)
\right).
$$
Summing \eqref{eq:fiberwisecounting} over the bounded-rank points in
$\Delta_W$ now yields
$$
\begin{aligned}
&\#\left\{
v\in\Delta:\ \|v\|<T,\
\rank\bigl(\iota(p_W(v))\bigr)\le k-1
\right\}\\
&\qquad\sim
\frac{1}{\operatorname{covol}(\Delta\cap W^\perp)}
\left(\int_W\psi\,d\mu_{\Delta_W}\right)
T^{n(n-1)}.
\end{aligned}
$$
Finally, the integral is positive because $\psi>0$ near $0$ and
$\mu_{\Delta_W}$ assigns positive measure to every neighborhood of
$0$.
\end{proof}

\subsection{Finite decompositions of the isotropic contribution}

Let $\Lambda_{\operatorname{iso}}^{(1)}$ and
$\Lambda_{\operatorname{iso}}^{(2)}$ denote, respectively, the lattice
points lying in a $\Lambda$-rational column- and row-isotropic subspace of
dimension $n,2n,\ldots,n(n-1)$. Recall that
$\Lambda_{\operatorname{iso}}$ consists of the lattice points lying in a
$\Lambda$-rational column- or row-isotropic subspace. Thus
\begin{equation}\label{eq:isotropic-union-column-row}
\Lambda_{\operatorname{iso}}
=
\Lambda_{\operatorname{iso}}^{(1)}\cup\Lambda_{\operatorname{iso}}^{(2)}.
\end{equation}

The next lemma packages the structural results of \Cref{s:quasinull} in the
form needed for counting.  It says that, up to lower-dimensional pieces, the
column-isotropic contribution is a finite union of fibered bounded-rank
conditions.

\begin{lemma}[Finite decomposition of the isotropic contribution]
\label{lem:finite-isotropic-piece-decomposition}
Assume that a lattice $\Lambda<\M_n(\br)$ is not of $\mathbb Q$-split type and satisfies the column
isotropic noncoincidence condition. Then there are
\begin{enumerate}[label=\textnormal{(\roman*)}]
\item finitely many proper $\Lambda$-rational subspaces
$L_1,\ldots,L_q$, each contained in a column-isotropic subspace and each of
dimension at most $n(n-2)$; 
\item finitely many column-type subspaces
$$
W_j=\mathcal L(\mathsf U_j)=\mathsf U_j\otimes(\mathbb R^n)^*,
\quad 1\le j\le J,
\quad
1\le k_j:=\dim \mathsf U_j\le n-1,
$$
together with rank-preserving tensor-product isomorphisms
$$
\iota_j:W_j\to\M_{k_j,n}(\mathbb R),
$$
such that $\iota_j(p_{W_j}(\Lambda))$ is commensurable with a full-rank
rational lattice.
\end{enumerate}
These objects may be chosen so that
\begin{equation}\label{eq:finite-column-decomposition}
\Lambda_{\operatorname{iso}}^{(1)}
=
\left(\Lambda\cap\bigcup_{m=1}^qL_m\right)
\cup
\bigcup_{j=1}^J\Sigma_j
\end{equation}
where
$
\Sigma_j
:=
\left\{
v\in\Lambda:
\rank\bigl(\iota_j(p_{W_j}(v))\bigr)\le k_j-1
\right\}$.
Moreover, the decomposition may be chosen so that, whenever $j\ne j'$,
\begin{equation}\label{eq:column-piece-overlap}
\#\{v\in\Sigma_j\cap\Sigma_{j'}:\|v\|< T\}
\ll
T^{n(n-2)}.
\end{equation}
The analogous statement holds for the row-isotropic contribution.
\end{lemma}

\begin{proof}
We prove the column statement. First consider maximal proper
$\Lambda$-rational column-isotropic subspaces of dimension $kn$ with
$k\le n-2$. For $\lceil n/2\rceil\le k\le n-2$, their finiteness follows
from \Cref{prop:finitely-many-maximal-isotropic}. For $k<n/2$, there is at
most one such maximal subspace for each $k$. Indeed, if $V_1$ and $V_2$
were distinct maximal subspaces of this dimension, then $V_1^\perp$ and
$V_2^\perp$ would be minimal nonzero $\Lambda^*$-rational
column-isotropic subspaces of dimension $(n-k)n$.
The corresponding
$(n-k)$-dimensional subspaces of $\mathbb R^n$ intersect nontrivially,
because $n-k>n/2$. Their intersection is again $\Lambda^*$-rational and
column-isotropic, contradicting minimality unless $V_1^\perp=V_2^\perp$.
Taking all maximal subspaces in these finitely many families gives
$L_1,\ldots,L_q$, and we include the zero subspace if necessary. Every
column-isotropic point not contained in an $n(n-1)$-dimensional
column-isotropic subspace lies in one of these $L_m$.

It remains to treat the $n(n-1)$-dimensional column-isotropic subspaces. By
\Cref{lem:global-isotropic-structure} and  the construction in its
proof, there are subspaces $\mathsf U_1,\ldots,\mathsf U_J$, of dimensions
between $2$ and $n-1$ such that $\mathsf U_j\cap\mathsf U_{j'}=\{0\}$ for $j\ne j'$ and, apart from finitely
many exceptional ones, every $\Lambda$-rational
column-isotropic subspace $V$ of dimension $n(n-1)$ satisfies
$$
V^\perp\subset \mathcal L(\mathsf U_j)
$$
for a unique $j$. Moreover, $\Lambda^*\cap\mathcal L(\mathsf U_j)$ is
column-$\mathbb Q$-split. We discard from the exceptional list any member
whose orthogonal complement is already contained in one of the
$\mathcal L(\mathsf U_j)$. For each remaining exceptional subspace $V$, add to the list the
one-dimensional factor $\mathsf U$ determined by $V^\perp=\mathcal L(\mathsf U)$. Thus every $n(n-1)$-dimensional column-isotropic subspace is assigned
to one of the resulting spaces $W_j=\mathcal L(\mathsf U_j)$, possibly with
$k_j=1$.

Fix one such space $W_j=\mathcal L(\mathsf U_j)$, and write $p_j=p_{W_j}$. Since
$$
p_j(\Lambda)^*=\Lambda^*\cap W_j,
$$
and the tensor-product splitting property is preserved by duality, there is a
rank-preserving tensor-product isomorphism
$$
\iota_j=a_j\otimes b_j:
W_j\to\M_{k_j,n}(\mathbb R)
$$
for which $\iota_j(p_j(\Lambda))$ is commensurable with a full-rank rational
lattice. When $k_j=1$, this is automatic.

We claim that
\begin{equation}\label{eq:rankdecomposition}
\Lambda\cap
\bigcup
\{V : {V\in\mathcal X^{(1)}_{n(n-1),\mathrm{iso}}(\Lambda),\; V^\perp\subset W_j}\}
=
\Sigma_j.
\end{equation}
If $v\in V\cap\Lambda$ and $V^\perp=\mathcal L(\mathbb Ru)$ with
$0\ne u\in \mathsf U_j$, then every column of $p_j(v)$ lies in
$\mathsf U_j\cap u^\perp$. Since $\iota_j$ preserves rank,
$$
\rank\bigl(\iota_j(p_j(v))\bigr)\le k_j-1, \quad\text{so $v\in\Sigma_j$.}
$$

Conversely, let $v\in\Sigma_j$. Since $\iota_j(p_j(\Lambda))$ is
commensurable with a rational lattice, the left kernel of
$\iota_j(p_j(v))$ contains a nonzero rational covector. Pulling this covector
back through $a_j$, and using the Euclidean identification of $\mathsf U_j^*$ with
$\mathsf U_j$, we obtain a nonzero $u\in \mathsf U_j$ such that
$\mathcal L(\mathbb Ru)$ is $p_j(\Lambda)^*$-rational and orthogonal to
$p_j(v)$. Because $p_j(\Lambda)^*=\Lambda^*\cap W_j$, the space
$\mathcal L(\mathbb Ru)$ is $\Lambda^*$-rational. Hence $V:=\mathcal L(\mathbb Ru)^\perp$ is a $\Lambda$-rational column-isotropic subspace of dimension $n(n-1)$,
with $V^\perp\subset W_j$. Since $v-p_j(v)\in W_j^\perp$, the
orthogonality of $p_j(v)$ to $\mathcal L(\mathbb Ru)$ implies $v\in V$.
This proves \eqref{eq:rankdecomposition}, and hence
\eqref{eq:finite-column-decomposition}.

It remains to estimate the overlap of two distinct main pieces. By
construction, $\mathsf U_j\cap \mathsf U_{j'}=\{0\}$ for $j\ne j'$. Set
$$
W:=W_j\oplus W_{j'}=\mathcal L(\mathsf U_j\oplus \mathsf U_{j'}).
$$
The lattice $\Lambda^*\cap W$ is commensurable with $(\Lambda^*\cap W_j)\oplus(\Lambda^*\cap W_{j'})$. Using the two tensor-product coordinate maps on the two summands, the joint
quotient lattice is commensurable with a full-rank rational lattice in $\M_{k_j,n}(\mathbb R)\oplus\M_{k_{j'},n}(\mathbb R)$. Thus, after scaling, we may overcount its points by pairs of integral
matrices. The number of pairs $(A,A')$ of norm $O(T)$ satisfying
$$
\rank(A)\le k_j-1,
\qquad
\rank(A')\le k_{j'}-1
$$
is $O\bigl(T^{n(k_j-1)+n(k_{j'}-1)}\bigr)$ by \Cref{rational-lattice}.
If  $k_j=1$ or $k_{j'}=1$, the
corresponding matrix is zero and contributes $O(1)$. The kernel of the joint
quotient map has dimension $n(n-k_j-k_{j'})$, so each pair has
$O(T^{n(n-k_j-k_{j'})})$ lifts in $\Lambda$. Therefore
$$
\#\{v\in\Sigma_j\cap\Sigma_{j'}:\|v\|< T\}
\ll
T^{n(k_j-1)+n(k_{j'}-1)+n(n-k_j-k_{j'})}
=
T^{n(n-2)}.
$$
This proves \eqref{eq:column-piece-overlap}. The row statement follows after
transposition.
\end{proof}

Combining \Cref{lem:fibered-rank-count} with
\Cref{lem:finite-isotropic-piece-decomposition}, we obtain the following:
\begin{proposition}\label{prop:column-row-isotropic-asymptotic}
Assume that $\Lambda$ is not of $\mathbb Q$-split type and satisfies the isotropic
noncoincidence condition.
There exist constants
$c_{\Lambda,\mathrm{col}}^{\mathrm{iso}},
 c_{\Lambda,\mathrm{row}}^{\mathrm{iso}}\ge0$ such that
$$
\#\{v\in\Lambda_{\operatorname{iso}}^{(1)}:\|v\|< T\}
=
c_{\Lambda,\mathrm{col}}^{\mathrm{iso}}T^{n(n-1)} + o (T^{n(n-1)}) ;$$ 
$$\#\{v\in\Lambda_{\operatorname{iso}}^{(2)}:\|v\|< T\}
=
c_{\Lambda,\mathrm{row}}^{\mathrm{iso}}T^{n(n-1)}+ o (T^{n(n-1)}).$$
\end{proposition}
\begin{proof}
Each main piece $\Sigma_j$ has the asserted asymptotic, the finitely
many lower-dimensional spaces contribute $O(T^{n(n-2)})$, and
\eqref{eq:column-piece-overlap} gives the same bound for every pairwise
overlap of main pieces.\end{proof}

\subsection{Mixed column--row intersections}

The column and row asymptotics must now be combined by inclusion--exclusion.
The following compatibility lemma gives rational coordinates for the joint
projection associated with one column piece and one row piece. 

\begin{lemma}[Compatible split coordinates]
\label{lem:compatible-split-coordinates}
Let $\mathsf C$ and $\mathsf R$ be $n$-dimensional real vector spaces,
let $\mathsf U<\mathsf C$ and $\mathsf V<\mathsf R$ be subspaces, and put
$$
W:=\mathsf U\otimes\mathsf R\quad\text{and}\quad
Z:=\mathsf C\otimes\mathsf V.
$$
Let $\Delta<\mathsf C\otimes\mathsf R$ be a lattice. Assume that $W$ and
$Z$ are $\Delta^*$-rational and that
$\Delta^*\cap W$ and
$\Delta^*\cap Z $
are commensurable with tensor-product lattices in $W$ and $Z$, respectively.
Let $q_{W,Z}:
\mathsf C\otimes\mathsf R
\to
(\mathsf C\otimes\mathsf R)/(W^\perp\cap Z^\perp)
$
be the quotient map. One can choose bases of $\mathsf C$ and
$\mathsf R$, adapted to $\mathsf U$ and $\mathsf V$, such that
the quotient is identified with
$$
\bigl(\mathsf U\otimes\mathsf V\bigr)
\oplus
\bigl(\mathsf U\otimes(\mathsf R/\mathsf V)\bigr)
\oplus
\bigl((\mathsf C/\mathsf U)\otimes\mathsf V\bigr),
$$
and $q_{W,Z}(\Delta)$ is a full-rank rational lattice in these
coordinates.
\end{lemma}

\begin{proof}
Choose tensor-product lattices $L_{\mathsf U}\otimes L_{\mathsf R}$ and
$M_{\mathsf C}\otimes M_{\mathsf V}$ commensurable with $\Delta^*\cap W$ and
$\Delta^*\cap Z$, respectively. Here $L_{\mathsf U}<\mathsf U$,
$L_{\mathsf R}<\mathsf R$, $M_{\mathsf C}<\mathsf C$, and $M_{\mathsf V}<\mathsf V$ are
full lattices in the indicated spaces. Since
$W\cap Z=\mathsf U\otimes\mathsf V$ is $\Delta^*$-rational, $\Delta^*\cap(W\cap Z)$ is a
full lattice in $W\cap Z$. It is commensurable with each of the lattices
$$
L_{\mathsf U}\otimes(L_{\mathsf R}\cap\mathsf V)
\qquad\text{and}\qquad
(M_{\mathsf C}\cap\mathsf U)\otimes M_{\mathsf V}.
$$
In particular, $\mathsf V$ is rational with respect to $L_{\mathsf R}$, $\mathsf U$ is
rational with respect to $M_{\mathsf C}$, and the two displayed
product lattices are commensurable.

We use the elementary fact that if two full tensor-product lattices
$L_1\otimes L_2$ and $L_1'\otimes L_2'$ are commensurable, then there is a
scalar $c\ne0$ such that $L_1$ is commensurable with $cL_1'$ and
$L_2$ is commensurable with $c^{-1}L_2'$.
Equivalently, if a Kronecker product $A_1\otimes A_2$ of two invertible change-of-basis matrices has rational entries, then, after multiplying $A_1$ by a scalar and $A_2$ by the inverse scalar, both $A_1$ and $A_2$ have rational entries. Indeed, fix $(i_0,j_0)$ and $(k,l)$ with $(A_1)_{i_0j_0}(A_2)_{kl}\ne0$; then every ratio $(A_1)_{ij}/(A_1)_{i_0j_0}=\bigl((A_1)_{ij}(A_2)_{kl}\bigr)\big/\bigl((A_1)_{i_0j_0}(A_2)_{kl}\bigr)$ is a ratio of entries of $A_1\otimes A_2$, hence rational, so $A_1$ is a scalar multiple of a rational matrix. The same argument applies to $A_2$, with mutually inverse scalars modulo $\mathbb Q^\times$.

It follows that the lattice structures induced on $\mathsf U$ by $L_{\mathsf U}$ and
$M_{\mathsf C}\cap \mathsf U$ are homothetically commensurable.  Similarly the
lattice structures induced on $\mathsf V$ by $L_{\mathsf R}\cap\mathsf V$ and $M_{\mathsf V}$
are homothetically commensurable. Choose a basis of $\mathsf C$ adapted to
the rational subspace $\mathsf U$ and the lattice $M_{\mathsf C}$, and choose a
basis of $\mathsf R$ adapted to $\mathsf V$ and the lattice $L_{\mathsf R}$.
Rescaling one tensor-factor basis absorbs the common homothety. In these
coordinates, both $\Delta^*\cap W$ and $\Delta^*\cap Z$ are rational
lattices.

Since $W+Z$ is $\Delta^*$-rational,
$W^\perp\cap Z^\perp$ is $\Delta$-rational, so
$q_{W,Z}(\Delta)$ is a lattice. Its dual is
$$
q_{W,Z}(\Delta)^*=\Delta^*\cap(W+Z).
$$
This lattice is commensurable with
$(\Delta^*\cap W)+(\Delta^*\cap Z)$,
which is rational in the chosen coordinates. Hence
$q_{W,Z}(\Delta)$ is also rational.
\end{proof}

\begin{lemma}[Mixed-piece estimate]
\label{lem:mixed-isotropic-pieces}
Let $\Sigma_{\mathrm c}$ be a column piece and
$\Sigma_{\mathrm r}$ a row piece supplied by
\Cref{lem:finite-isotropic-piece-decomposition}. Then
$$
\#\{v\in\Sigma_{\mathrm c}\cap\Sigma_{\mathrm r}:\|v\|< T\}
\ll
T^{n^2-n-1}(\log T)^{3n}.
$$
\end{lemma}

\begin{proof}
Set $\mathsf C:=\mathbb R^n$ and $\mathsf R:=(\mathbb R^n)^*$. Let
$
W:=\mathsf U\otimes\mathsf R$ and
$Z:=\mathsf C\otimes\mathsf V$
be the $\Lambda^*$-rational subspaces used in the definitions of
$\Sigma_{\mathrm c}$ and $\Sigma_{\mathrm r}$, respectively, and put
$$
k:=\dim\mathsf U,
\qquad
\ell:=\dim\mathsf V.
$$
Thus $1\le k,\ell\le n-1$, and $\Lambda^*\cap W$ and
$\Lambda^*\cap Z$ are commensurable with tensor-product lattices. By
\Cref{lem:compatible-split-coordinates} applied with $\Delta=\Lambda$,  choose bases $e_1,\ldots,e_n$ of $\mathsf C$ and
$f_1,\ldots,f_n$ of $\mathsf R$ such that
$$
\mathsf U=\operatorname{span}\{e_1,\ldots,e_k\},
\qquad
\mathsf V=\operatorname{span}\{f_1,\ldots,f_\ell\}.
$$
For $k<i\le n$ and $\ell<j\le n$, let $\overline e_i$ and
$\overline f_j$ denote the classes of $e_i$ and $f_j$ in
$\mathsf C/\mathsf U$ and $\mathsf R/\mathsf V$, respectively. Under the identification
of the quotient in \Cref{lem:compatible-split-coordinates} with
$$
(\mathsf U\otimes\mathsf V)
\oplus
(\mathsf U\otimes(\mathsf R/\mathsf V))
\oplus
((\mathsf C/\mathsf U)\otimes\mathsf V),
$$
the three summands have respective bases
$$
\{e_i\otimes f_j:1\le i\le k,\ 1\le j\le\ell\},
$$
$$
\{e_i\otimes\overline f_j:1\le i\le k,\ \ell<j\le n\},
\qquad
\{\overline e_i\otimes f_j:k<i\le n,\ 1\le j\le\ell\}.
$$
With respect to their union, $q_{W,Z}(\Lambda)$ is a full-rank rational
lattice. After multiplying all coordinates by a fixed positive integer, we
may identify it with a sublattice of the corresponding integral lattice and
hence overcount it by integral triples.

The kernel $W^\perp\cap Z^\perp$ of $q_{W,Z}$ has dimension
$(n-k)(n-\ell)$. Consequently, each projected point has
$O(T^{(n-k)(n-\ell)})$ lifts in $\Lambda$ of norm at most $T$. Write the
three quotient coordinates as
$$
A\in\M_{k,\ell}(\mathbb Z),
\qquad
B\in\M_{k,n-\ell}(\mathbb Z),
\qquad
C\in\M_{n-k,\ell}(\mathbb Z).
$$
The two piece conditions become
\begin{equation}\label{eq:mixed-rank-conditions}
\rank\begin{pmatrix}A&B\end{pmatrix}\le k-1,
\qquad
\rank\begin{pmatrix}A\\ C\end{pmatrix}\le\ell-1.
\end{equation}
Interchanging the column and row pieces if necessary, assume $k\ge\ell$.

We stratify by $r:=\rank(A)$, where $0\le r\le\ell-1$. Applying
\Cref{cor:crude-bounded-rank} to $A^{\intercal}\in\M_{\ell,k}(\mathbb Z)$
gives
\begin{equation}\label{eq:A-rank-stratum}
\#\{A:\|A\|< T,\ \rank(A)=r\}
\ll
T^{kr}(\log T)^n.
\end{equation}

First suppose that $r\ge k+\ell-n$. Then $n-\ell\ge k-r$ and
$n-k\ge\ell-r$. For fixed $A$, apply \Cref{lem:rank-extension} to $(A,B)$, and
its transposed version to $(A,C)$ to obtain
$$\begin{aligned}
    \#\{B: \rank\begin{pmatrix}A&B\end{pmatrix}\le k-1\}
&\ll
T^{(n-\ell)(k-1)}(\log T)^n,\\\#\{C:\rank\begin{pmatrix}A\\ C\end{pmatrix}\leq\ell-1\}
&\ll
T^{(n-k)(\ell-1)}(\log T)^n.
\end{aligned}$$
Including the lifts in the omitted block, the rank-$r$ stratum contributes
at most
$$
T^{kr+(n-\ell)(k-1)+(n-k)(\ell-1)+(n-k)(n-\ell)}(\log T)^{3n}.
$$
Since $r\le\ell-1$, the exponent of $T$ is at most $n^2-2n+\ell
\le
n^2-n-1$.

Now suppose that $r<k+\ell-n$. Then the two rank conditions in
\eqref{eq:mixed-rank-conditions} are automatic once $A$ is fixed, because
$$
r+(n-\ell)\le k-1,
\qquad
r+(n-k)\le\ell-1.
$$
Thus $B$, $C$, and the omitted block contribute respectively
$$
O(T^{k(n-\ell)}),
\qquad
O(T^{\ell(n-k)}),
\qquad
O(T^{(n-k)(n-\ell)}).
$$
Together with \eqref{eq:A-rank-stratum}, the exponent of $T$ is
$$
kr+k(n-\ell)+\ell(n-k)+(n-k)(n-\ell)
=
n^2-k(\ell-r).
$$
Since $r\le k+\ell-n-1$, we have $\ell-r\ge n-k+1$. Hence
$$
n^2-k(\ell-r)
\le
n^2-k(n-k+1)
\le
n^2-n-1.
$$
The last inequality follows from
$$
k(n-k+1)-(n+1)=(k-1)(n-k)-1\ge0;
$$
this second case can occur only for $2\le k\le n-1$.

Summing over $0\le r\le\ell-1$ proves the lemma.
\end{proof}

\subsection{Proof of \Cref{thm:quasinullcontribution2}}

\begin{proof}[Proof of \Cref{thm:quasinullcontribution2}]
The column and row asymptotics are given in
\Cref{prop:column-row-isotropic-asymptotic}. We claim that the intersection of the two contributions is
negligible.

Apply the finite piece decomposition in both the column and row cases. The
intersection of a lower-dimensional piece with any other set contributes
only $O(T^{n(n-2)})$. There are only finitely many pairs of main column
and row pieces, and \Cref{lem:mixed-isotropic-pieces} gives
$$
\#\left\{
v\in\Lambda_{\operatorname{iso}}^{(1)}\cap\Lambda_{\operatorname{iso}}^{(2)}:
\|v\|< T
\right\}
\ll
T^{n^2-n-1}(\log T)^{3n}
+
T^{n(n-2)}.
$$
In particular,
\begin{equation}\label{eq:column-row-overlap-negligible}
\#\left\{
v\in\Lambda_{\operatorname{iso}}^{(1)}\cap\Lambda_{\operatorname{iso}}^{(2)}:
\|v\|< T
\right\}
=
o(T^{n(n-1)}).
\end{equation}

Using \eqref{eq:isotropic-union-column-row} and inclusion--exclusion, we
therefore obtain
$$
\begin{aligned}
\#\{v\in\Lambda_{\operatorname{iso}}:\|v\|< T\}
&=
\#\{v\in\Lambda_{\operatorname{iso}}^{(1)}:\|v\|< T\}
+
\#\{v\in\Lambda_{\operatorname{iso}}^{(2)}:\|v\|< T\}\\
&\qquad-
\#\{v\in\Lambda_{\operatorname{iso}}^{(1)}\cap\Lambda_{\operatorname{iso}}^{(2)}:
\|v\|< T\}\\
&=
\left(
 c_{\Lambda,\mathrm{col}}^{\mathrm{iso}}
 +c_{\Lambda,\mathrm{row}}^{\mathrm{iso}}
\right)T^{n(n-1)}
+o(T^{n(n-1)}).
\end{aligned}
$$
Thus the limit in \Cref{thm:quasinullcontribution2} exists, with
$$
c_{\Lambda,\mathrm{iso}}
=
c_{\Lambda,\mathrm{col}}^{\mathrm{iso}}
+c_{\Lambda,\mathrm{row}}^{\mathrm{iso}}
\ge0.
$$

If $\Lambda$ contains a
rank-$n(n-1)$ submodule on which the determinant vanishes identically,
then its real span is an $n(n-1)$-dimensional linear subspace of
$\{\det=0\}$, and is therefore column- or row-isotropic by Dieudonn\'e's
theorem on maximal singular subspaces of $\M_n (\br)$ \cite{dieudonne:1948}.
Being spanned by lattice vectors, this span is $\Lambda$-rational, so the
submodule is contained in $\Lambda_{\operatorname{iso}}$. Standard
lattice-point counting in this span then gives
$c_{\Lambda,\mathrm{iso}}>0$. 

Conversely, suppose that $c_{\Lambda,\mathrm{iso}}>0$. Since the
lower-dimensional pieces and all relevant overlaps contribute
$o(T^{n(n-1)})$, the finite-piece decomposition in
\Cref{lem:finite-isotropic-piece-decomposition} implies that at least one
main column- or row-isotropic piece has a positive $T^{n(n-1)}$-coefficient. By the construction of
these pieces, there is a $\Lambda$-rational column- or row-isotropic
subspace $V$ of dimension $n(n-1)$. Hence
$\Lambda\cap V$ is a rank-$n(n-1)$ submodule of $\Lambda$ on
which the determinant vanishes identically. This completes the proof.
\end{proof}

\section{Fiber integrals and integral identities}\label{s:fiberintegral}
In this section, we define the fiber kernel $J_f$, show that 
every continuous compactly supported function on $B_0^+\times\mathbb R$ arises this way, and derive the
integral identities needed for counting.

For $v\in \M_n(\R)$, we write
\be\label{vzero}
v=\begin{pmatrix}
v^0 & v^+\\
v^- & v_{nn}
\end{pmatrix},
\ee
where $v^0\in \M_{n-1}(\R)$, $v^+\in \R^{n-1}$, and
$v^-\in (\R^{n-1})^{\intercal}$.

Let
$$
\M_n^+=\{v\in \M_n(\R): \det v^0>0\}.
$$  For a test function
$f\in C_c(\M_n^+)$, the main object is the fiber integral
$J_f(\mathsf r,\zeta)$ defined in \Cref{fiberi}, where $\mathsf r\in B_0^+$ records the $(n-1)$-dimensional
singular-value data and $\zeta$ records the determinant.

By the singular value decomposition, for each $v\in \M_n(\R)$, there is a unique diagonal element $$b(v):=\diag\bigl(e^{\kappa_1(v)},\dots,e^{\kappa_{n-1}(v)},\operatorname{sgn}(\det v)e^{\kappa_n(v)}\bigr)
$$ with $\kappa_1(v)\ge \cdots \ge \kappa_n(v)\ge -\infty$ such that
$$
v=k\cdot b(v) \quad\text{for some $k\in K$.}
$$
For every $v\in\M_n(\mathbb R)$ with
$\operatorname{rank}(v)\ge n-1$, define
$$
b^0(v)
:=
\diag(e^{\kappa_1(v)},\dots,e^{\kappa_{n-1}(v)})
\in\GL_{n-1}^+(\mathbb R),
$$
where $\kappa_1(v)\ge \cdots \ge \kappa_{n-1}(v) >-\infty$. Since $\det v^0>0$ implies $\rank(v)\ge n-1$, this defines
$b^0(v)$ for every $v\in\M_n^+$.

\subsection{The fiber integrals $J_f$ and approximation of $K$-averages}
Let  \begin{align*} B_0^+ &=\{\mathsf r:=\diag (r_1, \ldots, r_{n-1})\in \GL_{n-1}^+(\br): r_1\ge \cdots \ge r_{n-1}>0\}\\ &=\{b^0(v): v\in \M_n^+\}.\end{align*}
Let
$
\mathsf M:=\SO(n-1)\times\SO(n-1)
$
act on $\GL_{n-1}^+(\mathbb R)$ by
$$
(m_1,m_2)\cdot A=m_1 Am_2^{\intercal}.
$$
By the singular-value decomposition, for any
$A\in \GL_{n-1}^+(\R)$, there exists a unique $\r\in B_0^+$ such that $A\in \mathsf M\cdot \r$.
Thus any continuous function on  $B_0^+$ determines a unique continuous $\mathsf M$-invariant function on $\GL_{n-1}^+(\R)$, and conversely. We denote by $dm$ the probability Haar measure on $\mathsf M$.

\begin{definition}\label{fiberi} Given $f\in C_c({\M_n^+})$,  we define the associated fiber integral function
$$J_f: B_0^+ \times \R\to \R$$ as follows: for any $ (\r ,\zeta)\in \bp\times  \R$,
$$
J_f(\r ,\zeta)
:=
\frac{1}{\det(\r)^2}
\int_{\R^{n-1}\times \R^{n-1}}\int_{\mathsf M}
f\begin{pmatrix}
m\cdot \r & x_2\\
x_1^{\intercal} & x_{nn}
\end{pmatrix}
\, dm \, dx_1\,dx_2,
$$
where $x_{nn}=x_{nn}(\r,\zeta, m,x_1,x_2)$ is determined by the condition
$$
\det\begin{pmatrix}
m\cdot \r & x_2\\
x_1^{\intercal} & x_{nn}
\end{pmatrix}
=\zeta.
$$
\end{definition}
We may consider $J_f$ as a function on $\GL_{n-1}^+(\R)\times \br$:
$$J_f(A,\zeta):=J_f(\r,\zeta),$$ where $\r\in B_0^+$ is the unique element such that $A\in \mathsf M\cdot\r$. The resulting function is $\mathsf M$-invariant and belongs to
$C_c(\GL_{n-1}^+(\mathbb R)\times\mathbb R)$. Indeed, compact support
follows from the compactness of $\operatorname{supp}(f)\subset\M_n^+$, while continuity
follows from the explicit formula for $x_{nn}$ together with dominated
convergence.

\begin{lemma}\label{lem:Jf-surjective}
For any $\mathsf M$-invariant function
$h\in C_c(\GL_{n-1}^+(\br)\times \R)$, there exists
$f\in C_c(\M_n^+)$ such that
$$
J_f(A,\zeta)=h(A,\zeta)
\qquad
\text{for all }(A,\zeta)\in \GL_{n-1}^+(\br)\times \R.
$$
\end{lemma}

\begin{proof}
Choose a nonnegative $\psi\in C_c(\R^{n-1}\times\R^{n-1})$ such that
$$
\int_{\R^{n-1}\times\R^{n-1}}\psi(x_1,x_2)\,dx_1\,dx_2=1.
$$
For an $\mathsf M$-invariant function
$h\in C_c(\GL_{n-1}^+(\br)\times \R)$, define $f_h:\M_n^+\to \br$ by
\be\label{fh}
f_h\begin{pmatrix}
m\cdot \r & x_2\\
x_1^{\intercal} & x_{nn}
\end{pmatrix}
:=
\det(\r)^2\,h\!\left(\r,
\det\begin{pmatrix}
m\cdot \r & x_2\\
x_1^{\intercal} & x_{nn}
\end{pmatrix}\right)\psi(x_1,x_2)
\ee where $\r\in B_0^+$, $m\in\mathsf M$, $x_1, x_2\in \br^{n-1}$.
Then $f_h\in C_c(\M_n^+)$, and the definition of $J_{f_h}$ and the normalization of $\psi$ give
$J_{f_h}(\r,\zeta)=h(\r,\zeta)$ for all $(\r, \zeta)\in B_0^+\times \br$. This proves the claim.
\end{proof}

For the remainder of this section and throughout the following counting section,
we normalize the fixed $K$-invariant norm so that
$
\|E_{nn}\|=1.
$
We set
$$
\mathbb S^{N-1}
=
\mathbb S(\M_n(\mathbb R))
=
\{v\in\M_n(\mathbb R):\|v\|=1\}.
$$

We denote by
$\widehat v$ the adjugate matrix $\operatorname{adj}(v)$, that is, the
transpose of the cofactor matrix of $v$. For $v\in\M_n(\R)$,
$$
\widehat v\ne0
\quad\Longleftrightarrow\quad
\rank(v)\ge n-1.
$$
Thus $\widehat v/\|\widehat v\|$ is defined precisely when
$\rank(v)\ge n-1$.
 Note that the locus $\{v\in\M_n(\mathbb R):\operatorname{rank}(v)\le n-2\}
=
\{v \in \M_n(\mathbb R):\widehat v=0\}$ has Lebesgue measure zero. Throughout this section, any integrand
involving $b^0(v)$ or
$\widehat v/\|\widehat v\|$ is understood to be zero on the rank-at-most-$(n-2)$ locus.

We compare the $K$-average of $f(a_tk\cdot v)$ with the
corresponding averaged kernel $J_f$.
\begin{proposition}[Uniform asymptotic for $K$-averages]\label{jf}
Let $f\in C_c(\M_n^+)$ and
$\nu\in C(\mathbb S^{N-1})$. For every $\varepsilon>0$, there exist $t_0,T_0>0$ such that, for all $t\ge t_0$ and
$v\in\M_n(\mathbb R)$ satisfying $\|v\|\ge T_0$ and $\widehat v\ne0$,
\begin{multline*}
\left|
\frac{{\omega}_{n-1}^2}{2} \, e^{2n(n-1)t}
\int_K f(a_t k \cdot v)\,
\nu\bigl((k^{\intercal} \cdot E_{nn})^{\intercal}\bigr)\, dk \right. \\
\left.
- J_f\bigl( b^0(e^{-2t}v), \det(v)\bigr)\,
\nu\!\left(\frac{\widehat v}{\|\widehat v\|}\right)
\right| < \e.
\end{multline*}
Here $\omega_{n-1}$ denotes the surface area of 
$\mathbb S^{n-1}=\{x\in\mathbb R^n:\|x\|=1\}$.
\end{proposition}

\begin{proof}
Choose $k_v\in K$ such that $k_v\cdot b(v)=v$.
Substituting $kk_v^{-1}$ for the original integration variable and
using Haar invariance, we obtain
\begin{align}
&\int_K f(a_tk\cdot v)\,
\nu\bigl((k^{\intercal}\cdot E_{nn})^{\intercal}\bigr)\,dk \notag\\
&\qquad=
\int_K f(a_tk\cdot b(v))\,
\nu\bigl(((kk_v^{-1})^{\intercal}\cdot E_{nn})^{\intercal}\bigr)\,dk.
\label{eq:jf-reduced-integral}
\end{align}

\smallskip
\noindent\emph{Uniform bounds from the support of $f$.}
Compactness of $\operatorname{supp}(f)$ gives $c,C,D>0$ such that for every $u\in\operatorname{supp}(f)$,
$$
c\le \sigma_{n-1}(u^0)\le \sigma_1(u^0)\le C,
\qquad
\|u\|\le C,
\qquad
|\det u|\le D.
$$
 We claim that, for all
sufficiently large $t$, each of the conditions
$$
f(a_tk\cdot b(v))\ne0\quad\text{for some }k\in K,
\qquad
J_{|f|}\bigl(b^0(e^{-2t}v),\det v\bigr)\ne0
$$
implies
\begin{equation}\label{eq:jf-uniform-support}
\frac c2\le e^{-2t}e^{\kappa_{n-1}(v)}
\le\cdots\le e^{-2t}e^{\kappa_1(v)}\le 2C,
\qquad
|\det v|\le D.
\end{equation}

Suppose first that $u=a_tk\cdot b(v)\in\operatorname{supp}(f)$. Since
$a_{-t}\cdot u=k\cdot b(v)$, the singular values of
$$
e^{-2t}(a_{-t}\cdot u)
=
\begin{pmatrix}
u^0&e^{-nt}u^+\\
e^{-nt}u^-&e^{-2nt}u_{nn}
\end{pmatrix}
$$
are $e^{-2t}e^{\kappa_i(v)}$. The displayed matrix converges uniformly,
for $u\in\operatorname{supp}(f)$, to $\operatorname{diag}(u^0,0)$.
Since
$$
\left|e^{-2t}e^{\kappa_i(v)}-\sigma_i(u^0)\right|
\ll_f e^{-nt}
\qquad(1\le i\le n-1)
$$
and
$
c\le\sigma_{n-1}(u^0)\le\sigma_1(u^0)\le C,
$
this gives the first part of \eqref{eq:jf-uniform-support} for all
sufficiently large $t$; the second follows from $\det v=\det u$.

If instead
$J_{|f|}\bigl(b^0(e^{-2t}v),\det v\bigr)\ne0$, then there exist
$m=(m_1,m_2)\in\mathsf M$ and $x^\pm\in\mathbb R^{n-1}$ such that
$$
f\!\begin{pmatrix}
m_1b^0(e^{-2t}v)m_2^{\intercal}&x^+\\
(x^-)^{\intercal}&x_{nn}
\end{pmatrix}
\ne0.
$$
The singular values of the upper-left block are the diagonal entries of
$b^0(e^{-2t}v)$, and the determinant of the full matrix is $\det v$.
This again gives \eqref{eq:jf-uniform-support}. It also shows that
$x^-$ and $x^+$ range over a fixed bounded set whenever the displayed
matrix belongs to $\operatorname{supp}(f)$.

If neither support condition holds, then the $K$-integrand vanishes
identically and $J_{|f|}=0$, hence also $J_f=0$. Thus both sides of the
desired asymptotic vanish. We may therefore assume that at least one
support condition holds; all subsequent estimates are uniform on this
set.

We next localize $k$ under the condition
$f(a_tk\cdot b(v))\ne0$. If
$u=a_tk\cdot b(v)\in\operatorname{supp}(f)$, covariance of the
adjugate gives
$$
\widehat{k\cdot b(v)}=b_t\widehat u\,b_t.
$$
Here $(\widehat u)_{nn}=\det(u^0)$ is bounded away from zero, whereas
the largest scaling factor of any other entry of
$b_t\widehat u\,b_t$ is $e^{(n-2)t}$. Consequently,
\begin{equation}\label{eq:jf-adj-kb}
\left\|
\frac{\widehat{k\cdot b(v)}}{\|\widehat{k\cdot b(v)}\|}-E_{nn}
\right\|
=O_f(e^{-nt}).
\end{equation}

Write
$$
b^0(e^{-2t}v)=\diag(r_1,\ldots,r_{n-1}).
$$
Then
$$
b(v)=\diag\left(
e^{2t}r_1,\ldots,e^{2t}r_{n-1},
\frac{\det v}{e^{2(n-1)t}r_1\cdots r_{n-1}}
\right).
$$
This includes the rank-$(n-1)$ case, in which $\det v=0$. By
\eqref{eq:jf-uniform-support},
\begin{equation}\label{eq:jf-adj-b}
\left\|
\frac{\widehat{b(v)}}{\|\widehat{b(v)}\|}-E_{nn}
\right\|
=O_f(e^{-2nt}).
\end{equation}
Indeed, the $(n, n)$-entry of $\widehat{b(v)}$ is
$e^{2(n-1)t}r_1\cdots r_{n-1}$, whereas its $ii$-entry is
$\det v/(e^{2t}r_i)$ for $i<n$.

Adjugate covariance also gives
$
k_2^{\intercal}\widehat{k\cdot b(v)}k_1=\widehat{b(v)}.
$
Therefore, \eqref{eq:jf-adj-kb} and \eqref{eq:jf-adj-b} imply
\begin{equation}\label{eq:jf-k-localization}
\|k_2^{\intercal}E_{nn}k_1-E_{nn}\|=O_f(e^{-nt}).
\end{equation}
Thus $k_1^{\intercal}e_n$ and $k_2^{\intercal}e_n$ are both
$O_f(e^{-nt})$-close to $e_n$, or both $O_f(e^{-nt})$-close to
$-e_n$. Hence the integrand on the right-hand side of
\eqref{eq:jf-reduced-integral} can be nonzero only when $k$ belongs to one of these two
neighborhoods.

The same calculation controls the $\nu$-factor. Write
$kk_v^{-1}=(q_1,q_2)$. Since
$k\cdot b(v)=(kk_v^{-1})\cdot v$, adjugate covariance gives
$$
q_2^{\intercal}
\frac{\widehat{k\cdot b(v)}}{\|\widehat{k\cdot b(v)}\|}
q_1
=
\frac{\widehat v}{\|\widehat v\|}.
$$
Moreover,
$
((kk_v^{-1})^{\intercal}\cdot E_{nn})^{\intercal}
=q_2^{\intercal}E_{nn}q_1.
$
Hence, by uniform continuity of $\nu$ and \eqref{eq:jf-adj-kb},
\begin{equation}\label{eq:jf-nu-uniform}
\nu\bigl(((kk_v^{-1})^{\intercal}\cdot E_{nn})^{\intercal}\bigr)
=
\nu\!\left(\frac{\widehat v}{\|\widehat v\|}\right)
+o_{f,\nu}(1)
\end{equation}
uniformly whenever $f(a_tk\cdot b(v))\ne0$.

\smallskip
\noindent\emph{Coordinates near $k_i^{\intercal}e_n=e_n$.}
For $y\in\mathbb R^{n-1}$ with $\|y\|<1$, set
$$
c(y):=(1-\|y\|^2)^{1/2},
\;
C(y):=I_{n-1}-\frac{yy^{\intercal}}{1+c(y)},\;
\sigma(y):=
\begin{pmatrix}
C(y)&-y\\
y^{\intercal}&c(y)
\end{pmatrix}
\in\SO(n).
$$
Thus
$$
\sigma(y)^{\intercal}e_n=(y,c(y)),
\qquad
\sigma(y)e_n=(-y,c(y)).
$$
Every $k=(k_1,k_2)\in K$ for which both
$k_i^{\intercal}e_n$ are sufficiently close to $e_n$ can be written uniquely as
\begin{equation}\label{eq:jf-hemisphere-coordinates}
k_1=m_1\sigma(y^-),
\qquad
k_2=m_2\sigma(y^+)^{\intercal},
\qquad
m_i\in\SO(n-1),
\end{equation}
with $y^\pm$ close to zero, where $\SO(n-1)$ is embedded in $\SO(n)$
via $m\mapsto\diag(m,1)$. In these coordinates, Haar probability
measure is
\begin{equation}\label{eq:jf-haar-hemisphere}
dk=\omega_{n-1}^{-2}
\frac{dm\,dy^-\,dy^+}{c(y^-)c(y^+)}.
\end{equation}

Throughout the remainder of the proof, write
$$
\mathsf r:=b^0(e^{-2t}v).
$$
Define
$$
x^-:=e^{nt}m_2\mathsf r y^-,
\qquad
x^+:=-e^{nt}m_1\mathsf r y^+.
$$
The corresponding fiber matrix is
\begin{equation}\label{eq:jf-fiber-matrix}
\Phi_{\mathsf r,\det v}(m,x^-,x^+)
:=
\begin{pmatrix}
m_1\mathsf r m_2^{\intercal}&x^+\\
(x^-)^{\intercal}&
\displaystyle
\frac{\det v}{\det\mathsf r}
+(x^-)^{\intercal}
(m_1\mathsf r m_2^{\intercal})^{-1}x^+
\end{pmatrix}.
\end{equation}
Its lower-right entry is the unique value that makes
 the determinant equal to
$\det v$.

If \eqref{eq:jf-uniform-support} holds and $x^\pm$ range over a fixed
bounded set, direct block multiplication gives
\begin{equation}\label{eq:jf-matrix-approximation}
\left\|a_tk\cdot b(v)-
\Phi_{\mathsf r,\det v}(m,x^-,x^+)\right\|
=O_f(e^{-2nt}).
\end{equation}
For completeness, set $c_\pm=c(y^\pm)$ and
$C_\pm=C(y^\pm)$. The four blocks of
$$
\diag(m_1^{\intercal},1)(a_tk\cdot b(v))\diag(m_2,1)
$$
are
$$
\begin{aligned}
&C_-\mathsf r C_+
-e^{-2nt}\frac{\det v}{\det\mathsf r}
y^-(y^+)^{\intercal},\\
&C_-m_1^{\intercal}x^+
-e^{-nt}\frac{\det v}{\det\mathsf r}y^-c_+,\\
&(m_2^{\intercal}x^-)^{\intercal}C_+
+e^{-nt}\frac{\det v}{\det\mathsf r}c_-(y^+)^{\intercal},\\
&(m_2^{\intercal}x^-)^{\intercal}\mathsf r^{-1}
m_1^{\intercal}x^+
+c_-c_+\frac{\det v}{\det\mathsf r}.
\end{aligned}
$$
Since $y^\pm=O_f(e^{-nt})$, these formulas prove
\eqref{eq:jf-matrix-approximation}. They remain valid when
$\rank(v)=n-1$, in which case $\det v=0$.

We now compare the supports of the two integrands. If
$$
\Phi_{\mathsf r,\det v}(m,x^-,x^+)
\in\operatorname{supp}(f),
$$
then $x^\pm$ lie in a fixed bounded set. For $x^\pm$ in this bounded
set, define
$$
y^-=e^{-nt}\mathsf r^{-1}m_2^{\intercal}x^-,
\qquad
y^+=-e^{-nt}\mathsf r^{-1}m_1^{\intercal}x^+.
$$
By \eqref{eq:jf-uniform-support}, $\|y^\pm\|<1$ for all sufficiently
large $t$, uniformly in $m$ and in $x^\pm$ in this bounded set. Thus
\eqref{eq:jf-hemisphere-coordinates} defines $k\in K$, and
\eqref{eq:jf-matrix-approximation} holds. Conversely, if
$f(a_tk\cdot b(v))\ne0$ and both $k_i^{\intercal}e_n$ are
$O_f(e^{-nt})$-close to $e_n$, then $y^\pm=O_f(e^{-nt})$, and the
definitions of $x^\pm$, together with \eqref{eq:jf-uniform-support},
show that $x^\pm$ lie in a fixed bounded set. Consequently, there is
$R>0$, independent of $t$ and $v$, such that both integrands vanish
outside
\begin{equation}\label{eq:jf-common-domain}
m\in\mathsf M,
\qquad
\|x^-\|,\|x^+\|\le R.
\end{equation}
For all sufficiently large $t$, the formulas for $y^\pm$ satisfy
$\|y^\pm\|<1$ throughout \eqref{eq:jf-common-domain}. Hence
the coordinate map \eqref{eq:jf-hemisphere-coordinates} is defined wherever either integrand can be nonzero.

The change of variables has Jacobian
$$
dx^-\,dx^+
=e^{2n(n-1)t}\det(\mathsf r)^2\,dy^-\,dy^+.
$$
Together with \eqref{eq:jf-haar-hemisphere}, this gives
\begin{equation}\label{eq:jf-final-jacobian}
dk=\omega_{n-1}^{-2}e^{-2n(n-1)t}
\frac{dm\,dx^-\,dx^+}
{\det(\mathsf r)^2c(y^-)c(y^+)}.
\end{equation}
On the fixed domain \eqref{eq:jf-common-domain},
$c(y^-)c(y^+)=1+O_f(e^{-2nt})$. Therefore
\eqref{eq:jf-matrix-approximation},
\eqref{eq:jf-common-domain}, uniform continuity of the zero extension
of $f$, and \eqref{eq:jf-final-jacobian} give, uniformly in $v$,
\begin{align}
&\omega_{n-1}^2e^{2n(n-1)t}
\int_{K^+}f(a_tk\cdot b(v))\,dk \notag\\
&\quad=
\frac{1}{\det(\mathsf r)^2}
\int_{\mathbb R^{n-1}\times\mathbb R^{n-1}}
\int_{\mathsf M}
f\bigl(\Phi_{\mathsf r,\det v}(m,x^-,x^+)\bigr)
\,dm\,dx^-\,dx^++o_f(1)\notag\\
&\quad=J_f(\mathsf r,\det v)+o_f(1),
\label{eq:jf-positive-component}
\end{align}
where
$$
K^+:=\left\{k=(k_1,k_2)\in K:
\langle k_1^{\intercal}e_n,e_n\rangle>0,
\ \langle k_2^{\intercal}e_n,e_n\rangle>0
\right\}.
$$
By \eqref{eq:jf-k-localization}, if $k\in K^+$ and
$f(a_tk\cdot b(v))\ne0$, then $k$ lies in the coordinate neighborhood
\eqref{eq:jf-hemisphere-coordinates} for all sufficiently large $t$.
The element
$
\left(
\diag(-1,I_{n-2},-1),
\diag(-1,I_{n-2},-1)
\right)
$  fixes $b(v)$, and the right multiplication by this element
interchanges the positive and negative components.  
The contributions from these two components are therefore equal, while
\eqref{eq:jf-k-localization} excludes all other contributions. Hence
\eqref{eq:jf-positive-component} gives
$$
\frac{\omega_{n-1}^2}{2}e^{2n(n-1)t}
\int_K f(a_tk\cdot b(v))\,dk
=
J_f\bigl(b^0(e^{-2t}v),\det v\bigr)+o_f(1).
$$
Applying the same estimate to $|f|$ shows that
$$
\int_K |f(a_tk\cdot b(v))|\,dk
\ll_f e^{-2n(n-1)t}
$$
uniformly under the support conditions above, since
$J_{|f|}(b^0(e^{-2t}v),\det v)$ is uniformly bounded there by
\eqref{eq:jf-uniform-support} and \eqref{eq:jf-common-domain}.
It follows from \eqref{eq:jf-nu-uniform} that
\begin{align*}
&\int_K f(a_tk\cdot b(v))\,
\nu\bigl(((kk_v^{-1})^{\intercal}\cdot E_{nn})^{\intercal}\bigr)\,dk\\
&\qquad=
\nu\!\left(\frac{\widehat v}{\|\widehat v\|}\right)
\int_K f(a_tk\cdot b(v))\,dk
+o_{f,\nu}\bigl(e^{-2n(n-1)t}\bigr).
\end{align*}
Combining the last two displays with \eqref{eq:jf-reduced-integral}
gives
$$
\frac{\omega_{n-1}^2}{2}e^{2n(n-1)t}
\int_K f(a_tk\cdot v)\,
\nu\bigl((k^{\intercal}\cdot E_{nn})^{\intercal}\bigr)\,dk
=
J_f\bigl(b^0(e^{-2t}v),\det v\bigr)
\nu\!\left(\frac{\widehat v}{\|\widehat v\|}\right)
+o_{f,\nu}(1),
$$
where the error tends to zero uniformly in $v$. After increasing
$t_0$, we may therefore assume that the error is smaller than
$\varepsilon$. The argument is uniform for all $v$ with
$\widehat v\ne0$; so taking $T_0=1$ proves the proposition.
\end{proof}
\subsection{Integral identities for $J_f$}
We now turn to the integral identities satisfied by $J_f$.
\begin{proposition}\label{emm3.9}
Let $f\in C_c({\M_n^+})$ and $\nu\in C(\S^{N-1})$ be nonnegative.
Then
$$
\lim_{t\to\infty}
e^{-2n(n-1)t}
\int_{\M_n(\R)}
{J}_f( b^0(e^{-2t} v),\det(v))\,
\,\nu\!\left(\frac{\widehat v}{\|\widehat v\|}\right)\,dv
$$
$$
=
\frac{\omega_{n-1}^2}{2}
\left(\int_{\M_n(\R)} f(v)\,dv\right)
\left(\int_K \nu\bigl((k^{\intercal}\!\cdot E_{nn})^{\intercal}\bigr)\,dk\right).
$$
\end{proposition}

We deduce this proposition from \Cref{siegelJf,hfunctional}. The first lemma shows that $J_f$ recovers the Lebesgue integral of $f$.

Let $d^\times A:=\det(A)^{-(n-1)}\,dA$ denote the Haar measure on
$\GL_{n-1}^+(\R)$, and let $d\r$ be the measure on $\bp$ characterized by
\begin{equation}\label{meas}
d^\times A=dm\,d\r
\qquad(A=m\cdot\r).
\end{equation}
If $\r=\diag(e^{a_1},\ldots,e^{a_{n-1}})$, then
\begin{equation}\label{eq:Haarmeasureformula}
d\r
=c\prod_{1\le i<j\le n-1}\sinh(a_i-a_j)
\,da_1\cdots da_{n-1}
\end{equation}
for some constant $c>0$.

\begin{lemma}
\label{siegelJf}
For any $f\in C_c(\M_n^+)$, we have
\begin{align*}
\int_{\M_n(\br)}  f \, dv  = \int_\R \int_{\r\in \bp} J_f(\r, \zeta) \det(\r)^{n} {d\r} \, d\zeta.
\end{align*}
\end{lemma}
\begin{proof}
For $v\in \M_n^+$, write
$$
v=
\begin{pmatrix}
A & x_2\\
x_1^{\intercal} & x_{nn}
\end{pmatrix},
\qquad
A\in \GL_{n-1}^+(\R),\quad x_1,x_2\in \R^{n-1},\quad x_{nn}\in \R.
$$
Set
$$
\zeta:=\det(v)
=
\det(A)\bigl(x_{nn}-x_1^{\intercal}A^{-1}x_2\bigr).
$$
For fixed $A,x_1,x_2,\zeta$, this determines $x_{nn}$ uniquely as
$$
x_{nn}(A,\zeta,x_1,x_2)
=
\det(A)^{-1}\zeta+x_1^{\intercal}A^{-1}x_2.
$$
Since $\frac{\partial \zeta}{\partial x_{nn}}=\det(A)$,
the change of variables
$
(A,x_1,x_2,x_{nn})\mapsto (A,x_1,x_2,\zeta)$
has Jacobian $\det(A)$, and therefore
$$
dv=dA\,dx_1\,dx_2\,dx_{nn}
={\det(A)}^{-1}\,dA\,dx_1\,dx_2\,d\zeta.
$$

Hence
\begin{equation*}
\begin{aligned}
\int_{\M_n^+} f(v)\,dv
&=
\int_{\R}
\int_{\GL_{n-1}^+(\R)}
\int_{\R^{n-1}\times \R^{n-1}}
f\big(\begin{smallmatrix}
A & x_2\\
x_1^{\intercal} & x_{nn}(A,\zeta,x_1,x_2)
\end{smallmatrix}\big)
\tfrac{dx_1\,dx_2\,dA}{\det(A)}\,d\zeta.
\end{aligned}
\end{equation*}

Now write
$$
{\det(A)}^{-1} {dA}
=
\det(A)^{n-2}\,d^\times A.
$$
Using  \eqref{meas}, we get
$$
\begin{aligned}
\int_{\M_n^+} f(v)\,dv
&=
\int_{\R}
\int_{\r\in B_0^+}\int_{m\in\mathsf M}
\int_{\R^{n-1}\times \R^{n-1}}
f\big(\begin{smallmatrix}
m\cdot \r & x_2\\
x_1^{\intercal} & x_{nn}(m\cdot \r,\zeta,x_1,x_2)
\end{smallmatrix}\big)  \\
&\hspace{4cm}\cdot
\det(\r)^{n-2}\,dx_1\,dx_2\,dm\,d\r\,d\zeta
\\ & =\int_{\R}
\int_{\r\in B_0^+} 
J_f(\r,\zeta)
\det(\r)^{n} d\r\,d\zeta
\end{aligned}
$$
as claimed.
\end{proof}

The next lemma converts the preceding asymptotic $K$-average into an explicit volume integral. We will later apply it with $h=J_f$.
\begin{lemma}\label{hfunctional}
Let $h\in C_c(\GL_{n-1}^+(\R)\times \R)$ be $\mathsf M$-invariant and let
$\nu\in C(\S^{N-1})$ be nonnegative.
Then
$$
\lim_{t\to\infty}
e^{-2n(n-1)t}
\int_{\M_n(\R)} h( b^0(e^{-2t}v),\det(v))
\nu\!\left(\frac{\widehat v}{\|\widehat v\|}\right)\,dv
$$
$$
=
\frac{\omega_{n-1}^2}{2}
\left(\int_K \nu\bigl((k^{\intercal}\!\cdot E_{nn})^{\intercal}\bigr)\,dk\right)
\left(
\int_{\R}\int_{\bp}
h(\r,\zeta)\,\det (\r)^n d\r \,d\zeta
\right).
$$
\end{lemma}
\begin{proof} Regard $h$ as a function on $B_0^+\times \br$, and
let  $f_h\in C_c(\M_n^+)$ be as in \eqref{fh}.
Apply \Cref{jf} to $f_h$ and $\nu$. Let
\begin{equation*}
\begin{aligned}
\Xi_t:=\bigl\{v\in\M_n(\mathbb R):\ {}&\widehat v\ne0,\\
&h(b^0(e^{-2t}v),\det v)\ne0
\text{ or }\exists k\in K:\ f_h(a_tk\cdot v)\ne0\bigr\}.
\end{aligned}
\end{equation*}
Since $h$ and $f_h$ are compactly supported, there exists $C>1$ such that
for all sufficiently large $t$,
\begin{equation}\label{xit}
\Xi_t\subset
\left\{
v\in \M_n(\R):
C^{-1}e^{2t}\le e^{\kappa_i(v)}\le Ce^{2t}\ (1\le i\le n-1),\
|\det v|\le C
\right\}.
\end{equation}
Applying the Weyl integration formula for singular values, with
$\zeta=\det v$ as the final radial coordinate, and
using \eqref{eq:Haarmeasureformula} and \eqref{xit}, we obtain 
\begin{equation}\label{eq:Xi-volume}
e^{-2n(n-1)t}\operatorname{vol}(\Xi_t)\ll1.
\end{equation}

Fix $\e>0$. By \eqref{eq:Xi-volume}, choose $\e_1>0$ such that
$\e_1e^{-2n(n-1)t}\operatorname{vol}(\Xi_t)<\e$ for all sufficiently large $t$.
Apply \Cref{jf} with $\varepsilon=\varepsilon_1$, and let
$t_0,T_0>0$ be the resulting constants. For $v\in\Xi_t$, the
definition of $\Xi_t$ gives $\widehat v\ne0$, while \eqref{xit} gives
$$
\|v\|\gg e^{\kappa_1(v)}\ge C^{-1}e^{2t}.
$$
After increasing $t_0$ so that $C^{-1}e^{2t_0}\ge T_0$ and \eqref{xit} holds for every $t\ge t_0$, all the hypotheses
of \Cref{jf} are satisfied for $v\in\Xi_t$. Thus, for all $t\ge t_0$ and
$v\in \Xi_t$, we have
\begin{equation*}
\begin{aligned}
&\left|h(b^0(e^{-2t}v),\det v)\,
\nu\!\left(\frac{\widehat v}{\|\widehat v\|}\right)\right.\\
&\qquad\left.-
\frac{\omega_{n-1}^2}{2}e^{2n(n-1)t}
\int_K f_h(a_tk\cdot v)\,
\nu\bigl((k^{\intercal}\!\cdot E_{nn})^{\intercal}\bigr)\,dk
\right|<\e_1.
\end{aligned}
\end{equation*}
By the definition of $\Xi_t$, both integrands above vanish outside $\Xi_t$.  Integrating this over $\Xi_t$, multiplying by $e^{-2n(n-1)t}$, and using
\eqref{eq:Xi-volume}, we obtain
$$
\begin{aligned}
&
\left|
e^{-2n(n-1)t}
\int_{\M_n(\R)}
h(b^0(e^{-2t}v),\det(v))\,
\nu\!\left(\frac{\widehat v}{\|\widehat v\|}\right)\,dv \right. \\
&\qquad\qquad\left.
-
\frac{\omega_{n-1}^2}{2}
\int_{\M_n(\R)}\int_K
f_h(a_tk\cdot v)\,
\nu\bigl((k^{\intercal}\!\cdot E_{nn})^{\intercal}\bigr)\,dk\,dv
\right|
<\e.
\end{aligned}
$$
Applying Fubini's theorem and the change of variables $u=a_tk\cdot v$, we obtain
$$
\begin{aligned}
&\frac{\omega_{n-1}^2}{2}
\int_{\M_n(\R)}\int_K
f_h(a_tk\cdot v)\,
\nu\bigl((k^{\intercal}\!\cdot E_{nn})^{\intercal}\bigr)\,dk\,dv \\
&\qquad=
\frac{\omega_{n-1}^2}{2}
\left(\int_K \nu\bigl((k^{\intercal}\!\cdot E_{nn})^{\intercal}\bigr)\,dk\right)
\left(\int_{\M_n(\R)} f_h(u)\,du\right).
\end{aligned}
$$
 Since $\e>0$ is arbitrary, the claim now follows
 from  \Cref{siegelJf} and the identity $J_{f_h}=h$.
\end{proof}

Applying \Cref{hfunctional} with $h=J_f$ and then using
\Cref{siegelJf} proves \Cref{emm3.9}.
\section{Modified Siegel transforms and counting}
\label{s:maincounting}

We deduce the counting asymptotics from the uniform moment bound in
\Cref{thm:modifieduniformboundedness}. We first extend Shah's
equidistribution theorem to modified Siegel transforms, then compare the
resulting averages with the fiber kernels of \Cref{s:fiberintegral}, and
finally remove the compact singular-value restrictions required in that
comparison.

Until \Cref{subsec:general-bounds}, assume that $\Lambda\in X$ is Diophantine,
and fix
$0<\eta<1$ and $M>1$ such that $\Lambda$ is
$(\eta,M)$-Diophantine. By \Cref{lem:quasinull-isotropic}, its
$(\eta,M)$-quasi-null rational subspaces are precisely its rational
isotropic subspaces.

\subsection{Equidistribution of modified Siegel transforms}

For $\Delta\in X$, define
\begin{equation}\label{ni}
\Delta_{\operatorname{ni}}
:=
\Delta\smallsetminus
\bigl(\Delta_{\operatorname{iso}}\cup\{0\}\bigr),
\end{equation}
where $\Delta_{\operatorname{iso}}$ is the union of all proper
$\Delta$-rational column- and row-isotropic subspaces, as in \eqref{iso}.
For $f\in C_c(\M_n(\R)\smallsetminus\{0\})$, define
$$
\widehat f_{\operatorname{ni}}(\Delta)
:=
\sum_{v\in\Delta_{\operatorname{ni}}}f(v).
$$
Since the $H$-action preserves the families of isotropic subspaces, for $h\in H$, we
have
$$
\widetilde f_{\operatorname{ni}}(h;\Delta)
:=
\widehat f_{\operatorname{ni}}(h\Delta)
=
\sum_{v\in\Delta_{\operatorname{ni}}}f(hv).
$$

\begin{lemma}[Modified Lipschitz principle]
\label{lem:modifiedlipschitz}
Let $\Delta\in X$, let $0<\eta<1$, and let $M\ge1$. Suppose that all
$(\eta,M)$-quasi-null rational subspaces of $\Delta$ are isotropic.
Then the following hold.

\begin{enumerate}[label=\textnormal{(\roman*)}]
\item For every $f\in C_c(\M_n(\R)\smallsetminus\{0\})$,
$$
|\widetilde f_{\operatorname{ni}}(h;\Delta)|
\ll_f
1+\widehat\alpha_{\eta,M}(h;\Delta)
\qquad(h\in H).
$$

\item For $R,L\ge1$, set
\begin{equation}\label{brl}
\mathscr B_R(L)
:=
\left\{x\in\M_n(\R):
|x_{ij}|\le R\ \text{for }(i,j)\ne(n,n),\quad
|x_{nn}|\le RL
\right\}.
\end{equation}
Then, uniformly for $h\in H$ and $L\ge1$,
$$
\#\{v\in\Delta_{\operatorname{ni}}:hv\in\mathscr B_R(L)\}
\ll_R
L\bigl(1+\widehat\alpha_{\eta,M}(h;\Delta)\bigr).
$$
\end{enumerate}
\end{lemma}
\begin{proof}
It suffices to prove \textnormal{(ii)}, since \textnormal{(i)} follows
by choosing $R\ge1$ such that
$\operatorname{supp}(f)\subset\mathscr B_R(1)$.
Fix $h\in H$ and $R,L\ge1$, and set
$\mathscr B:=\mathscr B_R(L)$ and $
\mathscr S:=h\Delta_{\operatorname{ni}}\cap\mathscr B$.
We may assume that $\mathscr S\ne\emptyset$. Put
$$
V:=\operatorname{span}_{\mathbb R}\mathscr S,
\qquad
r:=\dim V,
\qquad
\Delta':=h\Delta\cap V,
\qquad
C:=\mathscr B\cap V.
$$
Then $\Delta'$ is a lattice in $V$, and $C$ is a centrally symmetric
convex body containing $r$ linearly independent vectors of $\Delta'$.
Hence its successive minima satisfy
$0<\lambda_1\le\cdots\le\lambda_r\le1$.
The standard successive-minima counting bound and Minkowski's second
theorem give
$$
\#\mathscr S
\le \#(\Delta'\cap C)
\ll_n \prod_{j=1}^r\lambda_j^{-1}
\ll_n
\frac{\operatorname{vol}_r(C)}
     {\operatorname{covol}_V(\Delta')}.
$$

If $r<n^2$, then $W:=h^{-1}V$ is a proper nonzero
$\Delta$-rational subspace containing a vector of
$\Delta_{\operatorname{ni}}$. Thus
$W\notin\widetilde{\mathscr Q}_{r,\eta,M}(\Delta)$:
otherwise $W$ would be contained in a quasi-null rational subspace,
which is isotropic by hypothesis. Consequently,
$$
\frac{1}{\operatorname{covol}_V(\Delta')}
\le \widehat\alpha_{\eta,M}(h;\Delta).
$$
If $r=n^2$, then $\operatorname{covol}_V(\Delta')=1$.
Let $p_V:\M_n(\mathbb R)\to V$ denote the orthogonal projection.
In either case, since $C\subset p_V\mathscr B$,
$$
\#\mathscr S
\ll_n
\operatorname{vol}_r(p_V\mathscr B)
\bigl(1+\widehat\alpha_{\eta,M}(h;\Delta)\bigr).
$$
The $r$-dimensional volume of the orthogonal projection
$p_V\mathscr B$ is bounded by the sum of the products of
$r$ distinct side lengths of $\mathscr B$.
Since only one side length depends on $L$, we obtain
$
\operatorname{vol}_r(p_V\mathscr)\ll_R L
$
uniformly in $V$. This proves \textnormal{(ii)}.
\end{proof}

We denote by $m_X$ the $\SL_N(\R)$-invariant probability measure on
$X$.

\begin{lemma}[Regularity of the modified Siegel transform]
\label{lem:modified-transform-regularity}
For every $f\in C_c(\M_n(\R)\smallsetminus\{0\})$,
$\widehat f_{\operatorname{ni}}$ is Borel measurable and agrees
$m_X$-almost everywhere with the ordinary Siegel transform $\widehat f$.
In particular, it is continuous outside an $m_X$-null set.
\end{lemma}
\begin{proof}
We first verify measurability. Let
$
\pi:\SL_N(\mathbb R)\to X
$
be the quotient map. Fix a relatively compact open set
$U\subset X$ admitting a smooth local section
$
s:U\to\SL_N(\mathbb R),
$
so that $\pi\circ s=\operatorname{id}_U$. Such sets form a countable
cover of $X$. 
Let $\mathscr W$ be the countable family of proper nonzero
$\mathbb Q$-rational subspaces of $\mathbb R^N$. For
$q\in\mathbb Z^N\smallsetminus\{0\}$, set
$$
E_q
:=
\bigcup_{{W\in\mathscr W, q\in W}}
\{x\in U:s(x)W\text{ is column- or row-isotropic}\}.
$$
For each $W\in\mathscr W$, the set
$\{g\in\SL_N(\mathbb R):gW\text{ is column- or row-isotropic}\}$
is algebraic by \Cref{lem:isotropic_iff_pi1_pi2}. Its preimage under
the continuous map $s$ is therefore Borel. Hence $E_q$ is Borel.

For $x\in U$, we have
$$
\widehat f_{\operatorname{ni}}(x)
=
\sum_{q\in\mathbb Z^N\smallsetminus\{0\}}
f(s(x)q)\mathbf 1_{U\smallsetminus E_q}(x).
$$
At each $x$, this sum has only finitely many nonzero terms because
$s(x)\mathbb Z^N$ is discrete and $\operatorname{supp}(f)$ is compact.
Thus $\widehat f_{\operatorname{ni}}$ is Borel on $U$, and hence on
$X$.

For $q\in\Z^N\smallsetminus\{0\}$, the set
$$
\mathcal Z_q
:=
\{g\in\SL_N(\R):\det(gq)=0\}
$$
is a proper algebraic subvariety of $\SL_N(\R)$. The set
$\bigcup_{q\ne0}\mathcal Z_q$ is a countable union of  proper algebraic subvarieties and is right
$\SL_N(\Z)$-invariant. Its image $\mathcal Z\subset X$ therefore has
$m_X$-measure zero. If $\Delta\notin\mathcal Z$, then $\Delta$ contains
no nonzero singular vector. Since every vector belonging to an isotropic subspace
is singular,
$$
\widehat f_{\operatorname{ni}}(\Delta)=\widehat f(\Delta).
$$

To prove continuity,  fix $\Delta\in U\smallsetminus\mathcal Z$. Choose
an open neighborhood $U_\Delta$ of $\Delta$ such that
$\overline{U_\Delta}\subset U$ is compact.
Only finitely many
$q\in\mathbb Z^N\smallsetminus\{0\}$ satisfy
$$
s(x)q\in\operatorname{supp}(f)
\qquad\text{for some }x\in \overline{U_\Delta}.
$$
For each such $q$, we have
$
\det(s(\Delta)q)\ne0
$
because $\Delta\notin\mathcal Z$. After shrinking $U_\Delta$, continuity
therefore gives
$$
\det(s(x)q)\ne0
\qquad(x\in U_\Delta)
$$
for every relevant $q$. The corresponding vectors cannot belong to an isotropic
subspace. Hence
$$
\widehat f_{\operatorname{ni}}(x)=\widehat f(x)
\qquad(x\in U_\Delta).
$$
Since $\widehat f$ is continuous, so is
$\widehat f_{\operatorname{ni}}$ at $\Delta$.
\end{proof}

For $k\in K$, define
$$
\omega(k):=(k^{\intercal}\!\cdot E_{nn})^{\intercal}
\in\mathbb S^{N-1}.
$$

We use the following form of Shah's theorem
\cite[Corollary~1.2]{shah}.

\begin{theorem}[Shah]
\label{thm:Ratnerequidistributionbounded}
Let $F\in C_c(X)$ and $\nu\in C(\mathbb S^{N-1})$. If
$H\Lambda$ is dense in $X$, then
$$
\lim_{t\to\infty}
\int_KF(a_tk\Lambda)\nu(\omega(k))\,dk
=
\left(\int_XF\,dm_X\right)
\left(\int_K\nu(\omega(k))\,dk\right).
$$
\end{theorem}

\begin{theorem}[Equidistribution for modified Siegel transforms]
\label{thm:Ratnerequidistributionunbounded}
Let $f\in C_c(\M_n^+)$ and let
$\nu\in C(\mathbb S^{N-1})$. Suppose that $\Lambda$ is Diophantine and
not determinant-rational. Then
\begin{align}\label{eq:modified-Siegel-equidistribution}
&\lim_{t\to\infty}
\int_K
\widetilde f_{\operatorname{ni}}(a_tk;\Lambda)
\nu(\omega(k))\,dk=
\left(\int_{\M_n(\R)}f(v)\,dv\right)
\left(\int_K\nu(\omega(k))\,dk\right).
\end{align}
\end{theorem}

\begin{proof}
By linearity, it suffices to assume that $f$ and $\nu$ are nonnegative.
The zero extension of $f$, still denoted by $f$, belongs to
$C_c(\M_n(\R)\smallsetminus\{0\})$. Set
$$
c_\nu:=\int_K\nu(\omega(k))\,dk.
$$
If $c_\nu=0$, then both sides of
\eqref{eq:modified-Siegel-equidistribution} are zero. We may therefore
assume that $c_\nu>0$.
By \Cref{prop:closed-rational-csa}, the hypothesis implies that
$H\Lambda$ is dense in $X$. Define finite measures $\mu_t$ on $X$ by
$$
\int_XF\,d\mu_t
:=
\int_KF(a_tk\Lambda)\nu(\omega(k))\,dk.
$$
By \Cref{thm:Ratnerequidistributionbounded},
$\mu_t\to c_\nu m_X$ vaguely. Since
$\mu_t(X)=c_\nu$ for every $t$, the convergence is weak.
By \Cref{lem:modified-transform-regularity},
$\widehat f_{\operatorname{ni}}$ is Borel measurable and continuous
outside an $m_X$-null set. Moreover,
\Cref{lem:modifiedlipschitz}(i) gives
$$
\widehat f_{\operatorname{ni}}(a_tk\Lambda)
\ll_f
1+\widehat\alpha_{\eta,M}(a_tk;\Lambda).
$$
Choose $\theta>0$ as in
\Cref{thm:modifieduniformboundedness}. Since $\nu$ is bounded,
$$
\begin{aligned}
\int_X\widehat f_{\operatorname{ni}}^{\,1+\theta}\,d\mu_t
&=
\int_K
\widehat f_{\operatorname{ni}}(a_tk\Lambda)^{1+\theta}
\nu(\omega(k))\,dk \ll_{f,\nu}
1+\int_K
\widehat\alpha_{\eta,M}(a_tk;\Lambda)^{1+\theta}\,dk.
\end{aligned}
$$
Consequently,
\begin{equation}\label{eq:modified-Siegel-uniform-moment}
C:=
\sup_{t\ge0}
\int_X\widehat f_{\operatorname{ni}}^{\,1+\theta}\,d\mu_t
<\infty.
\end{equation}
Fix $s>0$. Since both
$
\min\{\widehat f_{\operatorname{ni}},s\}$ and $
\min\{\widehat f_{\operatorname{ni}}^{\,1+\theta},s\}$
are bounded and have $m_X$-null discontinuity sets, the Portmanteau
theorem gives
$$
\int_X\min\{\widehat f_{\operatorname{ni}},s\}\,d\mu_t
\to 
c_\nu\int_X\min\{\widehat f_{\operatorname{ni}},s\}\,dm_X
$$
and
$$
c_\nu\int_X
\min\{\widehat f_{\operatorname{ni}}^{\,1+\theta},s\}\,dm_X
=
\lim_{t\to\infty}\int_X
\min\{\widehat f_{\operatorname{ni}}^{\,1+\theta},s\}\,d\mu_t
\le C.
$$
Letting $s\to\infty$ and applying the monotone convergence theorem, we
obtain
$$
c_\nu
\int_X\widehat f_{\operatorname{ni}}^{\,1+\theta}\,dm_X
\le C.
$$
Since
$
0\le
\widehat f_{\operatorname{ni}}
-\min\{\widehat f_{\operatorname{ni}},s\}
\le
s^{-\theta}\widehat f_{\operatorname{ni}}^{\,1+\theta},
$
we have
$$
\int_X
\left(
\widehat f_{\operatorname{ni}}
-\min\{\widehat f_{\operatorname{ni}},s\}
\right)d\mu_t
\le Cs^{-\theta}\;\text{and}\;
c_\nu\int_X
\left(
\widehat f_{\operatorname{ni}}
-\min\{\widehat f_{\operatorname{ni}},s\}
\right)dm_X
\le Cs^{-\theta}.
$$
It follows that
$$
\limsup_{t\to\infty}
\left|
\int_X\widehat f_{\operatorname{ni}}\,d\mu_t
-
c_\nu\int_X\widehat f_{\operatorname{ni}}\,dm_X
\right|
\le2Cs^{-\theta}.
$$
Letting $s\to\infty$ yields
$$
\lim_{t\to\infty}
\int_X\widehat f_{\operatorname{ni}}\,d\mu_t
=
c_\nu\int_X\widehat f_{\operatorname{ni}}\,dm_X.
$$
Finally, \Cref{lem:modified-transform-regularity} and Siegel's formula
give
$$
\int_X\widehat f_{\operatorname{ni}}\,dm_X
=
\int_X\widehat f\,dm_X
=
\int_{\M_n(\R)}f(v)\,dv,
$$
which proves \eqref{eq:modified-Siegel-equidistribution}.
\end{proof}

\subsection{Counting singular-value shells}
\label{subsec:singular-value-shells}
We discretize the first $n-1$ logarithmic singular values and separate
the non-deep and deep ranges. 
For $v\in \M_n(\br)$, let
$$
\kappa_1(v)\ge\cdots\ge\kappa_n(v)
$$
denote the logarithms of its singular values, with the convention
$\log 0=-\infty$.
For an integer $\tau\ge1$ and
$$
\mathbf m=(m_1,\ldots,m_{n-1})\in\mathbb Z_{\ge0}^{n-1},
\qquad
0=m_1\le\cdots\le m_{n-1},
$$
set
$$
|\mathbf m|:=\sum_{i=2}^{n-1}m_i
$$
and, for $D>0$, define
\begin{equation}\label{eq:integer-shell-definition}
\begin{aligned}
\mathcal S_{\tau,\mathbf m}(D)
:=\bigl\{v\in\M_n(\R):\;&|\det v|\le D,\\
&\tau-m_i-1<\kappa_i(v)\le\tau-m_i
\quad(1\le i\le n-1)\bigr\}.
\end{aligned}
\end{equation}
Also set
\begin{equation}\label{eq:deep-shell-set}
\mathcal D_\tau
:=
\left\{0\ne v\in\M_n(\R):
\tau-1<\kappa_1(v)\le\tau,\quad
\kappa_{n-1}(v)\le-\frac{(n-2)}{2}\tau
\right\}.
\end{equation}

These sets give the required decomposition. If $\kappa_1(v)>0$ and
$\tau=\lceil\kappa_1(v)\rceil$, then either $v\in\mathcal D_\tau$ or,
on setting
$$
m_i:=\tau-\lceil\kappa_i(v)\rceil,
$$
one has
\begin{equation}\label{eq:non-deep-deficits}
0=m_1\le\cdots\le m_{n-1}<\frac n2\tau,
\qquad
v\in\mathcal S_{\tau,\mathbf m}(D)
\quad\text{if }|\det v|\le D.
\end{equation}

For $q\ge0$, $D>0$, and $R>0$, set
$$
\mathcal B_{q,D}(R)
:=
\left\{x\in\M_n(\mathbb R):
|x_{ij}|\le R \text{ for all }(i,j)\ne(n,n),
|x_{nn}|\le R(1+D)e^q
\right\}.
$$

\begin{lemma}[Geometric estimate for non-deep shells]
\label{lem:one-long-direction-box}
For every $D>0$, there exist $R=R_{n,D}>1$ and
$c=c_{n,D}>0$ such that, for every integer $\tau\ge1$ and 
$\mathbf m\in\mathbb Z_{\ge0}^{n-1}$ satisfying
$
0=m_1\le\cdots\le m_{n-1}\le n\tau/2,
$
every
$v\in\mathcal S_{\tau,\mathbf m}(D)$ satisfies
\begin{equation}\label{eq:one-long-direction-lower-bound}
\operatorname{vol}_K
\left\{k\in K:
a_{\tau/2}k\cdot v\in\mathcal B_{|\mathbf m|,D}(R)
\right\}
\ge
c e^{-n(n-1)\tau+2|\mathbf m|}.
\end{equation}
\end{lemma}
\begin{proof}
By the singular-value decomposition, it suffices to consider
$$
v=
\operatorname{diag}
\left(e^{\lambda_1},\ldots,e^{\lambda_{n-1}},
\epsilon e^{\lambda_n}\right),
$$
where $\epsilon\in\{-1,0,1\}$,
$\tau-m_i-1<\lambda_i\le\tau-m_i$ for $1\le i\le n-1$,
and the last entry is understood to be zero when $\det v=0$.
Set
$$
q_i:=\frac n2\tau-m_i\ge0,
\qquad
\rho_i:=\varepsilon_0e^{-q_i},
$$
where $\varepsilon_0>0$ is sufficiently small in terms of $n$. For
$x=(x_1,\ldots,x_{n-1})$, set
$$
\mathsf k(x):=\sum_{i=1}^{n-1}x_i(E_{in}-E_{ni})
\in\mathfrak{so}(n).
$$
Fix a sufficiently small neighborhood $\mathscr U$ of the
identity in $\SO(n-1)$. The map
$$
(h,x)\mapsto
\begin{pmatrix}h&0\\0&1\end{pmatrix}\exp\mathsf k(x)
$$
is a smooth coordinate chart near the identity in $\SO(n)$, with
Jacobian bounded above and below. Let $\mathcal K_{\tau,\mathbf m}$ be
the image of
$$
h\in\mathscr U,
\qquad
|x_i|\le\rho_i\quad(1\le i\le n-1).
$$
Then
\begin{equation}\label{eq:moderate-K-volume}
\operatorname{vol}_{\SO(n)}(\mathcal K_{\tau,\mathbf m})
\gg_n
\prod_{i=1}^{n-1}\rho_i
=
\varepsilon_0^{n-1}e^{-n(n-1)\tau/2+|\mathbf m|}.
\end{equation}

We claim that, for all $k_1,k_2\in\mathcal K_{\tau,\mathbf m}$,
$$
a_{\tau/2}(k_1,k_2)\cdot v
\in\mathcal B_{|\mathbf m|,D}(R)
$$
for a sufficiently large $R=R_{n,D}$. Every entry of the upper-left block is
$O_n(1)$ because $e^{-\tau+\lambda_i}=O(1)$. The contribution of the
$i$th singular direction to a mixed entry is
$$
O_n\left(e^{(n-2)\tau/2+\lambda_i}\rho_i\right)
=O_n(1).
$$
The final singular direction satisfies the same bound because
$\lambda_n\le\lambda_i$.

For the $(n,n)$-entry, the final singular direction contributes zero
when $\det v=0$ and otherwise contributes
$$
e^{(n-1)\tau+\lambda_n}
=
|\det v|
\exp\left((n-1)\tau-\sum_{i=1}^{n-1}\lambda_i\right)
\ll_{n,D}e^{|\mathbf m|}.
$$
For $i<n$, the corresponding contribution satisfies
$$
e^{(n-1)\tau+\lambda_i}\rho_i^2
\ll_n
e^{m_i}
\le e^{|\mathbf m|}.
$$
This proves the claim. Applying \eqref{eq:moderate-K-volume} to each factor of $K$ proves
\eqref{eq:one-long-direction-lower-bound}.
\end{proof}

\begin{lemma}[Deep shells]
\label{lem:deep-shell-weyl-box}
There exist $R_n>1$ and $c_n>0$ such that, for every $\tau\ge1$ and
$v\in\M_n(\R)$ satisfying
\begin{equation}\label{eq:deep-shell-hypotheses}
\tau-1<\kappa_1(v)\le \tau,
\qquad
\kappa_{n-1}(v)\le-\frac{(n-2)}{2}\tau,
\end{equation}
one has
\begin{equation}\label{eq:deep-shell-lower-bound}
\operatorname{vol}_K
\left\{k\in K:
\|a_{\tau/2}k\cdot v\|\le R_n
\right\}
\ge
c_n e^{-\bigl(n(n-1)-\frac n2\bigr)\tau}.
\end{equation}
\end{lemma}
\begin{proof}
By the singular-value decomposition, we may assume that
$$
v=
\operatorname{diag}
(e^{\lambda_1},\ldots,e^{\lambda_{n-1}},
\epsilon e^{\lambda_n}),
$$
where $\lambda_i=\kappa_i(v)$ and we use the convention
$e^{-\infty}=0$. We set
$\epsilon=\operatorname{sgn}(\det v)$ if $\det v\ne0$ and
$\epsilon=0$ otherwise.
Let $\varepsilon_0>0$ be a sufficiently small constant.
For $1\le i\le n-1$, write
$X_i=E_{in}-E_{ni}$. Let $w_0\in\SO(n)$ be the identity on
$\langle e_1,\ldots,e_{n-2}\rangle$ and satisfy
$$
w_0e_{n-1}=e_n,
\qquad
w_0e_n=-e_{n-1}.
$$
Near the identity and $w_0$, respectively, we use the coordinate charts
$$
\begin{aligned}
k_1&=
\begin{pmatrix}h_1&0\\0&1\end{pmatrix}
\exp\left(\sum_{i=1}^{n-2}x_iX_i\right)
\exp(\alpha X_{n-1}),\\
k_2&=
\begin{pmatrix}h_2&0\\0&1\end{pmatrix}
\exp\left(\sum_{i=1}^{n-2}y_iX_i\right)
\exp(\beta X_{n-1})w_0,
\end{aligned}
$$
where $h_1,h_2$ range over fixed sufficiently small neighborhoods of the identity in
$\SO(n-1)$. These are full-dimensional smooth charts, and their
Jacobians are bounded above and below on fixed coordinate boxes.

Restrict the coordinates by
\begin{equation}\label{eq:smallboxrestriction}
|x_i|,|y_i|\le\varepsilon_0e^{-n\tau/2}
\quad(1\le i\le n-2),
\qquad
|\alpha|,|\beta|\le\varepsilon_0.
\end{equation}
Define
$$
F_\epsilon(\alpha,\beta)
:=
-e^{\lambda_{n-1}}\sin\alpha\cos\beta
+\epsilon e^{\lambda_n}\cos\alpha\sin\beta.
$$
Direct multiplication shows that the contribution of the last two
singular directions to the $(n,n)$-entry is
$\cos\|x\|\cos\|y\|F_\epsilon(\alpha,\beta)$; the remaining
contribution is
$$
O\left(\sum_{i=1}^{n-2}e^{\lambda_i}|x_i||y_i|\right).
$$
We impose the additional condition
\begin{equation}\label{eq:deep-strip-condition}
|F_\epsilon(\alpha,\beta)|
\le
\varepsilon_0e^{-(n-1)\tau}.
\end{equation}

If $e^{\lambda_{n-1}}=0$, then also $e^{\lambda_n}=0$, so this
condition is automatic. Otherwise, for every $\beta$ in a fixed
smaller interval, the equation $F_\epsilon(\alpha,\beta)=0$ has a
solution satisfying
$$
\tan\alpha
=
\epsilon e^{\lambda_n-\lambda_{n-1}}\tan\beta.
$$
Since $e^{\lambda_n}\le e^{\lambda_{n-1}}$, this solution lies in the
chosen $\alpha$-interval. After decreasing $\varepsilon_0$ if necessary, we also
have
$$
|\partial_\alpha F_\epsilon(\alpha,\beta)|
\asymp e^{\lambda_{n-1}}
$$
throughout the coordinate square. Hence the admissible
$(\alpha,\beta)$-set has area
$$
\gg
\min\left\{1,e^{-(n-1)\tau-\lambda_{n-1}}\right\}
\gg e^{-n\tau/2},
$$
where the final inequality follows from
\eqref{eq:deep-shell-hypotheses}. It follows that the set of pairs
$(k_1,k_2)$ satisfying
\eqref{eq:smallboxrestriction} and
\eqref{eq:deep-strip-condition} has Haar measure
\begin{equation}\label{eq:deep-chart-volume}
\gg
e^{-n(n-2)\tau}e^{-n\tau/2}
=e^{-\bigl(n(n-1)-\frac n2\bigr)\tau}.
\end{equation}

It remains to bound the image of $v$. An entry in the upper-left
$(n-1)\times(n-1)$ block of
$b_{\tau/2}k_1vk_2^{\intercal}b_{\tau/2}$ is bounded by
$$
O(e^{-\tau}e^{\lambda_1})=O(1).
$$
A mixed entry has weight $e^{(n-2)\tau/2}$. The contribution of a
singular direction $i\le n-2$ contains one of the factors $x_i,y_i$
and is therefore
$$
O\left(
e^{(n-2)\tau/2}e^{\lambda_i}e^{-n\tau/2}
\right)
=
O(e^{\lambda_i-\tau})
=
O(1),
$$
since $\lambda_i\le\lambda_1\le\tau$. The contribution of either of
the last two singular directions is also $O(1)$, since
$e^{(n-2)\tau/2}e^{\lambda_{n-1}}\le1$.

Finally, the $(n,n)$-entry has weight $e^{(n-1)\tau}$. Its
contribution from the last two singular directions is
$$
\cos\|x\|\cos\|y\|F_\epsilon(\alpha,\beta),
$$
and hence is $O(1)$ after multiplication by this weight, by
\eqref{eq:deep-strip-condition}. For $i\le n-2$, both a left and a
right small coordinate occur, so the corresponding contribution is
$$
O\left(
e^{(n-1)\tau}e^{\lambda_i}e^{-n\tau}
\right)
=
O(e^{\lambda_i-\tau})
=
O(1).
$$
Thus all entries are uniformly bounded on the chosen coordinate set.
Choosing $R_n$ sufficiently large gives
$\|a_{\tau/2}(k_1,k_2)\cdot v\|\le R_n$. Together with
\eqref{eq:deep-chart-volume}, this proves
\eqref{eq:deep-shell-lower-bound}.
\end{proof}

\begin{lemma}[Averaging principle]
\label{lem:shell-averaging}
Let $E,\mathscr B\subset\M_n(\R)$ be Borel sets, and suppose that for some $t\ge0$ and $p>0$,
$$
\int_K\mathbf 1_{\mathscr B}(a_tk\cdot v)\,dk\ge p
\qquad(v\in E).
$$
 Then for every discrete set $S\subset\M_n(\R)$,
\begin{align*}
\operatorname{vol}(E)
&\le p^{-1}\operatorname{vol}(\mathscr B),\\
\#(S\cap E)
&\le p^{-1}\int_K
\#\{v\in S:a_tk\cdot v\in\mathscr B\}\,dk.
\end{align*}
\end{lemma}
\begin{proof}
By Tonelli's theorem,
$$
\begin{aligned}
p\,\operatorname{vol}(E)
&\le
\int_E\int_K
\mathbf 1_{\mathscr B}(a_tk\cdot v)\,dk\,dv \le
\int_K\operatorname{vol}\bigl((a_tk)^{-1}\mathscr B\bigr)\,dk
=
\operatorname{vol}(\mathscr B).
\end{aligned}
$$
The same calculation with counting measure on $S\cap E$, followed by
enlarging $S\cap E$ to $S$ in the resulting counting function, proves
the second assertion.
\end{proof}

\begin{proposition}[Local shell estimates]
\label{lem:count-one-long-direction-shell}
Let $D>0$, let $\tau\ge1$ be an integer, and let
$
\mathbf m\in\mathbb Z_{\ge0}^{n-1}
$
satisfy
$
0=m_1\le\cdots\le m_{n-1}\le n\tau/2.
$
Then
\begin{equation}\label{eq:moderate-shell-count}
\begin{aligned}
\#\bigl(\Lambda_{\operatorname{ni}}
\cap\mathcal S_{\tau,\mathbf m}(D)\bigr)
&\ll_{\Lambda,D}e^{n(n-1)\tau-|\mathbf m|}\text{ and }\operatorname{vol}\bigl(\mathcal S_{\tau,\mathbf m}(D)\bigr)
&\ll_D e^{n(n-1)\tau-|\mathbf m|}.
\end{aligned}
\end{equation}
Moreover, for every $\tau\ge1$,
\begin{equation}\label{eq:deep-shell-count}
\begin{aligned}
\#\bigl(\Lambda_{\operatorname{ni}}\cap\mathcal D_\tau\bigr)
&\ll_\Lambda e^{\bigl(n(n-1)-\frac n2\bigr)\tau}\;\; \text{ and }\;\; \operatorname{vol}(\mathcal D_\tau)
&\ll e^{\bigl(n(n-1)-\frac n2\bigr)\tau}.
\end{aligned}
\end{equation}
\end{proposition}

\begin{proof}
Let $R=R_{n,D}$ be as in \Cref{lem:one-long-direction-box}. Since
\begin{equation}\label{eq:long-box-containment}
\mathcal B_{|\mathbf m|,D}(R)
\subset\mathscr B_{R(1+D)}(e^{|\mathbf m|}),
\end{equation}
\Cref{lem:modifiedlipschitz}(ii) gives
$$
\#\left\{v\in\Lambda_{\operatorname{ni}}:
a_{\tau/2}k\cdot v\in\mathcal B_{|\mathbf m|,D}(R)
\right\}
\ll_D
e^{|\mathbf m|}
\bigl(1+\widehat\alpha_{\eta,M}(a_{\tau/2}k;\Lambda)\bigr).
$$
Apply \Cref{lem:shell-averaging} with
$$
E=\mathcal S_{\tau,\mathbf m}(D),
\quad
\mathscr B=\mathcal B_{|\mathbf m|,D}(R),
\quad
p\asymp e^{-n(n-1)\tau+2|\mathbf m|}.
$$
The counting assertion of that lemma, together with
\Cref{thm:modifieduniformboundedness}, gives
\eqref{eq:moderate-shell-count}. Since
$\operatorname{vol}(\mathscr B)\ll_D e^{|\mathbf m|}$, its volume
assertion gives \eqref{eq:moderate-shell-count}.

Choose $R\ge1$, depending only on $n$ and the fixed norm, such that
$$
\{u\in\M_n(\mathbb R):\|u\|\le R_n\}
\subset\mathscr B_R(1),
$$
where $R_n$ is as in \Cref{lem:deep-shell-weyl-box}.
Apply \Cref{lem:shell-averaging} with
$E=\mathcal D_\tau$, $\mathscr B=\mathscr B_R(1)$, and
$p\asymp e^{-\bigl(n(n-1)-\frac n2\bigr)\tau}$. The counting assertion, followed by
\Cref{lem:modifiedlipschitz}(ii) and
\Cref{thm:modifieduniformboundedness}, gives
the first estimate in \eqref{eq:deep-shell-count}; the volume assertion gives
the second estimate in \eqref{eq:deep-shell-count}.
\end{proof}

The shell estimates imply both the compact-window bound needed in the kernel
comparison and the cusp-tail bound needed for norm balls.

\begin{corollary}[Counting in a compact singular-value window]
\label{lem:weightedcontributioncount}
Let $\mathscr C\subset B_0^+$ be compact and let $D>0$.  Then uniformly for all $T\ge2$,
$$
\#\left\{v\in\Lambda_{\operatorname{ni}}:
\widehat v\ne0,\quad
T^{-1}b^0(v)\in\mathscr C,
\ |\det v|\le D
\right\}
\ll_{\Lambda,\mathscr C,D}
T^{n(n-1)}.
$$
\end{corollary}
\begin{proof}
Let $v$ be a matrix in the set being counted. Then
$
\kappa_i(v)=\log T+O_{\mathscr C}(1)$
for all $1\le i\le n-1$.
Consequently, all such matrices lie in
$O_{\mathscr C}(1)$ shells
$\mathcal S_{\tau,\mathbf m}(D)$ satisfying
$
\tau=\log T+O_{\mathscr C}(1)$ and $|\mathbf m|=O_{\mathscr C}(1)$.
For sufficiently large $T$, these shells are non-deep, and
\eqref{eq:moderate-shell-count} bounds each of them by
$$
e^{n(n-1)\tau-|\mathbf m|}
\ll_{\mathscr C}T^{n(n-1)}.
$$
Summing over the $O_{\mathscr C}(1)$ possible shells proves the result
for large $T$. Enlarging the implied constant proves it for all $T\ge 2$.
\end{proof}

\begin{lemma}[Shell summation]\label{lem:deficit-sum}
For every integer $\ell\ge1$ and $A\ge0$,
\begin{equation}\label{eq:deficit-sum}
\sum_{\substack{(m_1,\ldots,m_\ell)\in\mathbb Z_{\ge0}^{\ell}\\
0=m_1\le\cdots\le m_\ell,\ m_\ell\ge A}}
\exp\left(-\sum_{i=2}^{\ell}m_i\right)
\ll_\ell e^{-A/2}.
\end{equation}
Moreover, for $s>1/2$ and $C_0\ge0$,
\begin{equation}\label{eq:outer-shell-sum}
\sum_{\nu\ge0}e^{-s\nu}
e^{-\frac12(R-\nu-C_0)_+}
\ll_{s,C_0}e^{-R/2}
\qquad(R\ge0)
\end{equation} where $x_+=\max\{x, 0\}$.
\end{lemma}

\begin{proof}
For $\ell=1$, \eqref{eq:deficit-sum} is immediate.  If $\ell\ge2$,
then, for fixed $m_\ell=m$, there are at most $(m+1)^{\ell-2}$ choices
of $(m_2,\ldots,m_{\ell-1})$, and
$\sum_{i=2}^{\ell}m_i\ge m$.  Hence the left-hand side of
\eqref{eq:deficit-sum} is at most
$$
\sum_{m\ge A}(m+1)^{\ell-2}e^{-m}\ll_\ell e^{-A/2}.
$$
To prove \eqref{eq:outer-shell-sum}, split the sum into the ranges
$\nu\le R-C_0$ and $\nu>R-C_0$.  In the first range, the summand equals
$e^{-(R-C_0)/2}e^{-(s-1/2)\nu}$; in the second, it is at most
$e^{-s\nu}$.  Both resulting geometric series are $O_{s,C_0}(e^{-R/2})$.
\end{proof}

For $D,R,T>0$, set
$$
\mathcal E_{D,R}(T):=
\left\{0\ne v\in\M_n(\mathbb R):
\|v\|<T,\ |\det v|\le D,\
\kappa_{n-1}(v)<\log T-R\right\}.
$$

\begin{proposition}[Uniform singular-value off-diagonal tails]
\label{prop:singular-value-cusp-tail}
Fix $D>0$ and a $K$-invariant norm on $\M_n(\mathbb R)$.  For
$R\ge1$ and $T\ge2$,
\begin{align}
\#\bigl(\Lambda_{\operatorname{ni}}\cap\mathcal E_{D,R}(T)\bigr)
&\ll_{\Lambda,D}e^{-R/2}T^{n(n-1)}+T^{n(n-1)-\frac n2},
\label{eq:singular-value-cusp-tail}\\
\operatorname{vol}\bigl(\mathcal E_{D,R}(T)\bigr)
&\ll_D e^{-R/2}T^{n(n-1)}+T^{n(n-1)-\frac n2}.
\label{eq:volume-singular-value-cusp-tail}
\end{align}
\end{proposition}

\begin{proof}
We prove the two estimates simultaneously, using  counting
measure on $\Lambda_{\operatorname{ni}}$ and Lebesgue measure, respectively. Denote
either measure by $\mu$.  By equivalence of norms, there is $c_0\ge0$
such that
$$
\kappa_1(v)\le\log\|v\|+c_0\qquad(v\ne0).
$$
The set on which $\kappa_1(v)\le0$ and $|\det v|\le D$ is bounded, so
its $\mu$-measure is $O_{\Lambda,D}(1)$ in the lattice case and $O_D(1)$
in the Lebesgue case.  For every remaining $v$, set
$\tau=\lceil\kappa_1(v)\rceil$.  Then
$$
1\le\tau\le L_T:=\lfloor\log T+c_0+1\rfloor.
$$

By \eqref{eq:deep-shell-count}, the
total contribution of the deep shells is
$$
\ll\sum_{1\le\tau\le L_T}e^{\bigl(n(n-1)-\frac n2\bigr)\tau}
\ll T^{n(n-1)-\frac n2}.
$$
Every remaining $v$ has a unique deficit vector
$\mathbf m=(m_1,\ldots,m_{n-1})$ satisfying
\eqref{eq:non-deep-deficits}.  By
\eqref{eq:moderate-shell-count}, its shell has $\mu$-measure
$$
\ll e^{n(n-1)\tau-|\mathbf m|}.
$$
Set $\nu=L_T-\tau$. The condition
$\kappa_{n-1}(v)<\log T-R$ implies
$$
m_{n-1}
\ge\tau-\log T+R-1
\ge R-\nu-1.
$$
Thus \eqref{eq:deficit-sum} bounds the sum over the admissible deficit
vectors by $O(e^{-\frac12(R-\nu-1)_+})$. Since
$e^{n(n-1)\tau}\ll T^{n(n-1)}e^{-n(n-1)\nu}$, \eqref{eq:outer-shell-sum} gives
$$
\sum_{\nu\ge0}T^{n(n-1)}e^{-n(n-1)\nu}
e^{-\frac12(R-\nu-1)_+}
\ll T^{n(n-1)}e^{-R/2}.
$$
Adding the bounded and deep contributions to this estimate proves both assertions.
\end{proof}

\subsection{Comparison with the fiber kernel}
\begin{proposition}
\label{emm3.7}
Let $f\in C_c(\M_n^+)$, and let
$\nu\in C(\mathbb S^{N-1})$.  Then as $T\to\infty$,
\begin{align}\label{eq:kernel-lattice-comparison}
&T^{-n(n-1)}
\sum_{\substack{v\in\Lambda_{\operatorname{ni}},\widehat v\ne0}}
J_f(T^{-1}b^0(v),\det v)
\nu\!\left(\frac{\widehat v}{\|\widehat v\|}\right)\notag\\
&\quad-
\frac{\omega_{n-1}^2}{2}
\int_K
\widetilde f_{\operatorname{ni}}
   (a_{\frac12\log T}k;\Lambda)
\nu(\omega(k))\,dk
\to0.
\end{align}

\end{proposition}

\begin{proof}
Set $t=\frac12\log T$. Since $|J_f|\le J_{|f|}$, the support
implication \eqref{eq:jf-uniform-support}, applied to $|f|$, yields a compact set
$\mathscr C_f\subset B_0^+$ and $D_f>0$ such that every vector
contributing to either side of
\eqref{eq:kernel-lattice-comparison} belongs to
$$
\Xi_T
:=
\left\{v\in\Lambda_{\operatorname{ni}}:
\widehat v\ne0,\quad
T^{-1}b^0(v)\in\mathscr C_f,\quad
|\det v|\le D_f
\right\}.
$$
Therefore \Cref{lem:weightedcontributioncount} gives
\begin{equation}\label{eq:XiT-card}
\#\Xi_T\ll_{\Lambda,f} T^{n(n-1)}.
\end{equation}
Compactness of $\mathscr C_f$ gives $\|v\|\gg_f T$ on $\Xi_T$.
Hence \Cref{jf}
applies uniformly on $\Xi_T$ and gives
\begin{align*}
&\frac{\omega_{n-1}^2}{2}T^{n(n-1)}
\int_K f(a_tk\cdot v)\nu(\omega(k))\,dk=
J_f(T^{-1}b^0(v),\det v)
\nu\!\left(\frac{\widehat v}{\|\widehat v\|}\right)+r_T(v),
\end{align*}
with $\sup_{v\in\Xi_T}|r_T(v)|\to0$. Both sides of the required comparison are supported on $\Xi_T$.
After summing the last identity and multiplying by $T^{-n(n-1)}$, the
absolute value of the error is at most
$$
T^{-n(n-1)}\#\Xi_T\sup_{v\in\Xi_T}|r_T(v)|=o(1)
$$
by \eqref{eq:XiT-card}. This proves
\eqref{eq:kernel-lattice-comparison}.
\end{proof}

\begin{proposition}[Counting in the $(b^0,\det)$-parameter space]
\label{prop:countinglimit-quotient}
Suppose that $\Lambda$ is Diophantine and not determinant-rational.  Let
$h\in C_c(B_0^+\times\R)$ and
$\nu\in C(\mathbb S^{N-1})$.  Then
\begin{align}\label{eq:parameter-space-counting-limit}
&\lim_{T\to\infty}
T^{-n(n-1)}
\sum_{\substack{v\in\Lambda_{\operatorname{ni}}, \widehat v\ne0}}
h(T^{-1}b^0(v),\det v)
\nu\!\left(\frac{\widehat v}{\|\widehat v\|}\right)\notag\\
&\quad=
\frac{\omega_{n-1}^2}{2}
\left(\int_K\nu(\omega(k))\,dk\right)
\left(
\int_{\R}\int_{B_0^+}
h(\r,\zeta)\det(\r)^n\,d\r\,d\zeta
\right).
\end{align}
\end{proposition}
\begin{proof}
Choose $f\in C_c(\M_n^+)$ with $J_f=h$ by
\Cref{lem:Jf-surjective}. By \Cref{emm3.7}, the normalized sum in
\eqref{eq:parameter-space-counting-limit} differs by $o(1)$ from
$$
\frac{\omega_{n-1}^2}{2}
\int_K
\widetilde f_{\operatorname{ni}}
\left(a_{\frac12\log T}k;\Lambda\right)
\nu(\omega(k))\,dk.
$$
By \Cref{thm:Ratnerequidistributionunbounded}, this converges to
$$
\frac{\omega_{n-1}^2}{2}
\left(\int_{\M_n(\mathbb R)}f(v)\,dv\right)
\left(\int_K\nu(\omega(k))\,dk\right).
$$
Finally, \Cref{siegelJf} and $J_f=h$ give
$$
\int_{\M_n(\mathbb R)}f(v)\,dv
=
\int_{\mathbb R}\int_{B_0^+}
h(\mathsf r,\zeta)\det(\mathsf r)^n\,d\mathsf r\,d\zeta.
$$
This proves \eqref{eq:parameter-space-counting-limit}.
\end{proof}

For fixed $\nu\in C(\mathbb S^{N-1})$, the conclusion of
\eqref{eq:parameter-space-counting-limit} extends to bounded, compactly
supported Borel functions $h$ whose discontinuity set has
$d\mathsf r\,d\zeta$-measure zero. Indeed, for any continuous
approximation $h_j$ of $h$, the corresponding weighted approximation
error is at most
$$
\|\nu\|_\infty T^{-n(n-1)}
\sum_{{v\in\Lambda_{\operatorname{ni}}, \widehat v\ne0}}
|h-h_j|(T^{-1}b^0(v),\det v).
$$
The case $\nu\equiv1$ and the Portmanteau theorem control the last
quantity by the $L^1$-error of $h_j$. Letting $j\to\infty$ proves the
extension.

\subsection{Passage from compact windows to norm balls}
Fix a $K$-invariant norm $\|\cdot\|$ on $\M_n(\R)$.  For
$\r\in B_0^+$ and $\tau\in\R$, put
\begin{equation}
\Phi(\r,\tau)
:=
\left\|
\operatorname{diag}\left(\r,\frac{\tau}{\det \r}\right)
\right\|.
\end{equation}
Because the norm is $K$-invariant, for every matrix $v$ of rank at
least $n-1$,
\begin{equation}\label{eq:K-inv-reduction}
\|v\|=\Phi(b^0(v),\det v).
\end{equation}
Consequently,
\begin{equation}\label{eq:scaled-K-inv-reduction}
\|v\|<T
\quad\Longleftrightarrow\quad
\Phi(T^{-1}b^0(v),T^{-n}\det v)<1.
\end{equation}

\begin{theorem}
\label{thm:13.11-corrected}
Suppose that $\Lambda\in X$ is Diophantine and
non-determinant-rational. Then, for every $a<b$,
\begin{equation}\label{eq:norm-ball-counting}
\lim_{T\to\infty}T^{-n(n-1)}
\#\{v\in\Lambda_{\operatorname{ni}}:
\|v\|<T,\ a<\det v<b\}
=
C_{\|\cdot\|}(b-a),
\end{equation}
where
\begin{equation}\label{mconstant}
C_{\|\cdot\|}
:=
\frac{\omega_{n-1}^2}{2}
\int_{\{\mathsf r\in B_0^+:
\|\operatorname{diag}(\mathsf r,0)\|<1\}}
\det(\mathsf r)^n\,d\mathsf r.
\end{equation}
\end{theorem}

\begin{lemma}[Volume asymptotic]
\label{lem:volume-main-term}
For every $a<b$,
\begin{equation}\label{eq:volume-main-term}
\operatorname{vol}
\{v\in\M_n(\mathbb R):
\|v\|<T,\ a<\det v<b\}
\sim
C_{\|\cdot\|}(b-a)T^{n(n-1)}.
\end{equation}
\end{lemma}

\begin{proof}[Proof of
\Cref{thm:13.11-corrected,lem:volume-main-term}]
For $R>1$, define
$$
D_R
:=
\left\{\mathsf r\in B_0^+:
\mathsf r_{n-1}\ge e^{-R},\quad
\Phi(\mathsf r,0)<1
\right\}.
$$
Set also
$$
A_R
:=
\frac{\omega_{n-1}^2}{2}(b-a)\int_{D_R}\det(\mathsf r)^n\,d\mathsf r.
$$
Let
$$
N_R(T)
:=
\#\left\{v\in\Lambda_{\operatorname{ni}}:
\|v\|<T,\ a<\det v<b,\ 
\kappa_{n-1}(v)\ge\log T-R
\right\};
$$
$$
V_R(T)
:=
\operatorname{vol}\left\{v\in\M_n(\mathbb R):
\|v\|<T,\ a<\det v<b,\ 
\kappa_{n-1}(v)\ge\log T-R
\right\}.
$$

We first fix $R$. If $\mathsf r_{n-1}\ge e^{-R}$ and
$\zeta\in[a,b]$, then
$\det(\mathsf r)\ge e^{-(n-1)R}$, and therefore
\begin{equation}\label{eq:Phi-uniform-approximation}
\left|
\Phi(\mathsf r,T^{-n}\zeta)-\Phi(\mathsf r,0)
\right|
\le
\frac{T^{-n}|\zeta|}{\det(\mathsf r)}\|E_{nn}\|
\ll_{R,a,b}T^{-n}.
\end{equation}
If the integrand defining $N_R(T)$ or $V_R(T)$ is nonzero, then
\eqref{eq:Phi-uniform-approximation} gives
$
\Phi(\mathsf r,0)\le1+O_{R,a,b}(T^{-n}).
$
Thus, for all sufficiently large $T$, every relevant parameter
$\mathsf r$ belongs to the compact set
$$
\{\mathsf r\in B_0^+:
\mathsf r_{n-1}\ge e^{-R},\ \Phi(\mathsf r,0)\le2\}.
$$
The boundary of $D_R\times(a,b)$ is contained in
$$
\{\mathsf r_{n-1}=e^{-R}\}
\cup\{\Phi(\mathsf r,0)=1\}
\cup\{\zeta=a\}
\cup\{\zeta=b\},
$$
each of which is null with respect to
$d\mathsf r\,d\zeta$.

For $\varepsilon>0$, put
$$
D_R^\pm(\varepsilon)
:=
\left\{\mathsf r\in B_0^+:
\mathsf r_{n-1}\ge e^{-R},\quad
\Phi(\mathsf r,0)<1\pm\varepsilon
\right\}.
$$
By \eqref{eq:scaled-K-inv-reduction} and
\eqref{eq:Phi-uniform-approximation}, for all sufficiently large $T$,
the parameter set defining $N_R(T)$ and $V_R(T)$ lies between
$$
D_R^-(\varepsilon)\times(a,b)
\quad\text{and}\quad
D_R^+(\varepsilon)\times(a,b).
$$
Apply \Cref{prop:countinglimit-quotient}, with $\nu\equiv1$, to these
two fixed sets and then let $\varepsilon\downarrow0$. The boundary
nullity proved above gives
\begin{equation}\label{eq:truncated-norm-ball-limit}
\lim_{T\to\infty}T^{-n(n-1)}N_R(T)=A_R.
\end{equation}
The identical squeeze, applied to \Cref{hfunctional} with
$\nu\equiv1$, gives
\begin{equation}\label{eq:truncated-volume-limit}
\lim_{T\to\infty}T^{-n(n-1)}V_R(T)=A_R.
\end{equation}
The locus of matrices of rank at most $n-2$ is Lebesgue null and does
not affect the latter identity.

Define
$$
\begin{aligned}
N(T)
&:=
\#\{v\in\Lambda_{\operatorname{ni}}:
\|v\|<T,\ a<\det v<b\};\\
V(T)
&:=
\operatorname{vol}\{v\in\M_n(\mathbb R):
\|v\|<T,\ a<\det v<b\}.
\end{aligned}
$$
Choose $D>\max\{|a|,|b|\}$. The two estimates in
\Cref{prop:singular-value-cusp-tail} give the following bounds; the
zero matrix omitted from $\mathcal E_{D,R}(T)$ is a Lebesgue-null set.
\begin{equation}\label{eq:discrete-continuous-omitted-bounds}
\begin{aligned}
0\le N(T)-N_R(T)
&\ll_{\Lambda,D}e^{-R/2}T^{n(n-1)}+T^{n(n-1)-\frac n2},\\
0\le V(T)-V_R(T)
&\ll_D e^{-R/2}T^{n(n-1)}+T^{n(n-1)-\frac n2}.
\end{aligned}
\end{equation}
Combining this with
\eqref{eq:truncated-norm-ball-limit} and
\eqref{eq:truncated-volume-limit}, we obtain, for $Q=N$ and $Q=V$,
\begin{equation}\label{eq:norm-ball-sandwich}
A_R
\le
\liminf_{T\to\infty}T^{-n(n-1)}Q(T)
\le
\limsup_{T\to\infty}T^{-n(n-1)}Q(T)
\le
A_R+O(e^{-R/2}).
\end{equation}

If $R'>R$, then
$
0\le V_{R'}(T)-V_R(T)\le V(T)-V_R(T).
$
Equations \eqref{eq:truncated-volume-limit} and
\eqref{eq:discrete-continuous-omitted-bounds} therefore imply
$$
0\le A_{R'}-A_R\ll e^{-R/2}.
$$
Thus $A_R$ has a finite limit. Since
$$
D_R\uparrow
\{\mathsf r\in B_0^+:\Phi(\mathsf r,0)<1\},
$$
monotone convergence and \eqref{mconstant} give
$$
\lim_{R\to\infty}A_R=C_{\|\cdot\|}(b-a).
$$
Letting $R\to\infty$ in \eqref{eq:norm-ball-sandwich} proves
\eqref{eq:norm-ball-counting} and \eqref{eq:volume-main-term}.
\end{proof}
\subsection{Completion of the counting theorems}\label{isoo}

For a lattice $\Lambda$, put
$$
S_{\Lambda,\operatorname{ni}}(T)
:=
\#\{v\in\Lambda_{\operatorname{ni}}:
\|v\|<T,\ \det v=0\}.
$$
Assume first that $\Lambda\in X$ is Diophantine and
non-determinant-rational. For every $\varepsilon>0$,
$$
S_{\Lambda,\operatorname{ni}}(T)
\le
\#\{v\in\Lambda_{\operatorname{ni}}:
\|v\|<T,\ -\varepsilon<\det v<\varepsilon\}.
$$
Consequently, \Cref{thm:13.11-corrected} gives
$$
\limsup_{T\to\infty}T^{-n(n-1)}S_{\Lambda,\operatorname{ni}}(T)
\le2C_{\|\cdot\|}\varepsilon.
$$
Letting $\varepsilon\downarrow0$, we obtain
\begin{equation}\label{eq:nonisotropic-singular-negligible}
S_{\Lambda,\operatorname{ni}}(T)=o(T^{n(n-1)}).
\end{equation}

Every nonsingular matrix is non-isotropic. Hence the difference between
$$
\#\{v\in\Lambda:
\|v\|<T,\ a<\det v<b,\ \det v\ne0\}
$$
and the counting function in \eqref{eq:norm-ball-counting} is at most
$S_{\Lambda,\operatorname{ni}}(T)$. It follows from
\eqref{eq:norm-ball-counting} and
\eqref{eq:nonisotropic-singular-negligible} that the former is
asymptotic to $C_{\|\cdot\|}(b-a)T^{n(n-1)}$.

Now let $\Lambda$ be an arbitrary lattice satisfying the same
Diophantine and non-determinant-rational hypotheses, and set
$$
c:=\covol(\Lambda)^{-1/n^2}.
$$
Then $c\Lambda$ is unimodular. Homothety preserves both hypotheses and
the isotropic decomposition; in particular,
$
(c\Lambda)_{\operatorname{ni}}=c\Lambda_{\operatorname{ni}}.
$
Therefore
\begin{equation}\label{eq:nonisotropic-singular-negligible-general}
S_{\Lambda,\operatorname{ni}}(T)
=
S_{c\Lambda,\operatorname{ni}}(cT)
=
o(T^{n(n-1)}).
\end{equation}
Moreover,
$$
N_\Lambda^\times(a,b;T)
=
N_{c\Lambda}^\times(c^na,c^nb;cT).
$$
Applying the unimodular asymptotic gives
$$
N_\Lambda^\times(a,b;T)
\sim
C_{\|\cdot\|}c^n(b-a)(cT)^{n(n-1)}
=
\frac{C_{\|\cdot\|}}{\covol(\Lambda)}
(b-a)T^{n(n-1)},
$$
because $n(n-1)+n=n^2$. Together with \Cref{lem:volume-main-term}, this
proves \Cref{main}. Applying
\Cref{prop:algebraicityimpliesDiophantine} gives \Cref{m1}. The
determinant-form statement follows from the coordinate change defining
$F_{\Lambda,\mathcal B}$, as recorded in \Cref{m2general}.

Finally, suppose that $\Lambda$ satisfies the isotropic noncoincidence
condition. A non-determinant-rational lattice is not of
$\mathbb Q$-split type, so \Cref{thm:quasinullcontribution2} gives
$$
T^{-n(n-1)}\#\{v\in\Lambda_{\operatorname{iso}}:\|v\|<T\}
\to c_{\Lambda,\mathrm{iso}}.
$$
The singular set decomposes as
$$
\#\{v\in\Lambda:\|v\|<T,\ \det v=0\}
=
\#\{v\in\Lambda_{\operatorname{iso}}:\|v\|<T\}
+
S_{\Lambda,\operatorname{ni}}(T)
+O(1).
$$
Thus \eqref{eq:nonisotropic-singular-negligible-general} proves
that
$
c_\Lambda^{\operatorname{sing}}=c_{\Lambda,\mathrm{iso}}.
$
The positivity characterization follows from
\Cref{thm:quasinullcontribution2}; applying
\Cref{prop:algebraicityimpliesDiophantine} to the lattices in
\Cref{m11} proves that theorem.

\subsection{Bounds without a Diophantine hypothesis}
\label{subsec:general-bounds}
We conclude this section with bounds that do not require a
Diophantine hypothesis.

\begin{theorem}\label{thm:general-irrational-bounds}
Let $\Lambda<\M_n(\mathbb R)$ be a lattice. For every $a<b$ and 
$\varepsilon>0$,
$$
N_\Lambda(a,b;T)
\ll_{\Lambda,a,b,\varepsilon}
T^{n(n-1)+\varepsilon}
\qquad(T\ge2).
$$
If, in addition, $\Lambda$ is not determinant-rational, then
$$
\frac{C_{\|\cdot\|}}{\covol(\Lambda)}(b-a)
\le
\liminf_{T\to\infty}T^{-n(n-1)}
N_\Lambda^\times(a,b;T).
$$
In particular,
$$
N_\Lambda^\times(a,b;T)
=T^{n(n-1)+o(1)}
\quad\text{and}\quad
N_\Lambda(a,b;T)=T^{n(n-1)+o(1)}.
$$
\end{theorem}

\begin{proof}
We may assume that $\Lambda$ is unimodular without loss of generality. For every
$\varepsilon>0$,
\begin{equation}\label{eq:ordinary-height-subpower}
\int_K\bigl(1+\alpha(a_tk\Lambda)\bigr)\,dk
\ll_{\Lambda,\varepsilon}e^{\varepsilon t}
\qquad(t\ge0).
\end{equation}
To prove this, choose $0<\theta<(2n)^{-6}$ and let
$\rho=\rho(\theta)$ be given by
\eqref{eq:bar-alpha-comparison}, so that
$$
\max\{1,\alpha(\Delta)\}^{1-\rho}
\le\bar\alpha_\theta(\Delta).
$$
Since $\rho(\theta)\to0$ as $\theta\to0$, we may assume that
$\rho<1$ and $2N\rho<\varepsilon$. Set
$$
X_t(k):=\max\{1,\alpha(a_tk\Lambda)\},
\qquad
B_t:=e^{2Nt}\max\{1,\alpha(\Lambda)\}.
$$
The smallest singular value of $a_t$ on every
$\wedge^i\M_n(\mathbb R)$ is at least $e^{-2Nt}$; hence
$X_t(k)\le B_t$. Therefore
$$
X_t(k)
=
X_t(k)^{1-\rho}X_t(k)^\rho
\le
\bar\alpha_\theta(a_tk\Lambda)B_t^\rho.
$$
Since $1+\alpha\le2\max\{1,\alpha\}$, integration and
\eqref{eq:standardalphauniformboundednesscompact} prove
\eqref{eq:ordinary-height-subpower}.

We also require the ordinary Lipschitz estimate
\begin{equation}\label{eq:ordinary-long-box}
\#\bigl(\Delta\cap\mathscr B_R(L)\bigr)
\ll_R L\bigl(1+\alpha(\Delta)\bigr)
\qquad(\Delta\in X,\ L\ge1).
\end{equation}
Indeed, the projection calculation in the proof of
\Cref{lem:modifiedlipschitz} gives, uniformly over every nonzero linear
subspace $V\subset\M_n(\mathbb R)$,
$$
\operatorname{vol}_{\dim V}(p_V\mathscr B_R(L))\ll_R L.
$$
The ordinary successive-minima flag estimate, together with the bound
$d_\Delta(V)^{-1}\le\alpha(\Delta)$ for every proper
$\Delta$-rational subspace $V$, gives
\eqref{eq:ordinary-long-box}.

Fix $D>\max\{|a|,|b|\}$, and let $R=R_{n,D}$ be as in
\Cref{lem:one-long-direction-box}. We now derive ordinary versions of
the two local shell estimates. If
$
0=m_1\le\cdots\le m_{n-1}\le n\tau/2,
$
then \Cref{lem:shell-averaging,lem:one-long-direction-box}, \eqref{eq:ordinary-long-box} and \eqref{eq:long-box-containment} give
$$
\begin{aligned}
\#\bigl(\Lambda\cap\mathcal S_{\tau,\mathbf m}(D)\bigr)
&\ll_D
e^{n(n-1)\tau-2|\mathbf m|}
\int_K
\#\bigl(a_{\tau/2}k\Lambda
       \cap\mathcal B_{|\mathbf m|,D}(R)\bigr)\,dk\\
&\ll_D
e^{n(n-1)\tau-|\mathbf m|}
\int_K\bigl(1+\alpha(a_{\tau/2}k\Lambda)\bigr)\,dk.
\end{aligned}
$$
Applying \eqref{eq:ordinary-height-subpower} with $2\varepsilon$ in
place of $\varepsilon$ yields
\begin{equation}\label{eq:general-moderate-shell}
\#\bigl(\Lambda\cap\mathcal S_{\tau,\mathbf m}(D)\bigr)
\ll_{\Lambda,D,\varepsilon}
e^{(n(n-1)+\varepsilon)\tau-|\mathbf m|}.
\end{equation}

Choose $R_1\ge1$ such that
$
\{u:\|u\|\le R_n\}\subset\mathscr B_{R_1}(1),
$
where $R_n$ is as in \Cref{lem:deep-shell-weyl-box}. The same averaging
argument, now using \eqref{eq:deep-shell-lower-bound}, gives
$$
\begin{aligned}
\#\bigl(\Lambda\cap\mathcal D_\tau\bigr)
&\ll
e^{\bigl(n(n-1)-\frac n2\bigr)\tau}
\int_K\#\bigl(a_{\tau/2}k\Lambda
             \cap\mathscr B_{R_1}(1)\bigr)\,dk\\
&\ll_{R_1}
e^{\bigl(n(n-1)-\frac n2\bigr)\tau}
\int_K\bigl(1+\alpha(a_{\tau/2}k\Lambda)\bigr)\,dk.
\end{aligned}
$$
Thus
\begin{equation}\label{eq:general-deep-shell}
\#\bigl(\Lambda\cap\mathcal D_\tau\bigr)
\ll_{\Lambda,\varepsilon}
e^{(n(n-1)-\frac n2+\varepsilon)\tau}.
\end{equation}

It remains to sum the local estimates. Choose $c_0\ge0$ such that
$$
\kappa_1(v)\le\log\|v\|+c_0
\qquad(v\ne0).
$$
The nonzero lattice vectors with $\kappa_1(v)\le0$ lie in a fixed
compact set and contribute $O_\Lambda(1)$. For every remaining vector
with $\|v\|<T$, put $\tau=\lceil\kappa_1(v)\rceil$. Then
$$
1\le\tau\le\log T+c_0+1.
$$
The deep shells contribute, by \eqref{eq:general-deep-shell},
$$
\ll_{\Lambda,\varepsilon}
\sum_{\tau\le\log T+c_0+1}e^{(n(n-1)-\frac n2+\varepsilon)\tau}
\ll_{\Lambda,\varepsilon}T^{n(n-1)-\frac n2+\varepsilon}.
$$
If $v\notin\mathcal D_\tau$, define
$$
m_i:=\tau-\lceil\kappa_i(v)\rceil
\qquad(1\le i\le n-1).
$$
Then \eqref{eq:non-deep-deficits} holds and
$v\in\mathcal S_{\tau,\mathbf m}(D)$. Hence
\eqref{eq:general-moderate-shell} and
\eqref{eq:deficit-sum} with $A=0$ show that the contribution for fixed
$\tau$ is
$$
\ll_{\Lambda,D,\varepsilon}
e^{(n(n-1)+\varepsilon)\tau}
\sum_{\substack{
\mathbf m\in\mathbb Z^{n-1}\\
0=m_1\le\cdots\le m_{n-1}}}
e^{-|\mathbf m|}
\ll_{\Lambda,D,\varepsilon}e^{(n(n-1)+\varepsilon)\tau}.
$$
Summing over $\tau\le\log T+c_0+1$ proves
\begin{equation}\label{eq:general-upper-unimodular}
N_\Lambda(a,b;T)
\ll_{\Lambda,a,b,\varepsilon}T^{n(n-1)+\varepsilon}.
\end{equation}
Every nonzero matrix of rank at most $n-2$ with $\kappa_1(v)>0$
belongs to a deep shell. Those with $\kappa_1(v)\le0$ were already
absorbed into $O_\Lambda(1)$, and the zero matrix contributes at
most one point.

Suppose now that $\Lambda$ is not determinant-rational. Then
$H\Lambda$ is dense in $X$ by
\Cref{prop:closed-rational-csa}. Set
$$
\mathcal U
:=
\left\{
(\mathsf r,\zeta)\in B_0^+\times\mathbb R:
\Phi(\mathsf r,0)<1,\quad
a<\zeta<b,\quad
\zeta\ne0
\right\}.
$$
Let $0\ne h\in C_c(\mathcal U)$ satisfy $0\le h\le1$, and let
$f=f_h$ be the nonnegative function constructed in the proof of
\Cref{lem:Jf-surjective}. Then $J_f=h$, and
\begin{equation}\label{eq:f-support-projection}
u\in\operatorname{supp}(f),\quad
m\in\mathsf M,\quad \mathsf r\in B_0^+,\quad
u^0=m\cdot\mathsf r
\quad\to\quad
(\mathsf r,\det u)\in\operatorname{supp}(h).
\end{equation}
Since $\operatorname{supp}(h)$ is a compact subset of $\mathcal U$,
there exists $\delta>0$ such that
$$
\Phi(\mathsf r,0)\le1-\delta
\qquad
((\mathsf r,\zeta)\in\operatorname{supp}(h)).
$$
Moreover, since $\operatorname{supp}(f)$ is a compact subset of
$\M_n^+$,
$$
c_f
:=
\frac12\min_{u\in\operatorname{supp}(f)}
\left\|\operatorname{diag}(u^0,0)\right\|>0.
$$
Thus, for every $u\in\operatorname{supp}(f)$, writing
$u^0=m\cdot\mathsf r$ and using the $K$-invariance of the norm, we have
\begin{equation}\label{eq:f-inner-support}
a<\det u<b,\qquad
\det u\ne0,\qquad
2c_f
\le
\left\|\operatorname{diag}(u^0,0)\right\|
=
\Phi(\mathsf r,0)
\le1-\delta.
\end{equation}
The continuous zero extension of $f$ to $\M_n(\mathbb R)$, still
denoted by $f$, belongs to
$C_c(\M_n(\mathbb R)\smallsetminus\{0\})$. Set $T=e^{2t}$. If $u=a_tk\cdot v\in\operatorname{supp}(f)$, then
$$
T^{-1}k\cdot v
=
\begin{pmatrix}
u^0&e^{-nt}u^+\\
e^{-nt}u^-&e^{-2nt}u_{nn}
\end{pmatrix}.
$$
Consequently, uniformly for $u\in\operatorname{supp}(f)$,
$$
\left\|
T^{-1}k\cdot v-\operatorname{diag}(u^0,0)
\right\|
\ll_f e^{-nt}.
$$
The $K$-invariance of the norm, together with
\eqref{eq:f-inner-support}, controls the
norm of $v$; the $H$-action also gives $\det v=\det u$. Therefore, for
all sufficiently large $t$,
\begin{equation}\label{eq:f-contribution-inside-count}
f(a_tk\cdot v)\ne0
\quad\to\quad
c_fT\le\|v\|<T,\quad
a<\det v<b,\quad
\det v\ne0.
\end{equation}
Also $u^0$ is invertible, so $\widehat v\ne0$. Thus \Cref{jf} applies
uniformly to every contributing vector. Given $\eta>0$, it yields
$$
\frac{\omega_{n-1}^2}{2}T^{n(n-1)}
\int_Kf(a_tk\cdot v)\,dk
\le
h(T^{-1}b^0(v),\det v)+\eta
\le1+\eta
$$
for all sufficiently large $t$. Summing over $v\in\Lambda$ and using
\eqref{eq:f-contribution-inside-count}, we obtain
\begin{equation}\label{eq:count-dominates-K-average}
\frac{\omega_{n-1}^2}{2}T^{n(n-1)}
\int_K\widehat f(a_tk\Lambda)\,dk
\le
(1+\eta)N_\Lambda^\times(a,b;T).
\end{equation}

Choose functions $0\le\chi_j\le1$ in $C_c(X)$ such that
$\chi_j\uparrow1$. Since $\operatorname{supp}(f)$ avoids the origin,
$\widehat f$ is continuous, and
\Cref{thm:Ratnerequidistributionbounded} gives
$$
\begin{aligned}
\liminf_{t\to\infty}\int_K\widehat f(a_tk\Lambda)\,dk
&\ge
\lim_{t\to\infty}
\int_K(\chi_j\widehat f)(a_tk\Lambda)\,dk =
\int_X\chi_j\widehat f\,dm_X.
\end{aligned}
$$
Letting $j\to\infty$, monotone convergence and Siegel's formula give
\begin{equation}\label{eq:ordinary-Siegel-liminf}
\liminf_{t\to\infty}
\int_K\widehat f(a_tk\Lambda)\,dk
\ge
\int_{\M_n(\mathbb R)}f(v)\,dv.
\end{equation}
Combining \eqref{eq:count-dominates-K-average},
\eqref{eq:ordinary-Siegel-liminf}, and \Cref{siegelJf}, we obtain
$$
\liminf_{T\to\infty}T^{-n(n-1)}N_\Lambda^\times(a,b;T)
\ge
\frac{\omega_{n-1}^2}{2(1+\eta)}
\int_{\mathbb R}\int_{B_0^+}
h(\mathsf r,\zeta)\det(\mathsf r)^n\,d\mathsf r\,d\zeta.
$$
Letting $\eta\downarrow0$ and taking the supremum over
$0\le h\le1$ in $C_c(\mathcal U)$ gives
$$
\begin{aligned}
\liminf_{T\to\infty}T^{-n(n-1)}N_\Lambda^\times(a,b;T)
&\ge
\frac{\omega_{n-1}^2}{2}(b-a)
\int_{\{\mathsf r\in B_0^+:\Phi(\mathsf r,0)<1\}}
\det(\mathsf r)^n\,d\mathsf r\\
&=
C_{\|\cdot\|}(b-a).
\end{aligned}
$$
Here deleting the slice $\zeta=0$ does not change the integral.
\end{proof}

\section{Algebraic lattices and examples}\label{s:example}
In this section, we prove
\Cref{prop:algebraicityimpliesDiophantine} and deduce that every determinant
form with algebraic coefficients is Diophantine. We also give several
examples, including a cubic norm-form lattice with a positive singular
constant, and show that every diagonal lattice satisfies the isotropic
noncoincidence condition.

\begin{lemma}[Liouville-type lower bound for algebraic matrices]
\label{lem:algebraic-linear-map}
For every matrix $A\in\op{M}_{t\times s}(\overline{\q})$, there exist
 $c>0$ and $M\ge0$ such that for every $z\in \mathbb Z^s$ such that $Az\ne 0$,  
$$
\|Az\|\ge c\max\{1,\|z\|\}^{-M}.
$$
\end{lemma}
\begin{proof}
Choose a number field $\mathscr K\subset\mathbb C$ containing the entries
of $A$ and an integer $D\ge1$ such that
$DA\in\op{M}_{t\times s}(\mathcal O_{\mathscr K})$. Set
$d=[\mathscr K:\mathbb Q]$, and choose $C\ge1$ such that
$\|\sigma(DA)\|_{\mathrm{op}}\le C$ for every embedding
$\sigma:\mathscr K\hookrightarrow\mathbb C$.

Suppose that $Az\ne0$ for $z\in\mathbb Z^s$, and choose a nonzero
coordinate $\beta$ of $Az$. Then $D\beta$ is a nonzero algebraic integer,
and
$|\sigma(D\beta)|\le C\max\{1,\|z\|\}$ for every $\sigma:\mathscr K\hookrightarrow\mathbb C$.
Hence
$$
\begin{aligned}
1
&\le \bigl|N_{\mathscr K/\mathbb Q}(D\beta)\bigr|
=
|D\beta|
\prod_{{\sigma:\mathscr K\hookrightarrow\mathbb C, \sigma\ne\operatorname{id}}}
|\sigma(D\beta)| \le
|D\beta|\,C^{d-1}\max\{1,\|z\|\}^{d-1}.
\end{aligned}
$$
Thus
$|\beta|\ge D^{-1}C^{1-d}\max\{1,\|z\|\}^{1-d}$.
Since $\|Az\|\ge|\beta|$, the assertion follows.
\end{proof}

\begin{proof}[Proof of \Cref{prop:algebraicityimpliesDiophantine}]
Choose a $\mathbb Z$-basis of $\Lambda$ and
write
$
\Lambda=g\mathbb Z^N$ for some
$g\in\GL_N(\overline{\mathbb Q}\cap\mathbb R).
$ We first note that the projections appearing in
\Cref{lattice_Diophantine} are defined over $\mathbb Q$. Indeed, the
skew Cauchy decomposition in \eqref{eq:decompositionatdimnn-1} is obtained
by base change from the corresponding decomposition over $\mathbb Q$,
and each of its Schur-functor summands is defined over $\mathbb Q$; see,
for example, \cite{ABW,FultonYT}. Consequently, the maps
$\pi_{k,m}:\wedge^{kn}\mathbb R^N\to\mathcal M_{k,m}$
and $\pi_{k,m}^{(r)}=\wedge^r\pi_{k,m} $
have rational matrices with respect to the standard exterior bases.

Let $1\le k\le n-1$, $m\in\{1,2\}$, and $r\in\{1,\dim {\mathcal M_{k, m}} \}$. Let $V_1,\ldots,V_r$ be $kn$-dimensional
$\Lambda$-rational subspaces of $\br^N$. For each $j$, set
$$
L_j:=g^{-1}(\Lambda\cap V_j)
      =\mathbb Z^N\cap g^{-1}V_j.
$$
Then $L_j$ is a primitive sublattice of $\mathbb Z^N$ of rank
$kn$. Choose a $\mathbb Z$-basis of $L_j$, and let $q_j\in\wedge^{kn}\mathbb Z^N$
be the corresponding primitive Pl\"ucker vector, with its sign chosen so
that
$\mathsf w_{\Lambda,V_j}=(\wedge^{kn}g)q_j. $
Define
\begin{equation}\label{eq:algebraic-plucker-transfer}
\mathbf w
:=
\mathsf w_{\Lambda,V_1}\wedge\cdots\wedge
\mathsf w_{\Lambda,V_r}=\wedge^r\!\left(\wedge^{kn}g\right) q_1\wedge\cdots\wedge q_r.
\end{equation}

If $q_1\wedge\cdots\wedge q_r=0$, then $\mathbf w=0$, so the first alternative in
\Cref{lattice_Diophantine} holds. Suppose henceforth that $q_1\wedge\cdots\wedge q_r\ne0$. Since
$\wedge^r\!\left(\wedge^{kn}g\right)$ is invertible, there is a constant $C_{k,r}\ge1$ such that
\begin{equation}\label{eq:q-w-comparison}
\|q_1\wedge\cdots\wedge q_r\|\le C_{k,r}\|\mathbf w\|.
\end{equation}
In particular, because $q_1\wedge\cdots\wedge q_r$ is a nonzero integral vector,
\begin{equation}\label{eq:nonzero-w-lower-bound}
\|\mathbf w\|\ge C_{k,r}^{-1}.
\end{equation}
Since
the linear map $\bigl(I-\pi_{k,m}^{(r)}\bigr)\wedge^r\!\left(\wedge^{kn}g\right)
$
has algebraic matrix entries, 
 \Cref{lem:algebraic-linear-map} gives 
$c_{k,m,r}>0$ and $M_{k,m,r}\ge0$ such that
$$
\begin{aligned}
\|\mathbf w-\pi_{k,m}^{(r)}(\mathbf w)\| 
&\ge c_{k,m,r}\|q_1\wedge\cdots\wedge q_r\|^{-M_{k,m,r}}\ge
c_{k,m,r}C_{k,r}^{-M_{k,m,r}}
\|\mathbf w\|^{-M_{k,m,r}},
\end{aligned}
$$
whenever $\mathbf w\ne\pi_{k,m}^{(r)}(\mathbf w)$. 
There are only finitely many admissible triples $(k,m,r)$. Using
\eqref{eq:nonzero-w-lower-bound}, we may enlarge the finitely many
exponents to a common exponent $M_0>1$ and decrease the corresponding
constants to obtain $\eta_0>0$ such that
$$
\|\mathbf w-\pi_{k,m}^{(r)}(\mathbf w)\|
\ge
\eta_0\|\mathbf w\|^{-M_0}
$$
whenever $\mathbf w\ne\pi_{k,m}^{(r)}(\mathbf w)$. Taking
$\eta:=\frac12\min\{1,\eta_0\}$ gives the strict inequality required in
\Cref{lattice_Diophantine}, with $M=M_0$.
\end{proof}

\begin{corollary}\label{cor:algebraic-form-diophantine}
Every determinant form $F$ with algebraic coefficients is Diophantine.
\end{corollary}

\begin{proof} Let $g\in\GL_N(\mathbb R)$ be such
that
$
F=\det\circ g.
$
Comparing coefficients in this identity and introducing an auxiliary
variable $u$ satisfying $u\det(g)=1$, we obtain a finite system of
polynomial equations in the entries of $g$ and $u$ with coefficients in
$\overline\q\cap \mathbb R$. This system has a solution over $\mathbb R$.
Since $\overline\q\cap \mathbb R$ is real closed, Tarski's transfer
principle \cite{BCR1998} gives a solution over
$\overline\q\cap \mathbb R$. In particular, there exists
$
g_0\in\GL_N(\overline\q\cap \mathbb R)
$
such that $F=\det\circ g_0$.
Thus $F$ is the determinant form associated with the algebraic lattice
$g_0\mathbb Z^N$. The conclusion follows from
\Cref{prop:algebraicityimpliesDiophantine,form_Diophantine}.
\end{proof}

Consequently, \Cref{m1} follows from \Cref{main} and
\Cref{prop:algebraicityimpliesDiophantine}; the determinant-form version
follows from \Cref{m2general} and
\Cref{cor:algebraic-form-diophantine}.

\subsection{Examples}\label{sub:ex}

\begin{example}[Entrywise algebraic lattices]
Let $\alpha,\beta,\gamma,\delta\in
\overline{\mathbb Q}\cap\mathbb R$ be nonzero, and define the lattice
$$
\Lambda_{\alpha,\beta,\gamma,\delta}
:=
\begin{pmatrix}
\alpha\mathbb Z&\beta\mathbb Z&\mathbb Z\\
\gamma\mathbb Z&\delta\mathbb Z&\mathbb Z\\
\mathbb Z&\mathbb Z&\mathbb Z
\end{pmatrix}
<\M_3(\mathbb R).
$$

It is easy to check that unless all $\alpha,\beta,\gamma,\delta$ are rational,
$\Lambda_{\alpha,\beta,\gamma,\delta}$ is not determinant-rational, and
hence \Cref{m1} applies. 
The rank-six submodule obtained by setting the first row equal to zero is
contained in a column-isotropic subspace. Since $n=3$, the
isotropic noncoincidence condition is vacuous. Thus
\Cref{thm:quasinullcontribution2} applies, and
$c_{\Lambda_{\alpha,\beta,\gamma,\delta}}^{\operatorname{sing}}>0$.
\end{example}

\begin{example}[A cubic norm-form lattice]\label{ex:cubic-norm}
Let $\mathscr{K}=\mathbb Q(\varpi)$ where $\varpi^3=2$, and let
$$
\iota:\mathscr{K}\hookrightarrow\operatorname{End}_{\mathbb Q}(\mathscr{K})
\simeq\M_3(\mathbb Q)
$$
be the regular representation with respect to the basis
$1,\varpi,\varpi^2$. Then $\iota(1)=I$ and
$$
\iota(\varpi)=
T:=
\begin{pmatrix}
0&0&2\\
1&0&0\\
0&1&0
\end{pmatrix},
\qquad
\iota(\varpi^2)=T^2=
\begin{pmatrix}
0&2&0\\
0&0&2\\
1&0&0
\end{pmatrix}.
$$
In particular,
$$
\det(xI+yT+zT^2)
=N_{\mathscr{K}/\mathbb Q}(x+y\varpi+z\varpi^2)
=x^3+2y^3+4z^3-6xyz.
$$

Let $\alpha\in\overline{\mathbb Q}\cap\mathbb R$ satisfy
$\alpha^3\notin\mathbb Q$, and put
$$
W:=\{A\in\M_3(\mathbb Q):\text{the first column of }A\text{ is zero}\}.
$$
Since the first column of $\iota(\xi)$ is the coordinate vector of
$\xi\in \mathscr{K}$, one has $\iota(\mathscr{K})\cap W=\{0\}$, and therefore
$\M_3(\mathbb Q)=\iota(\mathscr{K})\oplus W$.
Define
$$
\Lambda_{\mathscr{K},\alpha}
:=
\mathbb Z(\alpha I)\oplus\mathbb ZT\oplus\mathbb ZT^2
\oplus
\bigoplus_{r=1}^3
\bigl(\mathbb ZE_{r2}\oplus\mathbb ZE_{r3}\bigr).
$$
Writing
$
x=(x_1,x_2,x_3)\in\mathbb R^3$ and
$y=(y_{rj})_{1\le r\le 3,\ 2\le j\le 3}\in\mathbb R^{3\times 2}$
for the coordinates with respect to the displayed basis of $\Lambda_{\mathscr{K},\alpha}$, the associated
determinant form is
\begin{align*}
F_{\Lambda_{\mathscr{K},\alpha}}(x,y) =
\det\!\begin{pmatrix}
\alpha x_1&2x_3+y_{12}&2x_2+y_{13}\\
x_2&\alpha x_1+y_{22}&2x_3+y_{23}\\
x_3&x_2+y_{32}&\alpha x_1+y_{33}
\end{pmatrix}.
\end{align*}

Since $\tfrac{\det(\alpha I)}{\det(T)}
=\tfrac{\alpha^3}{2}\notin\mathbb Q,$
$\Lambda_{\mathscr{K},\alpha}$ is not determinant-rational, so \Cref{m1} applies. Moreover, the rank-six submodule
$\bigoplus_{r=1}^3\bigl(\mathbb ZE_{r2}\oplus\mathbb ZE_{r3}\bigr) $
consists entirely of matrices with zero first column. Since $n=3$,
\Cref{thm:quasinullcontribution2} applies, and
$c_{\Lambda_{\mathscr{K},\alpha}}^{\operatorname{sing}}>0$.
\end{example}

\begin{example}[Diagonal lattices satisfy isotropic noncoincidence]
\label{ex:diagonal-isotropically-noncoincident}
Let $\boldsymbol\lambda=(\lambda_{ij})_{1\le i,j\le n}\in(\mathbb R^\times)^{n^2}$, and set
$$
\Delta_{\boldsymbol\lambda}
:=
\bigoplus_{i,j=1}^n \lambda_{ij}\mathbb Z E_{ij}
<\M_n(\mathbb R).
$$
 Then $\Delta_{\boldsymbol\lambda}$ satisfies the
isotropic noncoincidence condition.
\end{example}

\begin{proof}
We prove the column statement; the row statement is identical. Write
$$
\Lambda_j:=\bigoplus_{i=1}^n \lambda_{ij}\mathbb Z e_i<\mathbb R^n.
$$
 If
$\mathsf U<\mathbb R^n$, then
$\mathcal L(\mathsf U)=\{X\in\M_n(\mathbb R):\operatorname{Col}(X)\subset \mathsf U\}$
is $\Delta_{\boldsymbol\lambda}$-rational if and only if $\mathsf U$ is $\Lambda_j$-rational for every $1\le j\le n$. Indeed,
$$
\Delta_{\boldsymbol\lambda}\cap \mathcal L(\mathsf U)
=
\bigoplus_{j=1}^n(\Lambda_j\cap \mathsf U),
$$
under the decomposition of a matrix into its columns.
Now suppose that $\mathcal L(\mathsf U)$ is
$\Delta_{\boldsymbol\lambda}$-rational and that
$\dim\mathsf U\le n-2$. By the preceding characterization,
$\mathsf U$ is $\Lambda_j$-rational for every $j$. Choose a coordinate
vector $e_r\notin\mathsf U$. Since $\mathbb Re_r$ is
$\Lambda_j$-rational for every $j$, so is
$
\mathsf U':=\mathsf U+\mathbb Re_r.
$
Moreover,
$
\dim\mathsf U'=\dim\mathsf U+1\le n-1,
$
so $\mathsf U'$ is proper. Hence
$\mathcal L(\mathsf U')$ is a proper
$\Delta_{\boldsymbol\lambda}$-rational column-isotropic subspace that
properly contains $\mathcal L(\mathsf U)$.

Therefore no proper $\Delta_{\boldsymbol\lambda}$-rational
column-isotropic subspace of dimension $kn$ with $k\le n-2$ is maximal
among proper rational column-isotropic subspaces. Thus the column isotropic
noncoincidence condition is vacuous. The same argument applied to the row
coordinate lattices proves the row condition. Hence
$\Delta_{\boldsymbol\lambda}$ satisfies the isotropic noncoincidence
condition.
\end{proof}

\appendix

\section{Uniform real moments from complex singularity exponents}
\label{app:analytic-stability}
This appendix is notationally self-contained; its symbols are independent of those used in the body of the paper. 

For a germ of a holomorphic map $F$ at $z$, we define
$$
 c_z(F):=\sup\{c>0:\|F\|^{-2c}\text{ is locally integrable near }z\}.
$$
Here integrability is with respect to Lebesgue measure in complex
coordinates. Note that $c_z(F)=+\infty$ when $F(z)\ne0$ and $c_z(F)=0$ when $F$ is the zero
germ, meaning that $F$ vanishes identically on a neighborhood of $z$.

Let
$p,\ell,q\ge1$ and 
$$
\mathcal Q:\mathbb R^p\times\mathbb R^\ell\to\mathbb R^q
$$
be a polynomial map. Write $\mathcal Q_h(x):=\mathcal Q(h,x)$, and let
$\mathcal Q^{\mathbb C}$ denote its complexification, given by the same
polynomial formulas. 

The main goal of this appendix is to prove the following multiparameter
stability statement, which is used in
\Cref{lem:uniform-negative-moments-complex}.

\begin{proposition}[Multiparameter stability from a complex threshold]
\label{prop:multiparameter-stability}
Suppose that there exist $\alpha>0$ and bounded open sets
$B\subset\mathbb R^\ell$ and $V\subset\mathbb C^\ell$ with
$\overline B\subset V$ such that
$$
c_z(\mathcal Q_0^{\mathbb C})\ge\alpha
\qquad(z\in V).
$$
Then, for every $0<\beta<\alpha$, there is an open neighborhood $U$ of $0$
in $\mathbb R^p$ such that
$$
\sup_{h\in U}\int_B\|\mathcal Q_h(x)\|^{-\beta}\,dx<\infty,
$$
where $dx$ denotes Lebesgue measure on $\mathbb R^\ell$.
\end{proposition}

The proof uses two external inputs. Lemma 3.2 of Demailly--Koll\'ar
\cite{DemaillyKollar2001} gives lower semicontinuity and a uniform
complex negative-moment bound in a holomorphic family, while
Theorem 9.7 of Cutkosky \cite{Cutkosky2017} provides, locally near each
point, finitely many monomialization diagrams obtained by finite sequences
of local blow-ups, in which the prescribed analytic ideal pulls back to a
principal ideal generated by a coordinate monomial. Compatible complexifications of these
diagrams come from \cite[Proposition~9.6 and the proof of
Theorem~9.7]{Cutkosky2017}. The
remainder of the proof computes the resulting monomial exponents and then
applies the area formula and Fatou's lemma.

\begin{lemma}[Complex stability and parameter-weighted integrability]
\label{lem:complex-stability-slicing}
Let $S$ and $V$ be complex manifolds. Let $t_0\in S$, $K\subset V$ a
compact subset, and 
$$
F=(F_1,\ldots,F_q)
$$
be a holomorphic map on a neighborhood of $\{t_0\}\times K$.  Write
$F_t(z):=F(t,z)$.

\begin{enumerate}
\item If $c_z(F_{t_0})\ge\alpha>0$ for every $z\in K$, then, for every
$0<\eta<\alpha$, after replacing $S$ by a sufficiently small neighborhood of $t_0$, one has
$$
c_z(F_t)>\eta
\qquad(t\in S,\ z\in K).
$$

\item Suppose that $\dim_{\mathbb C}S=p$, and let
$\zeta_1,\ldots,\zeta_p$ be local coordinates centered at $t_0$.  If
$z_0\in K$ and $c_{z_0}(F_{t_0})>\eta>0$, then, for every $0<c<1$,
\begin{equation}\label{eq:coordinate-slicing-density}
\left(
 |\zeta_1(t) \cdots\zeta_p (t) |\,\|F(t,z)\|^\eta
\right)^{-2c}
\end{equation}
is locally integrable near $(t_0,z_0)$.
\end{enumerate}
\end{lemma}
\begin{proof}
Working in local holomorphic coordinates, set
$$
\varphi(t,z):=\log\|F(t,z)\|.
$$
Then $\varphi$ is a H\"older plurisubharmonic function in the sense of
\cite[Definition~2.3]{DemaillyKollar2001}. By
\cite[Proposition~1.2]{DemaillyKollar2001}, the complex singularity
exponent of $\varphi(t_0,\cdot)$ on a compact set is the infimum of the
pointwise exponents. Part~\textnormal{(1)} therefore follows from
\cite[Lemma~3.2(1)]{DemaillyKollar2001}, after covering $K$ by finitely
many coordinate neighborhoods.
For part~\textnormal{(2)}, fix $0<c<1$. Since
$c\eta<c_{z_0}(F_{t_0})$, \cite[Lemma~3.2(2)]{DemaillyKollar2001} gives
neighborhoods $S'\ni t_0$ and $V'\ni z_0$, and a constant $C>0$, such
that
$$
\sup_{t\in S'}\int_{V'}\|F(t,z)\|^{-2c\eta}\,dz\le C.
$$
Shrinking $S'$ in the coordinates $\zeta_1,\ldots,\zeta_p$ and applying
Fubini's theorem, we obtain
$$
\begin{aligned}
&\int_{S'\times V'}
\left(
|\zeta_1(t)\cdots\zeta_p(t)|\,\|F(t,z)\|^\eta
\right)^{-2c}
\,dt\,dz \le
C\prod_{j=1}^p
\int_{|\zeta_j|<\varepsilon}|\zeta_j|^{-2c}\,dA(\zeta_j)
<\infty,
\end{aligned}
$$
since $c<1$. This proves part~\textnormal{(2)}.
\end{proof}

We use the terminology of
\cite[Definition~3.2 and Section~9]{Cutkosky2017} for monomial morphisms,
coordinate divisors, and local blow-ups with smooth centers. A compact
coordinate box means a compact rectangle in analytic coordinates. By a
\emph{compatible local complexification}, we mean holomorphic extensions of
the local diagrams and coordinates for which the displayed identities remain
valid and whose restrictions to the real loci recover the original data.

\begin{lemma}[Finite monomial charts over the parameter space]
\label{lem:finite-relative-monomial-charts}
Let $T\subset\mathbb R^p$ be open, let $X$ be a compact real-analytic
$\ell$-manifold, and let $A\subset T\times X$ be a closed real-analytic
subspace with ideal sheaf $\mathcal I_A$. Let $K\subset T\times X$ be
compact, and suppose that $A$ contains no connected component of
$T\times X$ meeting $K$.

Then there are finitely many commutative diagrams
$$
\begin{array}{ccc}
M_i&\xrightarrow{\ \sigma_i\ }&T\times X\\
{\scriptstyle p_i}\downarrow&&
\downarrow{\scriptstyle\operatorname{pr}_T}\\
T_i&\xrightarrow{\ \tau_i\ }&T,
\end{array}
$$
where $M_i$ and $T_i$ are smooth real-analytic manifolds of dimensions
$p+\ell$ and $p$, respectively, together with compact coordinate boxes
$K_i\subset M_i$ such that
$$
K\subset\bigcup_i\sigma_i(K_i).
$$
Each $K_i$ lies in a coordinate neighborhood centered at a point
$m_i\in M_i$ with $\sigma_i(m_i)\in K$. On the corresponding coordinate
neighborhoods, $\sigma_i$ and $\tau_i$ are finite compositions of local
blow-ups with smooth centers.

There are analytic coordinates
$$
u=(u_1,\ldots,u_{p+\ell})
\quad\text{on }M_i,
\qquad
\zeta=(\zeta_1,\ldots,\zeta_p)
\quad\text{on }T_i,
$$
centered at $m_i$ and $p_i(m_i)$, respectively, such that
\begin{equation}\label{eq:relative-monomial-normal-form}
\zeta_j\circ p_i=u^{\alpha_j^{(i)}}
\qquad(1\le j\le p),
\qquad
\operatorname{rank}\mathsf A_i=p,
\end{equation}
where
$$
u^\nu:=\prod_{k=1}^{p+\ell}u_k^{\nu_k},
\qquad
\mathsf A_i
:=
\bigl(\alpha_{jk}^{(i)}\bigr)_{
1\le j\le p,\ 1\le k\le p+\ell}.
$$
Set
$$
D_i:=\{u_1\cdots u_{p+\ell}=0\}.
$$
Then $\sigma_i$ is an open embedding on $M_i\smallsetminus D_i$. Moreover,
\begin{equation}\label{eq:relative-ideal-and-jacobian}
\sigma_i^*\mathcal I_A=(u^{a_i}),
\qquad
\rho_i^*
\bigl(
d\zeta_1\wedge\cdots\wedge d\zeta_p\wedge\omega_X
\bigr)
=
\varepsilon_i u^{\kappa_i}
du_1\wedge\cdots\wedge du_{p+\ell},
\end{equation}
where
$
\rho_i
:=
\bigl(p_i,\operatorname{pr}_X\circ\sigma_i\bigr):
M_i\to T_i\times X,
$
$\omega_X$ is any nowhere-vanishing local analytic $\ell$-form on $X$,
$a_i,\kappa_i\in\mathbb Z_{\ge0}^{p+\ell}$, and $\varepsilon_i$ is an
analytic unit.

After further finite restrictions and relabeling, each $T_i$ may be taken
to be a relatively compact coordinate neighborhood, and there is a
nowhere-dense closed analytic subset $\mathcal E_i\subset T_i$ such that
$\tau_i$ is an open embedding on $T_i\smallsetminus\mathcal E_i$,
$p_i^{-1}(\mathcal E_i)$ is nowhere dense, and the Lebesgue measure of $\tau_i(\mathcal E_i)$ is zero.
The local diagrams and coordinates admit compatible complexifications for
which $\sigma_{i,\mathbb C}$ is an open embedding off
$
D_{i,\mathbb C}:=\{u_1\cdots u_{p+\ell}=0\}.
$
Consequently, $\rho_{i,\mathbb C}$ is injective and a local biholomorphism
off $D_{i,\mathbb C}$.
\end{lemma}

\begin{proof}
Apply \cite[Theorem~9.7]{Cutkosky2017} at each point of $K$ to the
real-analytic morphism
$$
\operatorname{pr}_T:T\times X\to T
$$
and the prescribed closed analytic subspace $A$. By compactness, finitely
many of the resulting diagrams and compact subsets suffice to cover $K$.

Combining that theorem with the local construction in
\cite[Proposition~9.6]{Cutkosky2017}, and restricting to connected
components whose images meet $K$, gives the local blow-up descriptions,
the compatible complexifications, the coordinate divisors, and the
monomial form of $p_i$. Moreover, either
$$
\sigma_i^{-1}(A)=M_i
\qquad\text{or}\qquad
\sigma_i^*\mathcal I_A=\mathcal O_{M_i}(-G_i),
$$
where $G_i$ is an effective divisor supported on $D_i$ and
 $\mathcal O_{M_i}(-G_i)$ is the ideal sheaf of analytic functions vanishing along $G_i$ with at least the multiplicities specified by $G_i$. The first alternative cannot occur. Indeed, since $\sigma_i$ is an open
embedding off $D_i$, it would imply that $A$ contains a nonempty open
subset of a connected component of $T\times X$ meeting $K$. By analytic
continuation, $A$ would then contain that entire component, contrary to the
hypothesis. Hence
$
\sigma_i^*\mathcal I_A=\mathcal O_{M_i}(-G_i).
$
Writing locally
$G_i=\sum_k a_{ik}\{u_k=0\}$ gives
$$\sigma_i^*\mathcal I_A=(u^{a_i}).$$

Theorem 9.7 of \cite{Cutkosky2017} also provides a nowhere-dense closed analytic subset
$\mathcal E_i\subset T_i$ such that $\tau_i$ is an open embedding on its
complement and $p_i^{-1}(\mathcal E_i)$ is nowhere dense. On the dense open
set
$
(M_i\smallsetminus D_i)\smallsetminus p_i^{-1}(\mathcal E_i),
$
both $\sigma_i$ and $\tau_i$ are local analytic isomorphisms. The
commutative diagram therefore gives locally
$$
p_i
=
\tau_i^{-1}\circ\operatorname{pr}_T\circ\sigma_i,
$$
so $p_i$ has rank $p$ there. In Cutkosky's monomial normal form
\cite[Definition~3.2]{Cutkosky2017}, the nonzero pullbacks of the coordinate
functions on $T_i$ have linearly independent exponent vectors. Since
$p_i$ has generic rank $p$, none of the $p$ coordinate pullbacks vanishes
identically, and hence
$$
\operatorname{rank}\mathsf A_i=p.
$$

Cover the compact subsets of the domains lying over $K$ by finitely many of
these monomial coordinate neighborhoods, choose compact coordinate boxes
centered at points mapping into $K$, and relabel. This gives the boxes
$K_i$ and the asserted finite covering of $K$.

It remains to establish the Jacobian identity and the final assertions
about complexification. By definition of $\rho_i$, the commutative diagram
gives
\begin{equation}\label{eq:rho-factorization-cutkosky}
(\tau_i\times\operatorname{id}_X)\circ\rho_i=\sigma_i.
\end{equation}
The compatible complexification supplied by
\cite[Proposition~9.6]{Cutkosky2017} makes
$\sigma_{i,\mathbb C}$ an open embedding off $D_{i,\mathbb C}$. After
complexifying \eqref{eq:rho-factorization-cutkosky}, the chain rule shows
that $D\rho_{i,\mathbb C}$ is invertible off $D_{i,\mathbb C}$.

Let $\mathcal J_i$ be defined by
$$
\rho_{i,\mathbb C}^*
\bigl(
d\zeta_1\wedge\cdots\wedge d\zeta_p\wedge\omega_X
\bigr)
=
\mathcal J_i(u)\,
du_1\wedge\cdots\wedge du_{p+\ell},
$$
where $\omega_X$ has been extended holomorphically after shrinking the
coordinate neighborhoods. The function $\mathcal J_i$ is nonzero, and its
zero set is contained in $D_{i,\mathbb C}$. Since the local ring of
convergent complex power series is a unique-factorization domain, it follows
that
$$
\mathcal J_i=\varepsilon_{i,\mathbb C}u^{\kappa_i}
$$
for some $\kappa_i\in\mathbb Z_{\ge0}^{p+\ell}$ and a holomorphic unit
$\varepsilon_{i,\mathbb C}$. Restriction to the real locus gives the second
identity in \eqref{eq:relative-ideal-and-jacobian}. If two points of
$M_{i,\mathbb C}\smallsetminus D_{i,\mathbb C}$ have the same image under
$\rho_{i,\mathbb C}$, then
\eqref{eq:rho-factorization-cutkosky} implies that they have the same image
under $\sigma_{i,\mathbb C}$. Thus they coincide. Hence
$\rho_{i,\mathbb C}$ is injective off $D_{i,\mathbb C}$, and the
invertibility of its differential makes it a local biholomorphism there.

Finally, after restricting $T_i$ to a relatively compact coordinate
neighborhood, the nowhere-dense analytic set $\mathcal E_i$ has dimension
at most $p-1$. Since $\tau_i$ is locally Lipschitz on this neighborhood, the Lebesgue measure of
$\tau_i(\mathcal E_i)$ is zero. This completes the proof.
\end{proof}

\begin{proof}[Proof of \Cref{prop:multiparameter-stability}]
Choose $\beta'$ with $\beta<\beta'<\alpha$, and choose a bounded
open set $V_0\subset\mathbb C^\ell$ such that
$\overline B\subset V_0$ and $\overline V_0\subset V$.
Applying \Cref{lem:complex-stability-slicing}\textnormal{(1)} with
$K=\overline V_0$ and shrinking a complex neighborhood
$U_{\mathbb C}\subset\mathbb C^p$ of $0$, we obtain
\begin{equation}\label{eq:nearby-complex-threshold}
c_z(\mathcal Q_h^{\mathbb C})>\beta'
\qquad
(h\in U_{\mathbb C},\ z\in V_0).
\end{equation}

Choose a connected relatively compact neighborhood
$U_0\subset U_{\mathbb C}\cap\mathbb R^p$ of $0$. Let $D$ be an integer at least as large as the
$x$-degree of every component of $\mathcal Q$. The degree-$D$
homogenizations of these components generate a well-defined analytic ideal
sheaf $\overline{\mathcal I}$ on
$U_0\times\mathbb P^\ell(\mathbb R).$
On the affine chart $x_0=1$, this is the ideal generated by the components
of $\mathcal Q$. Let $A$ be the corresponding closed analytic
subspace. Since $\mathcal Q_0^{\mathbb C}$ is not identically zero,
$A$ is proper.

Regard $\overline B\subset\mathbb R^\ell$ as a compact subset of the
affine chart $x_0=1$ in $\mathbb P^\ell(\mathbb R)$.
Choose an open neighborhood $U$ of $0$ with $\overline U\subset U_0$, and
apply \Cref{lem:finite-relative-monomial-charts} to the projection
$$
U_0\times\mathbb P^\ell(\mathbb R)\to U_0,
$$
the subspace $A$, and the compact set
$\overline U\times\overline B$. We obtain finitely many diagrams indexed by
$i$, with compact coordinate boxes $K_i\subset M_i$, such that
$$
\overline U\times\overline B
\subset
\bigcup_i\sigma_i(K_i).
$$
Fix one chart. Let $m_i$ be the center of its coordinates, and write
$$
\sigma_i(m_i)=(h_i,x_i),
\;\;
t_i:=p_i(m_i), \text{ so that $\tau_i(t_i)=h_i$.}
$$
 Let
$ u=(u_1,\ldots,u_{p+\ell})$ and $\zeta=(\zeta_1,\ldots,\zeta_p)$
be the source and target coordinates supplied by
\Cref{lem:finite-relative-monomial-charts}. Thus
\begin{equation}\label{eq:parameter-monomials-in-proof}
\zeta_j\circ p_i=u^{\alpha_j},
\qquad
\alpha_j=(\alpha_{j1},\ldots,\alpha_{j,p+\ell}),
\end{equation}
and the exponent matrix
$\mathsf A=(\alpha_{jk})_{1\le j\le p,\ 1\le k\le p+\ell}$ has rank $p$.

Pass to the compatible complexifications and shrink the coordinate
neighborhoods so that $\tau_{i,\mathbb C}$ takes values in
$U_{\mathbb C}$. For $t$ near $t_i$ and $z\in V_0$, define the
holomorphic map
$$
F_i(t,z)
:=
\mathcal Q^{\mathbb C}(\tau_{i,\mathbb C}(t),z).
$$
By \eqref{eq:nearby-complex-threshold},
$$
c_z(F_{i,t_i})>\beta'
\qquad
(z\in V_0).
$$
Since $x_i\in\overline B\subset V_0$, it follows from
\Cref{lem:complex-stability-slicing}\textnormal{(2)} that, for every
$0<\theta<1$,
\begin{equation}\label{eq:chart-slicing-density}
\left(
|\zeta_1(t) \cdots\zeta_p(t)|\,\|F_i(t,z)\|^{\beta'}
\right)^{-2\theta}
\end{equation}
is locally integrable near $(t_i,x_i)$ with respect to Lebesgue
measure in the complex coordinates $(\zeta,z)$.

Set
$$
K_{i,B}
:=
K_i\cap
\sigma_i^{-1}(\overline U\times\overline B).
$$
On the affine chart $x_0=1$, take
$\omega_X=dx_1\wedge\cdots\wedge dx_\ell$, and write
$$
|u|^\nu:=\prod_{k=1}^{p+\ell}|u_k|^{\nu_k}, \qquad (\nu\in\mathbb R^{p+\ell}).
$$
 The principalization identity in
\Cref{lem:finite-relative-monomial-charts} gives
\begin{equation}\label{eq:chart-monomial-data}
\|\mathcal Q\circ\sigma_i\|
\asymp
|u|^a
\qquad\text{on }K_{i,B}.
\end{equation}
Indeed, the components have the form
$$
\mathcal Q_r\circ\sigma_i=u^a g_r,
$$
where the analytic functions $g_1,\ldots,g_q$ have no common zero. The same
factorization holds after complexification, so $a_k$ is the vanishing order
of the ideal generated by the components of
$F_i\circ\rho_{i,\mathbb C}$ along $\{u_k=0\}$.
Define
$$
b_k:=\sum_{j=1}^p\alpha_{jk},
\qquad
e_k:=\kappa_k+1-b_k
\qquad
(1\le k\le p+\ell),
$$
and write $b=(b_k)$, $e=(e_k)$, and
$\mathbf 1=(1,\ldots,1)$. Dividing the Jacobian identity in
\Cref{lem:finite-relative-monomial-charts} by
$
\zeta_1\cdots\zeta_p=u^b
$
gives, as an identity of meromorphic top forms,
\begin{equation}\label{eq:chart-log-jacobian}
\rho_i^*\left(
\bigwedge_{j=1}^p\frac{d\zeta_j}{\zeta_j}\wedge\omega_X
\right)
=
\varepsilon u^e
\bigwedge_{k=1}^{p+\ell}\frac{du_k}{u_k}.
\end{equation}
Equivalently,
\begin{equation}\label{eq:chart-ordinary-jacobian}
|\det D\rho_i|
\asymp
|u|^\kappa
=
|u|^{b+e-\mathbf 1}.
\end{equation}

 Since
$\rho_{i,\mathbb C}$ is injective and biholomorphic off the complex
coordinate divisor, the complex change-of-variables formula may be applied
to \eqref{eq:chart-slicing-density}. Near a generic point of
$\{u_k=0\}$, its pullback, including the squared complex Jacobian, is
comparable in the variable $u_k$ to
$$
|u_k|^{2(e_k-1+(1-\theta)b_k-\theta\beta'a_k)}.
$$
Complex one-variable integrability therefore gives
$$
e_k+(1-\theta)b_k-\theta\beta'a_k>0.
$$
Letting $\theta\to1$ yields
\begin{equation}\label{eq:multiparameter-order-inequality}
e_k\ge\beta'a_k
\qquad
(1\le k\le p+\ell).
\end{equation}

Set
\begin{equation}\label{eq:delta-positive}
\delta_k:=e_k-\beta a_k.
\end{equation}
By \eqref{eq:multiparameter-order-inequality},
$$
\delta_k
\ge
(\beta'-\beta)a_k
\ge0.
$$
Let
$$
Z:=\{k:\delta_k=0\}.
$$
For $k\in Z$, one necessarily has $a_k=e_k=0$.

We claim that the columns of $\mathsf A$ indexed by $Z$ are linearly
independent. This is immediate when $Z=\emptyset$. Otherwise, in the
complexified chart, consider
$$
S_Z^\circ
:=
\{u_k=0\text{ for }k\in Z,\quad
  u_k\ne0\text{ for }k\notin Z\}.
$$
Extract from \eqref{eq:chart-log-jacobian} the coefficient of
$\bigwedge_{k\in Z}\frac{du_k}{u_k}$
and restrict it to $S_Z^\circ$. The resulting coefficient on the
right-hand side is nonzero because $e_k=0$ for $k\in Z$. On the left,
$$
\frac{d\zeta_j}{\zeta_j}
=
\sum_{k=1}^{p+\ell}
\alpha_{jk}\frac{du_k}{u_k},
$$
while $\rho_i^*\omega_X$ is holomorphic. Thus every logarithmic factor
indexed by $Z$ must come from one of the parameter forms
$d\zeta_j/\zeta_j$. It follows that some minor of $\mathsf A$ involving
the columns indexed by $Z$ is nonzero, proving the claim.

Since $\operatorname{rank}\mathsf A=p$, we may therefore choose a set
$$
J\subset\{1,\ldots,p+\ell\},
\qquad
|J|=p,
\qquad
Z\subset J,
$$
such that the $p\times p$ minor $\mathsf A_J$ is nonsingular. Away from the
target coordinate boundary $\{\zeta_1\cdots\zeta_p=0\}$,
\begin{equation}\label{eq:parameter-jacobian-normal-form}
\left|
\det
\frac{\partial(\zeta_1,\ldots,\zeta_p)}
     {\partial(u_k)_{k\in J}}
\right|
=
|\det\mathsf A_J|\,
\frac{|\zeta_1\cdots\zeta_p|}
     {\prod_{k\in J}|u_k|}
\asymp
|u|^b\prod_{k\in J}|u_k|^{-1}.
\end{equation}
Thus, after a finite subdivision into sign regions,
$(\zeta,(u_k)_{k\notin J})$ are local coordinates.

For $t\in T_i$ outside the target coordinate boundary, define
$$
\sigma_{i,t}
:=
\operatorname{pr}_X\circ
\sigma_i|_{p_i^{-1}(t)}
:
p_i^{-1}(t)\to\mathbb P^\ell(\mathbb R).
$$
The block-determinant formula gives
$$
|\det D\rho_i|
=
\left|
\det
\frac{\partial(\zeta_1,\ldots,\zeta_p)}
     {\partial(u_k)_{k\in J}}
\right|
|\det D\sigma_{i,t}|,
$$
where the last determinant is taken along the fiber of $p_i$. Combining
this identity with
\eqref{eq:chart-ordinary-jacobian},
\eqref{eq:parameter-jacobian-normal-form}, and
\eqref{eq:chart-monomial-data}, 
we obtain
$$\begin{aligned}
\|\mathcal Q\circ\sigma_i\|^{-\beta}|\det D\sigma_{i,t}|
&\ll
|u|^{-\beta a}|u|^{e-\mathbf 1}
\prod_{k\in J}|u_k|=
\prod_{k\in J}|u_k|^{\delta_k}
\prod_{k\notin J}|u_k|^{\delta_k-1}.
\end{aligned}$$
The fiberwise area formula therefore bounds the
contribution of the chart by
\begin{equation}\label{eq:multiparameter-fiber-integral}
C
\int
\prod_{k\in J}|u_k|^{\delta_k}
\prod_{k\notin J}|u_k|^{\delta_k-1}
\,du_{J^c},
\end{equation}
where $du_{J^c}:=\prod_{k\notin J}du_k$ denotes Lebesgue measure
in the real fiber coordinates $(u_k)_{k\notin J}$.
After rescaling the coordinate boxes, assume that $|u_k|\le1$. The factors
indexed by $J$ are then at most $1$. Moreover, since $Z\subset J$, one has
$\delta_k>0$ for every $k\notin J$. Consequently,
$$
\eqref{eq:multiparameter-fiber-integral}
\ll
\prod_{k\notin J}\frac{2}{\delta_k}
<\infty.
$$
The bound is uniform after the finite subdivision into sign regions and
coordinate boxes.

For each chart, define
$$
\mathcal B_i
:=
p_i(K_{i,B})\cap
\left(
\mathcal E_i\cup\{\zeta_1\cdots\zeta_p=0\}
\right).
$$
By \Cref{lem:finite-relative-monomial-charts} and the local Lipschitz
continuity of $\tau_i$ on the relevant compact target neighborhood,
$\tau_i(\mathcal B_i)$ has measure zero. Hence
$$
U_{\rm reg}
:=
U\smallsetminus\bigcup_i\tau_i(\mathcal B_i)
$$
is dense and has full measure in $U$. Fix $h\in U_{\rm reg}$. For each $i$, there is at most one
$$
t\in p_i(K_{i,B})
\quad\text{such that}\quad
\tau_i(t)=h,
$$
and any such $t$ lies outside the target coordinate boundary. Set
$$
E_{i,t,h}
:=
K_{i,B}\cap p_i^{-1}(t)
\cap\sigma_i^{-1}(\{h\}\times B).
$$
The images $\sigma_{i,t}(E_{i,t,h})$, over the finitely many relevant
pairs $(i,t)$, cover $B$. Applying the area formula on each smooth fiber
to the truncations
$$
\min\{\|\mathcal Q_h\|^{-\beta},R\},
$$
and then letting $R\to\infty$, gives
$$
\int_B\|\mathcal Q_h(x)\|^{-\beta}\,dx
\le C
\qquad
(h\in U_{\rm reg}),
$$
where $C$ is independent of $h$. Here the image of the critical set may be
discarded by Sard's theorem.

Finally, let $h\in U$ and choose $h_\nu\in U_{\rm reg}$ with
$h_\nu\to h$. Since
$$
\mathcal Q_{h_\nu}(x)\to\mathcal Q_h(x)
$$
pointwise, Fatou's lemma gives
$$
\int_B\|\mathcal Q_h(x)\|^{-\beta}\,dx
\le
\liminf_{\nu\to\infty}
\int_B\|\mathcal Q_{h_\nu}(x)\|^{-\beta}\,dx
\le C.
$$
This proves the proposition.
\end{proof}

\def\cprime{$'$}


\begin{thebibliography}{10}

\bibitem{ABW}
Kaan Akin, David~A. Buchsbaum, and Jerzy Weyman.
\newblock Schur functors and {S}chur complexes.
\newblock {\em Adv. Math.}, 44(3):207--278, 1982.

\bibitem{benoist-quint:2012}
Yves Benoist and Jean-Fran\cedilla{c}ois Quint.
\newblock Random walks on finite volume homogeneous spaces.
\newblock {\em Invent. Math.}, 187(1):37--59, 2012.

\bibitem{birch:1962}
B.~J. Birch.
\newblock Forms in many variables.
\newblock {\em Proc. Roy. Soc. London Ser. A}, 265:245--263, 1962.

\bibitem{BCR1998}
Jacek Bochnak, Michel Coste, and Marie-Fran\cedilla{c}oise Roy.
\newblock {\em Real Algebraic Geometry}, volume~36 of {\em Ergebnisse der Mathematik und ihrer Grenzgebiete (3)}.
\newblock Springer, Berlin, Heidelberg, 1998.

\bibitem{borel-harish:1962}
A.~Borel and Harish-Chandra.
\newblock Arithmetic subgroups of algebraic groups.
\newblock {\em Ann. of Math. (2)}, 75:485--535, 1962.

\bibitem{BrowningDH}
T.~D. Browning.
\newblock The Davenport--Heilbronn method: 80 years on.
\newblock {\em J. Lond. Math. Soc. (2)}, 113(1):e70371, 2026.

\bibitem{BrudnyiGanzburg1973}
Yu.~A. Brudnyi and M.~I. Ganzburg.
\newblock On an extremal problem for polynomials in $n$ variables.
\newblock {\em Izv. Akad. Nauk SSSR Ser. Mat.}, 37:344--355, 1973.

\bibitem{cassels:1971}
J. W. S. Cassels.
\newblock {\em An Introduction to the Geometry of Numbers}.
\newblock Springer-Verlag, Berlin-New York, 1971.
\newblock Second printing, corrected, Die Grundlehren der mathematischen Wissenschaften, Band 99.

\bibitem{Cutkosky2017}
Steven Dale Cutkosky.
\newblock Local monomialization of analytic maps.
\newblock {\em Adv. Math.}, 307:833--902, 2017.
\newblock DOI: 10.1016/j.aim.2016.11.023.


\bibitem{dani-margulis:1993}
S.~G. Dani and G.~A. Margulis.
\newblock Limit distributions of orbits of unipotent flows and values of quadratic forms.
\newblock In {\em I. M. Gel'fand Seminar}, volume~16 of {\em Advances in Soviet Mathematics}, pages 91--137. American Mathematical Society, Providence, RI, 1993.

\bibitem{davenport:1951}
H.~Davenport.
\newblock On a principle of Lipschitz.
\newblock {\em J. London Math. Soc.}, 26:179--183, 1951.

\bibitem{DemaillyKollar2001}
Jean-Pierre Demailly and J{\'a}nos Koll{\'a}r.
\newblock Semi-continuity of complex singularity exponents and
{K\"ahler--Einstein} metrics on {Fano} orbifolds.
\newblock {\em Ann. Sci. {\'E}cole Norm. Sup. (4)}, 34(4):525--556, 2001.
\newblock DOI: 10.1016/S0012-9593(01)01069-2.

\bibitem{dieudonne:1948}
Jean Dieudonn\'e.
\newblock Sur une g\'en\'eralisation du groupe orthogonal \`a quatre variables.
\newblock {\em Arch. Math.}, 1:282--287, 1948.

\bibitem{dynkin}
E.~B. Dynkin.
\newblock Maximal subgroups of the classical groups.
\newblock {\em Trudy Moskov. Mat. Ob\hacek{s}\hacek{c}.}, 1:39--166, 1952.

\bibitem{eskin-margulis-mozes:1998}
A.~Eskin, G.~Margulis, and S.~Mozes.
\newblock Upper bounds and asymptotics in a quantitative version of the Oppenheim conjecture.
\newblock {\em Ann. of Math. (2)}, 147(1):93--141, 1998.

\bibitem{eskin-margulis-mozes:2005}
A.~Eskin, G.~Margulis, and S.~Mozes.
\newblock Quadratic forms of signature $(2, 2)$ and eigenvalue spacings on rectangular 2-tori.
\newblock {\em Ann. of Math. (2)}, 161(2):679--725, 2005.

\bibitem{FultonYT}
William Fulton.
\newblock {\em Young tableaux}, volume~35 of {\em London Mathematical Society Student Texts}.
\newblock Cambridge University Press, Cambridge, 1997.
\newblock With applications to representation theory and geometry.

\bibitem{FultonHarris}
William Fulton and Joe Harris.
\newblock {\em Representation Theory: A First Course}, volume 129 of {\em Graduate Texts in Mathematics}.
\newblock Springer, New York, 1991.


\bibitem{katznelson:1994}
Y.~R. Katznelson.
\newblock Integral matrices of fixed rank.
\newblock {\em Proc. Amer. Math. Soc.}, 120(3):667--675, 1994.

\bibitem{Kim}
Wooyeon Kim.
\newblock Moments of Margulis functions and indefinite ternary quadratic forms.
\newblock {\em Preprint, arXiv:2403.16563v3}, 2025.

\bibitem{KimOh_two}
Wooyeon Kim and Hee Oh.
\newblock {Quantitative Oppenheim in signature $(2,2)$ via determinant values}.
\newblock {\em Preprint, arXiv:2607.18037}, 2026.


\bibitem{kleinbock:2008}
Dmitry Kleinbock.
\newblock An extension of quantitative nondivergence and applications to Diophantine exponents.
\newblock {\em Trans. Amer. Math. Soc.}, 360:6497--6523, 2008.

\bibitem{kleinbock-margulis:1998}
D.~Y. Kleinbock and G.~A. Margulis.
\newblock Flows on homogeneous spaces and {D}iophantine approximation on manifolds.
\newblock {\em Ann. of Math. (2)}, 148(1):339--360, 1998.


\bibitem{MarcusMoyls}
Marvin Marcus and B.~N. Moyls.
\newblock Linear transformations on algebras of matrices.
\newblock {\em Canad. J. Math.}, 11:61--66, 1959.

\bibitem{Margulis:1987}
G.~A. Margulis.
\newblock Formes quadratiques ind{\'e}finies et flots unipotents sur les espaces homog{\`e}nes.
\newblock {\em C. R. Acad. Sci. Paris S{\'e}r. I Math.}, 304(10):249--253, 1987.

\bibitem{Margulis:1989}
G.~A. Margulis.
\newblock Discrete subgroups and ergodic theory.
\newblock In {\em Number Theory, Trace Formulas and Discrete Groups ({O}slo, 1987)}, pages 377--398. Academic Press, Boston, MA, 1989.

\bibitem{ratner:1991}
M.~Ratner.
\newblock Raghunathan's topological conjecture and distributions of unipotent flows.
\newblock {\em Duke Math. J.}, 63(1):235--280, 1991.

\bibitem{RydinMyersonThesis}
S.~L. Rydin Myerson.
\newblock {\em Systems of forms in many variables}.
\newblock D.Phil. thesis, University of Oxford, 2016.

\bibitem{Sch68}
W. Schmidt, \emph{Asymptotic formulae for point lattices of bounded determinant and
subspaces of bounded height}, Duke Math. J. 35(1968), 327–339.

\bibitem{shah}
Nimish~A. Shah.
\newblock Limit distributions of expanding translates of certain orbits on homogeneous spaces.
\newblock {\em Proc. Indian Acad. Sci. Math. Sci.}, 106(2):105--125, 1996.

\bibitem{WaterhouseTwistedDeterminant}
William~C. Waterhouse.
\newblock Twisted forms of the determinant.
\newblock {\em J. Algebra}, 86(1):60--75, 1984.

\bibitem{zimmer1984ergodic}
Robert~J. Zimmer.
\newblock {\em Ergodic Theory and Semisimple Groups}, volume~81 of {\em Monographs in Mathematics}.
\newblock Birkh\"auser, Boston, 1984.



\end{thebibliography}
\end{document}